\documentclass[11pt, letterpaper, reqno]{amsart}
\usepackage[utf8]{inputenc} 	
\usepackage{microtype} 			
\usepackage{geometry} 			
\usepackage{amsmath}			
\usepackage{amsthm}		 		
\usepackage{amssymb}	 		
\usepackage{bm}					
\usepackage{mathrsfs}			
\usepackage{xcolor}				
\definecolor{darkred}{rgb}{0.6, 0.1, 0.1}
\definecolor{darkblue}{rgb}{0.2, 0.2, 0.6}
\definecolor{darkgreen}{rgb}{0.2, 0.4,.1}
\definecolor{mellowyellow}{rgb}{1,.8,.2}
\definecolor{bettercyan}{rgb}{0.1, 0.4, 0.7}
\usepackage{tikz}				
\usepackage{tikz-3dplot}
\usepackage{pgfplots}
\usepackage[all]{xy}
\pgfplotsset{compat=1.13}
\usetikzlibrary{decorations.markings,decorations.pathmorphing,shapes}
\usepackage{subcaption}
\usepackage{floatrow}
\usepackage{enumitem} 			
\usepackage{array}
\usepackage{comment}
\usepackage{lmodern}			
\usepackage[T1]{fontenc}		
\usepackage[colorlinks=true, citecolor=bettercyan, linkcolor=darkred, urlcolor=magenta]{hyperref}
\usepackage{cleveref}
\usepackage[backend=bibtex, style=alphabetic, date=year, backref=true, firstinits=true, isbn=false]{biblatex}
\makeatletter
\renewbibmacro{in:}{}
\DeclareFieldFormat{pages}{#1}
\renewcommand*{\bibnamedash}{%
	\leavevmode\raise +0.6ex\hbox to 5.5ex{\hrulefill}.\space\space}

\InitializeBibliographyStyle{\global\undef\bbx@lasthash}

\newbibmacro*{bbx:savehash}{\savefield{fullhash}{\bbx@lasthash}}

\renewbibmacro*{author}{%
	\ifboolexpr{
		test \ifuseauthor
		and
		not test {\ifnameundef{author}}
	}
	{%
		\iffieldequals{fullhash}{\bbx@lasthash}
		{\bibnamedash\addcomma\space}
		{\printnames{author}}%
		\usebibmacro{bbx:savehash}%
		\iffieldundef{authortype}
		{}
		{%
			\setunit{\addcomma\space}%
			\usebibmacro{authorstrg}%
		}%
	}
	{\global\undef\bbx@lasthash}%
}
\makeatother
\numberwithin{equation}{section}

\newtheorem{propositionx}{Proposition}[section]
\newenvironment{proposition}
{\pushQED{\qed}\propositionx}
{\popQED\endpropositionx}
\newenvironment{propositionp}
{\pushQED{\qed}\propositionx}
{\popQED\endpropositionx}
\newtheorem{theoremx}[propositionx]{Theorem}
\newenvironment{theorem}
{\pushQED{\qed}\theoremx}
{\popQED\endtheoremx}

\newtheorem{corollaryx}{Corollary}[propositionx]
\newenvironment{corollary}
{\pushQED{\qed}\corollaryx}
{\popQED\endcorollaryx}
\newtheorem{lemmax}[propositionx]{Lemma}
\newenvironment{lemma}
{\pushQED{\qed}\lemmax}
{\popQED\endlemmax}

\theoremstyle{definition}

\theoremstyle{remark}
\newtheorem{remarkx}[propositionx]{Remark}
\newenvironment{remark}
{\pushQED{\qed}\remarkx}
{\popQED\endremarkx}

\makeatletter
\@namedef{subjclassname@2020}{\textup{2020} Mathematics Subject Classification}
\makeatother

\newcommand{\dd}{\,\mathrm{d}}

\newcommand{\cl}{\mathrm{cl}}

\newcommand{\calctwo}{\natural\mathrm{2res}} 
\newcommand{\calc}{\natural\mathrm{res}} 
\newcommand{\calczero}{\natural} 
\newcommand{\calcshort}{\natural} 

\newcommand{\la}{\ensuremath{\langle}}
\newcommand{\ra}{\ensuremath{\rangle}}
\newcommand{\bbB}{\mathbb{B}}
\newcommand{\bbC}{\mathbb{C}}

\newcommand{\bbM}{\mathbb{M}}
\newcommand{\bbN}{\mathbb{N}}

\newcommand{\bbR}{\mathbb{R}}
\newcommand{\bbS}{\mathbb{S}}

\newcommand{\calE}{\mathcal{E}}
\newcommand{\calF}{\mathcal{F}}

\newcommand{\calO}{\mathcal{O}}

\newcommand{\calR}{\mathcal{R}}
\newcommand{\calS}{\mathcal{S}}

\newcommand{\calU}{\mathcal{U}}
\newcommand{\calV}{\mathcal{V}}

\newcommand{\calX}{\mathcal{X}}

\newcommand{\scrV}{\mathscr{V}}

\newcommand{\frakv}{\mathfrak{v}}

\newcommand{\bfw}{\mathbf{w}}
\newcommand{\bfx}{\mathbf{x}}

\newcommand{\blank}{\llcorner\hspace{-1.55pt}\lrcorner}

\newcommand\WF{\operatorname{WF}}

\newcommand\ang[1]{\langle #1 \rangle}
\newcommand\aang[1]{\langle \! \langle #1 \rangle \! \rangle}
\newcommand\taun{\tau_{\natural}}
\newcommand\xin{\xi_{\natural}}
\newcommand\zetan{\zeta_{\natural}}
\newcommand\Psiparres{\Psi_{\mathrm{par,I,res}}}
\newcommand\Id{\operatorname{Id}}
\newcommand\rhopf{\rho_{\mathrm{pf}}}
\newcommand\rhonf{\rho_{\natural\mathrm{f}}}
\newcommand\rhodf{\rho_{\mathrm{df}}}
\newcommand\rhobf{\rho_{\mathrm{bf}}}
\newcommand\sgn{\operatorname{sgn}}

\title[Microlocal analysis of the non-relativistic limit of Klein--Gordon]{Microlocal analysis of the non-relativistic limit of the Klein--Gordon equation: Estimates}
\author{Andrew Hassell} 
\address{Department of Mathematics, Australian National University, Canberra ACT 0200, AUSTRALIA}
\email{Andrew.Hassell@anu.edu.au}
\author{Qiuye Jia}
\address{Mathematical Sciences Institute, Australian National University, Canberra ACT 2601, AUSTRALIA}
\email{Qiuye.Jia@anu.edu.au}
\author{Ethan Sussman}
\address{Department of Mathematics, Northwestern University, Evanston, IL, USA}
\email{ethan.sussman@northwestern.edu}
\author{Andr\'as Vasy}
\address{Department of Mathematics, Stanford University, California, USA}
\email{andras@math.stanford.edu}

\date{September 10, 2025.}
\subjclass[2020]{Primary 35L05, 35L15. Secondary 35B25, 35Q40, 58J47, 58J50.}

\begin{document}

\begin{abstract}
	This is the more technical half of a two-part work in which we introduce a robust microlocal framework for analyzing the non-relativistic limit of relativistic wave equations with time-dependent coefficients, focusing on the Klein--Gordon equation. 
	Two asymptotic regimes in phase space are relevant to the non-relativistic limit: one corresponding to what physicists call ``natural'' units, in which the PDE is approximable by the \textit{free} Klein--Gordon equation, and a low-frequency regime in which the equation is approximable by the usual Schr\"odinger equation. Combining the analyses in the two regimes gives global estimates which are uniform as the speed of light goes to infinity. The companion paper gives applications.
	Our main technical tools are three new pseudodifferential calculi, $\Psi_{\calczero}$ (a variant of the semiclassical scattering calculus), $\Psi_{\calc}$, and $\Psi_{\calctwo}$, the latter two of which are created by ``second microlocalizing'' the first at certain locations. This paper and the companion paper can be read in either order, since the latter treats the former as a black box.
\end{abstract}

\maketitle

\tableofcontents

\section{Introduction}
\label{sec:true_intro}
This is one half of a two-part work on the microlocal analysis of the non-relativistic limit ($c\to\infty$) of the Klein--Gordon equation 
\begin{equation}
	Pu=f, \quad P\approx P_0=\square-c^2,\quad \square = \textstyle \sum_{j=1}^d \partial_{x_j}^2 - c^{-2}\partial_t^2,\quad c\gg 1
\end{equation}
in asymptotically Minkowski settings; solutions of $Pu=f$ describe massive waves sourced by $f$, and the $c^2$ term in $P$ is the Einsteinian rest energy of such a wave. (We work on $\bbR^{1,d}$, $d\geq 1$, with metrics, background fields, etc.\ that decay at spacetime infinity. This is what we mean when we write $P\approx P_0$.) Compared to the other, more elementary, half \cite{NRL_II}, which focuses on applications, the focus of this half is on the microlocal details.
Correspondingly, we omit from this introduction a broad discussion of the non-relativistic limit and relevant literature. This can be found in \cite[\S1]{NRL_II}. Instead, the purpose of this introduction is to sketch the microlocal framework in which we work. 
(Unlike in the companion paper, we assume here that the reader is familiar with the basics of microlocal analysis. See \cite[\S3]{NRL_II} for an introduction.)

The central features of our microlocal framework are:
\begin{itemize}
	\item three new (interrelated) pseudodifferential calculi $\Psi_{\calczero},\Psi_{\calc},\Psi_{\calctwo}$, the \emph{natural}- calculus $\Psi_{\calczero}$ and its ``second- microlocalizations'',\footnote{The symbol `$\natural$' is the musical notation for ``natural.''}
	\item Fredholm setups with estimates which are uniform in the non-relativistic limit.
\end{itemize}
It can be thought of as interpolating between two existing microlocal frameworks, the treatment of the Klein--Gordon equation for fixed speed of light -- due to Vasy \cite{VasyGrenoble}, using the sc-calculus $\Psi_{\mathrm{sc}}$ of Parenti--Shubin--Melrose \cite{MelroseSC} and the corresponding compactification of phase space,  
\begin{equation}
	{}^{\mathrm{sc}}\overline{T}^* \bbM \hookleftarrow T^* \bbR^{1,d} 
\end{equation}
-- and the more recent treatment of the non-relativistic Schr\"odinger equation due to Gell-Redman--Gomes--Hassell \cite{Parabolicsc, gell2023scattering}, using the sc-analogue $\Psi_{\mathrm{par}}$ of Lascar's anisotropic calculus \cite{Lascar} and the corresponding compactification of phase space,
\begin{equation}
	{}^{\mathrm{par}}\overline{T}^* \bbM\hookleftarrow T^* \bbR^{1,d}.
\end{equation}
In each case, $\blank\in \{\mathrm{sc},\mathrm{par}\}$, the main task is the analysis in the $\blank$-characteristic set of the operator, the portion of the closure of the vanishing set of its symbol contained in the boundary of the $\blank$-phase space, outside of which elliptic regularity trivializes the problem. Here, propagation of $\blank$-Sobolev regularity holds along integral curves of the (suitably rescaled) Hamiltonian flow, ${}^{\blank}\mathsf{H}_p$, 
which has a smooth extension to the boundary faces of the $\blank$-phase space. 
A key role is played by the sources/sinks of the flow within the characteristic set, called
\emph{radial sets}. Each component of the characteristic set has two radial sets, one source and one sink, one corresponding to ``incoming'' waves and another to ``outgoing'' waves.

The possibility of interpolating between the existing analysis of the Klein--Gordon and Schr\"odinger equations is strongly suggested by the qualitative similarities between their Hamiltonian flows ${}^{\blank}\mathsf{H}_p$; see \Cref{fig:individual_flows}, which compares 
\begin{enumerate}[label=(\alph*)]
	\item the sc-Hamiltonian flow in one connected component of the sc-characteristic set of the Klein--Gordon equation to
	\item  the par-Hamiltonian flow in the par-characteristic set of the Schr\"odinger equation. 
\end{enumerate}
The latter looks like the former with the light cone widened, which means that the speed of light $c>0$ has been taken $c\to\infty$. 
Each flow is source-to-sink, with the sink over the future hemisphere of the boundary of spacetime infinity and the source over the past hemisphere. This source-to-sink flow enables the operator to be represented as a Fredholm operator between certain anisotropic Sobolev spaces, from which the well-posedness of the fundamental inhomogeneous problems (the advanced/retarded problems for the Schr\"odinger equation, and the advanced/retarded and Feynman/anti-Feynman problems for the Klein--Gordon equation) follows, with corresponding anisotropic Sobolev estimates on the solutions.

The similarity of (a) and (b) belies an important difference: the characteristic set of the Klein--Gordon equation, being a second-order hyperbolic equation, has two connected components, whereas the characteristic set of the Schr\"odinger equation, being parabolic, has only a single connected component. 
In \Cref{fig:individual_flows}, we are only showing one half of the Klein--Gordon characteristic set.
So, very roughly, one expects the analysis of the Klein--Gordon equation to degenerate to \emph{two} copies of the analysis of the Schr\"odinger equation. (This has to do with the fact that the non-relativistic Schr\"odinger equation has a ``branch,'' namely the choice of sign of the $i\partial_t$ term in the PDE, whereas the Klein--Gordon equation has no such branch.) 
More precisely, the Fredholm estimates for the Klein--Gordon equation should degenerate to a Fredholm estimate for the Schr\"odinger equation on each component of the characteristic set.
This expectation is correct, but some work will be required to make it precise.

\begin{figure}[t]
	\begin{center} 
		\includegraphics[scale=.95]{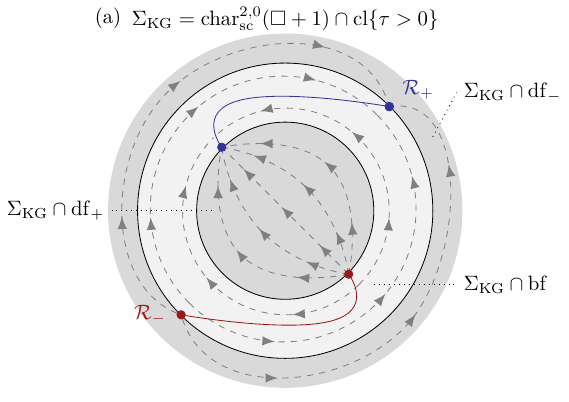}
		\hspace{-5.6em}
		\includegraphics[scale=.95]{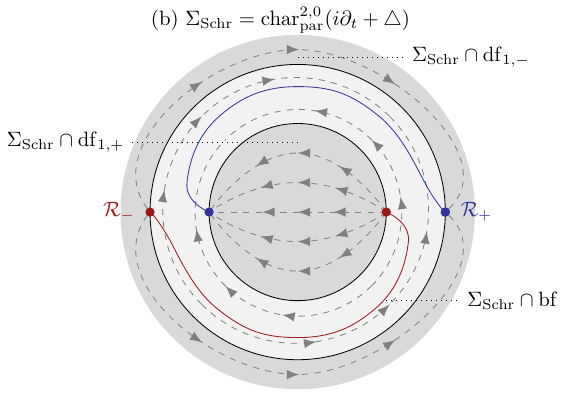}
	\end{center}
	\caption{The phase-space dynamics governing the non-relativistic limit, in the $d=1$ case. (a) The Hamiltonian flow within the positive-energy, $\{\tau>0\}$, sheet of the characteristic set $\Sigma_{\mathrm{KG}}\subset {}^{\mathrm{sc}}\overline{T}^* \bbM$ for the Klein--Gordon equation with $c=1$. The central disk (dark gray) represents one component of the portion of the characteristic set at fiber infinity (which is disconnected if $d=1$), the ``left-moving'' component $\mathrm{df}_+$. The right-moving component $\mathrm{df}_-$ is mostly hidden from view. The lighter gray annulus depicts the portion of the characteristic set over spatial infinity, $\mathrm{bf}$. (b) The Hamiltonian flow in ${}^{\mathrm{par}}\overline{T}^* \bbM$ (in which fiber infinity is $\mathrm{df}_1$ from \S\ref{subsec:phase}) for the Schr\"odinger equation, with similar notational conventions. \textit{Notice the qualitative similarity} between the two diagrams. Roughly, (a) becomes (b) upon twisting the inner disk by $45^\circ$ counter-clockwise and twisting the outermost annulus $45^\circ$ clockwise, which corresponds to the widening of the light cone as $c\to\infty$.}
	\label{fig:individual_flows}
\end{figure}

Precisification will involve a certain compactification 
\begin{equation} 
	{}^{\calctwo}\overline{T}^* \bbM \hookleftarrow T^* \bbR^{1,d}\times (0,\infty)_c
	\label{eq:misc_0}
\end{equation} 
of phase space with five boundary hypersurfaces:
\begin{itemize}
	\item fiber infinity, df\footnote{Standing for `differential face'.}, i.e.\ the boundaries of the fibers of an appropriately compactified cotangent bundle, 
	\item base (i.e.\ spacetime) infinity, bf, 
	\item  the ``natural face'' $\natural \mathrm{f} \subset \{c=\infty\}$, whose significance is explained below, and at which the Klein--Gordon operator can be approximated by the \emph{free} Klein--Gordon operator, 
	\item two ``parabolic faces'' $\mathrm{pf}_\pm$, at which the Klein--Gordon operator can be approximated by one of the two branches of the non-relativistic Schr\"odinger operator. An important fact about $\mathrm{pf}_\pm$ is that each is canonically identifiable with the par-phase space ${}^{\mathrm{par}}\overline{T}^* \bbM$.
\end{itemize}
See \Cref{fig:phase2}. That the Klein--Gordon operator $P$ can be approximated by the free Klein--Gordon operator at $\natural\mathrm{f}$ and the non-relativistic Schr\"odinger operator at $\mathrm{pf}_\pm$ is a consequence of the principal symbol/normal operator of $P$ at each face. This is just a calculation. It appears below.

Quantizing symbols on the compactification \cref{eq:misc_0} yields the pseudodifferential calculus that we are calling $\Psi_{\calctwo}$,  
\begin{equation}
	\Psi_{\calctwo} = \bigcup_{m,\ell,q_-,q_+\in \bbR}\bigcup_{\mathsf{s}}  \Psi_{\calctwo}^{m,\mathsf{s},\ell;q_-,q_+}.
\end{equation}
This has 5 orders, associated to the 5 boundary hypersurfaces of the $\calctwo$-phase space.
Here, $m,\ell,q_\pm\in \bbR$ are the orders at df, $\natural\mathrm{f}$, and $\mathrm{pf}_\pm$, respectively, 
and 
\begin{equation} 
	\mathsf{s} \in C^\infty({}^{\calctwo}\overline{T}^* \bbM)
\end{equation} 
is a variable order at bf, telling us how much decay we have in different directions/ at different frequencies.

Associated to this pseudodifferential calculus is a hierarchy of Sobolev spaces
\begin{equation}
	H_{\calctwo}^{m,\mathsf{s},\ell;q_-,q_+} = \{H_{\calctwo}^{m,\mathsf{s},\ell;q_-,q_+} (c) \}_{c>0} = \{ ( H_{\mathrm{sc}}^{m,\mathsf{s}(c)}(\bbR^{1,d}), \lVert - \rVert_{H_{\calctwo}^{m,\mathsf{s},\ell;q_-,q_+}(c) } ) \}_{c>0} ,
\end{equation}
(The precise definition of these spaces appears below.) 
Thus, the
$\calctwo$-Sobolev norms are certain norms on the usual weighted $L^2$-based Sobolev spaces 
\begin{align} 
	\begin{split} 
	H_{\mathrm{sc}}^{m,s}=H_{\mathrm{sc}}^{m,s}(\bbR^{1,d}) &= (1+t^2+\lVert x \rVert^2)^{-s/2} H^m(\bbR^{1,d}) \\&= (1+t^2+\lVert x \rVert^2)^{-s/2}  (1-\partial_t^2 +\triangle)^{-m/2} L^2(\bbR^{1+d})
	\end{split} 
\end{align} 
on spacetime. (Here we took $\mathsf{s}$ constant, for simplicity. When $\mathsf{s}$ is variable, then $H_{\mathrm{sc}}^{m,\mathsf{s}}$ is what is called an \emph{anisotropic} Sobolev space.) At the level of sets,
\begin{equation}
	H_{\calctwo}^{m,\mathsf{s},\ell;q_-,q_+}(c)  = H_{\mathrm{sc}}^{m,\mathsf{s}(c)} \subset \calS'(\bbR^{1+d}),
\end{equation}
but the key point is that we have a $c$-dependent norm which captures the regularity and decay as measured using $\Psi_{\calctwo}$. 
The meaning of the subscript ``$\calctwo$'' will be explained in due course. 

The key properties of the $\calctwo$-norms are:
\begin{itemize}
	\item for each individual $c>0$, the norm $\lVert - \rVert_{H_{\calctwo}^{m,\mathsf{s},\ell;q_-,q_+}(c)}$ is equivalent to $\lVert - \rVert_{H_{\mathrm{sc}}^{m,\mathsf{s}(c)}}$, 
	\item as $c\to\infty$, the norms degenerate, in a certain sense, to two copies of $\mathrm{par}$-Sobolev norms: 
	\begin{equation}
		\lim_{c\to\infty} H_{\calctwo}^{\infty,\mathsf{s},\ell;0,0}(\bbR^{1,d}) = e^{-ic^2 t} H_{\mathrm{par}}^{\ell,\mathsf{s}_-}(\bbR^{1,d})+ e^{ic^2 t} H_{\mathrm{par}}^{\ell,\mathsf{s}_+}(\bbR^{1,d})
		\label{eq:first_rough_splitting}
	\end{equation}
	(roughly) for some variable orders $\mathsf{s}_\pm$ gotten by restricting $\mathsf{s}$ to $\mathrm{pf}_\pm$. (See \S\ref{subsec:calc_relations} for precise statements.)
\end{itemize}
We will often leave the $c$-dependence implicit in the notation. 

As is standard in microlocal analysis, the main consideration is the microlocal regularity of solutions; here, this takes place in the $\calctwo$-phase space, \cref{eq:misc_0}.

Our main result is the uniform invertibility, in the nonrelativistic limit $c \to \infty$, of the Klein--Gordon operator $P$ between suitable function spaces based on the $\calctwo$-Sobolev spaces just introduced. To state this we introduce the spaces 
\begin{equation}\begin{aligned}
		\mathcal{X}^{m,\mathsf{s},\ell} &= \big\{ u \in H_{\calctwo}^{m,\mathsf{s},\ell;0,0} : Pu \in H_{\calctwo}^{m-1,\mathsf{s}+1,\ell-1;0,0} \big\}, \\
		\mathcal{Y}^{m,\mathsf{s},\ell} &= H_{\calctwo}^{m,\mathsf{s},\ell;0,0}.
\end{aligned}
\end{equation}
We prove the following theorem:
\begin{theorem}\label{thm:inhomog}
	Let $P$ be as in \S\ref{sec:setup} (see \cref{eq:P_def}); in other words, $P$ is the Klein--Gordon operator (possibly with additional zeroth and first-order terms) associated to some asymptotically Minkowski metric $g$. 
	Let 
	\begin{equation} 
		\mathsf{s} \in C^\infty({}^{\calctwo}\overline{T}^* \bbM)
	\end{equation} 
	denote a variable order 
	satisfying the following conditions: 
	\begin{itemize}
		\item $\mathsf{s}$ is monotonic under the $\calctwo$-Hamiltonian vector field $\mathsf{H}_p$ near each component of the characteristic set, and 
		\item (threshold condition) on each component of the characteristic set, $\mathsf{s}>-1/2$ on one radial set and $\mathsf{s}<-1/2$ on the other.
	\end{itemize}
	(It is allowed that $\mathsf{s}$ be decreasing on one component of the characteristic set and increasing on the other.)	
	 We impose two more technical conditions: 
	\begin{itemize}
		\item $\mathsf{s}$ is constant near each radial set,
		\item $|\mathsf{H}_p \mathsf{s} |^{1/2}$ is smooth near the characteristic set.\footnote{These two technical conditions can easily be arranged. For any $\delta>0$, there exists a variable order $\mathsf{s}$ satisfying our criteria and staying within a $\delta$ neighborhood of $-1/2$.}
	\end{itemize}
	Then, $P$ is an invertible operator 
	\begin{equation}
		P : \mathcal{X}^{m,\mathsf{s},\ell} \to \mathcal{Y}^{m-1,\mathsf{s}+1,\ell-1}
	\end{equation}
	for $c$ large enough.
	Moreover, the invertibility of $P$ is uniform in the sense that the norm of the inverse is uniformly bounded as $c \to \infty$; equivalently, 
	there exist $c_0,C>0$ such that the estimate
	\begin{equation}\label{eq:uniform inv}
		\lVert   u \rVert_{H_{\calctwo}^{m,\mathsf{s},\ell;0,0}} \leq C \lVert  Pu \rVert_{H_{\calctwo}^{m-1,\mathsf{s}+1,\ell-1;0,0}}
	\end{equation} 
	holds for any $c>c_0$ and 
	$u \in H_{\mathrm{sc}}^{m,\mathsf{s}(c)}$ such that $Pu \in H_{\mathrm{sc}}^{m-1,\mathsf{s}(c)+1}$.
\end{theorem}

\begin{figure}
	\includegraphics{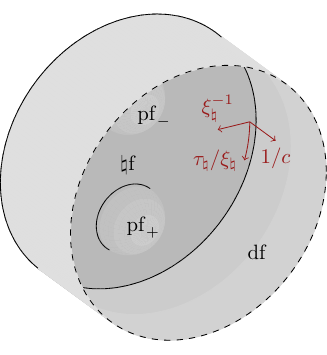}\qquad 
	\caption{The $\calctwo$-frequency space. The $\calctwo$-phase space, ${}^{\calctwo}\overline{T}^* \bbM$, is the frequency space times the radially compactified base $\bbM=\bbR^{1,d}\cup \infty \bbS^d$.}
	\label{fig:phase2}
\end{figure}

The invertibility, for each fixed $c > c_0$, of $P$ is standard. The novel aspect of \Cref{thm:inhomog} is the choice of norms such that the inverse is uniformly bounded.  

See \S\ref{sec:flow} for detailed descriptions of the characteristic sets and radial sets mentioned in this theorem, as well as the vector field $\mathsf{H}_p$.

\begin{remark} If $f \in(1+t^2+\lVert x \rVert^2)^{-1/4-\epsilon} L^2(\bbR^{1,d})$, say, uniformly in $c$ (for example, it could be independent of $c$), then, given an order $\mathsf{s} \geq -1/2 - \epsilon$ satisfying the conditions of the theorem, we can take $\ell = 1$, $m = 1$ and apply \eqref{eq:uniform inv} to find that 
	\begin{equation}
		P^{-1} f \in H_{\calctwo}^{1,\mathsf{s},1;0,0}.
	\end{equation}	
	In particular, it has order $0$ at the two parabolic faces $\mathrm{pf}_\pm$ at $c = \infty$,  which is where the Klein--Gordon operator behaves like the Schr\"odinger operator to leading order, but is order $1$ --- that is, it is $O(c^{-1})$ microlocally --- at the natural face $\natural \mathrm{f}$. This confirms the heuristic that it is the Schr\"odinger equation that governs the leading asymptotic as $c \to \infty$, at least for such $f$. This is developed in much greater detail in \cite{NRL_II}.
\end{remark}
\begin{remark}
	The two technical conditions imposed on $\mathsf{s}$ in the theorem can likely be removed. We impose them 
	in order to easily justify throwing out terms in our propagation/radial point estimates arising from derivatives falling on $\mathsf{s}$. However, these terms have the ``right'' semidefinite sign, even if we do not impose the additional technical conditions. It should therefore be possible to justify throwing them out via an appeal to the sharp G{\aa}rding inequality, adapted to our calculi.  
	The reason we have not carried this out is that we view the variable order as a technical tool rather than something of intrinsic interest.
\end{remark}

We thus obtain a solution operator $P^{-1} : \mathcal{Y}^{m-1,\mathsf{s}+1,\ell-1} \to \mathcal{X}^{m,\mathsf{s},\ell}$, for each choice of exponents $(m, \mathsf{s}, \ell)$ satisfying the conditions above. As one would expect, changing $m$, $\ell$, and changing $\mathsf{s}$ to another variable order \emph{satisfying the same monotonicity and threshold properties as the original} leads to ``the same'' inverse operator acting on a slightly different space (i.e. the two inverses agree on their common domain). However, changing $\mathsf{s}$ to a different variable order with a different threshold property (that is, below threshold instead of above threshold on a component of the radial set) leads to a genuinely different inverse. This leads to the four well-known distinct inverses to the Klein--Gordon operator mentioned above:  the forward solution, for which $\mathsf{s}$ is non-increasing with respect to time, the backward solution, for which $\mathsf{s}$ is non-decreasing with respect to time, the Feynman solution, for which $\mathsf{s}$ is non-increasing with respect to the Hamiltonian flow (forwards in time on one sheet of the characteristic set and backwards on the other), and the anti-Feynman solution, for which $\mathsf{s}$ is non-decreasing with respect to the Hamiltonian flow.
  
\begin{remark}
	\label{rem:history}
	This strategy of representing an operator whose Hamiltonian flow has a simple source-to-sink structure as a Fredholm operator between anisotropic Sobolev spaces is by now a standard technique in microlocal analysis. 
	The first example was due to Faure--Sj\"ostrand \cite{faure-sjostrand2011anosov} in the context of Anosov flows. In \cite{vasy2013asympthyp}, the fourth-named author introduced a more general framework in order to study wave equations. This was rapidly developed with coauthors and extended to a variety of settings. For an incomplete list: 
	\begin{itemize}
		\item for the Helmholtz equation on asymptotically Euclidean spaces at fixed positive energy, \cite{VasyLA}, 
		\item for the Helmholtz equation on asymptotically hyperbolic spaces and de Sitter and Kerr-de Sitter we have, besides \cite{vasy2013asympthyp}, the works \cite{ hintz-vasy2015semilinear}, 
		\item for the Helmholtz equation at low-energy, \cite{VasyN0, VasyN0L, SussmanACL}, 
		\item regarding the Schr\"odinger equation, the works of Gell-Redman--Gomes--Hassell \cite{Parabolicsc, gell2023scattering} have already been mentioned,
		\item for the Klein--Gordon equation, we have, in addition to the already mentioned \cite{VasyGrenoble}, we have \cite{BDGR, M-V-Feynman}
		\item for the massless wave equation, \cite{baskin-vasy-wunsch2015asymptrad}
		\item for Anosov flows, besides \cite{faure-sjostrand2011anosov},  we have  \cite{dyatlov-zworski2016dynamical}.
	\end{itemize} 
\end{remark}

\subsection{The laboratory vs.\ natural scales}

Part of our goal is to understand the scale-dependence of the non-relativistic limit: 
\begin{itemize}
	\item on the \emph{laboratory} scale, meaning length scales which are $\Omega(1)$ as $c\to\infty$, the non-relativistic limit is governed by the two (non-relativistic) Schr\"odinger equations, 
	\item  on the \emph{natural} scale, meaning length scales which are $O(1/c)$ as $c\to\infty$, the non-relativistic limit is governed by the \emph{free} Klein--Gordon equation.
\end{itemize}
In frequency-space terminology, the natural scale corresponds to frequencies which are large in the $c\to\infty$ limit. Specifically, letting $\xi_1,\dots,\xi_d$ denote the frequency coordinates dual to $x_1,\dots,x_d$:
\begin{enumerate}[label=(\Roman*)]
	\item the laboratory scale corresponds to $O(1)$ frequencies, 
	\item the natural scale corresponds to $\lVert \xi \rVert = \Omega(c)$.
\end{enumerate}
In other words, the natural scale corresponds to non-small values of the natural frequency $\xi_\natural = \xi/c$, which is dual to the natural coordinate $x_\natural = cx$, which is measured in what physicists call \emph{natural units}, hence our choice in terminology.

One of the upshots of this manuscript is that these are the \emph{only} frequency regimes that need to be distinguished to understand the $c\to\infty$ limit, at least as far as basic solvability theory is concerned. (We leave open the possibility that the consideration of an additional very-large frequency regime is useful for understanding scattering data in the non-relativistic limit. The usefulness for individual $c>0$ has been demonstrated in \cite{desc}, cf.\ \cite{DangVasyWrochna2024dirac}.) Indeed:
\begin{itemize}
	\item intermediate frequencies, e.g.\ $e^{i c^{1/2} x_1}$, live at the corners $\mathrm{pf}_\pm \cap \natural\mathrm{f}$ and are controlled by our estimates there, 
	\item very large frequencies, e.g.\ $e^{ic^2 x_1}$, live at the corner $\natural\mathrm{f}\cap\mathrm{df}$ and are controlled by our estimates there.
\end{itemize}

In the descriptions above, we avoided mentioning the temporal frequency $\tau$, dual to $t$. (That is, $\tau$ is the energy.) We consider 
\begin{enumerate}[label=(\Roman*)]
	\item the laboratory scale to be associated with $O(1)$ temporal frequencies and 
	\item the natural scale to be associated with $\Omega(c^2)$ temporal frequencies.
\end{enumerate}
That is, the natural scale is associated with finite, nonzero values of $\tau_\natural = \tau/c^2$. 
However, the two \emph{particular} natural-scale oscillations 
\begin{equation}
	e^{\pm ic^2 t}=e^{\pm it_\natural}
	\label{eq:usual_oscillations}
\end{equation} 
are ubiquitous in the analysis of the non-relativistic limit; see e.g.\ \cref{eq:first_rough_splitting}. These oscillations appear prominently in solutions of the Cauchy problem with $c$-independent initial data --- see \cite{NRL_II} and the wider literature cited therein --- multiplied by functions varying on the spatiotemporal laboratory scale.
So, the temporal frequencies corresponding to the spatial laboratory scale are not $\tau=O(1)$, but rather 
\begin{equation} 
	\tau= \pm c^2 + O(1),
\end{equation} 
i.e.\ laboratory-scale perturbations of the two special temporal frequencies $\tau=\pm c^2$.

\begin{remark}
	From the physicists' point-of-view, the phase factors $e^{\pm ic^2 t}$ arise because of a difference of convention between non-relativistic and relativistic quantum mechanics: in relativistic quantum mechanics, the Einsteinian ``rest energy'' $+c^2$ is counted as part of the particle's energy, whereas in non-relativistic quantum mechanics it is not. 
\end{remark}

We can now explain the correspondence between the faces $\natural\mathrm{f}$, $\mathrm{pf}_\pm$ of the $\calctwo$-phase space and the scales discussed above: the faces $\mathrm{pf}_\pm$ of the $\calctwo$-phase space correspond to laboratory scale perturbations of the frequencies 
\begin{equation}
	(\tau,\xi) = (\pm c^2,0)
	\label{eq:misc_b}
\end{equation}
associated with the oscillations \cref{eq:usual_oscillations}, and $\natural\mathrm{f}$ captures the rest of the natural scale, hence the label `$\natural\mathrm{f}$'.

\subsection{More about the natural scale}
It is straightforward to see why the natural scale is important for the analysis of the PDE. 
Because $g$ is asymptotically Minkowski, the d'Alembertian $\square_g$ differs from the Minkowski d'Alembertian 
\begin{equation}
	\square = \sum_{j=1}^d \frac{\partial^2}{\partial x_j^2} - \frac{1}{c^2} \frac{\partial^2}{\partial t^2} = \triangle - \frac{1}{c^2}  \frac{\partial^2}{\partial t^2} 
\end{equation}
by decaying terms.
Consequently, the Klein--Gordon operator $P$ differs from the free Klein--Gordon operator 
\begin{equation}
	P_0 =  \square - c^2
	\label{eq:misc_a}
\end{equation}
by decaying terms, which we can ignore for the sake of this heuristic. The essence of the non-relativistic limit lies in the competition
between the Einsteinian $c^2$ term in \cref{eq:misc_a}, which is large in the $c\to\infty$ limit, and the terms $c^{-2}\partial_t^2,\partial_{x}^2$ in the d'Alembertian, which are $O(c^{-2}),O(1)$ in the $c\to\infty$ limit, but highest order in the sense that they have the most derivatives. Replacing $\partial_x\rightsquigarrow\xi$ and $\partial_t \rightsquigarrow \tau$,
\begin{itemize} 
	\item $\partial_x^2 \rightsquigarrow\xi^2 = c^2 \xi_\natural^2$
	\item  $c^{-2}\partial_t^2 \rightsquigarrow c^{-2} \tau^2= c^2 \tau_\natural^2$.
\end{itemize}
The natural scale is therefore that on which $c^2$ is comparable to $\partial_x^2,c^{-2}\partial_t^2$. 

In order to analyze the Klein--Gordon operator on the natural scale without having to worry about the laboratory scale, we use the $\calczero$-pseudodifferential calculus that we denote by 
\begin{equation}
	\Psi_{\calczero} = \bigcup_{m,\ell\in \bbR}\bigcup_{\mathsf{s}} \Psi_{\calczero}^{m,\mathsf{s},\ell}.  
\end{equation}
This is c constructed by quantizing symbols on the compactification 
\begin{equation}
	{}^{\calczero} \overline{T}^* \bbM \hookleftarrow T^* \bbR^{1,d} \times (0,\infty)_c
\end{equation}
that results from ``blowing down'' $\mathrm{pf}_\pm$ in the $\calctwo$-phase space. 
This new compactification is called the ``$\natural$-phase space''.
The $\natural$-phase space can be described simply as a product
\begin{equation}\label{eq:Planck_intro}
	\overline{\bbR^{1,d}_{t,x}} \times \overline{\bbR^{1,d}_{\tau_{\natural},\xi_{\natural} }}  \times (0,\infty]_c,
\end{equation} 
and this is actually how we construct it.
This phase space has three boundary hypersurfaces, which we call df, bf, and $\natural\mathrm{f}$. The first two of these are canonically identifiable with the identically named boundary hypersurfaces of the $\calctwo$-phase space. The third is not literally the same as the corresponding boundary hypersurface of the $\calctwo$-phase space, but they are canonically identifiable away from $\mathrm{pf}_\pm$, and no confusion will come from using the same name.
\begin{remark}
	Although, as a manifold-with-corners (mwc), the $\calczero$-phase space is diffeomorphic to 
	\begin{equation}
		\overline{\bbR^{1,d}_{t,x}}\times \overline{\bbR^{1,d}_{\tau,\xi}}\times (0,\infty]_c \cong \bbB^{1+d}\times \bbB^{1+d}\times [0,1), 
		\label{eq:misc_024}
	\end{equation}
	the two are inequivalent as compactifications $T^* \bbR^{1,d} \times (0,\infty)_c\hookrightarrow \bbB^{1+d}\times \bbB^{1+d}\times [0,1)$. Indeed, the curves 
	\begin{equation}
	h\mapsto (\tau=\tau_\natural/h^2,\xi=\xi_\natural/h,h)
	\end{equation}
	of fixed ``natural frequency'' $(\tau_\natural,\xi_\natural)\in \bbR\times \bbR^d$ 
	hit the interior of the $\{h=0\}$ boundary hypersurface of the $\calczero$-phase space but hit the corner of the compactification in \cref{eq:misc_024}.
	What we care about is the compactification, not the topology of the ambient space.
\end{remark}

Typical vector fields in $\Psi_{\calczero}$ are
\begin{equation}
	\partial_{t_\natural} = c^{-2} \partial_t,\quad \partial_{x_{\natural,j}} = c^{-1} \partial_{x_j} 
\end{equation}
with coefficients that are symbolic functions on spacetime. The
calculus $\Psi_{\calczero}$ of $\calczero$-pseudodifferential
operators is very similar to the semiclassical calculus \cite{Zworski}
with $h=c^{-1}$ (or $h=c^{-2}$), with the difference being that, in
the $\calczero$-calculus, the time and space derivatives are treated
on unequal footing, with different powers of the ``semiclassical
parameter'' $c^{-1}$ on each. Such a setup, though without this
precise base space infinity behavior, was introduced in
\cite{Vasy:Semiclassical-X-Ray}, called the semiclassical foliation
pseudodifferential algebra there, and was developed (with a different
structure at infinity) in \cite{Vasy-Zachos:Conic} (and further
employed in \cite{Jia-Vasy:Conic-Tensor}) to study the geodesic X-ray
transform. In these X-ray transform settings there is a natural foliation by
level sets of a function; in the present setting the time function would be the analogue. To emphasize the similarities between $\natural$-analysis and semiclassical analysis, we will write $h$ interchangeably with $1/c$ below.

Thus, the free Klein--Gordon operator $P_0=\square-c^2$ lies in 
\begin{equation}
	 \operatorname{Diff}_{\calczero}^{2,0,2}=\Psi_{\calczero}^{2,0,2}\cap \operatorname{Diff}(\bbR^{1,d}),
\end{equation}
and does so in a non-degenerate way, with each term, $c^{-2} \partial_t^2$, $\partial_{x}^2$, $c^2$ being order 2 at $\natural\mathrm{f}$. (The variable-coefficient operator $P$ will be similar.)

At this point (forgetting what we know about the Cauchy problem), it may not be clear why it is necessary to consider the laboratory scale at all. Indeed, one can almost carry out a full analysis of the Klein--Gordon equation on the natural scale, using $\Psi_{\calczero}$. However, the $\natural$-Hamiltonian flow, ${}^{\calczero}\mathsf{H}_p$, degenerates, as $c\to\infty$, at two special frequencies discussed above, i.e.\ 
\begin{equation} 
	(\tau_\natural,\xi_\natural)=(\pm 1,0)
	\label{eq:misc_c}
\end{equation} 
See Appendix~\ref{sec:whynot}. Rather than being a technical glitch, this is an insurmountable obstacle to a full $\calczero$-analysis of the Klein--Gordon equation. 
The presence of the oscillations \cref{eq:usual_oscillations} in the solutions of the Cauchy problem guarantees the occurrence of a degeneration at the frequencies \cref{eq:misc_c} --- if degeneration did not occur, we would be able to use $\calczero$-analysis to prove to absence of the oscillations which we know may, in actuality, be present.

We will see in \S\ref{sec:flow} that \emph{second-microlocalizing} parabolically at the problematic points, i.e.\ blowing them up in some particular way, is sufficient to remove the degeneracy of the Hamiltonian flow: the $\calctwo$-Hamiltonian flow is source-to-sink, without degeneration. 
This second-microlocalization results in the $\calctwo$-phase space depicted in \Cref{fig:phase2}. We have described the $\calczero$-calculus as arising from blowing down $\mathrm{pf}_\pm$ in the $\calctwo$-phase space, but the reverse, constructing the $\calctwo$-phase space out of the $\calczero$-phase space, is simpler and closer to what we actually do below.

\subsection{The conjugated perspective}
\label{subsec:description}
Since the Fourier transform intertwines translation with multiplication by oscillatory exponentials, conjugating an operator via a natural oscillation has the effect of translating the corresponding natural frequency to the zero section. So, if instead of studying $P$, we study the conjugated operator
\begin{equation}
	P_\pm = e^{\mp ic^2 t} P e^{\pm i c^2 t} ,
	\label{eq:P+-_form}
\end{equation}
meaning the maps $\calS'\ni u\mapsto  e^{\mp i c^2 t} P ( e^{\pm i c^2 t} u)$,
we can second- microlocalize at the zero section, rather than at the points \cref{eq:misc_c}. 

This ``conjugated perspective'' turns out to be computationally convenient --- 
see for instance \cite{VasyLA, VasyN0L}, which applied it to analyze the Schr\"odinger--Helmholtz equation.
Second-microlocalization is often considered somewhat esoteric, even by experts, but second-microlocalization at the zero-section is a particularly simple form of second-microlocalization. 
For example, second-microlocalizing the semiclassical Sobolev spaces
at the zero section results in spaces with ordinary Sobolev regularity
in addition to ``background'' semiclassical Sobolev regularity away
from the zero section, see \cite{Vasy:Relationship}. Since the $\calczero$-calculus is similar to the semiclassical calculus, second-microlocalization at the zero section achieves something analogous here: the resultant Sobolev spaces, which we will call $H_{\calc}^{m,\mathsf{s},\ell}$,  have par-Sobolev regularity in addition to background $\calczero$-Sobolev regularity.

We will apply the conjugated perspective throughout this paper.

Of course, we can only conjugate one of the two points  \cref{eq:misc_c} to the zero section at a time. Because these two points are contained in different connected components of the $\calctwo$-characteristic set, this is not an issue; using  microlocalizers 
\begin{equation} 
	O\in \Psi_{\calczero}^{0,0,0},
\end{equation} 
we can decompose all functions $f=Of+(1-O)f$ into two parts, $Of$ and $(1-O)f$, each of which is microlocally trivial at one of the two components (and the two being microlocally trivial on \emph{different} components). 
Then, we can use the appropriate one of $P_\pm$ to study $Of$ and the other to study $(1-O)f$.  Thus, the ``bad'' component of the $\calczero$-characteristic set of $P_\pm$, the one that does not pass through the zero section, can be ignored entirely, because it corresponds to the ``good'' component of the $\calczero$-characteristic set of $P_\mp$.

The phase space 
\begin{equation}
	{}^{\calc}\overline{T}^* \bbM\hookleftarrow T^* \bbR^{1,d}\times (0,\infty)_c 
\end{equation}
that arises from second- microlocalizing the $\calczero$-phase space at the zero section we call the ``$\calc$-phase'' space. As opposed to ``$\calctwo$,'' the label ``$\calc$'' is supposed to indicate that only a single natural frequency is blown up. Correspondingly, at $c=\infty$, the $\calc$-phase space has two boundary hypersurfaces, $\natural\mathrm{f}$, and the ``parabolic face'' pf. In addition, we use bf to denote spacetime infinity and df to denote fiber infinity, as elsewhere. 
See \Cref{fig:phase} for depictions of the $\calc$-phase space.

The compactification 
\begin{equation}
	\mathrm{pf} \hookleftarrow T^* \bbR^{1,d}
\end{equation}
is equivalent to that ${}^{\mathrm{par}}\overline{T}^* \bbM \hookleftarrow T^* \bbR^{1,d}$ used by Gell-Redman--Gomes--Hassell to study the Schr\"odinger equation in \cite{Parabolicsc}; see \Cref{prop:pf=par_phase_space}. So, we may canonically identify 
\begin{equation}
	\mathrm{pf} \cong {}^{\mathrm{par}}\overline{T}^* \bbR^{1,d}. \label{eq:pf_id}
\end{equation}
(This is why the faces $\mathrm{pf}_\pm$ in the $\calctwo$-phase space corresponding to $\mathrm{pf}$ are also identifiable with the par-phase space.)
Since each positive level set of $c$ in ${}^{\calc}\overline{T}^* \bbM$ is canonically identifiable with ${}^{\mathrm{sc}}\overline{T}^* \bbM$, one way of thinking about ${}^{\calc}\overline{T}^* \bbM$ is that it interpolates between
\begin{itemize}
	\item ${}^{\mathrm{sc}}\overline{T}^* \bbM$ for $c<\infty$ and
	\item ${}^{\mathrm{par}}\overline{T}^* \bbM$ at $c=\infty$.
\end{itemize}

Correspondingly, the pseudodifferential calculus 
\begin{equation}
	\Psi_{\calc} = \bigcup_{m,\ell,q\in \bbR}\bigcup_{\mathsf{s}} \Psi_{\calc}^{m,\mathsf{s},\ell,q} 
\end{equation}
that arises from quantizing symbols on the $\calc$-phase space can be thought of as interpolating between $\Psi_{\mathrm{sc}}$ for $c<\infty$ and $\Psi_{\mathrm{par}}$ for $c=\infty$. 

Thus, $\Psi_{\calc}$ is a natural choice to interpolate between the sc-analysis of the Klein--Gordon equation and the par-analysis of the Schr\"odinger equation in the non-relativistic limit.

\subsection{Outline for remainder of manuscript}
\label{subsec:outline}

The sections in the rest of the manuscript are as follows:
\begin{itemize}
	\item In \S\ref{sec:calculus}, we develop the theory of $\Psi_{\calczero}$, $\Psi_{\calc}$, $\Psi_{\calctwo}$.  Because the $\calczero$-calculus is similar to the standard semiclassical calculus, our discussion will be telegraphic. The $\calctwo$-calculus is built out of two copies of the first. Hence, our focus is on $\Psi_{\calc}$. The gist of this section is fairly standard, but we need to check that the familiar elements of pseudodifferential calculus hold in this unfamiliar setting. 
	\item In \S\ref{sec:setup}, we state the set of Klein--Gordon operators $P$ to which our analysis applies, and we record some elementary propositions about where various related operators lie in our three pseudodifferential calculi.
	\item In \S\ref{sec:flow}, we will study ${}^{\calc}\mathsf{H}_p$, the generator of the Hamiltonian flow, minimally rescaled so as to define a vector field on the $\calc$-phase space tangent to all boundary hypersurfaces, which means that it induces a flow on the boundary. We will show that the flow has a source-to-sink structure in each component of the $\calc$-characteristic set. 
	\item In \S\ref{sec:estimates}, we prove the microlocal estimates that follow from the dynamical structure of the flow investigated in the previous section.  These include propagation (Duistermaat--H\"ormander) estimates, as well as radial point (microlocal Mourre) estimates. It is in this section that we prove \Cref{thm:inhomog}. 
\end{itemize}

In addition, we have several appendices:
\begin{itemize}
	\item \S\ref{sec:notation} is an index of notation. 
	\item \S\ref{sec:schrodinger_estimate} discusses how to generalize a certain estimate in \cite{Parabolicsc} so that we can apply it in the generality here.  
	\item \S\ref{sec:further} has some variants of the estimates in \S\ref{sec:estimates} that will be used in \cite{NRL_II}.
	\item \S\ref{sec:whynot} explains why (as must be the case for the non-relativistic limit to hold) the Klein--Gordon equation cannot be analyzed entirely in $\Psi_{\calczero}$. It is necessary to second microlocalize, i.e.\ to pass from $\Psi_{\calczero}$ to $\Psi_{\calc}$. 
	\item \S\ref{sec:parabolic} lists, for the reader's convenience, a few facts about the parabolic compactification $(\overline{\mathbb{R}^{1,d}})_{\mathrm{par}}$ of $\mathbb{R}^{1,d}$.
	\item \S\ref{sec:reduction} provides a second proof, via the formula for the Moyal star product, of the fact that elements of $\Psi_{\calc}$ can be composed, with the composition having the expected orders. This provides an alternative for \S\ref{subsec:composition}.
\end{itemize}

\section{\texorpdfstring{Development of $\Psi_{\calczero},\Psi_{\calc},\Psi_{\calctwo}$}{Development of pseudodifferential algebras} }
\label{sec:calculus}

In this technical section, we develop the theory of our three new pseudodifferential calculi $\Psi_{\calczero},\Psi_{\calc},\Psi_{\calctwo}$.
For readers that are comfortable with the construction of pseudodifferential calculi associated to various phase spaces, or willing to take for granted that such calculi work analogously to the usual H\"ormander calculus, only \S\ref{subsec:phase}, \S\ref{subsec:calc}, and \S\ref{sec:2-res_construction} are  necessary reading. 
In \S\ref{subsec:Sobolev}, we introduce the scale 
\begin{equation}
	\lVert - \rVert_{H_{\calc}^{m,\mathsf{s},\ell,q}} = \{ \lVert - \rVert_{H_{\calc}^{m,\mathsf{s},\ell,q}}(h) \}_{h\in (0,\infty)}
\end{equation}
of Sobolev norms in terms of which the main estimates of this paper are phrased, so the reader may wish to refer to that subsection as well. 

\subsection{\texorpdfstring{The natural calculus $\Psi_{\calczero}$}{The natural calculus}}\label{subsec:natural_calculus}

We begin our discussion with the phase space $\smash{{}^{\calczero}\overline{T}^* \bbM}$, which we have already introduced in the introduction.
As a manifold-with-corners, $\smash{{}^{\calczero}\overline{T}^* \bbM}$ is diffeomorphic to 
\begin{equation}
	\bbB^{1+d}\times \bbB^{1+d}\times [0,\infty)\cong \overline{\bbR^{1+d}_{t,x}}\times \overline{\bbR^{1+d}_{\tau,\xi}} \times [0,\infty)_h = {}^{\mathrm{sc}} \overline{T}^* \bbM \times [0,\infty)_h,
	\label{eq:ph0mwcid}
\end{equation}
but the partial compactification 
\begin{equation}
	T^* \bbR^{1,d}\times (0,\infty)_h \hookrightarrow \overline{\bbR^{1+d}}\times \overline{\bbR^{1+d}} \times [0,\infty) 
\end{equation}
used is
\begin{equation}
	T^* \mathbb{R}^{1,d} \times (0,\infty)_h  \ni (t,x,\tau,\xi,h) \mapsto (t,x,h^2 \tau,h\xi,h) \in {}^{\calczero}\overline{T}^* \bbM.
\end{equation}

In the sense of compactifications, 
\begin{equation}
	{}^{\calczero}\overline{T}^* \bbM = [0,1)_h \times \bbM_{t,x} \times \overline{\bbR^{1,d}_{\tau_{\calcshort},\xi_{\calcshort}  }}, 
\end{equation}
where $\tau_{\calcshort} = h^2 \tau$, $\xi_{\calcshort} = h \xi$, and $\bbM_{t,x} = \overline{\bbR^{1,d}}$. 

Notice that when one substitutes in $h=1/c$, the frequencies $\tau_{\natural}=c^{-2} \tau,\xi_{\natural}=c^{-1} \xi$ are precisely the natural scale frequency coordinates discussed in the introduction.

The three boundary hypersurfaces of ${}^{\calczero}\overline{T}^* \bbM$  are $\mathrm{zf}=\{h=0\}$, $\mathrm{bf} = \{t^2+\lVert x \rVert^2=\infty\}$, and $\mathrm{df} = \{ | \zeta_{\calcshort} |= \infty \}$, where $\zeta_{\calcshort} = (\tau_{\calcshort},\xi_{\calcshort}) \in \bbR^{1+d}$. As  boundary-defining functions for these hypersurfaces, we can take
\begin{equation} \begin{aligned}
\rho_{\mathrm{zf}} &= h, \\
	\rho_{\mathrm{df}} &= (1+h^4 \tau^2 + h^2 \xi^2)^{-1/2} = (1+\tau_{\calcshort}^2+\xi_{\calcshort}^2)^{-1/2} , \\
	\rho_{\mathrm{bf}} &= 1/\langle z \rangle = (1+t^2+\lVert x \rVert^2)^{-1/2}. 
\end{aligned}\end{equation} 

Next, we define our $\calczero$-symbol classes. Let $\delta \in (0,1/2)$ and a variable spacetime order $\mathsf{s} \in C^\infty({}^{\calczero}\overline{T}^* \bbM)$ be given. The symbol class $S^{m,\mathsf{s},0}_{\calczero,\delta}(\bbM)$ is set of conormal functions of $h$ with values in $S^{m,\mathsf{s}}_{\mathrm{sc}}$ symbols \emph{in the $\taun, \xin$ variables}. That is, we have estimates
\begin{equation} \label{eq:nat_symbol}
	|\partial_z^\alpha\partial_{\zetan}^\beta (h\partial_h)^j a(z,\zetan,h) | \leq C_{\alpha\beta j} \la z \ra^{\mathsf{s}-(1-\delta)|\alpha| + \delta |\beta|+\delta j} \la \zetan \ra^{m-|\beta|}, \quad \zetan = (\taun, \xin), \quad z=(t,x), 
\end{equation}
for all $j\in \bbN$, multi-indices $\alpha,\beta\in \bbN^{1+d}$.
This is provided with the Fr\'echet topology induced by the seminorms given by the best constants $C_{\alpha\beta j}$ in the inequality above. Then, we define 
\begin{equation}
	S_{\calczero,\delta}^{m,\mathsf{s},\ell} = h^{-\ell} S_{\calczero,\delta}^{m,\mathsf{s},0}.
\end{equation}
We then take the intersection over all $\delta > 0$ to obtain our desired symbol class:
\begin{equation}\label{eq:nat_symbol_intersection_delta}
	S_{\calczero}^{m,\mathsf{s},\ell} = \bigcap_{\delta \in (0,1/2)} S_{\calczero,\delta}^{m,\mathsf{s},\ell}.
\end{equation}

\begin{remark}
	If $\mathsf{s}$ were constant, then we could take $\delta=0$ above. The construction involving $\delta$ is needed to handle logarithmic terms that arise when differentiating $\langle z \rangle^{\mathsf{s}}$. 
\end{remark}

There is an alternative way of specifying the symbol class that we find more convenient when dealing with second- microlocalized phase spaces. That is, we express \eqref{eq:nat_symbol} in terms of regularity, in the sense of remaining in a fixed weighted $L^\infty$ space, under the repeated action of a class of vector fields. When $\delta = 0$ this class would simply be the set $\calV_{\mathrm{b}}({}^{\calczero}\overline{T}^* \bbM)$ of $b$-vector fields on ${}^{\calczero}\overline{T}^* \bbM)$, that is, those smooth vector fields tangent to the boundary. These vector fields are spanned, over $C^\infty({}^{\calczero}\overline{T}^* \bbM)$, by $h \partial_h$, $\ang{z} \partial_{z_j}$ and $\ang{\zetan} \partial_{\zeta_{\natural, k}}$ and thus \eqref{eq:nat_symbol}, for $\delta = 0$, is equivalent to the condition that the repeated action of $\calV_{\mathrm{b}}({}^{\calczero}\overline{T}^* \bbM)$ on $a$ remains in 
\begin{equation} 
\la z \ra^{\mathsf{s}} \la \zetan \ra^{m} L^\infty({}^{\calczero}\overline{T}^* \bbM).
\end{equation} 
To deal with the $\delta$ loss we take a subspace of $\calV_{\mathrm{b}}({}^{\calczero}\overline{T}^* \bbM)$ requiring an arbitrarily small but positive extra decay at the face $\mathrm{bf}$:  we denote  
\begin{equation} \label{eq:tildeVb-nat}
\tilde{\calV}_{\mathrm{b}}({}^{\calczero}\overline{T}^* \bbM) = \bigcup_{\delta \in (0,1/2)} \rho_{\mathrm{bf}}^\delta \calV_{\mathrm{b}}({}^{\calczero}\overline{T}^* \bbM).
\end{equation}
Then we have the alternative definition of the space in \eqref{eq:nat_symbol_intersection_delta}: 
\begin{equation}
S_{\calczero}^{m,\mathsf{s},\ell} = \big\{ a \in \la z \ra^{\mathsf{s}} \la \zetan \ra^{m} L^\infty({}^{\calczero}\overline{T}^* \bbM) : \big( \tilde{\calV}_{\mathrm{b}}({}^{\calczero}\overline{T}^* \bbM) \big)^M a \in \la z \ra^{\mathsf{s}} \la \zetan \ra^{m} L^\infty({}^{\calczero}\overline{T}^* \bbM) \text{ for all } M \in \mathbb{N} \big\}.
\end{equation}

For $a \in S^{m,\mathsf{s},\ell}_{\calczero}$, its (left) quantization $A=\operatorname{Op}(a)$ acts on functions by
\begin{equation}
	A(h) f(t, x)  = \frac{1}{(2\pi)^{d+1} h^{d+2}} \int_{\bbR^{d+1}} \int_{\bbR^{d+1}}   e^{i \taun (t-t')/h^2 + \xin \cdot (x-x')/h} a(h)(t, x, \taun, \xin, h) f(z') \dd^{d+1} z' \dd^{d+1} \zetan ,
	\label{eq:quant_calczero}
\end{equation}
where $z' = (t', x')$. This is, of course, nothing more than the standard quantization written in the natural coordinates $\taun, \xin$. These are the natural variables to use since the symbol $a$ is uniformly symbolic in these variables; see \cref{eq:nat_symbol}. See the discussion after \cref{eq:quant} for more details.

Let $\Psi_{\calczero}^{{m,\mathsf{s},\ell}}$ denote the set of all quantizations of elements of $S_{\calczero}^{m,\mathsf{s},\ell}$, and let
\begin{equation} \label{eq:_calczero_PsiDO_union}
	\Psi_{\calczero}=\bigcup_{m,s,\ell\in \bbR} \Psi_{\calczero}^{{m,s,\ell}}. 
\end{equation}
This is almost the usual semiclassical-scattering calculus, except with $\tau_\natural=h^2 \tau$ being the relevant semiclassical temporal frequency instead of $h\tau$. 

In contrast to the calculi $\Psi_{\calc},\Psi_{\calctwo}$ mentioned above and defined below,
$\Psi_{\calczero}$ is ``fully symbolic'', for the same reason that the usual semiclassical calculus is fully symbolic. This means that one has a ``principal symbol'' map
\begin{equation}
	\sigma_{\calczero}^{m,\mathsf{s},\ell}:\Psi_{\calczero}^{m,\mathsf{s},\ell} \to S_{\calczero}^{m,\mathsf{s},\ell} / \bigcap_{\delta>0}  S_{\calczero}^{m-1,\mathsf{s}-1+\delta,\ell-1}
\end{equation} 
which captures operators up to operators of lower order in each of the three indices.
Here, and in similar formulas below (we will not repeat mention of this), we could take $\delta=0$ when $\mathsf{s}$ is constant. 

The $\calczero$-calculus $\Psi_{\calczero}$ is a multi-graded algebra, in which $\Psi$DOs can be composed with the usual addition of orders:
\begin{propositionp} \label{prop:Planck_composition}
	For $A =\operatorname{Op}(a)\in \Psi_{\calczero}^{m_1,\mathsf{s}_1,\ell_1}, \, B=\operatorname{Op}(b) \in \Psi_{\calczero}^{m_2,\mathsf{s}_2,\ell_2}$, we have
	\begin{equation}
		AB \in \Psi_{\calczero}^{m_1+m_2,\mathsf{s}_1+\mathsf{s}_2,\ell_1+\ell_2}.
	\end{equation}
	The symbol $c$ of $C=A B$ has (in the usual precise sense) an asymptotic expansion 
	\begin{equation}
		c(z, \zetan, h) \sim  \sum_{\alpha\in \bbN^{d+1}} \frac{i^{|\alpha|}}{ \alpha! } \big((h^2 D_{\taun}, hD_{\xin})^\alpha a(z,\zetan,h) \big) D_{z}^{\alpha} b(z,\zetan,h),
		\label{eq:moyal_explicit_natural}
	\end{equation}
	where $D_{\bullet}=-i\partial_{\bullet}$ as usual and the $\alpha$-th term in the expansion is in \begin{equation} 
		\bigcap_{\delta > 0} S_{\calczero}^{m_1 + m_2 - |\alpha|,\mathsf{s}_1 + \mathsf{s}_2 - |\alpha| + \delta, \ell_1 + \ell_2 - |\alpha|}.
	\end{equation} 
	Consequently,
	\begin{equation}
		\sigma_{\calczero}^{m_1+m_2,\mathsf{s}_1+\mathsf{s}_2,\ell_1+\ell_2}(A \circ B) = \sigma_{\calczero}^{m_1,\mathsf{s}_1,\ell_1}(A) \sigma_{\calczero}^{m_2,\mathsf{s}_2,\ell_2}(B), 
	\end{equation}
	\begin{equation}
		[A,B] \in\bigcap_{\delta>0} \Psi_{\calczero}^{m_1+m_2-1,\mathsf{s}_1+\mathsf{s}_2-1+\delta,\ell_1+\ell_2-1},
	\end{equation}
	with the principal symbol of the commutator $[A,B]$ being proportional to the Poisson bracket of the principal symbols of $A,B$.
\end{propositionp}

Ellipticity in symbol classes $S^{m,\mathsf{s},\ell}_{\calczero}$ is defined in the usual way: for $a \in S^{m,\mathsf{s},\ell}_{\calczero}$ and $p \in \partial({}^{\calczero} \overline{T}^* \bbM)$, we say that $a$ is elliptic at $p$ if there exists a neighborhood $U$ of $p$ and some $\epsilon>0$ such that  
\begin{equation}
	| \rho_{\mathrm{df}}^{m} \rho_{\mathrm{bf}}^{\mathsf{s}}\rho_{\mathrm{zf}}^{l} a |  \geq \epsilon  > 0 \text{ on } U.
\end{equation}
In this case, we also say that $A = \operatorname{Op}(a)$ is elliptic at $p$. 
When this holds for every $p$, then we simply say $A$ is (totally) elliptic (as an operator in $\Psi_{\calczero}^{m,\mathsf{s},l}$). Note that $\calczero$-ellipticity is stronger than ordinary ellipticity, as it involves elliptic behavior at spacetime infinity as well as $h\to 0^+$, which applies at all finite frequencies, not just high frequencies, just as in the usual semiclassical case.

We define the operator wavefront set in an analogous way: 
the $\calczero$-operator wavefront set \begin{equation} 
	\operatorname{WF}'_{\calczero}(A) \subset \partial ({}^{\calczero} \overline{T}^*\bbM)
\end{equation} 
of $A=\operatorname{Op}(a)$ is defined to be the essential support of $a\in S_{\calczero}^{m,\mathsf{s},\ell}$. More concretely, for $p \in \partial ({}^{\calczero} \overline{T}^*\bbM)$, 
\begin{equation}
	p \notin \operatorname{WF}'_{\calczero}(A)  \iff  |\rho_{\mathrm{df}}^{-N} \rho_{\mathrm{bf}}^{-N}\rho_{\mathrm{zf}}^{-N}a| \leq C_N \text{ near } p, \text{ for all } N \in \bbR ,\text{ for some }C_N.
\end{equation}
Note that $\operatorname{WF}'_{\calczero}(A)$ is always a closed subset of $\partial ({}^{\calczero} \overline{T}^*\bbM)$.

We next define the $\calczero$-wavefront set and Sobolev norms for (parametrized) tempered distributions on $\mathbb{M}$. By a parametrized tempered distribution $u(h)$, depending on $h \in (0, 1)$, we mean a family of distributions $u(h)$ on $\bbR^{1,d}$ such that each $u(h)$ is contained in a fixed weighted Sobolev space $H^{m, s}_{\mathrm{sc}}(\bbR^{1,d})$, and its norm is tempered, i.e.\ bounded by $C h^{-N}$ for some $C, N \in \bbR$. We will omit the `parametrized' from here on, and refer to the family $u(h)$ as a tempered distribution. 

The $(m,\mathsf{s},\ell)$-order $\calczero$-wavefront set of a tempered distribution $u$ is the closed set $\operatorname{WF}_{\calczero}^{m,\mathsf{s},\ell}(u) \subset \partial ({}^{\calczero} \overline{T}^*\bbM)$ defined by

\begin{multline}
	p \notin \operatorname{WF}_{\calczero}^{m,\mathsf{s},\ell}(u) \iff \exists A \in \Psi_{\calczero}^{m,\mathsf{s},\ell} \text{ that is elliptic at } p \text{ such that } Au \in L^2(\bbM), \text{ with } \\
	\| Au(h) \|_{L^2(\bbM)} \text{ uniformly bounded in } h. 
\end{multline}

We next define a family of norms on tempered distributions. The $\calczero$-Sobolev norms with indices $(m,\mathsf{s},\ell)$ will be defined by choosing a  totally elliptic operator $\Lambda_{m,\mathsf{s},\ell}$ in $\Psi_{\calczero}^{m,\mathsf{s},\ell}$ (which can just be taken to be the quantization of a product of powers of the boundary-defining-functions), together with an invertible elliptic operator $A_{m',s',\ell'} \in  \Psi_{\calczero}^{m',s',\ell'}$ where $m' \leq m, s' \leq \mathsf{s}, \ell' \leq \ell$ and where $s$ is constant: for example, we can take 
\begin{equation} 
	A_{m',s',\ell'} = h^{-\ell'} \ang{z}^{s'} (1 + h^2 \triangle_x + h^4 D_t^2)^{m'/2}.
\end{equation}  
Then, the $\calczero$-Sobolev norms are defined by
\begin{equation}\label{eq:msl_norm}
	\lVert u \rVert_{H^{m,\mathsf{s},\ell}_{\calczero}} =\lVert \Lambda_{m,\mathsf{s},\ell}u \rVert_{L^2(\bbR^{1,d})} + 
\lVert A_{m',s',\ell'} u \rVert_{L^2(\bbR^{1,d})}	.
\end{equation}
In \cref{eq:msl_norm}, the purpose of the $A_{m',s',\ell'} u$ term is to ensure that the norm vanishes only when $u \equiv 0$. It could be removed if instead we modified $\Lambda_{m,\mathsf{s},\ell}$ to make it injective. 

We remark that there is a slight abuse of notation here:  we are \emph{not} defining a single Sobolev space $H^{m,\mathsf{s},\ell}_{\calczero}$, merely a one-parameter family of norms, $\lVert - \rVert_{H_\natural^{m,\mathsf{s},\ell}(h)}$
\begin{remark}
	Such a space could be defined, e.g.\ by defining the norm to be the supremum, over $h > 0$, of the norms in \cref{eq:msl_norm}. We do not wish to do this, because we are, in fact, interested in the value of the norm for each $h$, not merely in its supremum; and besides, we also want to consider tempered distributions for which \cref{eq:msl_norm} blows up, as $h \to 0$, in a tempered way. 
\end{remark}

Next, we state some basic facts about $\Psi_{\calczero}$.
\begin{proposition} 
\hfill
\begin{itemize}
              \item 
		For $A,B \in \Psi_{\calczero}$, $\operatorname{WF}'_{\calczero}(AB)\subseteq \operatorname{WF}'_{\calczero}(A)\cap \operatorname{WF}'_{\calczero}(B)$.

		\item If $A\in \Psi_{\calczero}^{m,\mathsf{s},\ell}$ is elliptic, then there exists a $B\in \Psi_{\calczero}^{-m,-\mathsf{s},-\ell}$ such that $AB-1,BA-1\in \Psi_{\calczero}^{-\infty,-\infty,-\infty}$.

		Consequently, for all $m',\ell'\in \bbR$ and variable order $\mathsf{s}'$, there exists an $h_0>0$ and $C > 0$ such that for all $h\in (0,h_0)$ and $u\in \calS'$, the estimate 
		\begin{equation} 
			\lVert u \rVert_{H_{\calczero}^{m',\mathsf{s}',\ell'} }\leq C \lVert Au \rVert_{H_{\calczero}^{m'-m',\mathsf{s}'-\mathsf{s},\ell'-\ell} } 
			\label{eq:misc_11n} 
		\end{equation}
		holds, in the strong sense that if the right-hand side is finite, then so is the left-hand side, and the stated inequality holds.  	
			\item More generally, if $A \in \Psi_{\calczero}^{m,\mathsf{s},\ell}$, $B \in \Psi_{\calczero}^{m_1,\mathsf{s_1},\ell_1}$  are such that $A$ is elliptic on $\operatorname{WF}'_{\calczero}(B)$, then we have for all $N \in \mathbb{R}$  
\begin{equation} 
			\lVert B u \rVert_{H_{\calczero}^{m'-m_1,\mathsf{s}' - \mathsf{s_1},\ell' - \ell_1} }\leq C \big( \lVert Au \rVert_{H_{\calczero}^{m'-m',\mathsf{s}'-\mathsf{s},\ell'-\ell} } + \lVert u \rVert_{H_{\calczero}^{-N, -N, -N}} \big). 
			\label{eq:micro-elliptic-est} 
		\end{equation}
			\end{itemize}
\end{proposition}

\begin{remark} We emphasize that the constants $C$ in \cref{eq:misc_11n} and \cref{eq:micro-elliptic-est}  are uniform in $h$. Although we usually will not state it explicitly, \emph{this uniformity as $h \to 0$ holds for all inequalities of norms we state in this paper.}
\end{remark}

\begin{proof}
The first point follows straightforwardly from the reduction formula \cref{eq:moyal_explicit_natural}. For the second point,   the standard elliptic parametrix construction directly gives 
\begin{equation} 
	\lVert u \rVert_{H_{\calczero}^{m',\mathsf{s}',\ell'} } \lesssim \lVert Au \rVert_{H_{\calczero}^{m'-m',\mathsf{s}'-\mathsf{s},\ell'-\ell} } +\lVert u \rVert_{H_{\calczero}^{-N_1,-N_2,-N_3}}, 
	\label{eq:misc_118}
\end{equation}
for all $N_1,N_2,N_3\in \bbR$. If we take $N_1,N_2,N_3$ sufficiently large, then 
\begin{equation} 
	\lVert u \rVert_{H_{\calczero}^{-N_1,-N_2,-N_3}}\lesssim h 	\lVert u \rVert_{H_{\calczero}^{m',\mathsf{s}',\ell'} }. 
\end{equation} 
So, for $h$ sufficiently small, the last term in \cref{eq:misc_118} can be absorbed into the left-hand side, giving \cref{eq:misc_11n}. The last point follows similarly via the standard microlocal elliptic parametrix. 
\end{proof}

Note that the absence of an ``error term'' on the right-hand side of \cref{eq:misc_11n} means that, for each $N\in \bbR$, it must be the case that $A(h)$ is injective on the ordinary Sobolev space $H_{\mathrm{sc}}^{-N,-N} = (1+t^2+r^2)^{N/2} H^{-N}$ for $h$ sufficiently small, where what ``sufficiently small'' means can depend on $N$.

For any $B\in \Psi_{\calczero}^{m,\mathsf{s},\ell}$ and $\alpha\in \bbR$, let 
\begin{equation} 
	B_\alpha=M_{\exp(i\alpha t/h^2)}\operatorname{Op}( b)M_{\exp(-i\alpha t/h^2)},
\end{equation} 
where $M_\bullet:u\mapsto \bullet u$ is the multiplication map. Multiplication by oscillatory functions corresponds to translation in frequency space. An important fact about the $\calczero$-calculus is that it is invariant under $\natural$-scale translations in frequency space:
\begin{proposition}
	If $b\in S_{\calczero}^{m,\mathsf{s},\ell}$ and $B=\operatorname{Op}(b)\in \Psi_{\calczero}^{m,\mathsf{s},\ell}$, then $B_\alpha \in \Psi_{\calczero}^{m,\mathsf{s},\ell}$ as well, and $B_\alpha= \operatorname{Op}(b_\alpha)$, 
	where $b_\alpha(\tau_{\natural},\xi_{\natural},t,x) = b(\tau_{\natural} -\alpha,\xi_{\natural},t,x)$. 
	\label{prop:translation}
\end{proposition}
\begin{proof}
	This is clear from the quantization formula \cref{eq:quant_calczero}. 
\end{proof}
We will use this later.

\begin{proposition}
	For any $\alpha\in \bbR$, the map $\calS \ni u \mapsto e^{i\alpha t / h^2} u$ extends to a uniformly bounded transformation on every $\calczero$-Sobolev space. In addition, this map is bounded from below: 
	\begin{equation}
		\lVert e^{i\alpha t/h^2}w \rVert_{H_{\calczero}^{m,\mathsf{s},\ell}}
		\geq C^{-1} \lVert w \rVert_{H_{\calczero}^{m,\mathsf{s},\ell}},
		\label{eq:misc_058}
	\end{equation}
	for a constant $C>0$ depending on $\alpha,m,\mathsf{s},\ell$.
	\label{prop:oscillation_mapping_boundedness}
\end{proposition}
\begin{proof}
	Choose $B\in \Psi_{\calczero}^{m,\mathsf{s},\ell}$ elliptic. Then, there exists some $h_0>0$ such that for all $h\in (0,h_0)$ and $u\in \calS'$, we have 
	\begin{equation}
		\lVert e^{i\alpha t/h^2} u \rVert_{H_{\calczero}^{m,\mathsf{s},\ell}} \lesssim \lVert B (e^{i\alpha t/h^2} u) \rVert_{L^2} =\lVert e^{-i\alpha t/h^2} B (e^{i\alpha t/h^2} u) \rVert_{L^2} = \lVert B_{-\alpha} u \rVert_{L^2}\lesssim \lVert u \rVert_{H_{\calczero}^{m,\mathsf{s},\ell}}.
	\end{equation}

	The lower bound \cref{eq:misc_058} follows applying this to $u=e^{-i\alpha t/h^2} w$.
\end{proof}

\subsection{\texorpdfstring{The resolved parametrized parabolic calculus, $\Psiparres$}{The resolved parametrized parabolic calculus}}\label{subsec:parIres}
In this section we introduce the second ingredient of $\Psi_{\calc}$. This is the resolved parametrized parabolic calculus, $\Psiparres$, which is constructed from the parabolic calculus involving a small parameter $h$. 
The parabolic calculus on Minkowski space $\bbR^{1,d}$ is a variant of the scattering calculus,  in which we treat the time frequency variable $\tau$ differently to the spatial coordinates $\xi_i$. This is done so that ``one time derivative is worth two spatial derivatives''; more precisely, the differential operator $D_t$ is second order in the calculus, while spatial derivatives $D_{x_j}$ are first order. This ensures that the time-derivative term in the Schr\"odinger operator $D_t + \triangle$ is a principal term that is treated as the same ``strength'' as the Laplace term, essential so that the operator is of real principal type (away from radial sets) and can be analyzed using microlocal propagation estimates. Anisotropic calculi, more general than the particular parabolic anisotropy we consider here, were introduced on local patches of $\bbR^{1,d}$ by Lascar \cite{Lascar}. 

The parabolic compactification of phase space considered here was introduced by Gell-Redman--Gomes--Hassell in \cite{Parabolicsc}. 
It is defined by
\begin{equation} 
	{}^{\mathrm{par}}\overline{T}^* \bbM = \bbM \times \overline{(\bbR^{1,d}_{\tau,\xi  })}_{\mathrm{par}},
\end{equation} 
where $\overline{(\bbR^{1,d}_{\tau,\xi  })}_{\mathrm{par}}$
is the parabolic compactification of $\bbR^{1,d}$. 

The points at fiber-infinity in this compactification can be identified with 
``parabolic'' rays   $\{ (\tau_0 s^2, v s) :s\in \bbR^+  \} $, $\tau_0\in \bbR,v\in \bbR^d$ (not both zero).
An explicit atlas for this compactification is given by 
\begin{equation} 
	\overline{(\bbR^{1,d}_{\tau,\xi  })}_{\mathrm{par}} =  \bbR^{1,d}_{\tau,\xi }  \cup \Big( \bigcup_\pm  (\overline{(\bbR^{1,d}_{\tau,\xi  })}_{\mathrm{par}} \cap \{\pm \tau>0\}) \Big) \cup \Big( \bigcup_{\pm, k\in \{1,\ldots,d\}} (\overline{(\bbR^{1,d}_{\tau,\xi  })}_{\mathrm{par}} \cap \{\pm \xi_k>0\}) \Big)\;\,
	\label{eq:misc_005}
\end{equation} 
where, in the sense of equality of compactifications,  
\begin{align}
	&\overline{(\bbR^{1,d}_{\tau,\xi  })}_{\mathrm{par}} \cap \{\pm\tau>0\} = [0,\infty)_{1/(\pm \tau)^{1/2}} \times \bbR^d_{\xi/(\pm\tau)^{1/2}}, \label{eq:misc_006} \\
	&\overline{(\bbR^{1,d}_{\tau,\xi  })}_{\mathrm{par}} \cap \{\pm \xi_k>0\} =  [0,\infty)_{\pm 1/\xi_k} \times \bbR_{\tau/\xi_k^2}\times  \bbR^{d-1}_{\hat{\xi}_k}, \label{eq:misc_007}
\end{align}
where $\hat{\xi}_k$ is the $(d-1)$-tuple $\{\xi_j/\xi_k\}_{j\neq k}$. 

We can observe that $\rho_{\mathrm{bf}} := \ang{z}^{-1}$ is a boundary defining function for spacetime infinity, and $\aang{\zeta}^{-1}$ is a boundary defining function for fiber-infinity, where 
\begin{equation} 
\aang{\zeta} = (1 + \tau^2 + |\xi|^4)^{1/4}, \quad \zeta = (\tau, \xi).
\end{equation} 

We now let $\mathsf{s}$ denote a smooth function on ${}^{\mathrm{par}}\overline{T}^* \bbM \times I$, where $I = [0, 1]_h$ denotes the space of parameters.  We also choose $0 < \delta < 1/2$ as before. We define the space of parametrized parabolic symbols  with differential order $m$, spacetime order $\mathsf{s}$ and $h$-order $q$ , denoted  $S_{\mathrm{par, I, \delta}}^{m,\mathsf{s},q}(\bbM)$ as the set of smooth functions $a(z, \zeta, h)$ on $T^* \bbR^{1,d} \times (0, 1]$ satisfying 
\begin{equation} \label{eq:par_symbol}
|\partial_z^\alpha \partial_{\xi}^\beta \partial_\tau^k (h \partial_h)^j a(z,\zeta,h)| \leq C_{\alpha k \beta} h^{-q} \ang{z}^{\mathsf{s}- |\alpha|+ \delta(|\alpha| +  |\beta| + j+ k)} \aang{\zeta}^{m-2k-|\beta|}. 
\end{equation}
These estimates are just parabolic symbol estimates of order $(m, \mathsf{s})$ with conormal behaviour in $h$. Note, however, that $\mathsf{s}$ may depend on $h$. 
Also notice a key difference with scattering-type symbol estimates is that a $\tau$-derivative reduces the order of growth in frequency by $2$, since $\tau$ has order two in the parabolic calculus.  
The best constants in \cref{eq:par_symbol} define a sequence of seminorms that give a Fr\'echet space structure to $S_{\mathrm{par, I, \delta}}^{m,\mathsf{s}}(\bbM)$. We then define 
\begin{equation}
S_{\mathrm{par, I}}^{m,\mathsf{s},q}(\bbM) = \bigcap_{\delta > 0} S_{\mathrm{par, I, \delta}}^{m,\mathsf{s},q}(\bbM).
\end{equation}

As with the natural calculus, we can express this in terms of fixed regularity under repeated applications of vector fields. We define  $\calV_{\mathrm{b}}({}^{\mathrm{par}}\overline{T}^* \bbM \times [0,1]_h)$ to be the smooth vector fields on this space tangent to the boundary,\footnote{We are not interested in the boundary at $h=1$ and ignore it completely.} and 
\begin{equation}  \label{eq:tildeVb-par}
\tilde{\calV}_{\mathrm{b}}({}^{\mathrm{par}}\overline{T}^* \bbM \times [0,1]_h) = \bigcup_{\delta > 0} \rho_{\mathrm{bf}}^\delta \calV_{\mathrm{b}}({}^{\mathrm{par}}\overline{T}^* \bbM \times [0,1]_h).
\end{equation} 
The symbol classes are characterized by 
\begin{multline}\label{eq:par_symbols_Vb_defn}
S_{\mathrm{par, I}}^{m,\mathsf{s},\ell} = \big\{ a \in h^{-q} \la z \ra^{\mathsf{s}} \la \zeta \ra^{m} L^\infty({}^{\mathrm{par, I}}\overline{T}^* \bbM) : \big( \tilde{\calV}_{\mathrm{b}}({}^{\mathrm{par, I}}\overline{T}^* \bbM) \big)^M a \in h^{-q} \la z \ra^{\mathsf{s}} \la \zeta \ra^{m} L^\infty({}^{\mathrm{par}}\overline{T}^* \bbM \times [0,1]_h) \\ \text{ for all } M \in \mathbb{N} \big\}.
\end{multline}

Then the space of parametrized parabolic pseudodifferential operators of order $(m,\mathsf{s},q)$, which is denoted by $\Psi_{\mathrm{par, I}}^{m,\mathsf{s}, q}(\bbM)$, consists of parametrized operators acting by the standard (left) quantization 
\begin{equation}
\operatorname{Op}(a) f(z, h)  = \frac{1}{(2\pi)^D} \int_{\bbR^D} \int_{\bbR^D}   e^{i \zeta\cdot (z-z')} a(z,\zeta,h) f(z') \bbR^D z' \dd^D \zeta, \quad D = 1+d. 
\label{eq:quant, par}
\end{equation}
where $a \in S_{\mathrm{par, I}}^{m,\mathsf{s}, q}(\bbM)$. 
This class is of course nothing more than $h^{-q}$ times a conormal function of $h$ with values in $\Psi_{\mathrm{par}}^{m,\mathsf{s}}(\bbM)$. 

The principal symbol of $A = \operatorname{Op}(a)$, $a \in S_{\mathrm{par,I}}^{m,\mathsf{s},q}(\bbM)$, is its image in 
\begin{equation} 
S_{\mathrm{par,I}}^{m,\mathsf{s},q}(\bbM) / \bigcap_{\delta > 0} S_{\mathrm{par,I}}^{m-1,\mathsf{s}-1+\delta,q}(\bbM)
\end{equation} 
Notice that this calculus is \emph{not} ``fully symbolic'' in the sense of the discussion at the beginning of this section. In fact, if we restrict attention to smooth, rather than conormal functions, of $h$, then 
the restriction at $h=0$ is an operator, not a function, and while this restriction or `normal operator' is multiplicative under composition, it takes values in a noncommutative algebra, namely $\Psi_{\mathrm{par}}$.

It is straightforward to check that the composition of $A \in \Psi_{\mathrm{par,I}}^{m,\mathsf{s},q}$ with $B \in \Psi_{\mathrm{par,I}}^{m',\mathsf{s}',q'}$ is an operator $C \in \Psi_{\mathrm{par,I}}^{m+m',\mathsf{s}+\mathsf{s}',q+q'}$. Moreover, the symbol of $C$ satisfies an asymptotic expansion 
\begin{multline}\label{eq:symbol exp parI}
	c(z,\zeta,h) \sim  \sum_{|\alpha| \leq k-1} \frac{i^{|\alpha|}}{ \alpha! } \big(D_\zeta^\alpha a(z,\zeta,h) \big) D_{z}^{\alpha} b(z,\zeta,h)  
	+ r(z,\zeta,h), \\ r \in S_{\mathrm{par, I}}^{m+m'-k,\mathsf{s}+\mathsf{s}' -k+\delta,q+q'}(\bbM) \ \forall \delta > 0
\end{multline}
in which the terms decrease in order both at fiber-infinity and at spacetime infinity. It follows that the principal symbol of $AB$ is the product of the principal symbols of $A$ and $B$, and the principal symbol of the commutator $i[A, B]$ is the Poisson bracket of the principal symbols of $A$ and $B$.

That was all about $\Psi_{\mathrm{par,I}}$.
We now turn to describing the resolved parabolic calculus $\Psi_{\mathrm{par,I,res}}$. 
The resolved parametrized phase space ${}^{\mathrm{par,I}}\overline{T}^*\bbM={}^{\mathrm{par,I,res}}\overline{T}^* \bbM$ is obtained from ${}^{\mathrm{par}}\overline{T}^* \bbM \times [0,1]_h$ by blowing up the intersection of $\{ h = 0 \}$ and the face at fiber-infinity, that is, $\aang{\zeta}^{-1} = 0$:
\begin{equation}
{}^{\mathrm{par,I,res}}\overline{T}^* \bbM = \Big[ {}^{\mathrm{par}}\overline{T}^* \bbM \times [0,1]_h; \{ h = 0, \aang{\zeta}^{-1} = 0 \} \Big]. 
\end{equation}
This ``resolved'' phase space has four boundary hypersurfaces: the differential face $\mathrm{df_1}$ at fiber-infinity (the lift of fiber-infinity from ${}^{\mathrm{par,I}}\overline{T}^* \bbM$); the spacetime face $\mathrm{bf}$ (the lift of spacetime-infinity from ${}^{\mathrm{par,I}}\overline{T}^* \bbM$); the parabolic face $\mathrm{pf}$ (the lift of $\{ h = 0 \}$ from ${}^{\mathrm{par,I}}\overline{T}^* \bbM$); and the new face created by blowup, which we call the front face or the natural face, and denote $\natural\mathrm{f}$ (or ff). We have boundary defining functions 
\begin{equation}\begin{aligned}
\rho_{\mathrm{df}_1} &= \frac{\aang{\zeta}^{-1}}{h +  \aang{\zeta}^{-1}}, &\qquad 
\rho_{\mathrm{bf}} &= \ang{z}^{-1}, \\
\rho_{\mathrm{par}} &= \frac{h}{h +  \aang{\zeta}^{-1}} , &\qquad
\rho_{\natural\mathrm{f}} &= h + \aang{\zeta}^{-1}.
\end{aligned}\end{equation}

Accordingly we have symbol classes with four different orders, which we denote $m$ (the order at the differential face), $\mathsf{s}$ (the order at the spacetime face), $q$ (the order at the parabolic face) and $l$ (the order at the natural face). 

We define the symbol classes on ${}^{\mathrm{par,I,res}}\overline{T}^* \bbM$  in terms of fixed regularity under repeated applications of vector fields. We define 
$\calV_{\mathrm{b}}({}^{\mathrm{par,I,res}}\overline{T}^* \bbM)$ to be the smooth vector fields on this space tangent to the boundary, and 
\begin{equation}  \label{eq:tildeVb-parIres}
\tilde{\calV}_{\mathrm{b}}({}^{\mathrm{par,I,res}}\overline{T}^* \bbM ) = \bigcup_{\delta > 0} \rho_{\mathrm{bf}}^\delta \calV_{\mathrm{b}}({}^{\mathrm{par,I,res}}\overline{T}^* \bbM ).
\end{equation} 
The symbol classes $\tilde S_{\mathrm{par, I,res}}^{m,\mathsf{s},\ell,q}$ are defined analogously to \cref{eq:par_symbols_Vb_defn}:
\begin{multline}
\tilde S_{\mathrm{par, I,res}}^{m,\mathsf{s},\ell,q} = \Big\{ a \in \rho_{\mathrm{df}}^{-m} \rho_{\mathrm{bf}}^{-\mathsf{s}} \rho_{\natural\mathrm{f}}^{-\ell}   \rho_{\mathrm{par}}^{-q} L^\infty({}^{\mathrm{par, I}}\overline{T}^* \bbM) : \big( \tilde{\calV}_{\mathrm{b}}({}^{\mathrm{par, I}}\overline{T}^* \bbM) \big)^M a \\ \in \rho_{\mathrm{df}}^{-m} \rho_{\mathrm{bf}}^{-\mathsf{s}} \rho_{\natural\mathrm{f}}^{-\ell}   \rho_{\mathrm{par}}^{-q} L^\infty({}^{\mathrm{par}}\overline{T}^* \bbM \times [0,1]_h) \text{ for all } M \in \mathbb{N}. \Big\}
\end{multline}

Let the corresponding operators, obtained from the standard quantization, be denoted $\tilde \Psi_{\mathrm{par,I, res}}^{m,\mathsf{s},l, q}$. (The tilde is used as we will modify the class of symbols, and hence operators, shortly.) 

We claim the following:
\begin{proposition}\label{eq:parIres calculus}
The operators $\tilde \Psi_{\mathrm{par,I, res}}^{\bullet, \bullet, \bullet, \bullet}$ form a multi-graded algebra, i.e. if $A \in \tilde \Psi_{\mathrm{par,I, res}}^{m,\mathsf{s},\ell, q}$ and $B \in  \tilde \Psi_{\mathrm{par,I, res}}^{m',\mathsf{s}',\ell', q'}$, then $C = A \circ B$ satisfies 
\begin{equation}
	C\in \tilde \Psi_{\mathrm{par,I, res}}^{m+m',\mathsf{s} + \mathsf{s}',\ell+\ell', q+q'}.
\end{equation}
If the left reduced symbols are $a$, $b$ and $c$ respectively, then we have an expansion similar to \eqref{eq:symbol exp parI}:  
\begin{multline}\label{eq:symbol exp parIres}
	c(z,\zeta,h) \sim  \sum_{|\alpha| \leq k-1} \frac{i^{|\alpha|}}{ \alpha! } \big(D_\zeta^\alpha a(z,\zeta,h) \big) D_{z}^{\alpha} b(z,\zeta,h)  
	+ r(z,\zeta,h), \\ r \in S_{\mathrm{par, I,res}}^{m+m'-k,\mathsf{s}+\mathsf{s}' -k+\delta, \ell + \ell' - k, q+q'}(\bbM) \ \forall \delta > 0. 
\end{multline}
Moreover, we have 
\begin{equation}\label{eq:when_parIres_=_parI}
\tilde \Psi_{\mathrm{par,I, res}}^{m,\mathsf{s},m+q, q} = \Psi_{\mathrm{par,I}}^{m,\mathsf{s}, q}.
\end{equation}
\end{proposition}

We will deduce this using the following lemma on lifting  functions under blowup. 
Let $X$ be a manifold with corners with boundary hypersurfaces $\{H_0,...,H_\ell\}$ among which $H_1, H_2$ are two intersecting boundary hypersurfaces of $X$. 
Let $H_0$ be a boundary hypersurface other than $H_1,H_2$ and $\rho_0$ is the boundary defining function of it. 
Define
\begin{equation} \label{eq:tildeVb-general}
\tilde{\calV}_{\mathrm{b}}(X) = \bigcup_{\delta>0} \rho_0^\delta \calV_{\mathrm{b}}(X).
\end{equation}
Let $\mathcal{A}^{0,0,..,0}(X)$ (where each index $0$ corresponds to a boundary hypersurface) denote the space of functions that is smooth in the interior of $X$ and remain bounded under repeated application of elements in $\tilde{\calV}_{\mathrm{b}}(X)$. 
Fixing a finite generating set of $\calV_{\mathrm{b}}(X)$, $\mathcal{A}^{0,0,..,0}(X)$ is equipped with the Fr\'echet topology induced from $L^\infty$-norm of products of $\rho_0^{\delta_j}$ times those vector fields applied to functions, where $\delta_j \to 0$ as $j \to \infty$.
Then we set
\begin{equation}
\mathcal{A}^{m_0,...,m_\ell}(X)
= \Pi_{i=0}^\ell \rho_i^{m_i} \mathcal{A}^{0,0,..,0}(X),
\end{equation}
where $m_0 \in C^\infty(X)$ and all other $m_i$ are constants. This has a Fr\'echet induced from that of $\mathcal{A}^{0,0,..,0}(X)$.

\begin{remark}
We write this part in this slightly abstract way so that it applies to various phase spaces that we are using. In applications, we always take $H_0$ to be the spacetime infinity, $\mathrm{bf}$, and $H_1,H_2$ are boundary faces intersecting at the place where we perform resolutions. 
\end{remark}

Now we consider 
\begin{equation}
Y = [X; H_1 \cap H_2], \qquad \beta : Y \to X,
\end{equation}
which is the manifold with corners obtained by blowing up the codimension two corner $H_1 \cap H_2$, together with its blowdown map $\beta$.
We define $\mathcal{A}^{m_0,m_1, m_2, m_3,\dots}(Y)$
in the same manner as above with $\calV_{\mathrm{b}}(X)$ (resp. $\tilde{\calV}_{\mathrm{b}}(X)$) replaced by $\calV_{\mathrm{b}}(Y)$ (resp. $\tilde{\calV}_{\mathrm{b}}(Y)$),
with weight $m_0$ at the lift of $H_0$, $m_1$ at the lift of $H_1$, $m_2$ at the lift of $H_2$, $m_3$ at the new boundary hypersurface obtained by replacing $H_1 \cap H_2$ by its inward-pointing spherical normal bundle, and weights $\mathsf{m}$ at the other boundary hypersurfaces. 
Again, we allow $m_0 \in C^\infty(Y)$ but all other orders are constants.

\begin{lemma}\label{lemma:blow up conormal} Let $X,Y$ be as above, then  we have
\begin{equation}\label{eq:conormal_lift}
\beta^* \mathcal{A}^{m_0,m_1, m_2, \dots }(X) = \mathcal{A}^{\beta^*m_0,m_1, m_2, m_1 + m_2, \dots}(Y).
\end{equation}
The key point is the equality, not mere containment, in \cref{eq:conormal_lift}.
\end{lemma}

\begin{proof} The essence of the proof is that a spanning set of b-vector fields on $X$ lifts to a spanning set of b-vector fields on $Y$. This is easily checked in local coordinates. 
This tells us that a generating set of $\rho_0^\delta \calV_{\mathrm{b}}(X)$ lifts to a generating set of $\rho_0^\delta \calV_{\mathrm{b}}(Y)$, which gives the desired equivalence of seminorms induced from $\tilde{\calV}_{\mathrm{b}}(X)$ and $\tilde{\calV}_{\mathrm{b}}(Y)$, which in turn gives \cref{eq:conormal_lift}.
\end{proof}

\begin{proof}[Proof of Proposition~\ref{eq:parIres calculus}]
Using \Cref{lemma:blow up conormal}, we see that symbols of order $(m,\mathsf{s}, q)$ on ${}^{\mathrm{par,I}}\overline{T}^* \bbM$ correspond exactly to symbols of order $(m,\mathsf{s}, m+q, q)$ on ${}^{\mathrm{par,I}}\overline{T}^* \bbM$. This immediately implies \cref{eq:when_parIres_=_parI}.

 More generally, symbols in $\tilde S_{\mathrm{par,I, res}}^{m, \mathsf{s}, l, q}(\bbM)$ are contained in $S_{\mathrm{par,I}}^{M, \mathsf{s}, q}(\bbM)$, where \begin{equation} 
 	M = \max(m, l-q).
 \end{equation} 
 It follows that $AB$ is contained in $\Psi_{\mathrm{par,I}}^{M+M', \mathsf{s}+\mathsf{s}', q+q'}(\bbM)$, where $M' = \max(m', l'-q')$. This is a weaker composition property  than claimed in Proposition~\ref{eq:parIres calculus}. However, we can apply the symbol expansion \eqref{eq:symbol exp parI}, and since differential operators $D_z$ reduce the order at the spacetime face, while differential operators $D_{\zeta}$ reduce the order at the differential and natural faces, the error term $r$ is eventually in 
 \begin{equation} 
 	S_{\mathrm{par,I}}^{M+M'-N, \mathsf{s}+\mathsf{s}'-N+\delta, q+q'}(\bbM),
 \end{equation} 
 which is contained in $\tilde S_{\mathrm{par,I,res}}^{m+m', \mathsf{s}+\mathsf{s}', l+l', q+q'}(\bbM)$ for large enough $N$. 
 
 Thus the quantization of the remainder term for large enough $N$ is contained in $\tilde \Psi_{\mathrm{par,I, res}}^{m+m',\mathsf{s} + \mathsf{s}',l+l', q+q'}$ --- indeed, it is contained in $\tilde \Psi_{\mathrm{par,I, res}}^{m+m' - N',\mathsf{s} + \mathsf{s}' - N',l+l' - N', q+q'}$ for large enough $N$ depending on $N'$. 
 
 On the other hand, the $k$th term in the symbol expansion \eqref{eq:symbol exp parI} is in $\tilde S_{\mathrm{par,I, res}}^{m+m'-k,\mathsf{s} + \mathsf{s}' -k+\delta,l+l'-k, q+q'}$, for every $\delta > 0$. It follows that $c$ is in $\tilde S_{\mathrm{par,I, res}}^{m+m',\mathsf{s} + \mathsf{s}',l+l', q+q'}$ and hence $C$ is in $\tilde \Psi_{\mathrm{par,I, res}}^{m+m',\mathsf{s} + \mathsf{s}',l+l', q+q'}$ and the expansion \eqref{eq:symbol exp parIres} is valid.  
\end{proof}

Let us pause to note a useful consequence of \cref{eq:when_parIres_=_parI}: it implies a uniform control on the operator norm of $A \in \tilde \Psi_{\mathrm{par,I,res}}^{m,\beta^*\mathsf{s},\ell,q}$ on parabolic Sobolev spaces. Namely, when $m'\geq m$, $\beta^* \mathsf{s} \leq	\mathsf{s}'$, $q'\geq q$, $m'+q'\geq \ell$ then 
\begin{equation}\label{eq:action on par Sob spaces}
A \in \tilde \Psi_{\mathrm{par,I,res}}^{m,\beta^*\mathsf{s},\ell,q} \Rightarrow  \lVert A(h) \rVert_{H_{\mathrm{par}}^{m'',\mathsf{s}''}\to H_{\mathrm{par}}^{m''-m',\mathsf{s}''-\mathsf{s}' } }  \leq C h^{-q'}
\end{equation}
uniformly for $h_0\in (0,1)$ and $m'' \in \bbR$, $\mathsf{s}', \mathsf{s}'' \in C^\infty({}^{\mathrm{par,I}}\overline{T}^* \bbM)$.

In fact, we will work with a slightly smaller class of $\mathrm{par, I, res}$-symbols that is classical, that is, smooth up to an overall power, at the parabolic face $\mathrm{pf}$. This will allow us to define a normal operator at $\mathrm{pf}$. These operators will be denoted $\Psi_{\mathrm{par,I, res}}^{m,\mathsf{s},l, q}$ (with no tilde), and we will show that this is also a multi-graded algebra.

We define the symbol class $S_{\mathrm{par,I,res}}^{m,\mathsf{s},\ell,q}$ (with no tilde) by 
\begin{equation}
S_{\mathrm{par,I,res}}^{m,\mathsf{s},\ell,q} = \Big\{ a \in \tilde{S}_{\mathrm{par,I,res}}^{m,\mathsf{s},\ell,q}: \Big(\frac{\mathrm{d}}{\mathrm{d}h}\Big)^k a \in \tilde{S}_{\mathrm{par,I,res}}^{m,\mathsf{s} + \delta,\ell+k,q} 
\text{ for all } \delta>0, k \in \mathbb{N} \Big\}, 
\end{equation}
where $\mathrm{d}/\mathrm{d} h$ is the partial derivative with respect to $h$ while fixing the spatial and frequency variables $\tau,\xi$.
This condition enforces smoothness at  $\mathrm{pf}$ but not at the natural face, due to the increase in order by $k$ at the natural face when we differentiate $k$ times  in $h$. 
Its quantization is by definition the operator class $\Psi_{\mathrm{par,I,res}}^{m,\mathsf{s},\ell,q} = \operatorname{Op}( S^{m,\mathsf{s},\ell,q}_{\mathrm{par,I,res}} )$, and we set
\begin{equation} 
\Psi_{\mathrm{par,I,res}} = \bigcup_{m,s,\ell,q \in \mathbb{Z}} \Psi_{\mathrm{par,I,res}}^{m,\mathsf{s},\ell,q}.
\end{equation}
By definition $\Psi_{\mathrm{par,I,res}}^{m,\mathsf{s},\ell,q}$ consists of operators $A$ such that
\begin{equation}
A \in \tilde{\Psi}_{\mathrm{par,I,res}}^{m,\mathsf{s},\ell,q} \ , \quad \Big(\frac{\mathrm{d}}{\mathrm{d}h}\Big)^k A \in \bigcap_{\delta>0} \tilde{\Psi}_{\mathrm{par,I,res}}^{m,\mathsf{s}+\delta,\ell+k,q}.
\end{equation}

We verify that this is a sub-algebra of $\tilde{\Psi}_{\mathrm{par,I,res}}$. The only point that needs a verification is that it is closed under the composition. Suppose 
\begin{equation} \label{eq:A,B_in_small_calculus}
A \in \Psi_{\mathrm{par,I,res}}^{m,\mathsf{s},\ell,q}, B \in \Psi_{\mathrm{par,I,res}}^{m',\mathsf{s}',\ell',q'}, 
\end{equation}
and we need to show
\begin{equation}
\Big(\frac{\mathrm{d}}{\mathrm{d}h}\Big)^k (A \circ B) \in \bigcap_{\delta>0} \tilde{\Psi}_{\mathrm{par,I,res}}^{m,\mathsf{s}+\delta,\ell+k,q}.
\end{equation}
But the left hand side is a sum of terms of the form
\begin{equation}
\Big(\frac{\mathrm{d}}{\mathrm{d}h}\Big)^{j}A \circ \Big(\frac{\mathrm{d}}{\mathrm{d}h}\Big)^{k-j}B,
\end{equation}
and by assumption \eqref{eq:A,B_in_small_calculus}, we have
\begin{equation}
\Big(\frac{\mathrm{d}}{\mathrm{d}h}\Big)^{j}A \in \bigcap_{\delta>0} \tilde{\Psi}_{\mathrm{par,I,res}}^{m,\mathsf{s}+\delta,\ell+j,q},
\quad  \Big(\frac{\mathrm{d}}{\mathrm{d}h}\Big)^{k-j}B \in \bigcap_{\delta>0} \tilde{\Psi}_{\mathrm{par,I,res}}^{m',\mathsf{s}'+\delta,\ell'+k-j,q'},
\end{equation}
which shows $(\frac{\mathrm{d}}{\mathrm{d}h})^{j}A \circ (\frac{\mathrm{d}}{\mathrm{d}h})^{k-j}B \in \bigcap_{\delta>0} \tilde{\Psi}_{\mathrm{par,I,res}}^{m+m',\mathsf{s}+\mathsf{s}'+\delta,\ell+\ell'+k,q+q'}$, as desired.

The advantage of the $\Psi_{\mathrm{par,I,res}}$ calculus is that there is a well-defined normal operator at the parabolic face. The normal operator $N(A)$ is defined for $A \in \Psi^{m, \mathsf{s}, \ell, 0}_{\mathrm{par,I,res}}$ and is given by restricting the full (left-reduced) symbol of $A$ to the parabolic face, and then quantizing to obtain $N(A) \in \Psi_{\mathrm{par}}^{\ell, \mathsf{s}}$. This is well-defined, since the symbol is by definition smooth at the parabolic face and so can be restricted to this face, and the restriction belongs in the parabolic symbol class of order $(\ell, \mathsf{s})$. 

\begin{proposition}\label{prop:N mult parIres} The normal operator in the $\Psi_{\mathrm{par,I,res}}$ calculus is multiplicative, that is, for $A \in \Psi^{m, \mathsf{s}, \ell, 0}_{\mathrm{par,I,res}}$, $B \in \Psi^{m', \mathsf{s}', \ell', 0}_{\mathrm{par,I,res}}$ we have 
\begin{equation}\label{eq:mult normal op parIres}
N(AB) = N(A) N(B) \in \Psi_{\mathrm{par}}^{\ell + \ell', \mathsf{s} + \mathsf{s}'}.
\end{equation}
\end{proposition}

\begin{proof}
Consider the action of $A \in \Psi^{m, \mathsf{s}, \ell, 0}_{\mathrm{par,I,res}}$ on a Schwartz function $\phi$; we claim that this is a $C^\infty$ function of $h$ with values in $\mathcal{S}(\bbR^{1,d})$. To see this, we recall that, as above, 
\begin{equation} 
\Big(\frac{\mathrm{d}}{\mathrm{d}h}\Big)^{j}A \in \bigcap_{\delta>0} \tilde{\Psi}_{\mathrm{par,I,res}}^{m,\mathsf{s}+\delta,\ell+j,0} \subset \tilde{\Psi}_{\mathrm{par,I}}^{\max(m, \ell+j),s, 0}
\end{equation} 
provided that $s > \mathsf{s}(h)$ for every $h$. 

This is therefore a conormal function of $h$ with values in $\Psi_{\mathrm{par}}^{\max(m, \ell+j),s, 0}$ and its action on $\phi$ produces a conormal function of $h$ with values in $\mathcal{S}(\bbR^{1,d})$. Therefore, $A \phi$ is a $C^{j-1}$ function of $h$ with values in $\mathcal{S}(\bbR^{1,d})$, and since this is true for arbitrary $j$, the claim is proved. 

We next claim that $N(A) \phi$ coincides with $A\phi$ evaluated at $h=0$. To show this, we regard $N(A)$ as an element of $\Psi_{\mathrm{par,I}}^{\ell, \mathsf{s}, 0}$ independent of $h$. As such, it is also an element of $\Psi_{\mathrm{par,I,res}}^{\ell, \mathsf{s}, \ell, 0}$. Since it agrees with $A$ to leading order at the parabolic face, and since the symbol of $A$ is smooth at the parabolic face, we have 
\begin{equation} 
A - N(A) \in \Psi_{\mathrm{par,I,res}}^{\max(m,\ell), \mathsf{s}, \ell, -1}.
\end{equation} 
We can remove a factor of $h$ and we find that 
\begin{equation} 
(A - N(A)) \in h \Psi_{\mathrm{par,I,res}}^{\max(m,\ell), \mathsf{s}, \ell +1, 0}.
\end{equation}
Following the reasoning in the previous paragraph, we see that $A \phi - N(A) \phi$ is $O(h)$, which establishes the second claim. 

Finally we see that $AB \phi = A (N(B) \phi + O(h)) = N(A) N(B) \phi + O(h)$, and this shows that $N(AB) = N(A) N(B)$. 
\end{proof}

\subsection{\texorpdfstring{The resolved natural phase space ${}^{\calc}\overline{T}^* \bbM$}{The resolved natural phase space} } \label{subsec:phase}
We now turn to the main task of this section, developing the theory of the $\calc$-calculus.
First, we give a careful definition of the partially compactified phase space 
\begin{equation}
	{}^{\calc}\overline{T}^* \bbM \hookleftarrow \mathbb{R}^{1,d}_{t,x} \times \bbR^{1,d}_{\tau,\xi} \times [0,\infty)_h.
\end{equation}
We sketched
its definition already in the introduction: this phase space is the result of blowing up, in a quasihomogeneous way, the portion of the zero section 
\begin{equation} 
	{}^{\calczero} o^* \bbM = \mathrm{cl}_{{}^{\calczero}\overline{T}^* \bbM} \{\xi=0=\tau\} = \{\xi_{\calcshort} = 0 = \tau_{\calcshort} \}
\end{equation} 
in $\{h=0\}$. 
``Quasihomogeneous'' roughly means that the blowup is done parabolically in the $\tau_{\calcshort}$ direction relative to the $\xi_{\calcshort}$ directions. There are several ways of making this notion precise. As it is convenient later when constructing $\Psi_{\calc}$, we will use a gluing construction.

In this subsection, and only in this subsection, we will refer to the $\{h=0\}$ boundary hypersurface of the $\calczero$-phase space as $\mathrm{zf}$. Elsewhere, we use ``$\natural\mathrm{f}$,'' but while constructing the $\calc$-phase space it is useful to distinguish the $\{h=0\}$ face of the $\calczero$-phase space from its lift to the $\calc$-phase space.

\begin{figure}[t]
	\includegraphics{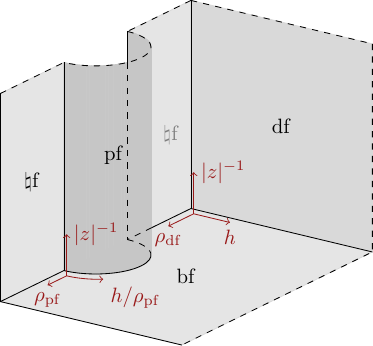}
	\quad
	\includegraphics{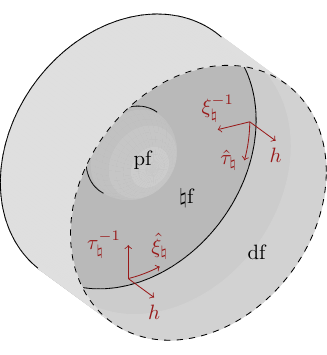}
	\caption{(Left) The combinatorial structure of the phase space ${}^{\calc}\overline{T}^* \bbM$, showing the various boundary hypersurfaces and their intersections. 
	(Right) The same phase space in the case with one spatial dimension, showing both frequency dimensions and suppressing both spacetime dimensions. Here, $\smash{\hat{\tau}_{\calczero}} =\tau_{\calcshort}/\xi_{\calcshort}$ and $\smash{\hat{\xi}_{\calcshort} }= \xi_{\calcshort}/\tau_{\calcshort}$.
	}
	\label{fig:phase}
\end{figure}

Recall from Section~\ref{subsec:parIres} the $\mathrm{par, I, res}$- phase space ${}^{\mathrm{par,I,res}}\overline{T}^* \bbM$. 
We record a system of coordinate charts covering $\mathrm{pf}\cap \mathrm{ff}$ in ${}^{\mathrm{par,I,res}}\overline{T}^* \bbM$. One way to find such a system is to note that a neighborhood of $\mathrm{pf}\cap\mathrm{ff}$ is given by 
\begin{equation}
U=\bbM\times \overline{(\bbR^{1,d}_{\tau,\xi  })}_{\mathrm{par}} \times [0,1)_{h/\varrho},
\end{equation}
where $\varrho$ is a global boundary-defining-function for the boundary of the frequency factor. So, by using the atlas given in \cref{eq:misc_006}, \cref{eq:misc_007} for the first factor, and choosing $\varrho$ to be given locally by something convenient, we get:
\begin{enumerate}[label=(\Roman*)]
	\item $\bbM\times  [0,\infty)_{1/ |\tau|^{1/2}} \times \bbR^d_{\xi/ |\tau|^{1/2}} \times [0,1)_{h|\tau|^{1/2}}$ is a smooth coordinate chart covering an open half of $U$ (one chart for the $\tau>0$ half, one for the $\tau<0$ half), 
	\label{it:parres_chart_1}
	\item $\bbM\times  [0,\infty)_{1/|\xi_k|}\times \bbR_{\tau/\xi_k^2}\times \bbR^{d-1}_{\hat{\xi}_k} \times [0,1)_{h |\xi_k|}$, as $k$ varies, covers the rest. 
	\label{it:parres_chart_2}
\end{enumerate}

Combinatorially, ${}^{\mathrm{par,I,res}}\overline{T}^* \bbM$ resembles the desired ${}^{\calc} \overline{T}^* \bbM$. 
But they are different because the former
has the ``wrong'' structure, as a compactification of $T^* \bbR^{1,d}\times (0,1)_h$, at $\mathrm{df}_1$. For example, for $h_0>0$,
\begin{equation}
{}^{\mathrm{par,I,res}}\overline{T}^* \bbM\cap \{h=h_0\} = 
{}^{\mathrm{par,I}}\overline{T}^* \bbM\cap \{h=h_0\} = {}^{\mathrm{par}} \overline{T}^* \bbM, 
\end{equation}
whereas we wanted the cross-section $\{h=h_0\}$ of the $\calc$-phase space to be ${}^{\mathrm{sc}}\overline{T}^* \bbM$. In the former, but not in the latter, $\aang{\zeta}^{-1}=(1+\tau^2+\xi^4)^{-1/4}$ is smooth.
In order to fix this, we replace a neighborhood 
\begin{equation} 
U\supset \mathrm{df}_1
\end{equation} 
of $\mathrm{df}_1$ in ${}^{\mathrm{par,I,res}}\overline{T}^* \bbM$ with a neighborhood of $\mathrm{df}$ in ${}^{\calczero}\overline{T}^* \bbM$. To do this, we use the following lemma:
\begin{lemma} In the sense of partial compactifications,  ${}^{\mathrm{par,I,res}}\overline{T}^* \bbM \backslash (\mathrm{pf}\cup \mathrm{df}_1) = {}^{\calczero } \overline{T}^* \bbM \backslash (\mathrm{df} \cup( {}^{\calczero}o^* \bbM \cap \mathrm{zf}))$. 
	\label{lem:phase_space_gluing}
\end{lemma}
Thus, ${}^{\mathrm{par,I,res}}\overline{T}^* \bbM$ is indistinguishable and ${}^{\calczero } \overline{T}^* \bbM$ except at fiber infinity and at the zero section of ${}^{\calczero } \overline{T}^* \bbM$ over $\{h=0\}$.
\begin{proof}
	We first show that the identity map on the interior extends to an immersion. Note that the decomposition in \cref{eq:misc_005} gives an atlas 
	\begin{multline}
		{}^{\mathrm{par,I,res}}\overline{T}^* \bbM = {}^{\mathrm{par,I}}T^* \bbM \cup \Big( \bigcup_\pm ({}^{\mathrm{par,I,res}}\overline{T}^* \bbM \cap \{\pm \tau>0\}) \Big) \\ 
		\cup \Big( \bigcup_{\pm, k\in \{1,\ldots,d\}} ( {}^{\mathrm{par,I,res}}\overline{T}^* \bbM \cap \{\pm \xi_k>0\}) \Big).
		\label{eq:misc_010}
	\end{multline}
	More precisely, letting $\beta: {}^{\mathrm{par,I,res}}\overline{T}^* \bbM \to {}^{\mathrm{par,I}}\overline{T}^* \bbM$ denote the blowdown map and 
	\begin{equation} 
		\pi: {}^{\mathrm{par,I}}\overline{T}^* \bbM \to  \overline{(\bbR^{1,d}_{\tau,\xi  })}_{\mathrm{par}} 
	\end{equation} 
	denote the projection onto the frequency variables, each term in \cref{eq:misc_010} is the preimage of the corresponding term in \cref{eq:misc_005} under the map $\pi\circ \beta$. 
	
	We now consider each term on the right-hand side of \cref{eq:misc_010}, after subtracting off $\mathrm{pf} \cup \mathrm{df}_1$:
	\begin{itemize}
		\item First consider $({}^{\mathrm{par,I}}T^* \bbM ) \backslash (\mathrm{pf} \cup \mathrm{df}_1)=({}^{\mathrm{par,I}}T^* \bbM ) \backslash \{h=0\}$. This is just $(0,\infty)_h\times \bbM_z\times \bbR^{1,d}_{\tau,\xi}$. 
		
		Since ${}^{\calczero}\overline{T}^* \bbM \backslash \{h=0\}$ is just ${}^{\mathrm{sc}}\overline{T}^* \bbM = (0,1)_h\times \bbM_z\times \bbR^{1,d}_{\tau,\xi}$,  we have $({}^{\mathrm{par,I}}T^* \bbM ) \backslash \{h=0\} = {}^{\calczero}\overline{T}^* \bbM \backslash \{h=0\}$.
		\item Now consider $({}^{\mathrm{par,I,res}}\overline{T}^* \bbM \backslash (\mathrm{pf}\cup \mathrm{df}_1)) \cap \{\pm \tau>0\}$. We discuss the $+$ case, the other being similar. Here, \cref{eq:misc_006} yields 
		\begin{equation} 
			({}^{\mathrm{par,I,res}}\overline{T}^* \bbM \backslash (\mathrm{pf}\cup \mathrm{df}_1)) \cap \{\tau>0\} = \bbM_z \times \bbR^+_{h^2 \tau } \times \bbR^d_{\xi/\tau^{1/2}} \times [0,1)_h,
			\label{eq:j4w1dc}
		\end{equation} 
		in the sense of compactifications (cf.\ \cref{it:parres_chart_1} above).\footnote{Here, we are excluding $\mathrm{pf}$ from consideration, which is why it is permissible to use $h^2 \tau$ as a final coordinate instead of $h\tau^{1/2}$. It is the latter that is a boundary-defining-function of $\mathrm{pf}$ in ${}^{\calc}\overline{T}^* \bbM$.} 
		But, $\bbR^+_{h^2 \tau } \times \bbR^d_{\xi/\tau^{1/2}} = \bbR^+_{h^2 \tau} \times \bbR^d_{h \xi}$, and 
		\begin{equation} 
			\bbM_z\times  \bbR^+_{h^2 \tau} \times \bbR^d_{h \xi} \times [0,1)_h = ({}^{\calczero}\overline{T}^* \bbM \backslash (({}^{\calczero}o^* \bbM \cap \mathrm{zf})\cup \mathrm{df} )) \cap \{\tau>0\} 
		\end{equation}  
		in the sense of compactifications.
		\item  
		Finally, consider ${}^{\mathrm{par,I,res}}\overline{T}^* \bbM \cap \{\pm \xi_k>0\}$. Again, we discuss the $+$ case, the other being similar. In the stated sector, \cref{eq:misc_007} yields 
		\begin{equation}
			({}^{\mathrm{par,I,res}}\overline{T}^* \bbM  \backslash (\mathrm{pf}\cup \mathrm{df}_1 ))\cap \{\pm \xi_k>0\} = \bbM_z\times \bbR^+_{h\xi_k} \times \bbR_{\tau/\xi_k^2}\times  \bbR^{d-1}_{\hat{\xi}_k}\times [0,1)_h 
		\end{equation}
		in the sense of compactifications (cf.\ \cref{it:parres_chart_2} above).
		Now use 
		\begin{equation} 
		\bbR^+_{h\xi_k}  \times \bbR_{\tau/\xi_k^2}\times  \bbR^{d-1}_{\hat{\xi}_k} = \bbR_{h^2 \tau}\times \bbR^{d}_{h\xi}
		\end{equation} 
		in the sense of compactifications. As 
		\begin{equation}
			\bbM_z \times \bbR_{h^2 \tau}\times \bbR^{d}_{h\xi} \times [0,1)_h = {}^{\calczero}\overline{T}^* \bbM \backslash (({}^{\calczero}o^* \bbM \cap\mathrm{zf}) \cup \mathrm{df} ) \cap \{\xi_k >0\}
		\end{equation}
		in the sense of compactifications, we conclude the desired result. 
	\end{itemize}
	So, the identity map on the interior of ${}^{\mathrm{par,I,res}}\overline{T}^* \bbM \backslash (\mathrm{pf}\cup \mathrm{df}_1)$ extends to an immersion, as claimed. 
	
	Reversing the argument, one concludes that the identity map on the interior of ${}^{\mathrm{par,I,res}}\overline{T}^* \bbM \backslash (\mathrm{pf}\cup \mathrm{df}_1)$ extends to an immersion in the other direction. So the immersion constructed previously is actually a diffeomorphism. 
\end{proof}

By \Cref{lem:phase_space_gluing}, we can glue the two mwcs ${}^{\mathrm{par,I,res}}\overline{T}^* \bbM \backslash  \mathrm{df}_1$, ${}^{\calczero}\overline{T}^* \bbM \backslash {}^{\calczero}o^* \bbM \cap \mathrm{zf}$ together, giving us a new mwc which we call ${}^{\calc}\overline{T}^* \bbM$:
\begin{equation} \label{eq:calc_phase_space_definition}
	{}^{\calc}\overline{T}^* \bbM = (({}^{\mathrm{par,I,res}}\overline{T}^* \bbM \backslash  \mathrm{df}_1)\sqcup ({}^{\calczero}\overline{T}^* \bbM \backslash ({}^{\calczero}o^* \bbM \cap \mathrm{zf}) )) / \!\sim, 
\end{equation}
where $\sim$ identifies the points identified by the diffeomorphism discussed in the previous lemma. The smooth structure is inherited from its constituents, as is the inclusion 
\begin{equation} 
T^* \bbR^{1,d}\times [0,\infty)_h \hookrightarrow 
{}^{\calc}\overline{T}^* \bbM
\end{equation} 
needed to make ${}^{\calc}\overline{T}^* \bbM$ a partial compactification of $T^* \bbR^{1,d}\times [0,\infty)_h$.

The only boundary hypersurface of ${}^{\calc}\overline{T}^* \bbM$ which has not been defined so far is 
\begin{equation} 
	\natural\mathrm{f}=((\mathrm{zf}\backslash {}^{\calczero}o^* \bbM  )\sqcup (\mathrm{ff} \backslash \mathrm{df}_1) )/\!\sim,
\end{equation} 
i.e.\ the closure of $\mathrm{zf} \backslash {}^{\calczero}o^* \bbM     \subseteq {}^{\calczero} T^* \bbM$ in ${}^{\calc} \overline{T}^* \bbM$. 

\begin{figure}[t]
	\includegraphics[scale=.95]{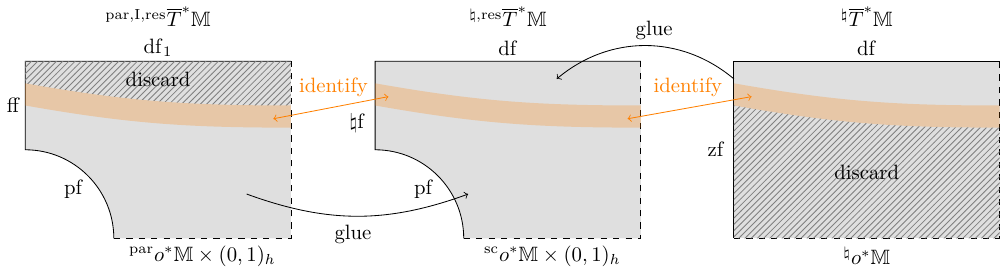}
	\caption{Illustration of the gluing procedure used to construct ${}^{\calc} \overline{T}^* \bbM$. A neighborhood of ${}^{\calczero} \overline{T}^* \bbM$  replaces a neighborhood of $\mathrm{df}_1$ in ${}^{\mathrm{par,I,res}}\overline{T}^* \bbM$. In order to accomplish this, it is important to be able to identify neighborhoods of compact subsets of the interior of $\mathrm{ff}$ and $\mathrm{zf}$, e.g.\ the orange set in the figure. This identification is done using \Cref{lem:phase_space_gluing}.
	Thus, we have canonical identifications $\natural\mathrm{f}^\circ = \mathrm{ff}^\circ = \mathrm{zf}^\circ$. Moreover, $\natural\mathrm{f}\backslash \mathrm{pf} = \mathrm{zf} \backslash {}^{\natural}o^* \bbM$ and $\natural\mathrm{f}\backslash \mathrm{df} = \mathrm{ff}\backslash \mathrm{df}_1$. 
	}
\end{figure}

We have already given boundary-defining-functions of $\mathrm{df}$ and $\mathrm{bf}$ (the latter omitted from most of the figures above). In addition:
\begin{itemize}
	\item We can take 
	\begin{equation} 
		\rho_{\mathrm{pf}} = \frac{h}{h+\chi(\tau_{\natural},\xi_{\natural} ) (1+\tau^2+\xi^4)^{-1/4}} = \frac{h}{h+\chi(\zeta_{\natural}) \aang{\zeta}^{-1} }
	\end{equation} 
	as a boundary-defining function of $\mathrm{pf}$, where $\chi\in C_{\mathrm{c}}^\infty(\bbR^{1,d})$, $0\notin \operatorname{supp}(1-\chi)$.  Indeed, this follows from a standard construction of boundary-defining-functions for the boundary hypersurfaces of mwcs constructed by blowing up a corner. Note that the cutoff $\chi$ ensures smoothness at $\mathrm{df}$.
	\item  We can take $h+\chi(\tau_{\natural},\xi_{\natural} ) (1+\tau^2+\xi^4)^{-1/4}$ as a boundary-defining-function of $\natural\mathrm{f}$, for similar reasons.
\end{itemize}

To get a better understanding of the structure of ${}^{\calc}\overline{T}^* \bbM$, note the following immediate consequence of the gluing construction:
\begin{proposition}
	The boundary hypersurface $\mathrm{pf} \subset{}^{\calc}\overline{T}^* \bbM$ is canonically diffeomorphic to ${}^{\mathrm{par}} \overline{T}^* \bbM$, and, for each $h_0>0$, the slice ${}^{\calc}\overline{T}^* \bbM\cap \{h=h_0\}$ is canonically diffeomorphic
	to ${}^{\mathrm{sc}} \overline{T}^* \bbM$. 
	\label{prop:pf=par_phase_space}
\end{proposition}

The identification $\mathrm{pf}\cong {}^{\mathrm{par}} \overline{T}^* \bbM$ was remarked on already in \cref{eq:pf_id}.

\subsection{\texorpdfstring{The resolved natural calculus $\Psi_{\calc}$}{The resolved natural calculus}}\label{subsec:calc}

We now explain the construction of the pseudodifferential calculus $\Psi_{\calc}$.
We give two different constructions, each of which is useful in different contexts.
One construction will define $\Psi_{\calc}$ as the image of $\operatorname{Op}$, where $\operatorname{Op}$ is the Kohn--Nirenberg quantization scheme on $\bbR^{1+d}$, of a certain class of one-parameter families of symbols on $\bbR^{1+d}$. Thus, $\Psi_{\calc}$ (in the case without variable orders) is just some set of one-parameter families of elements of the H\"ormander's uniform calculus $\Psi_\infty(\bbR^{1+d})$: 
\begin{equation}
A \in \Psi_{\calc} \Rightarrow A(h) \in \Psi_\infty \text{ for all }h>0,
\end{equation}
and in fact $A(h) \in \Psi_{\mathrm{sc}}$ for all $h>0$. (When working with variable orders, it is necessary, as is usually true, to work with a slightly larger class of operators.)
The other construction will define $\Psi_{\calc}$ by ``gluing together'' elements of the more basic pseudodifferential calculi $\Psi_{\calczero}$ (see \S\ref{subsec:natural_calculus}) and $\Psi_{\mathrm{par,I,res}}$ (see \S\ref{subsec:parIres}), each of which is associated naturally to one of the compactified phase spaces that was used to construct ${}^{\calc}\overline{T}^* \bbM$.

We begin by presenting the first construction. Let $\mathcal{V}_{\calc}$ be the Lie algebra of smooth vector fields on ${}^{\calc} \overline{T}^* \bbM$ that are tangent to all boundary faces except for $\mathrm{pf}$. This means that we are imposing smoothness at $\mathrm{pf}$ but only conormal regularity at the other faces. 
We denote
\begin{equation} \label{eq:tildeV-calc}
\tilde{\mathcal{V}}_{\calc} = \cup_{\delta > 0} \rho_{\mathrm{bf}}^\delta \mathcal{V}_{\calc}.
\end{equation}
Then we define 
\begin{multline} \label{eq:def-calc-symbol}
S_{\calc}^{m,\mathsf{s},\ell,q} = \big\{ a \in \rho_{\mathrm{df}}^{-m} \rho_{\mathrm{bf}}^{-\mathsf{s}} \rho_{ \calczero\mathrm{f} }^{-\ell} \rho_{\mathrm{pf}}^{-q} L^\infty({}^{\calc}\overline{T}^* \bbM) : \big( \tilde{\mathcal{V}}_{\calc}  \big)^M a  \in \rho_{\mathrm{df}}^{-m} \rho_{\mathrm{bf}}^{-\mathsf{s}} \rho_{ \calczero\mathrm{f} }^{-\ell} \rho_{\mathrm{pf}}^{-q} L^\infty({}^{\calc}\overline{T}^* \bbM ) \\ \text{ for all } M \in \mathbb{N} \big\}.
\end{multline}
Moreover, we define the space of classical symbols by 
\begin{equation} \label{eq:def-calc-symbol-classical}
S_{\calc, \mathrm{cl}}^{m,\mathsf{s},\ell,q} =  \rho_{\mathrm{df}}^{-m} \rho_{\mathrm{bf}}^{-\mathsf{s}} \rho_{ \calczero\mathrm{f} }^{-\ell} \rho_{\mathrm{pf}}^{-q} C^\infty({}^{\calc}\overline{T}^* \bbM).
\end{equation}
It is not hard to check that $S_{\calc, \mathrm{cl}}^{m,\mathsf{s},\ell,q}$ is contained in $S_{\calc}^{m,\mathsf{s},\ell,q}$ (and that the factor $\rho_{\mathrm{bf}}^\delta$ in \eqref{eq:tildeV-calc} is essential for this to be true). 

In this paper, we will only work with variable $\mathsf{s}$ (recall this is order capturing spacetime decay).   
Were we also to allow $m$ and $\ell$ to be variable, then we would also need to work with symbols which are instead bounded under application of vector fields in $\rho_{\calczero\mathrm{f}}^{\delta}\rho_{\mathrm{df}}^\delta \rho_{\mathrm{bf}}^\delta \mathcal{V}_{\calc}$.

For $a \in S^{m,\mathsf{s},\ell,q}_{\calc}$, we define its (left) quantization by
\begin{equation}
\operatorname{Op}(a(h)) f(z)  = \frac{1}{(2\pi)^D} \int_{\bbR^D} \int_{\bbR^D}   e^{i (z-z') \cdot \zeta} a(h)(z,\zeta) f(z') \bbR^D z' \dd^D \zeta \in \calS(\bbR_z^D),\quad D=d+1
\label{eq:quant}
\end{equation}
is defined for any $f\in \calS(\bbR^D)$. This integral is absolutely convergent if $m,\mathsf{s}$ are sufficiently negative, but otherwise it still makes sense as an iterated integral. 
By the discussion in  \cite[p.\ 246]{VasyGrenoble}, $\operatorname{Op}(a(h))$ can be extended to a continuous linear map $\calS'(\bbR^D)\to \calS'(\bbR^D)$. 
In this way, quantizing elements of $S^{m,\mathsf{s},\ell,q}_{\calc}$ results in families $\operatorname{Op}(a) = \{\operatorname{Op}(a(h)) \}_{h>0}$ of pseudodifferential operators of scattering, or Parenti--Shubin, type. Let 
\begin{equation}
\Psi_{\calc}^{m,\mathsf{s},\ell,q} = \{\operatorname{Op}(a): a \in S^{m,\mathsf{s},\ell,q}_{\calc}\}
\end{equation}
denote the set of all such families. This inherits its Fr\'echet space structure from $S^{m,\mathsf{s},\ell,q}_{\calc}$. 

Of particular importance are the ``residual'' operators
\begin{equation}
\Psi_{\calc}^{-\infty,-\infty,-\infty,q} = \bigcap_{m,s,\ell\in \bbR} \Psi_{\calc}^{m,s,\ell,q}.
\end{equation}
By definition, $S_{\calc}^{-\infty,-\infty,-\infty,0} = C^\infty([0,1)_h;\calS(T^* \bbR^{1+d}))$, which means that $\Psi_{\calc}^{-\infty,-\infty,-\infty,0}$ just consists of smooth families of elements of $\Psi_{\mathrm{sc}}^{-\infty,-\infty}$.

We now turn the the second construction of ${}^{\calc}\overline{T}^* \bbM$, using gluing. 
 Referring to the gluing construction of the $\calc$-phase space ${}^{\calc}\overline{T}^* \bbM$ in \cref{eq:calc_phase_space_definition}, we choose small neighbourhoods  $U$ of $\mathrm{df}_1 \subset  {}^{\mathrm{par,I,res}}\overline{T}^* \bbM$ and a neighbourhood $V$ of ${}^{\calczero}o^* \bbM \cap \mathrm{zf} \subset {}^{\calczero}\overline{T}^* \bbM$,  and denote the two parts in the gluing by
\begin{equation}
T_{\mathrm{par}}= {}^{\mathrm{par,I,res}}\overline{T}^* \bbM \backslash  U, \quad 
T_{\natural} = {}^{\calczero}\overline{T}^* \bbM \backslash V.
\end{equation} 
Then we have the following fact, which we state as a lemma:
\begin{lemma} \label{lemma:glue_par_calczero_orders}
For $\mathsf{s} \in C^\infty({}^{\calc}\overline{T}^* \bbM)$, there are $\mathsf{s}_{\mathrm{par}} \in C^\infty({}^{\mathrm{par,I,res}}\overline{T}^* \bbM)$ and $\mathsf{s}_{\calczero} \in C^\infty({}^{\calczero}\overline{T}^* \bbM)$ that coincide with $\mathsf{s}$ on $T_{\mathrm{par}}$ and $T_{\calczero}$ respectively. 
\end{lemma}
\begin{proof}
We can just consider the restriction of $\mathsf{s}$ to $T_{\mathrm{par}}$ and $T_{\calczero}$ respectively and then take $\mathsf{s}_{\mathrm{par}}, \mathsf{s}_{\calczero}$ as their smooth extensions on ${}^{\mathrm{par,I,res}}\overline{T}^* \bbM$ and ${}^{\calczero}\overline{T}^* \bbM$ respectively.
\end{proof}

The second construction of $\Psi_{\calc}$ is to write elements of $\Psi_{\calc}^{m,\mathsf{s},\ell,q}$ as `glued' elements of $\Psi_{\mathrm{par,I,res}}^{-\infty, \mathsf{s}_{\mathrm{par}} ,\ell,q}$ and $\Psi_{\calczero}^{m,\mathsf{s}_{\calczero},\ell}$:
\begin{proposition} 
	The set $\Psi_{\calc}^{m,\mathsf{s},\ell,q}$ consists precisely of those one-parameter families \begin{equation} 
		A=\{A(h)\}_{h\in (0,1)}
	\end{equation} 
	of $A(h) \in \Psi_{\mathrm{sc}}^{m,\mathsf{s}(h)}$ ($\mathsf{s}(h)$ stands for $\mathsf{s}$ restricted to fixed $h$), of the form $A=B+C$ for 
	\begin{equation}
		B = \operatorname{Op}(b)\in \Psi_{\mathrm{par,I,res}}^{-\infty, \mathsf{s}_{\mathrm{par}} ,\ell,q}, \quad b\in S_{\mathrm{par,I,res}}^{-\infty,\mathsf{s}_{\mathrm{par}},\ell,q}, \quad C = \operatorname{Op}(c) \in \Psi_{\calczero}^{m,\mathsf{s}_{\calczero},\ell}, \quad c \in S_{\calczero}^{m,\mathsf{s}_{\calczero},\ell},
	\end{equation}
	where $\mathsf{s}_{\mathrm{par}}, \mathsf{s}_{\calczero}$ coincide with $\mathsf{s}$ on $T_{\mathrm{par}}$ and $T_{\calczero}$ respectively as in Lemma~\ref{lemma:glue_par_calczero_orders}, $b$ is supported away from $\mathrm{df}_1$ and $c$ is supported away from the zero section. 
	\label{prop:gluing_pseudos}
\end{proposition}
\begin{proof}
	First suppose that $A$ is given.
	Let $a$ be the full left symbol of $A$. Then, we can define $b=\chi a$ and $c= (1-\chi) a$ for $\chi \in C^\infty({}^{\calc}\overline{T}^* \bbM )$ supported away from $\mathrm{df}$ and identically $1$ near $\mathrm{pf}$. These have the desired properties.\\
	Conversely, if $b,c$ are as above, then $a=b+c\in S_{\calc}^{m,\mathsf{s},\ell,q}$ and correspondingly $A \in \Psi_{\calc}^{m,\mathsf{s},\ell,q}$. 
\end{proof}

\subsection{The composition law } 
\label{subsec:composition}

Next we verify that operators in $\Psi^{m, \mathsf{s}, \ell, q}_{\calc}$ satisfy basic properties we require for a pseudodifferential algebra. In order to do so, a further decomposition of elements in $\Psi^{m, \mathsf{s}, \ell, q}_{\calc}$ is useful, i.e.,
for $P \in \Psi^{m, \mathsf{s}, \ell, q}_{\calc}$, we can write it as
\begin{align} \label{eq:decomposition_resolved}
P = P_1 + P_2 + P_3,
\end{align}
where $P_1$ has left symbol supported near $\mathrm{pf}$, $P_2$ has left symbol supported near $\mathrm{df}$ (in particular, these two supports are disjoint), and $P_3$ has left symbol supported away from both $\mathrm{pf}$ and $\mathrm{df}$.
This decomposition can be obtained just by multiplying the left symbol of $P$ by a partition of unity.

In addition, suppose $\mathsf{s}_{\mathrm{par}}, \mathrm{s}_{\calczero}$ are as in Lemma~\ref{lemma:glue_par_calczero_orders}, then those parts in the decomposition belong to those building block calculi with following orders:
\begin{equation}
P_1 \in \Psi_{ \mathrm{par,I,res} }^{-\infty,\mathsf{s}_{\mathrm{par}} , \ell,q}, \quad
P_2 \in \Psi_{\calczero}^{m,\mathrm{s}_{\calczero},\ell}, \quad
P_3 \in \Psi_{ \mathrm{par,I,res} }^{-\infty,\mathsf{s}_{\mathrm{par}} , \ell,q} \cap \Psi_{\calczero}^{m,\mathrm{s}_{\calczero},\ell}.
\end{equation}


Now we use this decomposition to law in the $\calc$-calculus:
\begin{proposition} \label{prop:calc_composition}
Suppose $\mathsf{s}_1, \mathsf{s}_2 \in C^\infty({}^{\calc}\overline{T}^* \bbM)$, $A \in \Psi^{m_1, \mathsf{s}_1, \ell_1, q_1}_{\calc}, B \in \Psi^{m_2, \mathsf{s}_2, \ell_2, q_2}_{\calc}$, then 
\begin{equation} \label{eq:calc-composition}
A \circ B \in \Psi^{m_1+m_2, \mathsf{s}_1+\mathsf{s}_2 , \ell_1+\ell_2, q_1+q_2}_{\calc}.
\end{equation}
In addition, denoting the left symbol of $A,B$ by $a,b$, then the left symbol of $A \circ B$, which we denote by $a \star b$, has the asymptotic expansion
\begin{equation} \label{eq:expansion-calc-composition}
a \star b(t,x,\tau,\xi) \sim \sum_{\alpha \in \bbN^{d+1} } \frac{1}{\alpha !} \partial_{\tau,\xi}^\alpha a(t,x,\tau,\xi) D_{t,x}^\alpha b(t,x,\tau,\xi),
\end{equation} 
where $\sim$ means
\begin{equation} \label{eq:expainsion-calc-error-membership}
a \star b(t,x,\tau,\xi) - \sum_{ |\alpha| \leq N } \frac{1}{\alpha !} \partial_{\tau,\xi}^\alpha a(t,x,\tau,\xi) D_{t,x}^\alpha b(t,x,\tau,\xi)
\in \bigcap_{\delta > 0} S_{\calc}^{m_1+m_2-N-1,\mathsf{s}_1+\mathsf{s}_2-N-1+\delta,\ell_1+\ell_2-N-1,q_1+q_2}.
\end{equation}
\end{proposition}

\begin{proof}
We can write $A,B$ as
\begin{align*}
A = A_1+A_2+A_3 , \quad B= B_1+B_2+B_3,
\end{align*}
with $A_i,B_i$ having the property of $P_i$ in (\ref{eq:decomposition_resolved}).
We have 
\begin{equation}
A \circ B = \sum_{i,j=1}^3 A_i \circ B_j.
\end{equation}

Let $\mathsf{s}_{\mathrm{par},k}, \mathrm{s}_{\calczero,k}$, $k=1,2$ be variable orders as in Lemma~\ref{lemma:glue_par_calczero_orders} associated to $\mathsf{s}_k$.
Then for terms $A_1B_3,A_1B_1,A_3B_1$, 
we apply the composition law in $\Psi_{\mathrm{par,I,res}}$ to see these compositions are pseudodifferential operators are in 
\begin{equation}
\Psi_{\mathrm{par,I,res}}^{-\infty,\mathsf{s}_{\mathrm{par},1}+\mathsf{s}_{\mathrm{par},2}, \ell_1+\ell_2,q_1+q_2},
\end{equation}
and their left symbols, by \eqref{eq:symbol exp parIres}, modulo a term that is Schwartz at all boundary faces of ${}^{ \mathrm{par,I,res} } T^* \bbM$ except for $\mathrm{pf}$, is supported away from $\mathrm{df}_1$, thus this symbol is in $S_{\calc}^{-\infty,\mathsf{s}_{1}+\mathsf{s}_{2}, \ell_1+\ell_2,q_1+q_2}$ (here we changed the order associated to $\mathrm{bf}$, as they coincide on the region where the symbol is not residual), which in turn shows
\begin{equation}
A_1B_3,A_1B_1,A_3B_1 \in \Psi_{\calc}^{-\infty,\mathsf{s}_{1}+\mathsf{s}_{2}, \ell_1+\ell_2,q_1+q_2}.
\end{equation}

For terms $A_2B_3,A_2B_2,A_3B_2, A_3B_3$, we can apply the composition law in $\Psi_{\calczero}$ to see 
\begin{equation}
A_2B_3,A_2B_2,A_3B_2, A_3B_3 \in \Psi_{\calczero}^{m_1+m_2, \mathsf{s}_{\calczero,1}+\mathsf{s}_{\calczero,2}, \ell_1+\ell_2 },
\end{equation}
and their left symbol, by \eqref{eq:moyal_explicit_natural}, modulo a term that is Schwartz at all boundary faces in ${}^{\calczero} T^* \bbM$ is supported away from the $\{\tau_{\calczero}=0,\xi_{\calczero}=0,h=0\}$. Thus this left symbol is in $S_{\calc}^{m_1+m_2, \mathsf{s}_{1}+\mathsf{s}_{2}, \ell_1+\ell_2,-\infty}$, and we have 
\begin{equation}
A_2B_3,A_2B_2,A_3B_2, A_3B_3 \in \Psi_{\calc}^{m_1+m_2, \mathsf{s}_{1}+\mathsf{s}_{2}, \ell_1+\ell_2,-\infty}.
\end{equation}

Finally, we show that terms $A_1B_2$ and $A_2B_1$ give residual contributions.
We take $A_1B_2$ as example. Suppose $A_1$ has (left) full symbol $a_1$
and $B_2$ has full symbol $b_2$. Then $A_1B_2$ acts by

\begin{align}
\begin{split} 
A_1B_2u(t,x) = & (2\pi)^{-n-1}\int e^{ i((t-t')\tau'+(x-x')  \cdot \xi') }
a_1(h,t,x,\tau',\xi')
h^{-n-2} \Big(e^{ i((t'-t'')\frac{\tau_{\calczero}}{h^2}+(x'-x'')\cdot \frac{\xi_{\calczero}}{h}) } 
\\& b_2(h,t',x',\tau_{\calczero},\xi_{\calczero}) u(t'',x'') \dd{t}''\dd{x}'' \dd\tau_{\calczero}\dd\xi_{\calczero} \Big)
\dd{t}'\dd{x}' \dd\tau'\dd\xi'
\\ = &(2\pi)^{-2n-2} h^{-n-2}\int e^{i (t,x) \cdot(\frac{\tau_{\calczero}}{h^2}, \frac{\xi_{\calczero}}{h}) + i (t'-t,x'-x) \cdot (\frac{\tau_{\calczero}}{h^2}-\tau',\frac{\xi_{\calczero}}{h}-\xi')} a_1(h,t,x,\tau',\xi') 
\\ & b_2(h, t',x',\tau_{\calczero},\xi_{\calczero}) \hat{u}\Big(\frac{\tau_{\calczero}}{h^2},\frac{\xi_{\calczero}}{h}\Big) \dd\tau_{\calczero} \dd\xi_{\calczero} \dd{t}'\dd{x}' \dd\tau'\dd\xi',
\end{split} 
\end{align}
which is the action of the operator with symbol
\begin{multline}
c_{12}(h, t,x,\tau_{\calczero},\xi_{\calczero})
= (2\pi)^{-n-1} h^{-n-2}  \int e^{-i(t-t',x-x') \cdot (\frac{\tau_{\calczero}}{h^2}-\tau',\frac{\xi_{\calczero}}{h}-\xi')  } a_1(h, t,x,\tau',\xi')
\\  b_2(h, t',x',\tau_{\calczero},\xi_{\calczero}) \dd{t}'\dd{x}' \dd\tau'\dd\xi'.
\end{multline}
By the disjoint support condition, we know
that on the support of $a_1(h, t,x,\tau',\xi')  b_2(h, t',x',\tau_{\calczero},\xi_{\calczero})$, the phase does not have any critical point.
Then we apply a non-stationary argument, i.e., repeated integration by parts with respect to vector fields that leave the oscillatory factor 
\begin{equation} 
	e^{-i(t-t',x-x') \cdot (\frac{\tau_{\calczero}}{h^2}-\tau',\frac{\xi_{\calczero}}{h}-\xi')  }
\end{equation} 
invariant. This allows us to gain factors of 
%
%
%
$\la (\tau_{\calczero},\xi_{\calczero}) \ra^{-N}h^N \la (t,z) \ra^{-N}$ for any $N$, which means that the symbol $c$ is in $S^{-\infty,-\infty,-\infty,-\infty}_{\calc}$. Thus the composition $A_1 B_2$ is in $\Psi_{\calc}^{-\infty, -\infty, -\infty, -\infty}$. 


Now we turn to the proof of \eqref{eq:expansion-calc-composition}.
Notice that the asymptotic expansion of the same form (but of course, interpreted differently!) holds in $\mathrm{par,I,res}$ and $\calczero$-calculus respectively, thus we know the left symbol of $A_i \circ B_j$ has the expansion
\begin{equation}
c_{ij}(h,t,x,\tau,\xi) \sim \sum_{\alpha \in \bbN^{d+1} } \frac{1}{\alpha !} \partial_{\tau,\xi}^\alpha a_i(h,t,x,\tau,\xi)  D_{t,x}^\alpha b_j(h, t,x,\tau,\xi).
\end{equation}
For $(i,j) =(1,3),(1,1),(3,1)$, the $\sim$ means 
\begin{multline}
c_{ij}(h,t,x,\tau,\xi) - \sum_{ |\alpha| \leq N } \frac{1}{\alpha !} \partial_{\tau,\xi}^\alpha a_i(h,t,x,\tau,\xi)  D_{t,x}^\alpha b_j(h, t,x,\tau,\xi)
\\ \in \bigcap_{\delta > 0} S_{\mathrm{par,I,res}}^{ -\infty, \mathsf{s}_{\mathrm{par},1}+\mathsf{s}_{\mathrm{par},2} -N-1+\delta, \ell_1+\ell_2 -N-1, q_1+q_2},
\end{multline} 
and all terms in the expansion are supported away from $\mathrm{df}_1$. 
On the other hand, for $(i,j)=(2,3),(3,3),(3,2),(2,2)$, we have the same expansion but written in the $(\taun, \xin)$ coordinates, and $\sim$ means  
\begin{multline}
c_{ij}(h,t,x,\taun,\xin) - \sum_{ |\alpha| \leq N } \frac{1}{\alpha !} (h^2 \partial_{\taun}, h \partial_{\xin})^\alpha a_i(h,t,x,\taun,\xin)  D_{t,x}^\alpha b_j(h, t,x,\taun,\xin) \\
\in \bigcap_{\delta > 0} S_{\calczero}^{ m_1+m_2 -N-1, \mathsf{s}_{\calczero,1}+\mathsf{s}_{\calczero,2} -N-1+\delta, \ell_1+\ell_2 -N-1},
\end{multline} 
and all terms in the expansion are supported away from $\{\tau_{\calczero}=0,\xi_{\calczero}=0,h=0\}$.

As the expansion only happens when $\partial_{\tau,\xi}^\alpha a_i$ and $D_{t,x}^\alpha b_j$ are evaluated at the same place, we know the the corresponding expansions with $(i,j)=(1,2),(2,1)$ are identically zero (which is what we expect, as we verified that their contributions are residual). 
Thus, summing all those expansions, and using the support condition (which implies that $\mathsf{s}_{i}$ coincides with $\mathsf{s}_{\mathrm{par},i}$ and $\mathsf{s}_{{\calczero},i}$ on the region where they take effect), 
we know the symbol of $A \circ B$ has the expansion as in \eqref{eq:expansion-calc-composition}, and the error term has membership as in \eqref{eq:expainsion-calc-error-membership}.
\end{proof}

Consider also the Poisson bracket $\{a,b\}$ of $a,b \in C^\infty(T^* \bbR^{1,d}) $, which we define using the sign convention
\begin{equation}
\{a,b\} =  (\partial_\tau a)(\partial_t b) - (\partial_t a)(\partial_\tau b) + \sum_{j=1}^d ( (\partial_{\xi_j} a)(\partial_{x_j} b)-(\partial_{x_j} a)(\partial_{\xi_j} b)  )   .
\label{eq:poisson_convention}
\end{equation}
\begin{lemma}
	For $a\in S^{m,\mathsf{s},\ell,q}_{\calc}$, and $b\in S^{m',\mathsf{s}',\ell',q'}_{\calc}$, we have
	\begin{equation} 
		\{a,b\} \in S^{m+m'-1,\mathsf{s}+\mathsf{s}'-1,\ell+\ell'-1,q+q'}_{\calc}.
	\end{equation} 
\end{lemma}
\begin{proof}
	This follows from the symbol expansion \cref{eq:expansion-calc-composition}.
\end{proof}

\begin{remark}
There is another more `global' proof of the composition law stated above and also the expansion \cref{eq:expansion-calc-composition}. This alternative proof is included in \Cref{sec:reduction}.
\end{remark}

\subsection{\texorpdfstring{Some basic properties of the $\calc$ calculus}{Some basic properties of the resolved natural calculus}}  \label{sec:calc-basic-properties}
An immediate consequence of \Cref{prop:calc_composition} is the following: 
\begin{proposition}
	If $\mathsf{s},\mathsf{s}'$ are variable orders and $m,m',\ell,\ell',q,q'\in \bbR$, and if $a\in S_{\calc}^{m,\mathsf{s},\ell,q}, b\in S_{\calc}^{m',\mathsf{s}',\ell',q'}$, then, for every $\varepsilon>0$, 
	\begin{enumerate}[label=(\Roman*)]
		\item $\operatorname{Op}(a)\circ \operatorname{Op}(b) -\operatorname{Op}(ab)\in \bigcap_{\delta > 0} \Psi_{\calc}^{m+m'-1,\mathsf{s}+\mathsf{s}'-1 + \delta,\ell+\ell'-1,q+q'}$, 
		\item $[\operatorname{Op}(a), \operatorname{Op}(b)]- i \operatorname{Op}(\{a,b\}) \in \bigcap_{\delta > 0} \Psi_{\calc}^{m+m'-2,\mathsf{s}+\mathsf{s}'-2+\delta,\ell+\ell'-2,q+q'}$.	
			\end{enumerate}
	\label{prop:composition}
\end{proposition}



\begin{proof}
	\begin{enumerate}[label=(\Roman*)]
		\item Follows from \eqref{eq:expainsion-calc-error-membership} with $N=0$.
		\item Notice that $[\operatorname{Op}(a),\operatorname{Op}(b)] = \operatorname{Op}(a\star b) - \operatorname{Op}(b\star a)= \operatorname{Op}(a\star b - b\star a)$, the result follows from comparing $a\star b$ and $b\star a$ up to terms with $N=1$ in \eqref{eq:expainsion-calc-error-membership}.


	\end{enumerate}
\end{proof}

For $A=\operatorname{Op}(a)$, $a\in S_{\calc}^{m,\mathsf{s},\ell,q}$, we define its $\calc$-principal symbol by 
\begin{equation} 
\sigma_{\calc}^{m,\mathsf{s},\ell,q}(A) = a \bmod \bigcap_{\delta > 0} S_{\calc}^{m-1,\mathsf{s}-1+\delta,\ell-1,q} 
\in S_{\calc}^{m,\mathsf{s},\ell,q} / \bigcap_{\delta > 0} S_{\calc}^{m-1,\mathsf{s}-1+\delta,\ell-1,q}.
\end{equation} 
Note that  the sequence 
\begin{equation}\label{eq:short exact natres}
0 \to \bigcap_{\delta > 0} \Psi_{\calc}^{m-1,\mathsf{s}-1 + \delta,\ell-1,q} \to \Psi_{\calc}^{m,\mathsf{s},\ell,q} \overset{\sigma_{\calc}^{m,\mathsf{s},\ell,q}}{\longrightarrow} S_{\calc}^{m,\mathsf{s},\ell,q} / \bigcap_{\delta > 0} S_{\calc}^{m-1,\mathsf{s}-1+\delta,\ell-1,q} \to 0 
\end{equation}
is a short exact sequence and $\sigma_{\calc}^{m,\mathsf{s},\ell,q}$ is an algebra homomorphism. 

The preceding proposition gives us the basic algebraic properties of principal symbols:
\begin{corollary}
	For any $A\in \Psi_{\calc}^{m,\mathsf{s},\ell,q} ,B\in \Psi_{\calc}^{m',\mathsf{s}',\ell',q'}$, \begin{align}
	\sigma_{\calc}^{m+m',\mathsf{s}+\mathsf{s}',\ell+\ell',q+q'}(AB) &= \sigma_{\calc}^{m,\mathsf{s},\ell,q}(A)\sigma_{\calc}^{m',\mathsf{s}',\ell',q'}(B) \\ 
	\sigma_{\calc}^{m+m'-1,\mathsf{s}+\mathsf{s}'-1,\ell+\ell'-1,q+q'}([A,B]) &= i\{\sigma_{\calc}^{m,\mathsf{s},\ell,q}(A),\sigma_{\calc}^{m',\mathsf{s}',\ell',q'}(B)\}
	\end{align}
	both hold.
\end{corollary}

We also record here that the usual formula for the (left-reduced) symbol of the adjoint of a pseudodifferential operator, namely $a^*(z, \zeta) = e^{i\ang{D_z, D_\zeta}} a(z, \zeta)$, shows that the adjoint of $A \in \Psi_{\calc}^{m, \mathsf{s}, \ell, q}$ is also in $\Psi_{\calc}^{m, \mathsf{s}, \ell, q}$; see \cite[Theorem 18.1.7]{Ho3}. We omit the proof, but note that the asymptotic expansion of $e^{i\ang{D_z, D_\zeta}} a(z, \zeta)$ works similarly to \eqref{eq:expansion-calc-composition} in that the terms, and the remainder term, decrease in order at $\mathrm{df}$, $\mathrm{bf}$ and $\mathrm{{\natural}f}$,  but not at $\mathrm{pf}$.

\begin{proposition}
	Suppose that $A \in \Psi_{\calc}^{0,0,0,0}$. Then, for any $u\in L^2(\bbR^D)$, we have $A(h) u\in L^2(\bbR^D)$ for each $h>0$. Moreover, for each $h_0\in (0,1)$, there exists a constant $C>0$, depending on $A$, such that 
	\begin{equation}
	\lVert A(h) u \rVert_{L^2(\bbR^D)} \leq C \lVert u \rVert_{L^2(\bbR^D)}
	\end{equation}
	for all $u\in L^2(\bbR^D)$ and $h\in (0,h_0)$. Thus, $A(h)$ is a family of bounded linear operators on $L^2(\bbR^D)$, uniformly bounded as $h\to 0^+$. 
	\label{prop:L2_boundedness}
\end{proposition}
\begin{proof}
This follows from H\"ormander's ``square root trick'' (see \cite[Theorem 18.1.11]{Ho3}). Namely, using the symbol calculus we construct an operator $B \in \Psi_{\calc}^{0,0,0,0}$ such that 
\begin{equation}\label{eq:square root trick}
A^* A + B^* B = C^2 \cdot \Id + R,
\end{equation}
where $C^2$ is strictly larger than $\sup \sigma_L(A^* A)$ and $R \in \Psi_{\calc}^{-m,-s,-l,0}$ is an error term where the orders $(-m, -s, -l)$ can be taken as negative as desired. Then for sufficiently negative $(-m, -s, -l)$, $R(h)$ is a family of Hilbert-Schmidt operators with uniformly bounded Hilbert-Schmidt norm, so is uniformly bounded on $L^2(\bbR^D)$, and \eqref{eq:square root trick} immediately implies the uniform boundedness of $A$. 
\end{proof}

Below, we will use the symbol `$\lesssim$' to mean less than or equal to some constant multiple of, where the constant is independent of $h\in (0,h_0)$ and the function $u$. However, the ``constant'' is not required to be uniform as $h_0\to 1^-$,  nor independent of any of the pseudodifferential operators appearing. 
Moreover, it will be understood that when we write $L^2$, we mean $L^2(\bbR^D)$.
For instance, the conclusion of the previous proposition can be written 
\begin{equation} 
	\lVert A u \rVert_{L^2}\lesssim \lVert u \rVert_{L^2}.
\end{equation} 

The differential operators in this calculus will be denoted $\operatorname{Diff}_{\calc}^{m,s,\ell,q} = \Psi_{\calc}^{m,s,\ell,q} \cap \operatorname{Diff}(\bbR^{1+d})$. We pause to note the membership of spatial and temporal partial derivatives in this calculus.

\begin{lemma} $\partial_t \in \operatorname{Diff}_{\calc}^{1,0,2,0}$, and $\partial_{x_j}\in \operatorname{Diff}_{\calc}^{1,0,1,0}$ for each $j\in \{1,\dots,d\}$.
	\label{lem:derivatives}
\end{lemma}
\begin{proof}
	\begin{itemize}
		\item The full symbol of $\partial_t$ is (up to a factor of $i$) just $\tau = h^{-2} \tau_{\natural}$, and $\tau_{\natural} \in S^{1,0,0}_{\calczero}$, so $\tau \in S_{\calczero}^{1,0,2}$. Also, $\tau \in S_{\mathrm{par,I,res}}^{2, 0, 2, 0}$. 
Combining these two pieces of information, we have $\tau \in S_{\calc}^{1,0,2,0}$, which implies $\partial_t \in \Psi_{\calc}^{1,0,2,0}$.
		\item Similarly, the full symbol of $\partial_{x_j}$ is (up to a factor of $i$) just $\xi_j = h^{-1} \xi_{j,\natural\mathrm{f}}$. Since $\xi_{j,\natural\mathrm{f}} \in S_{\calczero}^{1,0,0}$, we have $\xi_j \in S_{\calczero}^{1,0,1}$. Also, $\xi_j \in S_{\mathrm{par,I,res}}^{1, 0, 1, 0}$. Combining these two pieces of information, $\xi_j \in S_{\calc}^{1,0,1,0}$, which implies $\partial_{x_j} \in \Psi_{\calc}^{1,0,1,0}$.
	\end{itemize}
\end{proof}


\subsection{Elliptic estimates, wavefront sets and Sobolev spaces}
\label{subsec:Sobolev}

We say that a symbol $a\in S^{0,0,0,0}_{\calc}$ is elliptic at a point $p \in \mathrm{df}\cup \mathrm{bf}\cup \natural\mathrm{f}$ if $|a|>c$ in some neighborhood of $p$ for some $c>0$. 
More generally, we say that $a\in S^{m,\mathsf{s},\ell,q}_{\calc}$ is elliptic at $p$ if it has the form 
\begin{equation} 
	a = \rho_{\mathrm{df}}^{-m}\rho_{\mathrm{bf}}^{-\mathsf{s}} \rho_{\natural\mathrm{f}}^{-\ell} \rho_{\mathrm{pf}}^{-q} a_0
\end{equation} 
for $a_0$ elliptic at $p$. If $a$ is elliptic at all $p \in \mathrm{df}\cup \mathrm{bf}\cup \natural\mathrm{f}$, then it is just said to be elliptic. In this paper, when we use the term elliptic, it is in this strong sense.

Let $\operatorname{ell}_{\calc}^{m,\mathsf{s},\ell,q}(a)$ denote the set of $p$ at which $a$ is elliptic, and let  $\smash{\operatorname{char}_{\calc}^{m,\mathsf{s},\ell,q}(a)}\subseteq \mathrm{df}\cup \mathrm{bf}\cup \natural\mathrm{f}$ denote its complement in $\mathrm{df}\cup \mathrm{bf}\cup \natural\mathrm{f}$. Similarly, we can define the elliptic set and characteristic set of an operator $A \in \Psi_{\calc}$ as the elliptic set or characteristic set of any representative of the principal symbol.


If $A \in \Psi_{\calc}^{m,s,\ell,q}$ is the left quantization of $a \in S_{\calc}^{m,s,\ell,q}$, let $\operatorname{WF}'_{\calc}(A)$ denote the complement of the set of points in $\mathrm{df}\cup \natural\mathrm{f}\cup \mathrm{bf}$ around which there exists an open neighborhood $U$ such that
\begin{equation}  \label{eq: calc residual symbol class}
a \in S_{\calc}^{-\infty,-\infty,-\infty,q}= \bigcap_{k\in \bbN} S_{\calc}^{-k,-k,-k,q} \quad \text{locally in } U. 
\end{equation}

\begin{remark}
The set $\WF'_{\calc}(A)$ is independent of the parameter $q$ such that $A \in \Psi_{\calc}^{m,s,\ell,q}$: if $q' = q + m$, $m \in \mathbb{N}$ then we also have $A \in \Psi_{\calc}^{m,s,\ell,q'}$ and then the corresponding condition \eqref{eq: calc residual symbol class} with $q$ replaced by $q'$ is trivially satisfied. On the other hand, if $q' = q-m$ and it happens that $A \in \Psi_{\calc}^{m,s,\ell,q'}$, which is a stronger statement than $A \in \Psi_{\calc}^{m,s,\ell,q}$, then we have both that $a \in S_{\calc}^{m,s,\ell,q'}$ and, locally in $U$, $a \in S_{\calc}^{-\infty,-\infty,-\infty,q}$. From this we deduce, using the smoothness of $h ^{-q'}a$ at the parabolic face, that $a \in S_{\calc}^{-\infty,-\infty,-\infty,q'}$ in $U$, so the stronger statement \eqref{eq: calc residual symbol class} with $q$ replaced by $q'$ is satisfied. 
\end{remark}
This set is known as the essential support, microlocal support or operator wavefront set of $A$ (or of $a$). Note that 
\begin{equation} 
	\operatorname{WF}'_{\calc}(A) = \varnothing \iff  A \in \Psi_{\calc}^{-\infty,-\infty,-\infty,q}, 
\end{equation}
\emph{not} $\operatorname{WF}'_{\calc}(A) = \varnothing \iff  A \in \Psi_{\calc}^{-\infty,-\infty,-\infty,-\infty}$. 

An easy consequence of the symbol expansion \cref{eq:expansion-calc-composition} is the behaviour of the operator wavefront set of pseudodifferential operators under products:
\begin{equation}\label{eq:WFproducts}
\operatorname{WF}_{\calc}(\operatorname{Op}(a)\circ \operatorname{Op}(b) ) \subseteq \operatorname{WF}_{\calc}(\operatorname{Op}(a))\cap \operatorname{WF}_{\calc}(\operatorname{Op}(b) ).
\end{equation}

The standard microlocal elliptic parametrix construction goes through in $\Psi_{\calc}$:
\begin{proposition}
	Let $C\subset \mathrm{df}\cup \mathrm{bf}\cup\natural\mathrm{f}$ be closed. 
	If $A \in \Psi_{\calc}^{m,\mathsf{s},\ell,q}$ is elliptic at every point in $C$, then there exists a $B\in \smash{\Psi_{\calc}^{-m,-\mathsf{s},-\ell,-q}}$ such that 
	\begin{equation}
	\operatorname{WF}_{\calc}'(AB-1) \cap C  = \varnothing = \operatorname{WF}_{\calc}'(BA-1) \cap C. 
	\end{equation}
	In particular, if $A$ is elliptic, then $AB -1 = R_1$ and $BA-1 = R_2$ for $R_1,R_2 \in \Psi_{\calc}^{-\infty,-\infty,-\infty,0} = C^\infty([0,1)_h; \Psi_{\mathrm{sc}}^{-\infty,-\infty} )$. 
	\label{prop:parametrix_const}
\end{proposition}
\begin{proof}
	By the smooth version of Urysohn's lemma, there exists a $\chi \in C^\infty({}^{\calc}\overline{T}^* \bbM)$ such that $\chi=1$ identically on some neighborhood of $C$ and such that $\chi=0$ identically on a neighborhood of the characteristic set of $A$ as well as on a neighborhood of any points in the interior of $\mathrm{pf}$ and the interior of phase space at which $h^{q} a$ vanishes, where $a$ is a full left symbol of $A$. It follows that $b=\chi/a$ is a well-defined $\calc$-symbol of orders $-m,-\mathsf{s},-\ell,-q$. So, letting $B_0 = \operatorname{Op}(b_0)$, we have that $AB_0 -1$, $B_0A-1$ are both characteristic on some neighborhood $U$ of $C$. Now, via asymptotic summation, we can find $B_{\mathrm{R}} \in \Psi_{\calc}^{-m,-\mathsf{s},-\ell,-q}$ such that 
	\begin{equation}
	B_{\mathrm{R}} \sim \sum_{j=0}^\infty B_0 (1-AB_0)^j
	\end{equation}
	in some neighborhood of $C$. Then, a formal computation yields $A B_{\mathrm{R}} \sim 1$, so the operator $R_1 = AB_{\mathrm{R}} - 1$ satisfies $R_1 \in \Psi_{\calc}^{-\infty,-\infty,-\infty,0}$ near $C$, meaning than the full left symbol of $R_1$ lies in $S_{\calc}^{-\infty,-\infty,-\infty,0}$ near $C$. This is another way of saying that 
	\begin{equation} 
		\operatorname{WF}_{\calc}'(AB_{\mathrm{R}}-1) \cap C  = \varnothing.
	\end{equation} 
	Similarly, one constructs $B_{\mathrm{L}}$ such that $B_{\mathrm{L}} A - 1 = R_2$ satisfies  $R_2 \in \Psi_{\calc}^{-\infty,-\infty,-\infty,0}$ near $C$.

	It remains to show that it is possible to take $B_{\mathrm{L}} = B_{\mathrm{R}}$. We can write 
	\begin{equation}
	B_{\mathrm{R}} + R_2 B_{\mathrm{R}}=
	(B_{\mathrm{L}} A) B_{\mathrm{R}} = 
	B_{\mathrm{L}} A B_{\mathrm{R}} = B_{\mathrm{L}}( A B_{\mathrm{R}}) = B_{\mathrm{L}} + B_{\mathrm{L}} R_1, 
	\end{equation}
	so $B_{\mathrm{L}} - B_{\mathrm{R}}$ has essential support away from $C$. It follows that $B_{\mathrm{R}} A - 1 = R_2 + (B_{\mathrm{R}} - B_{\mathrm{L}} ) A$ does as well. So, we can take $B=B_{\mathrm{L}} = B_{\mathrm{R}}$. 
\end{proof}

For orders  $m,\mathsf{s},\ell,q$, and for $u\in \calS'(\bbR^D)$, we can define a function $\lVert u \rVert_{H_{\calc}^{m,\mathsf{s},\ell,q } } : (0,1)_h\to [0,\infty]$ by 
\begin{equation}
\lVert u \rVert_{H_{\calc}^{m,\mathsf{s},\ell,q} }(h)= \lVert u \rVert_{H_{\calc}^{m,\mathsf{s},\ell,q}(h) }  = \lVert \Lambda_{m,\mathsf{s},\ell,q}(h) u(h) \rVert_{L^2} + \lVert h^{-q} R_{m,\mathsf{s},\ell,q} u(h) \rVert_{L^2}
\end{equation}
for some fixed elliptic $\Lambda_{m,\mathsf{s},\ell,q} \in \smash{\Psi_{\calc}^{m,\mathsf{s},\ell,q}}$, where
$R_{m,\mathsf{s},\ell,q}\in \Psi_{\calc}^{-\infty,-\infty,-\infty,0}$ is the remainder in the left parametrix construction: 
\begin{equation} 
\tilde{\Lambda}_{-m,-\mathsf{s},-\ell,-q} \Lambda_{m,\mathsf{s},\ell,q} + R_{m,\mathsf{s},\ell,q} = 1
\end{equation} 
for a two-sided parametrix $\tilde{\Lambda}_{-m,-\mathsf{s},-\ell,-q} \in \smash{\Psi_{\calc}^{-m,-\mathsf{s},-\ell,-q}}$.

\begin{proposition}[Sobolev boundedness of $\calc$-$\Psi$DOs]
	If $A \in \Psi_{\calc}^{m,\mathsf{s},\ell,q}$, then, for all $m',\mathsf{s}',\ell',q'$ and $h_0 \in (0,1)$, there exists a $C>0$ such that 
	\begin{equation}
	\lVert A u \rVert_{H_{\calc}^{m'-m,\mathsf{s}'-\mathsf{s},\ell'-\ell,q'-q }  } \leq C \lVert u \rVert_{H_{\calc}^{m',\mathsf{s}',\ell',q'} }
	\end{equation}
	holds for all $h\in (0,h_0)$ and $u\in \calS'(\bbR^D)$. 
	\label{prop:Sobolev_boundedness}
\end{proposition}
\begin{proof}
	We have 
	\begin{equation} 
	\lVert A u \rVert_{H_{\calc}^{m'-m,\mathsf{s}'-\mathsf{s},\ell'-\ell,q'-q }  }  = 	\lVert \Lambda_{m'-m,\mathsf{s}'-\mathsf{s},\ell'-\ell,q'-q } A u \rVert_{L^2} + \lVert h^{q-q'} R_{m'-m,\mathsf{s}'-\mathsf{s},\ell'-\ell,q'-q} Au \rVert_{L^2}.
	\label{eq:id_220_445}
	\end{equation}  
	Abbreviating $R_0 = R_{m',\mathsf{s}',\ell',q'}$, we can bound the first term on the right-hand side of \cref{eq:id_220_445} as follows:
	\begin{multline}
	\lVert \Lambda_{m'-m,\mathsf{s}'-\mathsf{s},\ell'-\ell,q'-q } A u \rVert_{L^2} = 	\lVert \Lambda_{m'-m,\mathsf{s}'-\mathsf{s},\ell'-\ell,q'-q } A (\tilde{\Lambda}_{-m',-\mathsf{s}',-\ell',-q'}  \Lambda_{m',\mathsf{s}',\ell',q'} + R_0) u \rVert_{L^2} 
	\\ \lesssim \lVert \Lambda_{m',\mathsf{s}',\ell',q'} u \rVert_{L^2}+ \lVert R_1 u \rVert_{L^2}
	\end{multline}
	where $R_1 =  \Lambda_{m'-m,\mathsf{s}'-\mathsf{s},\ell'-\ell,q'-q } A R_0 \in \Psi_{\calc}^{-\infty,-\infty,-\infty,q'}$. Above, we used that
	\begin{equation}
	\Lambda_{m'-m,\mathsf{s}'-\mathsf{s},\ell'-\ell,q'-q } A \tilde{\Lambda}_{-m',-\mathsf{s}',-\ell',-q'}  \in \Psi_{\calc}^{0,0,0,0}, 
	\end{equation}
	which follows from \Cref{prop:composition} and implies, by \Cref{prop:L2_boundedness}, uniform $L^2$-boundedness. Now, 
	\begin{multline}
	\lVert R_1 u \rVert_{L^2} \leq \lVert R_1 \tilde{\Lambda}_{-m',-\mathsf{s}',-\ell',-q'} \Lambda_{m',\mathsf{s}',\ell',q'}  u \rVert_{L^2} + \lVert R_1 R_0 u \rVert_{L^2}   \\ 
	\lesssim \lVert \Lambda_{m',\mathsf{s}',\ell',q'}  u \rVert_{L^2} + \lVert h^{-q'} R_0 u \rVert_{L^2}  = \lVert u \rVert_{H^{m',\mathsf{s}',\ell',q'}_{\calc} },
	\end{multline}
	since $R_1 h^{q'},R_1 \tilde{\Lambda}_{-m',-\mathsf{s}',-\ell',-q'} \in \Psi_{\calc}^{-\infty,-\infty,-\infty,0}$.

	Likewise, we can bound the other term on the right-hand side of \cref{eq:id_220_445} as follows: 
	\begin{multline}
	\lVert h^{q-q'}R_{m'-m,\mathsf{s}'-\mathsf{s},\ell'-\ell,q'-q} A u \rVert_{L^2}  \lesssim \lVert h^{q-q'}R_{m'-m,\mathsf{s}'-\mathsf{s},\ell'-\ell,q'-q} A \tilde{\Lambda}_{-m',-\mathsf{s}',-\ell',-q'}  \Lambda_{m',\mathsf{s}',\ell',q'} u \rVert_{L^2} \\ 
	+ \lVert h^{q-q'}R_{m'-m,\mathsf{s}'-\mathsf{s},\ell'-\ell,q'-q} A R_0u \rVert_{L^2} \lesssim \lVert \Lambda_{m',\mathsf{s}',\ell',q'} u \rVert_{L^2} + \lVert h^{-q'}R_0 u \rVert_{L^2} =  \lVert u \rVert_{H^{m',\mathsf{s}',\ell',q'}_{\calc} },
	\end{multline}
	again using \Cref{prop:composition} and \Cref{prop:L2_boundedness}.

	Combining everything, we have shown $\lVert A u \rVert_{H_{\calc}^{m'-m,\mathsf{s}'-\mathsf{s},\ell'-\ell,q'-q }  } \lesssim \lVert u \rVert_{H_{\calc}^{m',\mathsf{s}',\ell',q'} }$.
\end{proof}

\begin{remark} Using Proposition~\ref{prop:Sobolev_boundedness} one can show that different choices of $\Lambda_{m,\mathsf{s},\ell,q}$ and its parametrix lead to an equivalent  $\lVert \cdot \rVert_{H_{\calc}^{m,\mathsf{s},\ell,q } } $ norm. 
\end{remark}

We can now prove a basic elliptic estimate.
\begin{proposition}
	If $A \in \Psi_{\calc}^{m,\mathsf{s},\ell,q}$ and $Q\in \Psi_{\calc}^{0,0,0,0}$ satisfies $\operatorname{WF}'_{\calc}(Q)\subseteq \operatorname{Ell}^{m,\mathsf{s},\ell,q}_{\calc}(A)$, then, for any $N\in \bbN$ and $h_0\in (0,1)$, there exists a $C>0$ such that 
	\begin{equation}
	\lVert Q u \rVert_{H_{\calc}^{m',\mathsf{s}',\ell',q'} } \leq C\big( 	\lVert A u \rVert_{H_{\calc}^{m'-m,\mathsf{s}'-\mathsf{s},\ell'-\ell,q'-q} } + \lVert u \rVert_{H_{\calc}^{-N,-N,-N,q'} } \big) 
	\end{equation}
	holds for all $h\in (0,h_0)$ and $u \in \calS'(\bbR^D)$, in the strong sense that if the right-hand side is finite then so is the left-hand side and the stated inequality holds.
	\label{prop:elliptic}
\end{proposition}
\begin{proof}
	By the parametrix construction, \Cref{prop:parametrix_const}, there exists a $B\in  \Psi_{\calc}^{-m,-\mathsf{s},-\ell,-q} $ such that $BA-1,AB-1$ have essential support disjoint from that of $Q$. Thus, 
	\begin{align}
	\begin{split} 
	\lVert Q u \rVert_{H_{\calc}^{m',\mathsf{s}',\ell',q'}}  &\leq \lVert Q (1-BA) u \rVert_{H_{\calc}^{m',\mathsf{s}',\ell',q'}} + \lVert Q BA u \rVert_{H_{\calc}^{m',\mathsf{s}',\ell',q'}} \\ &\lesssim \lVert u \rVert_{H_{\calc}^{-N,-N,-N,q'} }  + \lVert Q BA u \rVert_{H_{\calc}^{m',\mathsf{s}',\ell',q'}}.
	\end{split}
	\end{align}
	Using \Cref{prop:Sobolev_boundedness}, $\lVert Q BA u \rVert_{H_{\calc}^{m',\mathsf{s}',\ell',q'}} \lesssim 	\lVert A u \rVert_{H_{\calc}^{m'-m,\mathsf{s}'-\mathsf{s},\ell'-\ell,q'-q} } $, so the claim follows. 
\end{proof}

Later on, we will need the following:
\begin{lemma}
	If $A\in \Psi_{\calczero}^{m,\mathsf{s},\ell}$ has $\operatorname{WF}'_{\calczero}(A)$ disjoint from the zero section, then, for all $m',\ell',q',N\in \bbR$ and variable order $\mathsf{s}'$, 
	\begin{equation}
		\lVert A u \rVert_{H_{\calc}^{m',\mathsf{s}',\ell',q'} } \lesssim \lVert Au \rVert_{H_{\calczero}^{m',\mathsf{s}',\ell' }} + \lVert u \rVert_{H_{\calczero}^{-N,-N,-N}} 
	\end{equation}
	holds for all $u\in \calS'$, in the strong sense that if the right-hand side is finite, then so is the left-hand side, and the stated inequality holds.
	\label{lem:pf_elim}
\end{lemma}
\begin{proof}
	It suffices to consider the case $q'=0$. Note that 
	\begin{equation} 
		R_{m',\mathsf{s}',\ell',0} A \in \Psi_{\calc}^{-\infty,-\infty,-\infty,-\infty} = \Psi_{\calczero}^{-\infty,-\infty,-\infty}
	\end{equation} 
	and $\Lambda_{m',\mathsf{s}',\ell',0} Q \in \Psi_{\calczero}^{m',\mathsf{s}',\ell'}$ for any $Q\in \Psi_{\calczero}^{0,0,0}$ whose essential support is disjoint from the zero section. We can choose $Q$ such that $Q$ is elliptic on $\operatorname{WF}'_{\calczero}(A)$, and then the microlocal parametrix construction gives 
	\begin{equation} 
		\operatorname{WF}'_{\calczero}(1-QQ')\cap \operatorname{WF}'_{\calczero}(A)=\varnothing
	\end{equation} for some $Q'\in \Psi_{\calczero}^{0,0,0,0}$. 
	So, in 
	\begin{multline}
		\lVert A u \rVert_{H_{\calc}^{m',\mathsf{s}',\ell',0} } = \lVert \Lambda_{m',\mathsf{s}',\ell',0} A u \rVert_{L^2} +\lVert R_{m',\mathsf{s}',\ell',0} A u \rVert_{L^2 } \\ \lesssim 
		\lVert \Lambda_{m',\mathsf{s}',\ell',0} \underbrace{(1-QQ') A}_{\in \Psi_{\calczero}^{-\infty,-\infty,-\infty}} u \rVert_{L^2} + \lVert \Lambda_{m',\mathsf{s}',\ell',0} QQ' A u \rVert_{L^2} +\lVert \underbrace{R_{m',\mathsf{s}',\ell',0} A}_{\in \Psi_{\calczero}^{-\infty,-\infty,-\infty}} u \rVert_{L^2 }, 
	\end{multline}
	the only term on the right-hand side that is not bounded above by $\lesssim \lVert u \rVert_{H_{\calczero}^{-N,-N,-N}} $ is the middle one, and this satisfies 
	\begin{equation}
		\lVert \underbrace{\Lambda_{m',\mathsf{s}',\ell',0} Q}_{\in \Psi_{\calczero}^{m',\mathsf{s}',\ell'}}Q' A u \rVert_{L^2} \lesssim \lVert Q' A u \rVert_{H_{\calczero}^{m',\mathsf{s}',\ell'} } \lesssim \lVert A u \rVert_{H_{\calczero}^{m',\mathsf{s}',\ell'} }.
	\end{equation}
\end{proof}

\subsection{Normal operators}
\label{subsec:normal} 
The fact that principal symbols only capture elements of $\Psi_{\calc}$ modulo terms with better decay at $\mathrm{df}\cup \mathrm{bf}\cup \natural\mathrm{f}$, in particular \emph{not} at $\mathrm{pf}$, means that $\Psi_{\calc}$ is ``not symbolic'' at $\mathrm{pf}$. Morally, the reason is because there is an additional, operator-valued ``symbol'' or normal operator residing at $\mathrm{pf}$, inherited from the normal operator in the $\mathrm{par, I, res}$-calculus. 

Let $A$ be an element of $\Psi_{\calc}^{m,\mathsf{s},\ell,0}$ with left-reduced symbol $a\in S_{\calc}^{m,\mathsf{s},\ell,0}$. We define 
\begin{equation}
	N(A) = \operatorname{Op}( a |_{\mathrm{pf}}) \in \Psi_{\mathrm{par}}^{\ell, \mathsf{s}\big|_{\mathrm{pf}}}. 
\end{equation}

\begin{proposition}\label{prop:N mult natres} The normal operator in the $\Psi_{\calc}$ calculus is multiplicative, that is, for $A \in \Psi^{m, \mathsf{s}, \ell, 0}_{\calc}$, $B \in \Psi^{m', \mathsf{s}', \ell', 0}_{\calc}$ we have 
\begin{equation}\label{eq:mult normal op natres}
N(AB) = N(A) N(B) \in \Psi_{\mathrm{par}}^{\ell + \ell', \mathsf{s}|_{\mathrm{pf}} + \mathsf{s}'|_{\mathrm{pf}}}.
\end{equation}
\end{proposition}

\begin{proof} The proof is similar to the proof of Proposition~\ref{prop:N mult parIres}. We establish that for any fixed Schwartz function $\phi$, we have 
\begin{equation}\label{eq:normal A natres}
A \phi = N(A) \phi + O_{\mathcal{S}(\bbR^D)} (h). 
\end{equation}
The proof cannot work exactly as before, however, because $N(A)$ cannot be viewed as an element of the $\calc$-calculus, since $N(A)$ is associated to the parabolic scaling and its symbol is smooth with respect to the `wrong' differentiable structure at $\mathrm{df}$. Instead, we decompose $A = A_1 + A_2 + A_3$ as in the proof of Proposition~\ref{prop:calc_composition}, and note that only $A_1$ has a nonzero normal operator. Thus $N(A) = N(A_1)$, and we observe that the previous proof applies to $A_1$ since it is in $\Psi_{\mathrm{par,I,res}}$.  So it suffices to show that $(A_2 + A_3) \phi = O_{\mathcal{S}(\bbR^D)}(h)$. In fact, we will show that $(A_2 + A_3) \phi = O_{\mathcal{S}(\bbR^D)}(h^\infty)$. We will do this for $A_2$; the argument for $A_3$ is identical (in fact there is no reason to split $A_2 + A_3$ in this argument). 

To do this, we decompose $\Id = \Phi + (\Id - \Phi)$, where $\Phi = \varphi(h^2 \triangle_x + h^4 D_t^2)$ and $\varphi(\lambda) \in C_c^\infty(\bbR_\lambda)$ has support in a small neighbourhood of zero and satisfies $\varphi(\lambda) \equiv 1$ for small $\lambda$. We have $\Phi \in \Psi_{\mathrm{par,I,res}}^{-\infty, 0, 0, 0}$, while $\Id - \Phi \in \Psi_{\calc}^{0, 0, 0, -\infty}$, so this decomposes the identity operator into pieces similarly to the proof of Proposition~\ref{prop:calc_composition}. If the support of $\varphi$ is sufficiently small then $A_2 \Phi \in \Psi_{\calc}^{-\infty, 0, -\infty, -\infty}$ for sufficiently small $h$, so we can treat $A_2 \Phi$ as a smooth family of operators in $\Psi_{\mathrm{sc}}^{-\infty, 0}$ which vanish rapidly as $h \to 0$. It follows that $A_2 \Phi \phi$ is $O_{\mathcal{S}(\bbR^D)}(h^\infty)$. On the other hand, $(\Id - \Phi) \phi$ is easily seen to be $O_{\mathcal{S}(\bbR^D)}(h^\infty)$ and then the Sobolev estimates from Proposition~\ref{prop:Sobolev_boundedness} show that $A_2 (\Id - \Phi) \phi = O_{\mathcal{S}(\bbR^D)}(h^\infty)$. This proves \eqref{eq:normal A natres} and then the proof of the Proposition is completed exactly as for Proposition~\ref{prop:N mult parIres}.
\end{proof}

We can now state a more precise short exact sequence, improving on \eqref{eq:short exact natres}, which also captures improvement at the parabolic face. To state this we introduce the notation 
\begin{equation}
(S_{\calc}^{m,\mathsf{s},\ell,q} / S_{\calc}^{m-1,\mathsf{s}-1+\delta,\ell-1,q}) \times|_{\mathrm{con}} \Psi_{\mathrm{par}}^{\ell-q,\mathsf{s}|_{\mathrm{pf}} }
\end{equation}
where $\times|_{\mathrm{con}}$ (the `con' stands for `consistent') means that we are 
	looking at the subset of the direct product consists of pairs $([a],A_{\mathrm{par}})$, where $[a] \in S_{\calc}^{m,\mathsf{s},\ell,q} / S_{\calc}^{m-1,\mathsf{s}-1+\delta,\ell-1,q}$ and $A_{\mathrm{par}} \in \Psi_{\mathrm{par}}^{\ell-q,\mathsf{s}|_{\mathrm{pf}} }$, satisfying the matching condition $\sigma_{\mathrm{par}}^{\ell,\mathsf{s}|_{\mathrm{pf}}}(A_{\mathrm{par}}) = [a]|_{\mathrm{pf} \cap \natural\mathrm{f}}$. That is, the principal symbol of $A_{\mathrm{par}}$ in the parabolic calculus matches with the restriction of $[a]$ to the boundary of the parabolic face $\mathrm{pf}$. 

%
\begin{proposition}
	We have a short exact sequence of algebra homomorphisms
	\begin{equation}
		0 \to \bigcap_{\delta>0} \Psi_{\calc}^{m-1,\mathsf{s}-1+\delta,\ell-1,-1}\to  \Psi_{\calc}^{m,\mathsf{s},\ell,0}  \overset{(\sigma,N)}{\to} (S_{\calc}^{m,\mathsf{s},\ell,0} / \bigcap_{\delta > 0} S_{\calc}^{m-1,\mathsf{s}-1+\delta,\ell-1,0}) \times|_{\mathrm{con}} \Psi_{\mathrm{par}}^{\ell,\mathsf{s}|_{\mathrm{pf}} } \to 0.
	\end{equation}
	\end{proposition}
\begin{proof}
	From the principal symbol short-exact sequence \eqref{eq:short exact natres}, the kernel of $(\sigma,N)$ consists of those elements $A \in \bigcap_{\delta>0} \Psi_{\calc}^{m-1,\mathsf{s}-1+\delta,\ell-1,0}$ with $N(A)=0$. Let $a \in \bigcap_{\delta>0} S_{\calc}^{m-1,\mathsf{s}-1+\delta,\ell-1,0}$ be the full symbol of $A$ and 
\begin{equation}
\textbf{a} = \rho_{\mathrm{df}}^{(m-1)} \rho_{\mathrm{bf}}^{(\mathsf{s}-1)} \rho_{\calczero \mathrm{f} }^{(l-1)}a
\in \bigcap_{\delta>0} S_{\calc}^{0,\delta,0,0}
\end{equation}
be the normalized symbol of $A$. 
Then $N(A)=0$ is equivalent to 
\begin{equation} \label{eq:renomalized-a-pf-vanish}
\textbf{a} \big|_{\mathrm{pf}} = 0
\end{equation}
by the injectivity of left quantization. 
Since $\rho_{\mathrm{bf}}^\delta \textbf{a}$ is smooth in $\rho_{\mathrm{pf}}$ for any $\delta>0$ \footnote{This can be seen by noticing that $\partial_{\rho_{\mathrm{pf}}}^k(\rho_{\mathrm{bf}}^\delta \textbf{a}) = (\rho_{\mathrm{bf}}^{\delta/k}\partial_{\rho_{\mathrm{pf}}})^k(\textbf{a})$ in suitable coordinates and recalling \cref{eq:tildeV-calc}, \cref{eq:def-calc-symbol}.  }, in combination with \cref{eq:renomalized-a-pf-vanish}, we know
\begin{equation}
\textbf{a} \in  \bigcap_{\delta>0} S_{\calc}^{0,\delta,0,-1},
\end{equation}
which in turn gives
\begin{equation}
a \in \bigcap_{\delta>0} S_{\calc}^{m-1,\mathsf{s}-1+\delta,\ell-1,-1}
\end{equation}
and therefore $A \in \bigcap_{\delta>0} \Psi_{\calc}^{m-1,\mathsf{s}-1+\delta,\ell-1,-1}$.

	Similarly, if $([a_0], A_{\mathrm{par}}) \in (S_{\calc}^{m,\mathsf{s},\ell,0} / S_{\calc}^{m-1,\mathsf{s}-1 + \delta,\ell-1,0}) \times|_{\text{con}} \Psi_{\mathrm{par}}^{\ell,\mathsf{s}\big|_{\mathrm{pf}}}$ for every $\delta > 0$, then let $a_1$ be the left-reduced symbol of $A_{\mathrm{par}}$ in the parabolic calculus. Then $a_0$ can be restricted to $\mathrm{pf}$ and it agrees with $a_1$ up to an element $b$ of $S_{\mathrm{par}}^{\ell -1,\mathsf{s}\big|_{\mathrm{pf}} -1+\delta}$ for every $\delta > 0$. 
We can extend $b$ to a neighbourhood of $\mathrm{pf}$ and cut off, so that it is an element $\tilde b$ of $S_{\calc}^{m-1,\mathsf{s}-1 + \delta,\ell-1,0}$. Replacing $a_0$ by $a = a_0 +\tilde b$, this satisfies $[a] = [a_0]$ and it now agrees with $b$ at $\mathrm{pf}$. 
Such $a\in S_{\calc}^{m,\mathsf{s},\ell,0}$ satisfies
	\begin{equation}
	a \bmod \bigcap_{\delta>0} S_{\calc}^{m-1,\mathsf{s}-1+\delta,\ell-1,0} = [a_0] 
	\end{equation} 
	and $ a \big|_{\mathrm{pf}} = a_1$. Let $A=\operatorname{Op}(a)$. Then, $(\sigma,N)$ maps $A$ to $([a_0],a_1)$. 
	
	The algebra homomorphism property is a direct consequence of \eqref{eq:short exact natres} and Proposition~\ref{prop:N mult natres}. 
\end{proof}

\subsection{Relation between the $\Psi_{\calc}$ calculus and the parabolic calculus}
\label{subsec:calc_relations}

In order to utilize the main results of \cite{Parabolicsc}, we need to relate the calculus $\Psi_{\calc}$ used here to the parabolic calculus $\Psi_{\mathrm{par}}$ used there. There is however a technical complication to overcome, which is that we obliged to use different variable spacetime orders for the two calculi: a variable order in $\Psi_{\calc}$ will typically have nontrivial dependence in the $h$ direction (or more precisely, near $\mathrm{pf}$, in the $h \aang{\zeta}$ direction), so it will not coincide with any variable order lifted from the parabolic face. 
Simply allowing a family of variable orders on the parabolic phase space parametrized by $h$ won't resolve this issue since the smooth structure of our $\calc$-phase space at $\mathrm{df}$ at $h>0$ comes from the homogeneous, instead of parabolic, compactification.
To deal with this we accept a slight loss in the variable order and bound a given $\calc$-order above and below by two orders lifted from the parabolic phase space. 
The following lemma, which is virtually just the application of \Cref{prop:gluing_pseudos} to estimates of Sobolev regularity in the setting with $\mathsf{s}$ a variable order, is all we need:
\begin{lemma}
	Let $\mathsf{s} \in C^\infty({}^{\calc}\overline{T}^* \bbM;\bbR )$ denote a variable order, and fix variable orders
	\begin{equation} 
		\underline{\mathsf{s}}, \overline{\mathsf{s}} \in C^\infty({}^{\mathrm{par}}\overline{T}^* \bbM;\bbR )
		\label{eq:sss}
	\end{equation} 
	such that $\beta^*\underline{\mathsf{s}} \leq \mathsf{s} \leq \beta^*\overline{\mathsf{s}}$ in some neighborhood $\{\rho_{\mathrm{pf}} < \varepsilon \}$ of $\mathrm{pf}$, for some $\varepsilon>0$, where $\beta: {}^{\calc}\overline{T}^* \bbM\backslash \mathrm{df} \to {}^{\mathrm{par}}\overline{T}^* \bbM$ is the (smooth) map which blows down $\natural\mathrm{f}$ and then forgets the semiclassical parameter.
	Now fix $\chi \in S_{\calc}^{-\infty,0,0,0}$ such that $\operatorname{supp} \chi \Subset \{\rho_{\mathrm{pf}} <\varepsilon\}$, so that $Q=\operatorname{Op}(\chi)$ satisfies 
	\begin{equation}
		Q\in \Psi_{\calc}^{-\infty,0,0,0}, \quad \operatorname{WF}'_{\calc}(Q)\Subset \{\rho_{\mathrm{pf}}<\varepsilon\} .
	\end{equation}
	Then, for any $m,\ell,q \in \bbR$, 
	\begin{equation}
		A\in \Psi_{\calc}^{m,\mathsf{s},\ell,q} \Rightarrow AQ,QA \in h^{-q} L^\infty([0,1]; \Psi_{\mathrm{par}}^{\ell-q,\overline{\mathsf{s}} } )  .
	\end{equation}
	Conversely, if $A\in \Psi_{\mathrm{par}}^{\ell,\underline{\mathsf{s}}}$, then $AQ,QA \in \Psi_{\calc}^{-\infty,\mathsf{s},\ell,0 }$, where we consider $A$ to be an $h$-independent family of pseudodifferential operators.
	\label{lem:par_Planck_comp}
\end{lemma}

\begin{proof}
	Multiplying through by $h^q$, it suffices to prove the $q=0$ case. 
	
	Note that 
	\begin{equation} 
		Q(h) \in \Psi_{\mathrm{sc}}^{-\infty,0}\cap \Psi_{\mathrm{par}}^{-\infty,0}
	\end{equation} 
	for each individual $h>0$. Consequently, in each of the cases above, we can write $AQ = \operatorname{Op}(a\star \chi)$ and $QA = \operatorname{Op}(\chi \star a)$, where $a=\sigma_L(A)$ is the full left-reduced symbol of $a$. We just need to investigate the properties of the Moyal products $a\star \chi$ and $\chi \star a$.  We will discuss the case of $a\star \chi$, and $\chi \star a$ is analogous. 
	
	First suppose that $A\in \Psi_{\calc}^{m,\mathsf{s},\ell,0}$. Then $a\in S_{\calc}^{m,\mathsf{s},\ell,0}$. The Moyal product $a\star \chi$ lies in the same symbol spaces as $a\chi$, so 
	\begin{equation}
		a\star \chi \in S_{\calc}^{-\infty,\mathsf{s},\ell,0}  
	\end{equation} 
	
	We can weaken this slightly to
	
	\begin{align}
		\begin{split} 
			\rho_{\mathrm{bf}}^{ \beta^*\overline{\mathsf{s}} } a \star \chi 
           \in S_{\calc}^{-\infty,0,\ell,0} 
           =  S_{\mathrm{par,I,res}}^{-\infty,0,\ell,0}  
           \subseteq \tilde S_{\mathrm{par,I,res}}^{\ell,0,\ell,0} = S_{\mathrm{par,I}}^{\ell, 0,0}
		\end{split} 
	\end{align}
where we recall from Section~\ref{subsec:parIres} that $S_{\mathrm{par,I,res}}^{m,s,\ell,q}$ is the subspace of $\tilde S_{\mathrm{par,I,res}}^{m,s,\ell,q}$ consisting of symbols that are smooth, not just conormal, at $\mathrm{pf}$. In this display, we used the fact that $\beta^* \overline{\mathsf{s}} \geq \mathsf{s}$ on the essential support of $a\star \chi$. The final equality is an application of Lemma~\ref{lemma:blow up conormal}.  As a technical note, $S_{\mathrm{par,I}}^{\ell,0,0}$ here should be understood as the symbol class defined as in the variable order setting, i.e., bounded under iterated application of vector fields as in \eqref{eq:par_symbols_Vb_defn}. 
         
So, with $\mathcal{A}([0, 1)_h; \bullet)$ denoting the class of conormal functions taking value in certain Fr\'echet space $\bullet$ with each seminorm bounded, we have
	\begin{equation}
		a \star \chi \in \mathcal{A}([0,1)_h; S_{\mathrm{par}}^{\ell,\overline{\mathsf{s}}}) \subset L^\infty([0,1)_h, S_{\mathrm{par}}^{\ell,\overline{\mathsf{s}}}).
	\end{equation}
	Quantizing, we get $AQ \in L^\infty([0,1]_h;\Psi_{\mathrm{par}}^{\ell,\overline{\mathsf{s}} })$, as desired.
	
	We now discuss the ``converse'' direction. Suppose that $A\in \Psi_{\mathrm{par}}^{\ell,\underline{\mathsf{s}}}$, so that $a\in S_{\mathrm{par}}^{\ell,\underline{\mathsf{s}}}$. 
Viewing $\underline{\mathsf{s}}$ as a variable order on ${}^{\mathrm{par,I,res}}\overline{T}^* \bbM$ that is constant in $h$, then using Lemma~\ref{lemma:blow up conormal} we have 
	\begin{equation} 
		a \in \tilde S_{\mathrm{par,I,res}}^{\ell,\underline{\mathsf{s}},\ell,0}. 
	\end{equation} 
Smoothness in $h$ improves this to 	$a \in S_{\mathrm{par,I,res}}^{\ell,\underline{\mathsf{s}},\ell,0}$. 
		So, using the support condition on $\chi$, we know
	\begin{equation} 
		a \star \chi \in S_{\mathrm{par,I,res}}^{-\infty,\underline{\mathsf{s}},\ell,0}
	\end{equation}
	Because $\mathsf{s} \geq \beta^* \underline{\mathsf{s}}$ on the essential support of $\chi$, this can be weakened to
	\begin{equation} 
		a\star \chi \in 
		S_{\calc}^{-\infty,\ell,\mathsf{s},0}.
	\end{equation}  
	Quantizing, we get $AQ \in \Psi_{\calc}^{-\infty,\mathsf{s},\ell,0}$, as desired. 
\end{proof}

From \Cref{lem:par_Planck_comp}, we get:
\begin{proposition}
	For $j=0,1,2$, let $Q_j = \operatorname{Op}(\chi_j)$ for $\chi_j$ as in \Cref{lem:par_Planck_comp} satisfying 
	\begin{equation}
		\operatorname{supp} \chi_1 \Subset \chi^{-1}_2(\{1\}) \subset \operatorname{supp} \chi_2 \Subset \chi^{-1}_3(\{1\}) \subset \operatorname{supp} \chi_3 \Subset \{\rho_{\mathrm{pf}} < \varepsilon \}.
	\end{equation}
	In addition, let $m,\ell,q\in \bbR$, and let $\mathsf{s},\underline{\mathsf{s}},\overline{\mathsf{s}}$ be as in \Cref{lem:par_Planck_comp}. Then,  for any $N\in \bbN$, 
	\begin{align}
		h^{-q}\lVert Q_1(h) u \rVert_{H_{\mathrm{par}}^{\ell-q,\underline{\mathsf{s}}}} &\lesssim  \lVert  Q_2(h) u \rVert_{H_{\calc }^{m,\mathsf{s},\ell,q}  } + \lVert u \rVert_{H_{\calczero}^{-N,-N,-N} } \label{eq:k41}\\
		\lVert  Q_2(h) u \rVert_{H_{\calc }^{m,\mathsf{s},\ell,q}  } &\lesssim h^{-q} \lVert Q_3(h) u \rVert_{H_{\mathrm{par}}^{\ell-q,\overline{\mathsf{s}}} }  + \lVert u \rVert_{H_{\calczero}^{-N,-N,-N} }.
	\end{align}
	\label{prop:par_Planck_comp_full}
\end{proposition}
\begin{proof}
	It suffices to prove the $q=0$ case since $ \lVert  \bullet \rVert_{H_{\calc }^{m,\mathsf{s},\ell,q} }$ is equivalent to  $h^{-q}\lVert  \bullet \rVert_{H_{\calc }^{m,\mathsf{s},\ell-q,0} } $ . 
	\begin{itemize}
		\item 
		There exists an elliptic, invertible $\Lambda \in \Psi_{\mathrm{par}}^{\ell,\underline{\mathsf{s}}}$, so
		\begin{align} 
			\begin{split} 
				\lVert Q_1 u \rVert_{H_{\mathrm{par}}^{\ell,\underline{\mathsf{s}}} }  &\lesssim \lVert \Lambda Q_1u \rVert_{L^2}  \leq \lVert \Lambda Q_1Q_2 u  \rVert_{L^2} +\lVert \Lambda Q_1 (1-Q_2)u  \rVert_{L^2} \\
				&\lesssim \lVert Q_2 u \rVert_{H_{\calc}^{m,\mathsf{s},\ell,0} } + \lVert u \rVert_{H_{\calczero}^{-N,-N,-N}}
			\end{split} 
		\end{align} 
		Here, we used the fact that $\Lambda Q_1 \in \Psi_{\calc}^{-\infty,\mathsf{s},\ell,0}$ according to \Cref{lem:par_Planck_comp}, and also that -- since the microlocal supports of $Q_1$ and $1 - Q_2$ are disjoint,  the operator $ \Lambda Q_1 (1-Q_2) \in \Psi_{\calc}^{-\infty,-\infty,-\infty,-\infty}$, similar to the composition of $A_1 B_2$ in the proof of Proposition~\ref{prop:calc_composition}.
		\item Likewise, for elliptic $\Lambda \in \Psi_{\calc}^{m,\mathsf{s},\ell,0}$, then $\lVert Q_2 u \rVert_{H_{\calc}^{m,\mathsf{s},\ell,0} } \lesssim \lVert \Lambda Q_2 u \rVert_{L^2} +\lVert Q_2 u \rVert_{H_{\calc}^{-N,-N,-N,0}}$. We bound each of the terms on the right-hand side as follows:
		\begin{align}
			\begin{split} 
				\lVert \Lambda Q_2 u \rVert_{L^2 }&\leq  \lVert \Lambda Q_2 Q_3u \rVert_{L^2} + \lVert \Lambda Q_2 (1-Q_3)u \rVert_{L^2} \\
				&\lesssim \lVert Q_3 u \rVert_{H_{\mathrm{par}}^{\ell,\overline{\mathsf{s}}} } + \lVert u \rVert_{H_{\calczero}^{-N,-N,-N}},
			\end{split} 
		\end{align}
		since \Cref{lem:par_Planck_comp} gives $\Lambda Q_2 \in L^\infty([0,h_0]; \Psi_{\mathrm{par}}^{\ell,\overline{\mathsf{s}}} )$, and also using $\Lambda Q_2 (1-Q_3) \in \Psi_{\calczero}^{-\infty,-\infty,-\infty}$.
	\end{itemize}
\end{proof}

\subsection{The twice-resolved natural calculus $\Psi_{\calctwo}$} 
\label{sec:2-res_construction}
As discussed in the Introduction --- see the discussion around \eqref{eq:misc_c} --- we actually need to resolve the natural phase space at \emph{two} points in frequency space at $h=0$, namely $(\taun, \xin) = (\pm 1, 0)$. For any $\alpha \in \bbR$, the map $(\taun, \xin) \mapsto (\taun + \alpha, \xin)$ is a symplectomorphism of the natural phase space ${}^{\calczero}\overline{T}^* \bbM$ which is implemented analytically by $u \mapsto e^{i\alpha t/h^2} u$ or, at the operator level, by $B \mapsto B_\alpha=M_{\exp(i\alpha t/h^2)}\operatorname{Op}( b)M_{\exp(-i\alpha t/h^2)}$, where $M_\bullet:u\mapsto \bullet u$ is the multiplication map. This is an algebra isomorphism on $\Psi_{\calczero}$, as in Proposition~\ref{prop:translation}. Thus, we can equally well define the blow up of ${}^{\calczero}\overline{T}^* \bbM$ at any $\natural$-frequency point $(\taun, \xin) = (\alpha, 0)$ at $h=0$. Let us define 
the compactification  ${}^{\calc,\pm}\overline{T}^* \bbM \hookleftarrow T^* \bbR^{1,d}\times (0,\infty)_h$
to be the space obtained from ${}^{\calczero}\overline{T}^* \bbM$ by 
blowing up 
\begin{equation} 
	\{ \xi_{\calcshort}=0, \tau_{\calcshort} = \pm 1 ,h=0\} \subset {}^{\calczero}\overline{T}^* \bbM
\end{equation}
in the same manner used to obtain ${}^{\calc}\overline{T}^* \bbM$.
Then, we glue these two spaces together, with the `obvious' identification, to create 
\begin{equation}
{}^{\calctwo}\overline{T}^* \bbM=
({}^{\calc,+}\overline{T}^* \bbM \backslash  \{ \xi_{\calcshort}=0, \tau_{\calcshort} = - 1 \} ) \sqcup ({}^{\calc,-}\overline{T}^*\bbM \backslash \{ \xi_{\calcshort}=0, \tau_{\calcshort} =  1 \}) / \sim,
\end{equation}
where $\sim$ identifies points that are not in the lift of one of $\{ \xi_{\calcshort}=0, \tau_{\calcshort} = \pm 1 \}$. Thus, ${}^{\calctwo}\overline{T}^* \bbM$ is locally identical to ${}^{\calc}\overline{T}^* \bbM$ up to a translation in the $\taun$-direction of $\pm 1$, but it has \emph{two} parabolic faces rather than one. We denote these, i.e.\ the lift of $\{ \xi_{\calcshort}=0, \tau_{\calcshort} = \pm 1 ,h=0\}$ in ${}^{\calctwo}\overline{T}^*\bbM$ by $\mathrm{pf}_\pm$ and denote the blow-down map ${}^{\calctwo}\overline{T}^*\bbM \to {}^{\calczero}\overline{T}^*\bbM$ by $\beta_{\calctwo}$. We use the same notation for the other faces as for ${}^{\calc}\overline{T}^*\bbM$, i.e.\ $\mathrm{df},\mathrm{bf},\calczero\mathrm{f}$ are inherited from the two copies of ${}^{\calc}\overline{T}^*\bbM$ after the identification $\sim$ is made. Notice that the point $(\taun, \xin) = (0, 0)$ at $h=0$ in ${}^{\calctwo}\overline{T}^* \bbM$ is just another point on the $\calczero\mathrm{f}$-face; nothing special happens at this point. 
See \Cref{fig:phase2} for a depiction of ${}^{\calctwo}\overline{T}^* \bbM$. 

To define symbol spaces on this phase space, allowing a variable order $\mathsf{s}$ at $\mathrm{bf}$ but constant orders at the other boundary hypersurfaces, we use the same approach as for $\Psi_{\calc}$. 
Let $\mathcal{V}_{\calctwo}$ be the Lie algebra of smooth vector fields on ${}^{\calctwo} \overline{T}^* \bbM$ that are tangent to all boundary faces except for $\mathrm{pf}_+$ and $\mathrm{pf}_-$. This means that we are imposing smoothness at $\mathrm{pf}_\pm$ but only conormal regularity at the other boundary hypersurfaces. 
We then define
\begin{equation}
\tilde{\mathcal{V}}_{\calctwo} = \bigcup_{\delta > 0} \rho_{\mathrm{bf}}^\delta \mathcal{V}_{\calctwo}.
\end{equation}
Then we set, for orders $m$ at $\mathrm{df}$, $\ell$ at $\calczero\mathrm{f}$, and $q_\pm$ at $\mathrm{pf}_\pm$, and boundary defining functions $\rho_{\bullet}$ at the boundary hypersurface $\bullet$ of ${}^{\calctwo}\overline{T}^* \bbM$,
\begin{multline}\label{eq:symbol 2natres}
S_{\calctwo}^{m,\mathsf{s},\ell,q_+, q_-} = \Big\{ a \in \rho_{\mathrm{df}}^{-m} \rho_{\mathrm{bf}}^{-\mathsf{s}} \rho_{ \calczero\mathrm{f} }^{-\ell} \rho_{\mathrm{pf}_+}^{-q_+} \rho_{\mathrm{pf}_-}^{-q_-}  L^\infty({}^{\calctwo}\overline{T}^* \bbM) : \\ \big( \tilde{\mathcal{V}}_{\calctwo}  \big)^M a  \in \rho_{\mathrm{df}}^{-m} \rho_{\mathrm{bf}}^{-\mathsf{s}} \rho_{ \calczero\mathrm{f} }^{-\ell}\rho_{\mathrm{pf}_+}^{-q_+} \rho_{\mathrm{pf}_-}^{-q_-} L^\infty({}^{\calctwo}\overline{T}^* \bbM )  \text{ for all } M \in \mathbb{N}. \Big\}
\end{multline}

We have, under pointwise multiplication,
\begin{equation}
	S^{m,\mathsf{s},\ell;q_-,q_+}_{\calctwo}	\cdot S^{m',\mathsf{s}',\ell';q_+',q_-'}_{\calctwo} \subseteq  	S^{m+m',\mathsf{s}+\mathsf{s}',\ell+\ell';q_++q'_+,q_-+q'_-}_{\calctwo},
\end{equation}
for any other variable order $\mathsf{s}'\in C^\infty({}^{\calc} \overline{T}^* \bbM;\bbR)$. 
 
	Via translation up or down by one unit in $\tau_\natural$, any variable order $\mathsf{s} \in  C^\infty({}^{\calctwo} \overline{T}^* \bbM;\bbR)$ determines two variable orders 
	\begin{equation} \label{eq:spm}
		\mathsf{s}_\pm \in C^\infty({}^{\calc} \overline{T}^* \bbM\backslash \{\xi_\natural,\tau_\natural=\mp 2, h=0\} ;\bbR),
	\end{equation} 
	which, using the canonical identification $\mathrm{pf}\cong {}^{\mathrm{par}}\overline{T}^* \bbM$, yield variable par-orders. 
The behavior of $\mathsf{s}_\pm$ near $\{\xi_\natural,\tau_\natural=\mp 2, h=0\}$ will be unimportant for our purposes, so for convenience, we redefine $\mathsf{s}_\pm$ near this point so that each becomes a globally defined (smooth) variable order on the $\calc$-phase space.
In addition, $\mathsf{s}_0$ will denote a variable order on ${}^{\calctwo}\overline{T}^* \bbM$ that coincides with $\mathsf{s}$ on the complement of a small neighborhood of $\mathrm{pf}_+ \cup \mathrm{pf}_-$ (small enough to be contained in the region where $\mathsf{s}_\pm$ deviates from the relevant translate of $\mathsf{s}$) and is constant near $\mathrm{pf}_\pm$. Consequently, 
\begin{equation}\label{eq:s0}
	\mathsf{s}_0 \in C^\infty({}^{\calczero}\overline{T}^* \bbM;\bbR) 
\end{equation}
is (or more precisely may be viewed as) a variable order on ${}^{\calczero}\overline{T}^* \bbM$. We use these variable orders in the next proposition.

\begin{proposition} \label{prop:2_res_symbol_decomposition}
For any $b \in S^{m,\mathsf{s},\ell;q_-,q_+}_{\calctwo}$, we can decompose it as
\begin{equation} \label{eq:symbol_decomposition}
b = b_+(z,\tau_{\calcshort}-1,\xi_{\calcshort})
+b_-(z,\tau_{\calcshort}+1,\xi_{\calcshort})+b'(z,\tau_{\calcshort},\xi_{\calcshort}),
\end{equation}
where $b_\pm \in S^{-\infty,\mathsf{s}_\pm,\ell,q_\pm}_{\calc}$ have supports near $\mathrm{pf}$, and $b' \in S^{m,\mathsf{s}_0,\ell}_{\calczero}$ has  support away from $\mathrm{pf}_-\cup\mathrm{pf}_+$. In particular, 
\begin{equation} 
	b=b_\pm(z,\tau_{\calcshort}\mp 1,\xi_{\calcshort})
\end{equation} 
near $\mathrm{pf}_\pm$.
\end{proposition}
 
\begin{proof}
Multiply $b$ by a partition of unity on ${}^{\calctwo}\overline{T}^* \bbM$ with appropriate support conditions.
\end{proof}

Now we turn to consider the quantization of symbols in $S^{m,\mathsf{s},\ell;q_-,q_+}_{\calctwo}$.
For $b \in S^{m,\mathsf{s},\ell;q_-,q_+}_{\calctwo}$, decomposing it as in \Cref{prop:2_res_symbol_decomposition}, let
\begin{equation}  \label{eq:2_res_operator_definition}
\operatorname{Op}(b) = e^{it/h^2}B_+e^{-it/h^2} + e^{-it/h^2}B_-e^{it/h^2}+B',
\end{equation}
where $B_\pm = \operatorname{Op}(b_\pm)\in \Psi_{\calc}^{m,\mathsf{s}_\pm,\ell,q_\pm}$ and $B'=\operatorname{Op}(b')\in \Psi_{\calczero}^{m,\mathsf{s}_0,\ell}$ are defined using \cref{eq:quant}, \cref{eq:quant_calczero} respectively. As we always use the standard left quantization on $\bbR^D$ in all of our calculi, different choices of $(b_+,b_-,b')$ lead to the same $B$.

Then we define the class of $\calctwo$ pseudodifferential operators to be
\begin{equation}
\Psi^{m,\mathsf{s},\ell;q_-,q_+}_{\calctwo}
= \{  \operatorname{Op}(b): b \in S^{m,\mathsf{s},\ell;q_-,q_+}_{\calctwo}  \}.
\end{equation}
Evidently, $\Psi_{\calczero}^{m,\mathsf{s},\ell} \subseteq \Psi_{\calctwo}^{m,\mathsf{s},\ell;\ell,\ell}$, and if $A\in \Psi_{\calczero}^{m,\mathsf{s},\ell}$ has $\operatorname{WF}'_{\calczero}(A)$ disjoint from $\mathrm{pf}_\pm$, then 
\begin{equation}
	A \in \Psi^{m,\mathsf{s},\ell;-\infty,-\infty}_{\calctwo}.
\end{equation}

For $A \in \Psi^{m,\mathsf{s},\ell;q_-,q_+}_{\calctwo}$, we define the elliptic set, characteristic set and operator wavefront set analogously to the definitions for $\Psi_{\calc}$. These are open, closed and closed subsets,  respectively, of $\mathrm{df} \cup \natural\mathrm{f} \cup \mathrm{bf}$. 
For example, we define $\WF^{\prime}_{\calctwo}(A)$
similarly to the definition before \eqref{eq: calc residual symbol class},  as the subset of $\mathrm{df} \cup \natural\mathrm{f} \cup \mathrm{bf}$ determined by:
\begin{equation}
p\notin \WF^{\prime}_{\calctwo}(A) \iff \exists \chi \in C^\infty({}^{\calctwo}\overline{T}^* \bbM), \chi(p)=1, \chi a \in S_{\calctwo}^{-\infty,-\infty,-\infty;q_-,q_+}.
\end{equation}

Before stating the composition law of $\Psi_{\calctwo}$, we prove a few lemmas. The first one concerns the case when two operators has disjoint wavefront sets.
\begin{lemma} \label{lem:Planck-2res_composition_residual}
For each $j\in \{1,2\}$, let $B_j \in \Psi^{m_j,\mathsf{s}_j,\ell_j,q_j}_{\calc}$, $\alpha_j\in \bbR$, and $C_j$ denote the result of translating the image of 
\begin{equation}
	\operatorname{WF}^{\prime}_{\calc}(B_j)\cup \mathrm{pf} 
\end{equation}
under the blowdown map ${}^{\calc}\overline{T}^* \bbM\to {}^{\calczero}\overline{T}^* \bbM$ by $+\alpha_j$ units in the $\tau_\natural$-direction. Then, if $C_1\cap C_2=\varnothing$, 
\begin{equation}
	e^{-i \alpha_1 t/h^2} B_1 e^{i (\alpha_1-\alpha_2) t/h^2} B_2 e^{i\alpha_2 t/h^2}  \in \Psi_{\calczero}^{-\infty,-\infty,-\infty}.
\end{equation}
\end{lemma}
Note that the composition $e^{-i \alpha_1 t/h^2} B_1 e^{i (\alpha_1-\alpha_2) t/h^2} B_2 e^{i\alpha_2 t/h^2}  $ is \emph{a priori} well-defined as a one-parameter family of sc-pseudodifferential operators.
\begin{proof}
Follows immediately from \eqref{eq:expansion-calc-composition} as the entire expansion vanishes now.
\end{proof}

We also need the following result. 

\begin{proposition}
	For $A\in \Psi_{\calczero}^{m,\mathsf{s},\ell}$ and $B\in \Psi_{\calctwo}^{m',\mathsf{s}',\ell';q_-,q_+}$, if $\operatorname{WF}'_{\calczero}(A)$ is disjoint from $\mathrm{pf}_-\cup \mathrm{pf}_+$, then 
	\begin{equation}
		AB,BA \in \Psi_{\calczero}^{m+m',\mathsf{s}'',\ell+\ell'}, 
	\end{equation}
	where $\mathsf{s}''$ is any variable order given by $\mathsf{s}+\mathsf{s}'$ outside of some sufficiently small neighborhood of $\mathrm{pf}_-\cup \mathrm{pf}_+$. Moreover, if $B=\operatorname{Op}(b)$, then  
	\begin{equation}
		\sigma_{\calczero}^{m+m',\mathsf{s}'',\ell+\ell'}(AB)=
		\sigma_{\calczero}^{m+m',\mathsf{s}'',\ell+\ell'}(BA) = b \sigma_{\calczero}^{m,\mathsf{s},\ell}(a), 
		\label{eq:misc_180}
	\end{equation}
	\begin{equation}
		\operatorname{WF}'_{\calczero}(AB),\operatorname{WF}'_{\calczero}(BA) \subseteq \operatorname{WF}'_{\calczero}(A)\cap \operatorname{WF}_{\calctwo}^{\prime}(B).
		\label{eq:misc_190}
	\end{equation}
	 \label{prop:Planck-2res_composition_residual}
\end{proposition}
\begin{proof}
	We discuss $AB$, the alternate order $BA$ being similar.
	Write 
	\begin{equation}
			B = e^{it/h^2}B_+e^{-it/h^2} + e^{-it/h^2}B_-e^{it/h^2}+B',
	\end{equation}
	where $B_\pm \in \Psi_{\calc}$ is micro-supported near $\mathrm{pf}_\pm$ and $B'\in \Psi_{\calczero}$ is micro-supported away from $\mathrm{pf}_\pm$. We can choose $B_\pm$ micro-supported arbitrarily close to $\mathrm{pf}_\pm$, and in addition we may assume without loss of generality that 
	\begin{equation}
		\operatorname{WF}'_{\calczero}(B')\subseteq  \operatorname{WF}_{\calctwo}^{\prime }(B).
		\label{eq:misc_192}
	\end{equation}
	Then,
	\begin{equation}
		AB = Ae^{it/h^2}B_+e^{-it/h^2} + Ae^{-it/h^2}B_-e^{it/h^2}+AB'.
	\end{equation}
	We already know $AB'\in  \Psi_{\calczero}^{m+m',\mathsf{s}'',\ell+\ell'}$. As long as we choose $B_\pm$ to be microsupported sufficiently close to $\mathrm{pf}$ (which we can do), then \Cref{lem:Planck-2res_composition_residual} tells us that 
	\begin{equation}
		Ae^{it/h^2}B_+e^{-it/h^2} + Ae^{-it/h^2}B_-e^{it/h^2} \in \Psi_{\calczero}^{-\infty,-\infty,-\infty}.
	\end{equation}
	So, $AB \in  \Psi_{\calczero}^{m+m',\mathsf{s}'',\ell+\ell'}$, and its principal symbol is the same as that of $AB'$, which is $b \sigma_{\calczero}^{m,\mathsf{s},\ell}(a)$, so \cref{eq:misc_180} holds.
	
	 Similarly, the microsupport of $AB$ is the same as that of $AB'$, and, using the multiplicativity of essential support in the $\calczero$-calculus, 
	 \begin{equation}
	 	\operatorname{WF}'_{\calczero}(AB') \subseteq \operatorname{WF}'_{\calczero}(A) \operatorname{WF}'_{\calczero}(B') \subseteq \operatorname{WF}'_{\calczero}(A)\cap \operatorname{WF}_{\calctwo}^{\prime}(B),
	 \end{equation}
	 using \cref{eq:misc_192}.
\end{proof}

We are now in a position to show that $\Psi_{\calctwo}$ is a multi-graded algebra. 
\begin{proposition} \label{prop:2res composition}
	For $A \in \Psi^{m_1,\mathsf{s}_1,\ell_1;q_{1,-},q_{1,+}}_{\calctwo}, B \in \Psi^{m_2,\mathsf{s}_2,\ell_2;q_{2,-},q_{2,+}}_{\calctwo}$, we have
	\begin{equation}
		C = A \circ B \in \Psi^{m_1+m_2,\mathsf{s}_1+\mathsf{s}_2,\ell_1+\ell_2;q_{1,-}+q_{2,-},q_{1,+}+q_{2,+}}_{\calctwo}.
		\label{eq:misc_196}
	\end{equation}
	\begin{equation}
		\operatorname{WF}_{\calctwo}^{\prime}(C)\subseteq \operatorname{WF}_{\calctwo}^{\prime }(A)\cap  \operatorname{WF}_{\calctwo}^{\prime }(B).
		\label{eq:misc_197}
	\end{equation}
\end{proposition}

\begin{proof}
	By definition we can write
	\begin{align}
		A &= e^{it/h^2}A_+e^{-it/h^2} + e^{-it/h^2}A_-e^{it/h^2}+A', \\
		B &= e^{it/h^2}B_+e^{-it/h^2} + e^{-it/h^2}B_-e^{it/h^2}+B',
	\end{align}
	where $A_\pm, B_\pm \in \Psi_{\calc}$ are the quantizations of symbols supported near $\mathrm{pf}_\pm$ and $A',B'\in \Psi_{\calczero}$ are the quantizations of symbols supported away from $\mathrm{pf}_\pm$. 
	We can choose $A_\pm,B_\pm$ to be microsupported within frequency balls in $\bbR^{1,d}_{\tau_\natural,\xi_\natural}$ of radius $<1$ around the origin.

	We can apply
	\Cref{prop:Planck-2res_composition_residual} to handle the contributions from $A',B'$, so it suffices to consider the case $A',B'=0$. 
	Then we have
	\begin{align}
		\begin{split} 
		AB = &e^{it/h^2}A_+B_+e^{-it/h^2} + e^{-it/h^2}A_-B_-e^{it/h^2} 
		+ e^{it/h^2}A_+(e^{-2it/h^2}B_-e^{2it/h^2})e^{-it/h^2} 
		\\& + e^{-it/h^2}A_-(e^{2it/h^2}B_+e^{-2it/h^2})e^{it/h^2}. \label{eq:misc_stm}
		\end{split} 
	\end{align}
	Also, we have 
	\begin{equation}
		\operatorname{WF}_{\calctwo}^{\prime }(A) = \operatorname{WF}_{\calctwo}^{\prime }(e^{it/h^2}A_+e^{-it/h^2} )\sqcup \operatorname{WF}_{\calctwo}^{\prime }(e^{-it/h^2}A_-e^{it/h^2} ),
	\end{equation}
	and similarly for $B$.

	Using the composition law and wavefront set bounds in \Cref{prop:composition} we know that 
	\begin{equation} 
		A_\pm B_\pm \in \Psi^{m_1+m_2,\mathsf{s}_1+\mathsf{s}_2,\ell_1+\ell_2,q_{1,\pm}+q_{2,\pm}}_{\calc}.
	\end{equation}
	On the other hand, \Cref{lem:Planck-2res_composition_residual} tells us that 
	\begin{equation}
		e^{it/h^2}A_+(e^{-2it/h^2}B_-e^{2it/h^2})e^{-it/h^2} ,  e^{-it/h^2}A_-(e^{2it/h^2}B_+e^{-2it/h^2})e^{it/h^2}\in \Psi_{\calczero}^{-\infty,-\infty,-\infty}.
	\end{equation}
	We therefore get \cref{eq:misc_196}. The statement \cref{eq:misc_197} also follows.
\end{proof}

We define the $\calctwo$-Sobolev norm $\lVert \cdot \rVert_{H_{\calctwo}^{m,\mathsf{s},\ell;q_-,q_+}}$,  for $h$ sufficiently small, by choosing a decomposition $\Id = Q_+ + Q_-$, $Q_\pm \in \Psi_{\calczero}^{0,0,0}$, such that 
\begin{equation}\label{eq:Qpm-properties}
\begin{aligned}
\operatorname{WF}'_{\calczero}(\Id - Q_+) \cap \{ \taun = +1, \xin = 0, h = 0 \} &= \varnothing, \\
\operatorname{WF}'_{\calczero}(\Id - Q_-) \cap \{ \taun = -1, \xin = 0, h = 0 \} &= \varnothing.
\end{aligned}\end{equation}
We also suppose that, with $T_{\pm 1}$ denoting the map $(h, z, \taun, \xin) \mapsto (h, z, \taun \pm 1, \xin)$, we have 
\begin{equation}
\mathsf{s}_\pm = \mathsf{s} \circ T_{\pm 1} \text{ on } \operatorname{WF}'_{\calczero}(Q_\pm). 
\label{eq:spm orders}\end{equation}
We then define 
\begin{equation} \label{eq:definition_2res norm}
\lVert u \rVert_{H_{\calctwo}^{m,\mathsf{s},\ell;q_-,q_+}}
 =   \lVert e^{-it/h^2}Q_+u\rVert_{H_{\calc}^{m,\mathsf{s}_+,\ell,q_+}} +  \lVert e^{it/h^2} Q_-u\rVert_{H_{\calc}^{m,\mathsf{s}_-,\ell,q_-}}  .\end{equation}
 Notice that the modulation $e^{-it/h^2}$ in the first term has the effect of $T_{-1}$ on the phase space, which moves $\mathrm{pf}_+$ to $\mathrm{pf}$, the parabolic face in the $\calc$-phase space. Similarly, the modulation $e^{it/h^2}$ in the second term has the effect of moving $\mathrm{pf}_-$ to $\mathrm{pf}$. 
%
It is not hard to check that different choices of $Q_\pm$ lead to equivalent norms. In Section~\ref{sec:estimates} we will choose $Q_\pm$ with an additional property with respect to the characteristic variety of our operator $P$; see \eqref{eq:misc_320}.

\section{The operator}
\label{sec:setup}

\subsection{Geometric assumptions}
\label{subsec:geometric_assumptions}

For each $c>0$, let $\eta=\eta(c) = -c^2 \dd t^2 + \dd x^2$ denote the exact Minkowski metric on $\bbR^{1,d}$ with speed of light $c$.  
We consider geometric perturbations with coefficients that are classical symbols. 
Concretely, the class of $m$-th order classical symbols is defined by 
\begin{equation}
S^{m}_{\cl}(\bbR^{1,d};\bbR)=(1+t^2+r^2)^{m/2} C^\infty(\overline{\bbR^{1,d}};\bbR).
\end{equation}

The metric $g=\{g(c)\}_{c>1}$ that we will consider is a family of Lorentzian metrics on $\bbR^{1,d}$ of the form 
\begin{equation} 
	g(c) = \eta(c) +  \alpha \dd t^2 + \sum_{j=1}^d \frac{w_j}{c}\mathrm{d} t\, \mathrm{d} x_j +\frac{1}{c^2} \sum_{j,k=1}^d h_{j,k} \dd x_j \dd x_k
	\label{eq:metric_form}
\end{equation}
for 
\begin{align} 
	\begin{split}
		h_{j,k}=h_{k,j},\alpha,w_j &\in C^\infty([0,1)_{1/c};  S^{-1}_{\cl}(\bbR^{1,d};\bbR)) \\ 
		&= C^\infty([0,1)_{1/c}; (1+t^2+r^2)^{-1/2} C^\infty(\overline{\bbR^{1,d}};\bbR)) ,
	\end{split}
\end{align} 
where $j,k\in \{1,\dots,d\}$. 
So, $(1-\alpha/c^2)^{1/2}$ is the \emph{lapse function} of the time-function $t$, and $w_j/c^2$ is the spacetime \emph{wind}, to use the standard descriptive terminology.

We assume that the metric $g$ is non-trapping for each $c>0$. For applications to the Cauchy problem, we will also assume that $g$ is globally hyperbolic with the Minkowski coordinate $t$ a time-function for $g$ and $\{t=0\}$ a Cauchy surface for the spacetime. Consequently, the Cauchy problem with initial data on the spacelike hypersurface $\{t=0\}$ is globally well-posed. Under the stated assumptions on $g$, these conditions are automatically satisfied as long as $c$ is sufficiently large. Since we are concerned with the $c\to\infty$ limit, the assumptions in this paragraph do not limit applicability. 

Now consider the d'Alembertian $\square_{g}=-\mathrm{d}^* \mathrm{d}$, defined using the sign convention
\begin{equation}
	\square_g = \frac{1}{\sqrt{|g|}} \sum_{i,j=0}^d \frac{\partial}{\partial z^i} \Big( \sqrt{|g|} g^{ij} \frac{\partial}{\partial z^j} \Big) =  \sum_{i,j=0}^d \Big( g^{ij} \frac{\partial^2}{\partial z^i \partial z^j} + \frac{\partial g^{ij}}{\partial z^i} \frac{\partial}{\partial z^j} + \frac{g^{ij}}{2|g|} \frac{\partial |g|}{\partial z^i} \frac{\partial}{\partial z^j} \Big),
	\label{eq:4hk6nv}
\end{equation}
where $z^0 = t$ and $\smash{z^j=x_j}$, and where $g^{ij}$ are the entries of the inverse of the matrix $\smash{\{g_{ij}\}_{i,j=0}^d}$. 
We think of $\square_g$ as a one-parameter family of differential operators on $\bbR^{1,d}$, parametrized by $c$ (suppressing this $c$-dependence in the notation). 

It must be remembered that our Minkowski d'Alembertian $\square$ is also a one-parameter family of operators, depending on $c$. 

\begin{proposition}
	Letting $\square = -c^{-2}\partial_t^2  + \triangle = -c^{-2} \partial_t^2 +\sum_{j=0}^d \partial_{x_j}^2$ denote the exact Minkowski d'Alembertian, 
	\begin{multline}
		\square_g - \square \in  \operatorname{span}_{C^\infty([0,1)_{1/c} ; S_{\cl}^{-1}(\bbR^{1,d};\bbR)) } \Big\{ \frac{1}{c^4} \frac{\partial^2}{\partial t^2}, \frac{1}{c^3} \frac{\partial^2}{\partial t \partial x_j}, \frac{1}{c^2} \frac{\partial^2}{\partial x_j x_k} \Big\} \\ 
		+ \operatorname{span}_{C^\infty([0,1)_{1/c} ; S_{\cl}^{-2}(\bbR^{1,d};\bbR)) } \Big\{ \frac{1}{c^3}\frac{\partial}{\partial t},\frac{1}{c^2} \frac{\partial}{\partial x_j} \Big\},
		\label{eq:48k44ds}
	\end{multline}
	where $j,k$ vary over $1,\dots,d$. 
	\label{prop:dAlembert_dif_form}
\end{proposition}
\begin{proof}
	Let us work out the form of $|g|= |\!\operatorname{det}(g)|$ and the $(1+d)\times (1+d)$ matrix $g^{-1}$. 
	Let $\sqrt{\eta}$ be the diagonal matrix with entries $ci,1,\dots,1$. Then, $\det(g) = \det(\sqrt{\eta}) \det (\mathrm{id}+M) \det(\sqrt{\eta})= -c^2 \det(\mathrm{id} + M)$ for 
	\begin{equation} 
		M  = \eta^{-1/2} (g - \eta) \eta^{-1/2} = \frac{1}{c^2} 
		\begin{pmatrix}
			- \alpha & -i \bfw \\
			-i \bfw^\intercal  & h
		\end{pmatrix} \in c^{-2} C^\infty([0,1)_{1/c};S_{\cl}^{-1} (\bbR^{1,d}; \bbC^{(1+d)\times (1+d)}) ), 
	\end{equation} 
	where $\bfw = (w_1,\dots,w_j)$ and $h=\{h_{j,k}\}_{j,k=1}^d$. 
	It follows that 
	\begin{equation} 
		\det(\mathrm{id}+M) -1 \in c^{-2} C^\infty([0,1)_{1/c};S_{\cl}^{-1}(\bbR^{1,d})),
	\end{equation} 
	and also that $(\mathrm{id}+M)^{-1} -\mathrm{id} \in  c^{-2} C^\infty([0,1)_{1/c};S_{\cl}^{-1} (\bbR^{1,d}; \bbC^{(1+d)\times (1+d)}))$. Consequently, $g^{-1}$ has the form 
	\begin{multline}
		g^{-1} -\eta^{-1} = \eta^{-1/2} ((\mathrm{id}+M)^{-1} - \operatorname{id}) \eta^{-1/2} \\ \in 
		\begin{pmatrix}
			c^{-4} C^\infty([0,1)_{1/c};S_{\cl}^{-1}(\bbR^{1,d};\bbR) ) & c^{-3} C^\infty([0,1)_{1/c};S_{\cl}^{-1} (\bbR^{1,d} ; \bbR^d )) \\ 
			c^{-3} C^\infty([0,1)_{1/c};S_{\cl}^{-1} (\bbR^{1,d};\bbR^d))& c^{-2} C^\infty([0,1)_{1/c};S_{\cl}^{-1} (\bbR^{1,d}; \bbR^{d\times d}))
		\end{pmatrix} .
		\label{eq:ginv_form}
	\end{multline}
	So, the various terms in \cref{eq:4hk6nv} are given by
	\begin{align}
		\begin{split} 
			\sum_{i,j=0}^d  g^{ij} \frac{\partial^2}{\partial z^i \partial z^j} - \square &\in \operatorname{span}_{C^\infty([0,1)_{1/c} ; S_{\cl}^{-1}(\bbR^{1,d})) } \Big\{ \frac{1}{c^4} \frac{\partial^2}{\partial t^2}, \frac{1}{c^3} \frac{\partial^2}{\partial t \partial x_j}, \frac{1}{c^2} \frac{\partial^2}{\partial x_j \partial x_k} \Big\}, \\ 
			\frac{\partial g^{ij}}{\partial z^i} \frac{\partial}{\partial z^j} ,\frac{g^{ij}}{2|g|} \frac{\partial |g|}{\partial z^i} \frac{\partial}{\partial z^j} &\in  \operatorname{span}_{C^\infty([0,1)_{1/c} ; S_{\cl}^{-2}(\bbR^{1,d})) } \Big\{ \frac{1}{c^3}\frac{\partial}{\partial t}, \frac{1}{c^2} \frac{\partial}{\partial x^j} \Big\}.
		\end{split}
	\end{align}
\end{proof}

\begin{corollary}  For $g$ as in \eqref{eq:metric_form},
	$\square_g \in \operatorname{Diff}_{\calczero}^{2,0,2}$. 
\end{corollary}

The $O(c^{-4})$ term in the coefficient of $\partial_t^2$ in $\square_g - \square$ will end up being important.  We call this $\aleph \in S_{\cl}^{-1}(\bbR^{1,d};\bbR)$. It is defined by
\begin{multline} 
	\square_g -\square - c^{-4} \aleph \partial_t^2 \in \operatorname{span}_{C^\infty([0,1)_{1/c} ; S_{\cl}^{-1}(\bbR^{1,d};\bbR)) } \Big\{ \frac{1}{c^5} \frac{\partial^2}{\partial t^2}, \frac{1}{c^3} \frac{\partial^2}{\partial t \partial x_j} , \frac{1}{c^2} \frac{\partial}{\partial x_j x_k}\Big\} 
	\\ + \operatorname{span}_{C^\infty([0,1)_{1/c} ; S_{\cl}^{-2}(\bbR^{1,d};\bbR)) } \Big\{ \frac{1}{c^3}\frac{\partial}{\partial t},\frac{1}{c^2} \frac{\partial}{\partial x_j} \Big\}.
	\label{eq:aleph_def}
\end{multline} 
The physical interpretation of this coefficient is that it is proportional to the asymptotic mass of the spacetime.

\subsection{Operators}
\label{subsec:operators}

We now specify the class of differential operators treated in this paper. Let $P = \{P(c)\}_{c>1}$ denote a family of differential operators on $\bbR^{1,d}$ of the form 
\begin{equation}
	P = \square_{g} -  c^2 + \frac{i\beta}{c^2} \frac{\partial}{\partial t} + \sum_{j=1}^d iB_j \frac{\partial}{\partial x_j} + W,
	\label{eq:P_def}
\end{equation}
where  
\begin{equation} 
	\beta ,B_j,W\in C^\infty([0,1)_{1/c}; S_{\cl}^{-1}(\bbR^{1,d})).
	\label{eq:misc_062}
\end{equation}  
Note that, by assumption, these functions (as well as the coefficients of $\square_g$) can be restricted to $c=\infty$, i.e. there is a well-defined limit, in $S_{\cl}^{-1}(\bbR^{1,d})$, as $c \to \infty$. 

When proving symbolic estimates, it will be convenient for $P$ to be $L^2$-symmetric at the first subleading order in the sense of decay, so we require that 
\begin{equation}
	\Im \beta , \Im B_j, \Im W \in C^\infty([0,1)_{1/c}; S_{\cl}^{-2}(\bbR^{1,d})).
	\label{eq:misc_158}
\end{equation}
This assumption is not crucial for most developments below. Dropping it results in a shift of the thresholds in the radial point estimates in \S\ref{sec:estimates} (replacing $-1/2$ with some other $s_0\in \bbR$), but the analysis is similar.

In addition, since the analysis of the Schr\"odinger equation is cleanest when the operator is $L^2$-symmetric modulo a scalar term, we require that
\begin{equation}
	\Im B_j|_{c=\infty} = 0.
	\label{eq:misc_170}
\end{equation}
Then, $P$ is symmetric (with respect to the $L^2(\bbR^{1,d})$-pairing with density induced by $g$) to leading order in $c$. 
This will be used in \S\ref{sec:schrodinger_estimate} (see \eqref{eq: time derivative of mass-2}).\footnote{ 
As one can see from eq.\ \eqref{eq: time derivative of mass-2}, this is used to avoid a $H^{1/2}$-level term in the derivative of the $L^2$-norm of the Schrödinger solution. In principle, one can treat this term and avoid the assumption by exploiting local smoothing type estimates of solutions to Schrödinger equations. But that is an unnecessary complication for our purpose.  }

\begin{proposition}\label{prop:P form}
	The conjugated operators $P_\pm$  defined by \cref{eq:P+-_form} are given by
	\begin{equation} 
		P_\pm (c) = \square_g  \mp 2 i \partial_t +i \beta c^{-2} \partial_t \mp  \beta + i \sum_{j=1}^dB_j \partial_{x_j} + W -  \aleph+ \calE 
	\end{equation} 
	for a differential operator $\calE$ of the form 
	\begin{equation} 
		\calE\in c^{-1}\operatorname{span}_{ C^\infty([0,1)_{1/c} ; S_{\cl}^{-1}(\bbR^{1,d}))} \{ c^{-1}\partial_t, \partial_{x_j} ,1  \}.
	\end{equation} 
	In particular, $\calE \in \operatorname{Diff}_{\calc}^{1,-1,0,-1}$.
	\label{prop:conjugate_form}
\end{proposition}
The notation ``$\calE$'' is meant to indicate that it is to be regarded as a ``small'' error term. 
\begin{proof}
	First, consider the case where $P=\square- c^2$. Then, $P_\pm(c) = \square \mp 2 i  \partial_t$. 
	In the general case, we need to add the following terms (in addition to those in $P-\square - c^2$) which arise during conjugation: 
	\begin{itemize}
		\item The function $\exp(\mp i c^2 t ) (\square_g - \square) \exp(\pm i c^2 t)$. By \cref{eq:aleph_def}, this has the form 
		\begin{equation}
			\exp(\mp i c^2 t ) (\square_g - \square) \exp(\pm i c^2 t) = -  \aleph + c^{-1} C^\infty([0,1)_{1/c}; S_{\cl}^{-1}(\bbR^{1,d}) ).
		\end{equation}
		\item The function $e^{\mp ic^2 t}  c^{-2} (i\beta) \partial_t e^{\pm ic^2 t} = \mp  \beta$.
		\item The derivatives that show up when one conjugates the operator $\square_g - \square$, these resulting from half of the derivatives in the second-order terms in $\square_g-\square$ falling on the exponential. These terms have the form 
		\begin{equation}
			\pm \frac{2 i  \aleph}{c^2} \frac{\partial}{\partial t} + 
			\operatorname{span}_{ C^\infty([0,1)_{1/c}; S_{\cl}^{-1}(\bbR^{1,d}))}  \Big\{ \frac{1}{c^3} \frac{\partial}{\partial t}, \frac{1}{c} \frac{\partial}{\partial x_j}\Big\}
		\end{equation}
		by \cref{eq:aleph_def}. 
	\end{itemize}
	These terms can be taken part of $\calE$. 
\end{proof}

Recall the definition of classical $\calc$-symbols, equation \eqref{eq:def-calc-symbol-classical}. 
\begin{corollary}
	$P_\pm\in \operatorname{Diff}_{\calc}^{2,0,2,0}$, and we can select a representative $p$ of $\sigma_{\calc}^{2,0,2,0}(P_\pm )$ such that  
	\begin{equation} 
		p \in 
		S_{\calc,\cl}^{2,0,2,0} 
		\label{eq:principal_symbol_phg}
	\end{equation} 
	is given by $p= -g^{-1}(-,-) \pm 2  \tau $.
	\label{prop:principal_symbol}
\end{corollary}
\begin{proof}
	By \Cref{lem:derivatives}, we have $\partial_t^2 \in \operatorname{Diff}_{\calc}^{2,0,4,0}$ and $\triangle \in \operatorname{Diff}_{\calc}^{2,0,2,0}$. 	
	Since $\partial_t^2$ is suppressed by a factor of $c^{-2}$ in $\square$, which improves its orders at $\natural\mathrm{f},\mathrm{pf}$ by $2$, we have 
	\begin{equation} 
		\square \in \operatorname{Diff}_{\calc}^{2,0,2,0}.
	\end{equation} 
	The terms on the right-hand side of \cref{eq:48k44ds} are all in $\operatorname{Diff}_{\calc}^{2,-1,0,-2}$, so $\square_g \in \operatorname{Diff}_{\calc}^{2,0,2,0}$.
	The other terms in \Cref{prop:conjugate_form} lie in $\operatorname{Diff}_{\calc}^{2,0,2,0}$ as well. Note in particular that
	\begin{equation}
		\calE \in \operatorname{Diff}_{\calc}^{1,-1,0,-1}. 
		\label{eq:misc_177}
	\end{equation} 
	So, we conclude $P_\pm\in \operatorname{Diff}_{\calc}^{2,0,2,0}$.

	From this description,  
	\begin{equation}
		\sigma_{\calc}^{2,0,2,0} ( P_\pm )= \sigma_{\calc}^{2,0,2,0}(\square_g) \pm 2 \sigma_{\calc}^{2,0,2,0}(-i\partial_t) .
		\label{eq:misc_178}
	\end{equation}
	We could throw out terms involving $\beta,B,W$ because, though they are not subleading at $\mathrm{pf}$, the principal symbol of elements of $\Psi_{\calc}^{2,0,2,0}$ was defined only modulo elements of $S_{\calc}^{1,-1,1,0}$.  Evaluating each term on the right-hand side of \cref{eq:misc_178} yields the formula for $p$ in the proposition. The frequency coordinates $\tau,\xi$ are fully polyhomogeneous on ${}^{\calc}\overline{T}^* \bbM$, and we computed their orders in the proof of the previous lemma.
	Plugging this information into the formula for $p$ yields \cref{eq:principal_symbol_phg}.
\end{proof}

Note that the \emph{unconjugated} Klein--Gordon operator $\square-c^2$ satisfies 
\begin{equation} 
	\square-c^2\in \operatorname{Diff}_{\calc}^{2,0,2,2},
\end{equation}
not $\square-c^2\in \operatorname{Diff}_{\calc}^{2,0,2,0}$. This is one reason why it is useful to conjugate the $c^2$ term away, as well as the sense in which this constant term is dominant in the $c\to\infty$ limit; it is second-order at $\mathrm{pf}$, whereas all the other terms in the operator are zeroth-order.

\begin{proposition} \label{prop:normal}
	The $\calc$-normal operators of $P_\pm$ at $\mathrm{pf}$ are given by 
	\begin{equation}\label{eq:Ppm normal}
		N(P_\pm) = \mp 2 i \partial_t + \triangle  + \Big( \pm  \beta + \sum_{j=1}^d i B_j\partial_{x_j} + W\Big)\Big|_{c=\infty} - \aleph,  
	\end{equation}	
	where $\beta,B_j,W,\aleph$ are as above.
\end{proposition}
\begin{proof} 
	Immediate from \Cref{prop:dAlembert_dif_form} and \Cref{prop:conjugate_form}. 
\end{proof}
Notice that the term $\aleph$, whose physical interpretation we described above, appears in this normal operator. This is the mechanism via which gravity contributes an effective Newtonian potential to the non-relativistic limit.

For later use, let us examine the terms in $P_\pm$ that fail to be symmetric, with respect to the $L^2$-pairing in terms of the Euclidean measure.
Let 
\begin{equation} 
	P_1 = \Im \{ P_\pm \} = (2i)^{-1} ( P_\pm - P_\pm^*),
\end{equation} 
where $*$ denotes the Euclidean $L^2$-based adjoint (which is also the $L^2(\bbR^{1,d},\eta(c'))$-based adjoint for every $c'$). Then:
\begin{proposition} 
	Even if we do not assume \cref{eq:misc_170}, only \cref{eq:misc_158}, the operator $P_1$ satisfies 
	\begin{equation}
		P_1 \in \operatorname{span}_{ C^\infty([0,1)_{1/c}; S_{\cl}^{-2}(\bbR^{1,d};\bbC))}  \Big\{ \frac{1}{c^2} \frac{\partial}{\partial t},\frac{\partial}{\partial x_j},1\Big\} \subset \operatorname{Diff}_{\calc}^{1,-2,1,0}.
		\label{eq:gdf31}
	\end{equation}
	Assuming \cref{eq:misc_170}, then this is improved to $P_1 \in \operatorname{Diff}_{\calc}^{1,-2,0,-1}$.
	\label{prop:P1}
\end{proposition}
\begin{proof}
	First, we show (without using \cref{eq:misc_158} or \cref{eq:misc_170}) that 
	\begin{multline}
		P_1 - \Im \Big\{  \frac{i\beta}{c^2} \frac{\partial}{\partial t} \mp \beta + \sum_{j=1}^d iB_j \frac{\partial}{\partial x_j} + W \Big\} \\ \in \operatorname{span}_{ C^\infty([0,1)_{1/c}; S_{\cl}^{-2}(\bbR^{1,d};\bbC))}  \Big\{ \frac{1}{c^3} \frac{\partial}{\partial t}, \frac{1}{c^2} \frac{\partial}{\partial x_j},\frac{1}{c}\Big\} \subset \operatorname{Diff}_{\calc}^{1,-2,-1,-1}.
		\label{eq:P1_stronger}
	\end{multline}
	\Cref{eq:misc_158} then yields \cref{eq:gdf31}. 
	
	In order to prove \cref{eq:P1_stronger}, note that we can write 
	\begin{align}
		\begin{split} 
			P_1 - \Im \Big\{  \frac{i\beta}{c^2} \frac{\partial}{\partial t} \mp \beta + \sum_{j=1}^d iB_j \frac{\partial}{\partial x_j} + W \Big\}  = \Im \tilde{\square}_{g,\pm}  &= M_{e^{\mp i c^2 t}} (\Im \square_g) M_{e^{\pm ic^2 t} } \\ &= M_{e^{\mp i c^2 t}} (\Im (\square_g-\square) )M_{e^{\pm ic^2 t} } .
		\end{split} 
	\end{align}
	\Cref{prop:dAlembert_dif_form} (and using that the coefficients of $\square_g-\square$ are real) now gives
	\begin{equation}
		\Im (\square_g - \square) \in
		\operatorname{span}_{C^\infty([0,1)_{1/c} ; S_{\cl}^{-2}(\bbR^{1,d};\bbR)) } \Big\{ \frac{1}{c^3}\frac{\partial}{\partial t},\frac{1}{c^2} \frac{\partial}{\partial x_j}\Big\}.
	\end{equation}
	Consequently, 
	\begin{equation}
		M_{e^{\mp i c^2 t}} \Im (\square_g-\square) M_{e^{\pm ic^2 t} } \in \operatorname{span}_{C^\infty([0,1)_{1/c} ; S_{\cl}^{-2}(\bbR^{1,d};\bbC)) } \Big\{ \frac{1}{c^3}\frac{\partial}{\partial t},\frac{1}{c^2} \frac{\partial}{\partial x_j}, \frac{1}{c}\Big\},
	\end{equation}
	so we conclude \cref{eq:P1_stronger}. 	
	
	Using  \Cref{prop:dAlembert_dif_form} and \Cref{prop:conjugate_form}, we have $P_\pm  - (\square\pm 2i\partial_t) \in \scrV$ for 
	\begin{equation}
		\scrV= \operatorname{span}_{C^\infty([0,1 )_{1/c} ; S_{\cl}^{-1}(\bbR^{1,d}) ) } \Big\{ \frac{1}{c^4} \frac{\partial^2}{\partial t^2}, \frac{1}{c^3} \frac{\partial^2}{\partial t \partial x_j} , \frac{1}{c^2} \frac{\partial^2}{\partial x_jx_k}, \frac{1}{c^2} \frac{\partial}{\partial t}, \frac{\partial}{\partial x_j}, 1\Big\}.
	\end{equation}
	The operator $\square \mp 2 i \partial_t$ is symmetric with respect to the usual $L^2$-inner product, and the adjoint of any element of $\scrV$ is another member of $\scrV$, so
	\begin{equation} 
		P_\pm^* - (\square \mp 2i\partial_t) \in \scrV
	\end{equation}  
	as well. Consequently, $P_1\in \scrV$. But, it follows from the definition of $P_\pm$ that the coefficients of all of the second-order terms in it are real, and this implies that $\Im \{P_\pm\}$ is a first-order differential operator. \Cref{eq:gdf31} must therefore hold. 

	For the last part of the claim, just notice that the contribution from $\Im (i B_j\partial_{x_j})$-terms will have extra $c^{-1}$-decay when we assume \cref{eq:misc_170}. So it lies in $\operatorname{Diff}_{\calc}^{1,-2,0,-1}$. 
    The conclusion follows since all other terms in $P_1$  lies in this class as well.
\end{proof}

\section{Classical dynamics in the compactified phase space}
\label{sec:flow}

In this section, we discuss the dynamics associated to the conjugated Klein--Gordon operator $P_\pm$ from the $\calc$-perspective, which was introduced in \S\ref{sec:calculus}. The principal symbol $p$ was computed in \Cref{prop:principal_symbol}.\footnote{As a reminder, $p$ depends on the sign $\pm$ in $P_\pm$, but we will suppress this dependence in the notation.} It remains to discuss its characteristic set and the structure of its Hamiltonian vector field $H_p$. 
What we will see is an interpolation between the structure of the Hamiltonian flow for the Klein--Gordon equation and that for the Schr\"odinger equation; compare 
\begin{itemize}
	\item \Cref{fig:individual_flows}, which shows the appropriately scaled flows of the Klein--Gordon and Schr\"odinger equations within their characteristic sets (in ${}^{\mathrm{sc}}\overline{T}^* \bbM$ and ${}^{\mathrm{par}}\overline{T}^* \bbM$ respectively), with
	\item  \Cref{fig:radial_sets}, \Cref{fig:Schrflow_with_pf} which show the flow generated by $H_p$ in the corresponding portion of the $\calc$-characteristic set (in ${}^{\calc} \overline{T}^* \bbM$).
\end{itemize}
In short, \textit{we have a source-to-sink flow within the ``good'' component $\Sigma$ of the $\calc$-characteristic set.}

The terms in the operator $P_\pm$ 
that are not present in $(\square_g-c^2)_\pm$ are subprincipal at $\mathrm{df}$, $\natural\mathrm{f}$ and  $\mathrm{bf}$, so as far as the principal symbol of $P_\pm$ is concerned, these can be ignored. They play no role in this section.

The metric perturbation $g-\eta$ \emph{does} enter the principal symbol, but only via terms which are suppressed as $c\to\infty$ or $|r|,|t|\to\infty$.  
More precisely, 
\begin{equation} 
p-p_0 \in S_{\calc}^{2,-1,1,0}, 
\end{equation}  
where 
\begin{equation}
p_0 = h^2\tau^2 - \xi^2 \pm 2 \tau, 
\label{eq:p0}
\end{equation}
which equals to $p$ when $g$ is the exact Minkowski metric $\eta$. (As a reminder, $h=1/c$.)
Consequently, when doing computations at $\natural\mathrm{f}\cup \mathrm{bf}$, instead of $p$ we can usually work with $p_0$.

\subsection{Computation of the characteristic set}
\label{subsec:char}
The characteristic set $\operatorname{char}_{\calc}^{2,0,2,0}(P_\pm)$ is given by the  vanishing set of
	\begin{equation} 
	\mathsf{p} = \rho_{\mathrm{df}}^2 \rho_{\natural\mathrm{f}}^2 \sigma_{\calc}^{2,0,2,0}(P_\pm)= \rho_{\mathrm{df}}^2 \rho_{\natural\mathrm{f}}^2 p, 
\end{equation}
 minus the interior of the $\calc$-phase space.
This subsection is devoted to making precise and proving the following intuitive proposition: the characteristic set consists of hyperboloids over spacetime infinity and the corresponding points at fiber infinity and $\mathrm{pf}$. Consequently, the characteristic set consists of two connected components.

Only one of the two connected components of the characteristic set intersects $\mathrm{pf}$. We call it $\Sigma$. We call the other component $\Sigma_{\mathrm{bad}}$ and it can mostly be ignored.
Each of $\Sigma_{\mathrm{bad}},\Sigma$ corresponds to one of the components $\Sigma_-,\Sigma_+$ of the characteristic set 
\begin{equation} 
\Sigma_-\cup \Sigma_+= \operatorname{char}_{\calczero}^{2,0,2}(\square-c^2)
\end{equation} 
of the \emph{unconjugated} Klein--Gordon operator. Switching the sign of $P_\pm$ switches which of $\Sigma_-,\Sigma_+$ corresponds to the ``good'' set $\Sigma$ and which corresponds to $\Sigma_{\mathrm{bad}}$. (We will see later that, using a microlocal partition of unity, we never need to consider functions with $\calc$-wavefront set on both sheets, which is why it suffices to only study one at a time.)

Note that 
\begin{equation}
\mathsf{p}|_{\natural\mathrm{f} \cup \mathrm{bf}} = \mathsf{p}_0|_{\natural\mathrm{f} \cup \mathrm{bf}}, \label{eq:pp0}
\end{equation}
where $\mathsf{p}_0 =\rho_{\mathrm{df}}^2 \rho_{\natural\mathrm{f}}^2 p_0$; that is
\begin{equation} 
\rho_{\mathrm{df}}^{-2}\rho_{\mathrm{pf}}^2 \mathsf{p}_0 =\tau_{\natural}^2 - \xi_{\natural}^2 \pm 2 \tau_{\natural} = h^2( h^2 \tau^2 - \xi^2 \pm 2 \tau), 
\label{eq:8i80ac4}
\end{equation} 
assuming that we have chosen our boundary-defining-functions such that $h=\rho_{\mathrm{pf}}\rho_{\natural\mathrm{f}}$.

\begin{remark}
	Provided we assume that the terms $\alpha$, $w_j$ and $h_{ij}$ in the metric (\cref{eq:metric_form}) are classical symbols, the function $\mathsf{p}$ is a $C^\infty$ function on $\smash{{}^{\calc}\overline{T}^*\bbM}$. Moreover, we have 
	\begin{equation}
	\mathsf{p} = \mathsf{p}_0 \text{ modulo } \rho_{\natural\mathrm{f}}^2 \rho_{\mathrm{bf}} C^\infty({}^{\calc}\overline{T}^*\bbM)
	\end{equation} 
	Our calculations below show that the differential of $\mathsf{p}_0$ restricted to either $\mathrm{bf}$ or $\natural\mathrm{f}$ or $\natural\mathrm{f} \cap \mathrm{df}$ is nonvanishing. This and \cref{eq:pp0} show that, at least for $h$ sufficiently small, the differential of $\mathsf{p}$ near $\mathrm{bf} \cup \natural\mathrm{f}\cup \mathrm{df}$ is nonvanishing. It follows that $\Sigma,\Sigma_{\mathrm{bad}}$ are each a union of $C^\infty$ submanifolds of the respective boundary hypersurfaces (with appropriate behavior at their intersections). This also follows from the explicit formulae in local coordinates for the vanishing set of $\mathsf{p}$,     which can be found below. If we allow $\alpha,w_{j}$, $h_{ij}$ to be merely conormal, then analogous statements apply, replacing $C^\infty$ with $C^k$ for some $k\geq 0$.
	
	If $\alpha$, $w_{j}$, $h_{ij}$ are not classical, this only affects the characteristic set at $\mathrm{df}$.
\end{remark}

We can read off \cref{eq:8i80ac4} that, except possibly at fiber infinity and $\mathrm{pf}$, the symbol $\mathsf{p}_0$ vanishes at $\{(\tau_\natural \pm 1)^2 - \xi_\natural^2 = 1\}$. That is, 
\begin{equation}
\mathsf{p}^{-1}_0(\{0\}) \backslash (\mathrm{pf}\cup \mathrm{df}) \subseteq \{(\tau_{\natural}\pm 1)^2 -\xi_{\natural}^2 = 1\} \subseteq {}^{\calczero} T^* \bbM. 
\end{equation}
The equation 
\begin{equation} 
	(\tau_{\natural}\pm 1)^2 -\xi_{\natural}^2 = 1
\end{equation} 
defines a two-sheeted hyperboloid in each fiber of ${}^{\calczero} T^* \bbM$.
One of the sheets intersects the zero section $\{\zeta_{\natural}=0\}$ of ${}^{\calczero}T^* \bbM$. This sheet is the one in $\{\pm\tau_{\natural} \geq 0\}$. (This will be the one corresponding to $\Sigma$.)
The other sheet is in $\{\pm \tau_{\natural} \leq -2\}$. 
 (This will be the one corresponding to $\Sigma_{\mathrm{bad}}$.)

Next, we check that nothing unexpected happens at $\mathrm{pf}$ or $\mathrm{df}$. 
Let 
\begin{align}
\begin{split} 
\Sigma_0 &= \operatorname{cl}_{{}^{\calc}\overline{T}^* \bbM }(\{ (\tau_{\natural}\pm 1)^2 -\xi_{\natural}^2 = 1 : \pm  \tau_{\natural}\geq 0 \} \backslash (\mathrm{pf} \cup \mathrm{df})) \\
\Sigma_{0,\mathrm{bad}} &= \operatorname{cl}_{{}^{\calc}\overline{T}^* \bbM }(\{ (\tau_{\natural}\pm 1)^2 -\xi_{\natural}^2 = 1 :\pm \tau_{\natural}\leq -2 \} \backslash (\mathrm{pf} \cup \mathrm{df}))
\end{split}
\end{align}
denote the closures in the $\calc$-phase space of the two components of $\mathsf{p}^{-1}_0(\{0\}) \backslash (\mathrm{pf}\cup \mathrm{df}) $ discussed in the previous paragraph. 
\begin{figure}[h] 
	\floatbox[{\capbeside\thisfloatsetup{capbesideposition={left,bottom},capbesidewidth=8cm,capbesidesep=none}}]{figure}[\FBwidth]
	{\caption{
			The hyperboloid $\{(\tau_{\natural}^2\pm 1 )^2 - \xi_{\natural}^2 = 1 \}$ in the radially compactified $\calczero$-frequency space, with the portions corresponding to the sheets $\Sigma_0,\Sigma_{0,\mathrm{bad}}$ shown. Here, $\nu_{\natural},\eta_{\natural}$ are the components of $\xi_{\natural}$ dual to $r=\lVert x \rVert$, $\theta=x/r$ respectively.}}
	{
		\begin{tikzpicture}[scale=.85]
		\def\mm{5.5}
		\def\fram{7}
		\def\b{1} 	
		\begin{axis}[
		hide axis,
		axis equal image,
		xmin=-\fram,xmax=\fram,
		ymin=-\fram,ymax=\fram]
		\draw[fill=gray!10] (0,0) circle (\mm);
		\draw[dotted] (0,-\mm) -- (0, +\mm);
		\draw[dotted] (-\mm,0) -- (+\mm,0);
		\draw[darkred,dashed] (0,1.5) ellipse (28.5pt and 3pt);
		\draw[darkred,dashed] (0,3.2) ellipse (50.5pt and 3pt);
		\draw[gray,dashed] (0,-3.7) ellipse (28.5pt and 3pt);
		\addplot [domain=-1.6:1.6, color=gray] ({-1.41\b*sinh(\x)},{-1.41\b*(.8*cosh(\x)+1)});
		\addplot [domain=-1.86:1.86, color=darkred] ({-1.41\b*sinh(\x)},{1.41\b*(cosh(\x)-1)});
		\draw[->] (0,0) -- (0,2.5);
		\node[above right] (nu) at (0,1.7) {$\pm\tau_{\natural}$};
		\draw[->] (0,0) -- (2.5,0);
		\node[below] (aleph) at (2,0) {$\nu_{\natural}$};
		\draw[->] (0,0) -- (-1.2,-1.2);
		\node[below left] (eta) at (-1.2,-1.2) {$\eta_{\natural} $};
		\node[darkred] at (-4,1.8) {$\Sigma_0$};
		\node[gray] at (2.5,-2.5) {$\Sigma_{0,\mathrm{bad}}$};
		\end{axis}
		\end{tikzpicture} 
}
\end{figure}
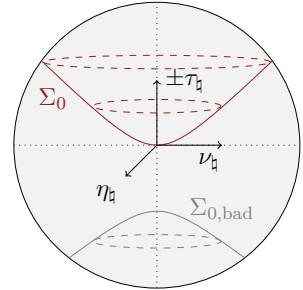

By the continuity of $\mathsf{p}_0$ on ${}^{\calc}\overline{T}^* \bbM$, $\mathsf{p}_0$ must vanish on each of $\Sigma_0,\Sigma_{0,\mathrm{bad}}$, so $\mathsf{p}_0^{-1}(\{0\}) \supseteq \Sigma_0 \cup \Sigma_{0,\mathrm{bad}}$.
\begin{proposition}
	$\mathsf{p}_0^{-1}(\{0\})= \Sigma_0 \sqcup \Sigma_{0,\mathrm{bad}}$. 
	\label{prop:char0}
\end{proposition}
We are saying two things here: first, that $\mathsf{p}_0$ does not vanish at any unexpected points at fiber infinity or $\mathrm{pf}$, and, second, that $\Sigma_0,\Sigma_{0,\mathrm{bad}}$ do not unexpectedly intersect at fiber infinity or $\mathrm{pf}$.

\begin{proof} The proof is via checking the situation in various coordinate charts covering the various corners of ${}^{\calc}\overline{T}^* \bbM$.

	First, we check at and near $\mathrm{df}$. In terms of $\hat{\tau}_{\natural\mathrm{f}} = \rho_{\mathrm{df}}\tau_{\natural}$ and $\hat{\xi}_{\natural\mathrm{f}} = \rho_{\mathrm{df}}\xi_{\natural}$, which are both smooth functions near $\mathrm{df}$, 
		\begin{equation}
		\mathsf{p}_0  =  \hat{\tau}_{\natural\mathrm{f}}^2 - \hat{\xi}_{\natural\mathrm{f}}^2 \pm 2 \rho_{\mathrm{df}} \hat{\tau}_{\natural\mathrm{f}}
		\end{equation}
		if we take $\rho_{\mathrm{pf}}=1$ near $\mathrm{df}$. At $\mathrm{df}$, this can only vanish away from the equator $\{\hat{\tau}_{\natural\mathrm{f}} = 0\}$. Away from the equator, we can take $\rho_{\mathrm{df}} = 1/|\tau_{\natural}|$. Then, $\hat{\tau}_{\natural\mathrm{f}}= \operatorname{sign}(\tau)$, so 
		\begin{equation}
		\mathsf{p}_0 = 1  - \hat{\xi}_{\natural\mathrm{f}}^2 \pm 2 \rho_{\mathrm{df}} \operatorname{sign}(\tau), \quad \mathsf{p}_0^{-1}(\{0\})\cap \mathrm{df} = \{\hat{\xi}_{\natural\mathrm{f}}^2 = 1 \}.
		\label{eq:p0char_df_form}
		\end{equation}
		So, if $\pm \hat{\tau}_{\natural\mathrm{f}}>0$, then $\Sigma_0 = \{\hat{\xi}_{\natural\mathrm{f}}^2 =1+2 \rho_{\mathrm{df}} \}$ locally. Otherwise, that is in the coordinate chart $\pm \hat{\tau}_{\natural\mathrm{f}}<0$, $\Sigma_0$ is not present, instead $\Sigma_{0,\mathrm{bad}}$ is.  
		Likewise, if $\pm \hat{\tau}_{\natural\mathrm{f}}>0$, then $\Sigma_{0,\mathrm{bad}}$ does not intersect this coordinate chart, and if $\pm \hat{\tau}_{\natural\mathrm{f}} <0$, then 
		\begin{equation} 
		\Sigma_{0,\mathrm{bad}} = \{\hat{\xi}_{\natural\mathrm{f}}^2 =1-2 \rho_{\mathrm{df}} \}
		\end{equation} 
		locally. In either case, the points in $\mathsf{p}_0^{-1}(\{0\})$ at $\mathrm{df}$ are precisely the limit points of $\mathsf{p}_0^{-1}(\{0\})\backslash \mathrm{df}$ and therefore lie in $\Sigma_0 \sqcup \Sigma_{0,\mathrm{bad}}$.

	Next, consider the situation in and near $\mathrm{pf}\backslash \natural\mathrm{f}$. Here, $h,\tau,\xi$ (together with the base variables) form a smooth coordinate chart. In terms of these, $\mathsf{p}_0 = p_0 = h^2 \tau^2-\xi^2 \pm 2 \tau$ if we choose $\rho_{\mathrm{df}}=1$ near $\mathrm{pf}$ and $\rho_{\natural\mathrm{f}}=1$ near the point in question. So,
		\begin{equation} 
		 \mathsf{p}_0^{-1}(\{0\}) = \{ h^2 \tau^2 \pm 2 \tau - \xi^2 = 0\}
		 \end{equation}
locally. 
		 For $h>0$, this just gives the two-sheeted hyperboloid found earlier. When $h=0$, we instead get the paraboloid $\{2\tau = \pm \xi^2\}$. This is what the sheet $\Sigma_0$ degenerates to as $h\to 0$ (see \Cref{fig:degen}), so each point in the paraboloid is a limit point of points in $\Sigma_0$ with $h>0$. So, $\{2\tau = \pm \xi^2\}\subset \Sigma_0$. The set $\Sigma_{0,\mathrm{bad}}$ evidently does not intersect $\mathrm{pf}\backslash \natural\mathrm{f}$.

	Finally, consider $\mathrm{pf}\cap \natural\mathrm{f}$. Since we are away from $\mathrm{df}$, we may take $\rho_{\mathrm{df}}=1$. Then, \cref{eq:8i80ac4} becomes $\rho_{\mathrm{pf}}^2 \mathsf{p}_0 = h^2( h^2 \tau^2 - \xi^2 \pm 2 \tau)$.
		 \begin{itemize}
		 	\item Using the coordinate system $\rho_{\natural\mathrm{f}} = 1/|\tau|^{1/2}$, $\hat{\xi} = \xi / |\tau|^{1/2}$, and $\rho_{\mathrm{pf}} = h |\tau|^{1/2}$ (\cref{it:parres_chart_1} in \S\ref{subsec:phase}, cf.\ \cref{eq:j4w1dc}), we have 
		 	\begin{equation} 
		 	\mathsf{p}_0 =  \rho_{\mathrm{pf}}^2-\hat{\xi}^2\pm 2\operatorname{sign}(\tau)  .
		 	\label{eq:z31rvz}
		 	\end{equation}
		 	Near $\mathrm{pf}$, this is only vanishing when $\operatorname{sign}(\tau) = \pm 1$, in which case $	\mathsf{p}_0^{-1}(\{0\}) = \Sigma_0= \{ - \rho_{\mathrm{pf}}^2+\hat{\xi}^2 = 2\}$
		 	locally. Again, we find that the points in $\mathsf{p}_0^{-1}(\{0\})$ are just limit points of those points discussed previously. 
		 	\item Using the coordinate system $\rho_{\natural\mathrm{f}} = 1/|\xi_k|$, $\hat{\tau} = \tau/ |\xi_k|^2$, $\hat{\xi}_k = \{\xi_j/\xi_k\}_{j\neq k}$, and $\rho_{\mathrm{pf}} = h |\xi_k|$ (\cref{it:parres_chart_2} in \S\ref{subsec:phase}), we have 
		 	\begin{equation}
		 	\mathsf{p}_0 =   \rho_{\mathrm{pf}}^2 \hat{\tau}^2 - 1 - \hat{\xi}_k^2 \pm 2 \hat{\tau} .
		 	\end{equation} 
		 	At $\mathrm{pf}$, this is given by  $1+\hat{\xi}_k^2 \mp 2 \hat{\tau}$, which vanishes for $\hat{\tau} = \pm 2^{-1}(1+\hat{\xi}_k^2)$. So, this coordinate chart does not contain any points in $\mathsf{p}_0^{-1}(\{0\})$ that we have not found previously.		 \end{itemize}

	So, $\mathsf{p}_0^{-1}(\{0\})= \Sigma_0 \cup \Sigma_{0,\mathrm{bad}}$.

	 The explicit expressions above show that $\Sigma_0 \cap \Sigma_{0,\mathrm{bad}}=\varnothing$ everywhere.
\end{proof}

\begin{proposition}
	The characteristic set
	$\operatorname{char}_{\calc}^{2,0,2,0}(P_\pm)$ consists of two connected components,  one of which is in the closure of the region where $\pm \tau_\natural<-1$ and the other of which is the region where $\pm \tau_\natural >-1$. More precisely, 
	\begin{align}
	\begin{split} 
	\Sigma_{\mathrm{bad}} &= \operatorname{char}_{\calc}^{2,0,2,0}(P_\pm) \cap \operatorname{cl}_{{}^{\calc}\overline{T}^* \bbM}\{\pm \tau_{\natural} <-1 \}, \\ \Sigma &=\operatorname{char}_{\calc}^{2,0,2,0}(P_\pm) \cap \operatorname{cl}_{{}^{\calc}\overline{T}^* \bbM} \{\pm \tau_{\natural} >-1 \}.
	\end{split} 
	\label{eq:l94945}
	\end{align}
	Moreover, the portions of $\Sigma,\Sigma_{\mathrm{bad}}$ at $\mathrm{bf}\cup \natural\mathrm{f}$ are what they are for the free Klein--Gordon equation. That is, $\Sigma=\Sigma_{0}$ and $\Sigma_{\mathrm{bad}}=\Sigma_{0,\mathrm{bad}}$ there. 
\end{proposition}
The inequality $\pm \tau_\natural \lessgtr -1$ in \cref{eq:l94945} can be replaced by  $\pm \tau_\natural \lessgtr E$ for any $E\in (-2,0)$ without changing $\Sigma,\Sigma_{\mathrm{bad}}$, since these sets are disjoint from $\cup_{\delta \in (0,1)} \operatorname{cl}_{{}^{\calc}\overline{T}^* \bbM} \{|\pm \tau_\natural+1|<\delta \}$.
\begin{proof}
	Since the characteristic set is $\mathsf{p}_0^{-1}(\{0\})$ at $\mathrm{bf}\cup \natural\mathrm{f}$, we only need to discuss the situation at $\mathrm{df}^\circ$. (Recall that the characteristic set, by definition, does not have any points in the interior of the phase space or in the interior of $\mathrm{pf}$; any point in the characteristic set in $\mathrm{pf}$ must be in at least one of $\mathrm{bf},\natural\mathrm{f}$.)
	
	At $\mathrm{df}$, $\operatorname{char}_{\calc}^{2,0,2,0}(P_\pm) = \mathsf{p}^{-1}(\{0\})$ consists precisely of null directions in the cosphere bundle. The null directions are all either future-directed or past-directed, and each set of null directions is a smooth submanifold of the cosphere bundle. Since the metric is asymptotically flat, at $\partial \bbM$ the future-directed null directions converge to those for the Minkowski metric, which are the points in $\mathsf{p}_0^{-1}(\{0\})$ found previously. 
	
	So, $\operatorname{char}_{\calc}^{2,0,2,0}(P_\pm)$ consists of two components. One, which we call $\Sigma$, consists of the points in $\Sigma_0$ in $\mathrm{bf}\cup \natural\mathrm{f}$ and, in the $+$ case, all of the future-directed null directions, and, in the $-$ case, all of the past-directed null directions. Similarly, $\Sigma_{\mathrm{bad}}$ consists of the points in $\Sigma_{0,\mathrm{bad}}$ in $\mathrm{bf}\cup \natural\mathrm{f}$ and all of the remaining null directions. From the description above, each of these is connected, they are disjoint, and their union is the whole characteristic set.
	
	The formula \cref{eq:l94945} is correct at $\mathrm{bf}\cup \natural\mathrm{f}$, since it holds there with $\Sigma_0,\Sigma_{0,\mathrm{bad}}$ in place of $\Sigma,\Sigma_{\mathrm{bad}}$, and at $\mathrm{df}$ these formulas pick out the future-directed or past-directed null directions (because we are assuming that $t$ is a global time function).
\end{proof}

The upshot of the discussion above is summarized in \Cref{fig:charFig}.

A worthwhile fact is that, as far as the frequency variables and semiclassical parameter in the $\calc$- phase space are concerned (meaning ignoring the fact that spacetime infinity needs multiple coordinate charts to cover), we can cover a neighborhood of $\Sigma$ with four coordinate charts.  Indeed, in the proof of \Cref{prop:char0}, we observed the following:
\begin{enumerate}
\item In the interior of $\natural\mathrm{f}$ we can use coordinates $\taun$, $\xin$ and $h$. \label{eq:nat_coord}
\item Away from the zero section of ${}^{\calczero}\overline{T}^* \bbM$ (that is, ``near'' $\mathrm{df}$) and near $\Sigma$, we can use 
	\begin{equation} 
	\rho_{\mathrm{df}} = 1/|\tau_{\natural}|,\; \hat{\xi}_{\natural\mathrm{f}} = \xi_{\natural}/|\tau_{\natural}|,\text{ and }\rho_{\natural\mathrm{f}} = h
	\label{eq:proj_coord}
	\end{equation} as smooth coordinates. (And we can take $\rho_{\mathrm{pf}} =1$ here.) Specializing \cref{eq:p0char_df_form} to the case relevant to $\Sigma$,  namely $\pm \hat{\tau}_{\natural\mathrm{f}}>0$,
	\begin{equation}
	\mathsf{p}_0 = 1 - \hat{\xi}^2_{\natural\mathrm{f}} +2\rho_{\mathrm{df}}, 
	\label{eq:p0_df_form}
	\end{equation}
	$\Sigma_0 =\{\hat{\xi}^2_{\natural\mathrm{f}} = 1+2\rho_{\mathrm{df}}\}$. Since $\Sigma_0=\Sigma$ at $\mathrm{bf}\cup \natural\mathrm{f}$, we have 
	\begin{equation}
	\Sigma \cap (\mathrm{bf} \cup \natural\mathrm{f}) = \{\hat{\xi}^2_{\natural\mathrm{f}} = 1+2\rho_{\mathrm{df}}\}
	\label{eq:misc_5454}
	\end{equation}
	near $\mathrm{df}$.
	\label{it:df_coord}
	\item 
	Except at $\natural\mathrm{f}\cup \mathrm{df}$, we can just use the usual frequencies $\tau,\xi$ as smooth coordinates, and we can take $\rho_{\mathrm{df}}=1$, $\rho_{\natural\mathrm{f}}=1$, $\rho_{\mathrm{pf}} = h$. This leads to $\mathsf{p}_0 = p_0 = -h^2 \tau^2 + \xi^2 \mp 2 \tau$. The portion of $\Sigma$ here is contained in $\mathrm{bf}$, where it is equal to $\Sigma_0$, so 
	\begin{equation}
	\Sigma\cap \mathrm{bf}\backslash (\natural\mathrm{f}\cup\mathrm{df})  = \{h^2 \tau^2 - \xi^2 \pm 2 \tau =0\}
	\end{equation}
	here.
	\label{it:pf_coord}
	\item Finally, to cover the corner $\Sigma\cap \natural\mathrm{f} \cap \mathrm{pf}$, we can use 
	\begin{equation} 
		\rho_{\natural\mathrm{f}} = 1/|\tau|^{1/2}, \hat{\xi} = \xi / |\tau|^{1/2},\text{ and }\rho_{\mathrm{pf}} = h |\tau|^{1/2}
		\label{eq:parabolic_proj_coord}
	\end{equation} 
	(\cref{it:parres_chart_1} in \S\ref{subsec:phase}, cf.\ \cref{eq:j4w1dc}), which leads to \cref{eq:z31rvz}, which, specializing to the $\pm \tau>0$ case relevant to $\Sigma$, becomes 
	\begin{equation} 
	\mathsf{p}_0 =  \rho_{\mathrm{pf}}^2-\hat{\xi}^2+2, \quad  \Sigma_0 \cap \mathrm{bf} \cup \natural\mathrm{f}  = \{ - \rho_{\mathrm{pf}}^2 +\hat{\xi}^2 = 2\}\subset \mathrm{bf} \cup \natural\mathrm{f} .
	\label{eq:misc_4545}
	\end{equation}
	In $\natural\mathrm{f}$,  $\Sigma=\Sigma_0$, so $\Sigma \cap \natural\mathrm{f}\cap \mathrm{pf} = \{  -\rho_{\mathrm{pf}}^2+\hat{\xi}^2 = 2\}$. 
\end{enumerate}

\begin{remark}
	Another useful fact is that $\Sigma$ intersects $\operatorname{cl}_{{}^{\calc}\overline{T}^* \bbM} \{\xi=0,h>0\}$
	only in $\mathrm{bf}\cap \mathrm{pf}\backslash \natural\mathrm{f}$. Concretely, this means that in the coordinate charts above, it is only in the coordinate chart (\ref{it:pf_coord}) that the symbol vanishes at some point with vanishing spatial frequency. This can be seen in \Cref{fig:charFig}, in which the characteristic set intersects the horizontal hyperplane bisecting the figure only at $\mathrm{bf}\cap \mathrm{pf}\backslash \natural\mathrm{f}$. 
	\label{rem:xi_vanishing}
\end{remark}

If one wants to use spherical coordinates in the base, instead of $\tau,\xi$, it is best to work with the frequency variables dual to spherical coordinates, specifically $\nu$ dual to $r=\lVert x \rVert$ and $r\eta \in T^* \bbS^{d-1}$ dual to the angular variable $\theta = x/r$. 
(The reason for letting $r\eta$ be the angular frequency rather than $\eta$ is this definition of $\eta$ is what makes $\eta$ a smooth coordinate even over base infinity. That is, this definition of $\eta$ is an ``sc-frequency.'' From a physicist's point of view, one wants all frequency coordinates to be measured using the same system of units.)
The coordinates $\nu,\eta$ extend smoothly to $\mathrm{bf}^\circ$. Then, $\nu_{\natural} = h \nu$ and $\eta_{\natural} = h \eta$ extend smoothly to $h=0$, and
\begin{equation}
\hat{\nu}_{\natural\mathrm{f}} = \rho_{\mathrm{df}} \nu_{\natural}, \quad \hat{\eta}_{\natural\mathrm{f}} = \rho_{\mathrm{df}} \eta_{\natural}, 
\end{equation}
also extend smoothly to $\mathrm{df}$. Similar coordinates can also be used near $\mathrm{pf}\cap \natural\mathrm{f}$.

\subsection{The Hamiltonian vector field}
\label{subsec:H}

We now study the Hamiltonian vector field $H_p$ associated to the $\calc$-principal symbol $p$ of $P_\pm$. The Hamiltonian vector field is defined by the usual formula,
\begin{equation} \label{eq:Hp-def}
H_p = (\partial_\tau p) \partial_t - (\partial_t p) \partial_\tau +\sum_{j=1}^d ( (\partial_{\xi_j}p)\partial_{x_j} - (\partial_{x_j} p) \partial_{\xi_j}   ),
\end{equation}
and likewise for $H_{p_0}$.
Since $p_0 = h^2 \tau^2 - \xi^2 \pm 2 \tau$,  
\begin{equation}
H_{p_0} =   2(h^2 \tau \pm 1  )\partial_t - 2\xi\cdot \partial_x.
\label{eq:Hp0}
\end{equation}
Here, $\xi \cdot \partial_x$ is the generator of the geodesic flow in the spatial directions.

It follows from \cref{eq:principal_symbol_phg} that the vector field $\mathsf{H}_p$ defined by 
\begin{equation} \label{eq:renormalized_Hvf}
	\mathsf{H}_p = 2^{-1}\rho_{\mathrm{df}} \rho_{\mathrm{bf}}^{-1} \rho_{\natural\mathrm{f}} H_p \in \calV_{\mathrm{b}}({}^{\calc}\overline{T}^* \bbM )
\end{equation} 
is a b-vector field on $\smash{{}^{\calc}\overline{T}^* \bbM}$. Consequently, $\mathsf{H}_p$ induces a flow on the boundary $\partial \smash{{}^{\calc}\overline{T}^* \bbM}$. And, since $\mathsf{H}_p \mathsf{p} \propto \mathsf{H}_p p = 0$ on the characteristic set, the characteristic set is invariant under the induced flow.
What we must understand is the dynamical structure of the flow within the good sheet $\Sigma$ of the characteristic set. In particular, the \emph{radial set} 
\begin{equation} 
	\calR = \text{vanishing set of }\mathsf{H}_p\text{ in }\Sigma
	\label{eq:R}
\end{equation} 
will be important. 
\begin{remark}\label{rem:vanishing sense}
	In \cref{eq:R}, we mean the ordinary vanishing set of $\mathsf{H}_p$, considered as a smooth vector field on the compactified phase space. We shall see that, \emph{as a b-vector field} on $\smash{{}^{\calc}\overline{T}^* \bbM}$, $\mathsf{H}_p$ never vanishes. That is, at points where it vanishes as an ordinary smooth vector field, the component normal to the boundary always vanishes simply, not quadratically. That is, the $\partial_{\rho_{\mathrm{bf}}}$ component will look like $\rho_{\mathrm{bf}} \partial_{\rho_{\mathrm{bf}}}$, not $\rho_{\mathrm{bf}}^2 \partial_{\rho_{\mathrm{bf}}}$. The former vanishes at bf when considered as a smooth vector field, but not a b-vector field, whereas the latter vanishes when considered as either.
	
	The fact that the coefficient of $\partial_{\rho_{\mathrm{bf}}}$ vanishes only simply is a crucial observation, as the action of this b-normal component on a power of $\rho_{\mathrm{bf}}$ is the main source of positivity in our positive commutator estimates of \S\ref{sec:estimates}. This is the usual way in which radial point estimates work.
\end{remark}

In this subsection, we will compute $\calR$ in various coordinate charts. Since the definition of $\calR$ does not depend on the choice of boundary-defining-functions, we are free to choose whatever definitions are convenient for the computation at hand. 
Those computations are simplified by noting that, at any one of $\mathrm{f} \in \{\mathrm{df},\mathrm{bf},\natural\mathrm{f}\}$, contributions to $p$ that are subprincipal at $\mathrm{f}$ do not affect $\mathsf{H}_p$ on $\mathrm{f}$. For example, 
\begin{equation} 
	\mathsf{H}_p = \mathsf{H}_{p_0}
\end{equation} 
on $\natural\mathrm{f}\cup \mathrm{bf}$.  Thus, for the purposes of this subsection, namely computing $\calR$, we can replace $\mathsf{H}_p$ with $\mathsf{H}_{p_0}$ when analyzing the situation restricted to any of $\mathrm{bf},\natural\mathrm{f}$. Later we will be concerned with the linearization of $\mathsf{H}_p$ at $\calR$, to which the deviation $\mathsf{H}_p-\mathsf{H}_{p_0}$ can contribute even at $\mathrm{bf},\natural\mathrm{f}$ (in a limited manner).

\begin{figure}[t]
	\begin{center}
		\includegraphics{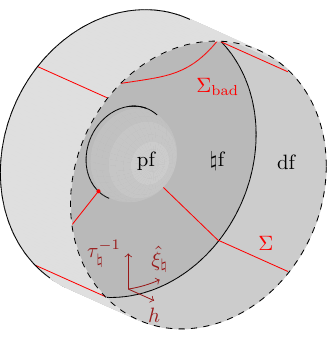}
		\quad 
		\includegraphics{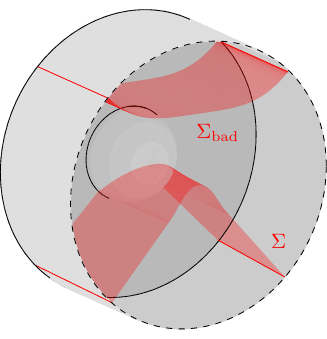}
	\end{center}
	\caption{The $\calc$-characteristic set of $P_+$ over the interior of $\bbM$ (left) and over $\partial \bbM$ (right), with spacetime variables ignored. The orientation of the figure is such that $\tau$ is increasing downwards, as in \Cref{fig:phase}. 
	}
	\label{fig:charFig}
\end{figure}


\begin{proposition}
	$\calR$ is contained in $\mathrm{bf}$.
\end{proposition}
\begin{proof} 
We need to show that $\mathsf{H}_p$ is nonvanishing on 
\begin{enumerate}[label=(\roman*)]
	\item the portion of the characteristic set at fiber infinity but away from the other boundary hypersurfaces, i.e.\  $\mathrm{df}^\circ \cap \Sigma$,
	\item the portion of the characteristic set at the corner between fiber infinity and $\natural\mathrm{f}$ but outside of $\mathrm{bf}$ i.e.\ $\mathrm{df} \cap \natural\mathrm{f} \cap \Sigma\backslash \mathrm{bf}$,
	\item  the portion of the characteristic set in $\natural\mathrm{f}$ but away from the other boundary hypersurfaces, i.e.\  $\natural\mathrm{f}^\circ \cap \Sigma$, and
	\item the portion of the characteristic set at the corner between $\natural\mathrm{f}$ and $\mathrm{pf}$ but outside of $\mathrm{bf}$, i.e.\  $\natural\mathrm{f} \cap \mathrm{pf} \cap \Sigma\backslash \mathrm{bf}$.
\end{enumerate}
No points in the interior of $\mathrm{pf}$ need to be considered, since $\calR$ is defined to be a subset of $\Sigma$ and $\Sigma$ is defined to be disjoint from $\mathrm{pf}^\circ$. 
Only for (i) do we need to consider $\mathsf{H}_p$. For the other three regions, it suffices to consider only $\mathsf{H}_{p_0}$.

In case (i), it suffices to note that $\mathsf{H}_p$ is a nonzero multiple of  the null-geodesic flow on the cosphere bundle of the base. This vector field is nonvanishing. 

In case (ii), we can use the projective coordinates \cref{eq:proj_coord}. In terms of these, 
\begin{equation}
\rho_{\mathrm{bf}} \mathsf{H}_{p_0} = \pm h (1+ \rho_{\mathrm{df}}) \partial_t - \hat{\xi}_{\natural\mathrm{f}} \cdot \partial_x.
\label{eq:Hp0_df_form}
\end{equation}
When $h=0$, $\rho_{\mathrm{bf}} \mathsf{H}_{p_0} |_{\mathrm{df}}$ vanishes when $\smash{\hat{\xi}_{\natural\mathrm{f}}} = 0$, but we see from \cref{eq:p0_df_form} that $\smash{\hat{\xi}_{\natural\mathrm{f}}}$ is never zero on $\Sigma$ in the region under investigation. 

For case (iii), we compute $\mathsf{H}_{p_0}$ using the natural coordinates 
$\tau_{\natural}  = h^2 \tau$ and $\xi_{\natural} = h \xi$.
Choosing $\rho_{\natural\mathrm{f}} = h$ as a boundary-defining-function of $\natural\mathrm{f}$ valid except at $\mathrm{pf}$, we get 
	\begin{equation}
	\rho_{\mathrm{df}}^{-1} \rho_{\mathrm{bf}} \mathsf{H}_{p_0} = h(\tau_{\natural} \pm 1 ) \partial_t  - \xi_{\natural} \cdot \partial_x .
	\label{eq:8tj44x3} 
	\end{equation}
	So, in the interior $\natural\mathrm{f}^\circ$, where $(x, t)$ are valid coordinates (as we are away from the spacetime boundary), 
	$\mathsf{H}_{p_0}$ can vanish only if $\xi_{\natural} = 0$. But Remark~\ref{rem:xi_vanishing} states that this does not happen on $\Sigma$. Indeed: on $\Sigma \cap \natural\mathrm{f}^\circ\subset \{\smash{\xi_{\natural}^2} = \smash{\tau_{\natural}^2} \pm 2 \tau_{\natural} \}$, for $\xi_\natural$ to vanish
	means that $\tau_{\natural}\in \{0,\mp2\}$. However, in  $\natural\mathrm{f}^\circ$, 
	\begin{equation}
	\xi_\natural = 0 \Rightarrow \tau_\natural \neq 0, 
	\end{equation} 
	so $\tau_{\natural}=0$ is not an option. If instead $\tau_{\natural} = \mp2$, then we would be in $\Sigma_{\mathrm{bad}}$ rather than $\Sigma$. So, $\calR$ is disjoint from $\natural\mathrm{f}^\circ$.

In case (iv), it suffices to use the coordinates in \cref{eq:parabolic_proj_coord}. (Indeed, this coordinate chart covers $\Sigma\cap \mathrm{pf}\cap \natural\mathrm{f}$, as we saw in the previous subsection.) In terms of these coordinates, and choosing $\rho_{\mathrm{df}}=1$ near $\mathrm{pf}$, we have 
	\begin{equation}
	\rho_{\mathrm{bf}} \mathsf{H}_{p_0} =  \pm \rho_{\natural\mathrm{f}}( \rho_{\mathrm{pf}}^2 + 1 ) \partial_t -\hat{\xi} \cdot \partial_x.
	\label{eq:Hp0_pfPlf_form}
	\end{equation} 
	At $\natural\mathrm{f}$, the right-hand side vanishes if and only if $\hat{\xi}=0$, but we see from \cref{eq:z31rvz} that  $\Sigma\cap \mathrm{pf}\cap \natural\mathrm{f}$ does not contain any points with $\smash{\hat{\xi}}=0$. This follows from Remark~\ref{rem:xi_vanishing} again.
\end{proof}

\begin{remark} We will note later that the radial set $\calR$ does intersect the boundary hypersurfaces $\mathrm{pf}$, $\mathrm{\natural\mathrm{f}}$ and $\mathrm{df}$ where they intersect $\mathrm{bf}$; that is, $\calR$ has a nonempty intersection with all three of $\mathrm{pf} \cap \mathrm{bf}$, $\mathrm{\natural\mathrm{f}} \cap \mathrm{bf}$ and $\mathrm{df} \cap \mathrm{bf}$.
\end{remark}
\begin{remark} Contrast with the situation on $\Sigma_{\mathrm{bad}}$; $\mathsf{H}_{p_0}$ \emph{does} vanish on $\Sigma_{\mathrm{bad}}$ over every point in $\bbR^{1,d}$, namely in $\natural\mathrm{f}^\circ$ when $\tau_{\natural} = \mp2$ and $\xi_{\natural}=0$. This is one sense in which $\Sigma_{\mathrm{bad}}$ is problematic. The advantage of working in ${}^{\calc}\overline{T}^* \bbM$, as opposed to ${}^{\calczero}\overline{T}^* \bbM$, is that this problematic behavior is only occurring on one component of the characteristic set, whereas in the latter both components have the same problem --- see \S\ref{sec:whynot}. Blowing up the zero section of  ${}^{\calczero}\overline{T}^* \bbM$ over $\{h=0\}$ turned out be precisely what is needed to remove this unwanted vanishing, at least on one component of the characteristic set.
	
Note that we could do the analogous blowup on $\Sigma_{\mathrm{bad}}$, which is equivalent to working with the $\calctwo$-calculus. This would resolve all problems. However, it is simpler to focus on one component of the characteristic set at a time (and to conjugate the operator so that the component intersects the zero section).	
\end{remark}

To analyze the situation over $\partial \bbM$, i.e.\ at $\mathrm{bf}$, we start by considering the region away from $\mathrm{df},\mathrm{pf}$, taking $\rho_{\mathrm{df}}=1$ and $\rho_{\natural\mathrm{f}}=h$ there. In region (\ref{eq:nat_coord}) enumerated above, we have by  \cref{eq:renormalized_Hvf} and \cref{eq:8tj44x3}, 
\begin{equation}\label{eq:Hvf_nat_face}
h \rho_{\mathrm{bf}}^{-1} H_{p_0} = 2\rho_{\mathrm{bf}}^{-1} (h(\tau_{\natural} \pm 1 ) \partial_t  - \xi_{\natural} \cdot \partial_x ) .
\end{equation}
In a local coordinate patch near $\partial \bbM$ we can take $\rho_{\mathrm{bf}}^{-1}$ to be either $\pm t$ or $\pm x_j$, where $t$, resp. $x_j$ is a dominant coordinate in the sense that $|t| \geq c |z|$, resp. $|x_j| \geq c |z|$ for some $c > 0$ locally, $z=(t,x)$, $|z|^2=t^2+\lVert x\rVert^2$. The vector field on the right-hand side of \cref{eq:Hvf_nat_face} is then homogeneous of degree zero in $z$ and hence extends to a smooth vector field on the radial compactification. This vector field vanishes (in the usual sense, cf. Remark~\ref{rem:vanishing sense}) at a point of $\mathrm{bf}$ if and only if it is \emph{radial} there, that is, if and only if it is a multiple of $t \partial_t + x \cdot \partial_x$. (Hence the term `radial set' for the points where this happens.) This gives a convenient way to identify $\calR$. 
For $\omega\in \partial \bbM$ and nonzero $z=(t,\bfx)\in \bbR^{1,d}$, we write $\omega \parallel z $ to mean that $\omega \in \mathrm{cl}_{\bbM} \operatorname{span}_\bbR z$, i.e.\ that $\omega$ is either parallel or anti-parallel to $z$, or 0. 
It follows from \cref{eq:8tj44x3} that $\mathsf{H}_{p_0}$ vanishes at a point in $\mathrm{bf}\backslash (\mathrm{pf} \cup \mathrm{df})$ with base coordinate $\omega$ and frequency coordinates $\tau_{\natural} \in \bbR,\xi_{\natural}\in \bbR^d$ if and only if 
\begin{equation}
\omega \parallel ( h (\tau_{\natural}\pm 1),- \xi_{\natural} ). 
\label{eq:id_008}
\end{equation}
If $h>0$, then, for each point in the frequency-space hyperboloid $\{\xi_{\natural}^2 = \smash{\tau_{\natural}^2} \pm 2 \tau_{\natural} \}\subset \bbR^{1+d}$, this picks out exactly two possible $\omega$. Specifically, we can parametrize the sheet $\Sigma$ using the map $\xi_{\natural} \mapsto ( \tau_{\natural} = \pm (\smash{(1+\xi_{\natural}^2 )^{1/2}}-1), \xi_{\natural} )$, and therefore the relevant $\omega$'s are given by
\begin{equation}\label{eq:omega_defn}
\omega_\varsigma(\xi_{\natural}) =  \varsigma \bbR^+ \cdot ( h (1+\xi_{\natural}^2 )^{1/2} , \mp \xi_{\natural} ), \qquad \varsigma \in \{-,+\},
\end{equation}
identifying points of $\partial \bbM$ with the corresponding rays through the origin. 
If $\xi_{\natural}=0$, then $\omega_\varsigma(\xi_{\natural})$ is the north $(+)$/south $(-)$ pole of $\bbM$. Otherwise, $\omega_\varsigma(\xi_{\natural})$ will be some other ray with the absolute value of the slope being strictly less than $h = c^{-1}$. The rays with `slope' $|t|/\lVert x \rVert$ exactly equal to $h$ are on the light cone, which is widening as $h^{-1} = c\to\infty$ in the expected way. 
If $h=0$, then \cref{eq:id_008} is satisfied as long as $\omega$ lies on the equator of $\partial \bbM$ and $\omega\parallel \xi_{\natural}$. In other words, the possible solutions are $\omega_\varsigma(\xi_{\natural}) = \varsigma \bbR^+ (0,\mp \xi_{\natural})$. 

Our sign conventions are such that $\omega_+$ is always in the future closed hemisphere of $\partial \bbM$ while $\omega_-$ is always in the past closed hemisphere. 

The vanishing points of $\mathsf{H}_{p_0}$ found in the previous paragraph, since they lie at $\mathrm{bf}$, must be vanishing points of $\mathsf{H}_p$, that is in our set $\calR$. 
To summarize, we found two portions of radial set, whose intersection with $\mathrm{bf}\backslash (\mathrm{pf}\cup \mathrm{df})$ is given by
\begin{equation}
\{(\omega_\varsigma(\xi_{\natural}), \pm(- 1 + (1+\xi^2_{\natural\mathrm{f}})^{1/2}),\xi_{\natural} , h  ) \in {}^{\calc} T^* \bbM \backslash (\mathrm{pf}\cup \mathrm{df})\} \subset {}^{\calczero} T^* \bbM.
\end{equation}
By the continuity of $\mathsf{H}_p$ on ${}^{\calc}\overline{T}^* \bbM$, if we define 
\begin{equation}\label{eq:Rclosure}
\calR_\varsigma = \mathrm{cl}_{{}^{\calc}\overline{T}^* \bbM} \{(\omega_\varsigma(\xi_{\natural}), \pm(- 1 + (1+\xi^2_{\natural\mathrm{f}})^{1/2}),\xi_{\natural} , h  ) 
\}, 
\end{equation}
then $\mathsf{H}_p$ is vanishing on the entirety of $\calR_\pm$. I.e.\ $\calR_\pm \subseteq \calR$. 

\begin{figure}[t]
	\begin{center}
		\hspace{4em}
		\includegraphics[scale=1.1]{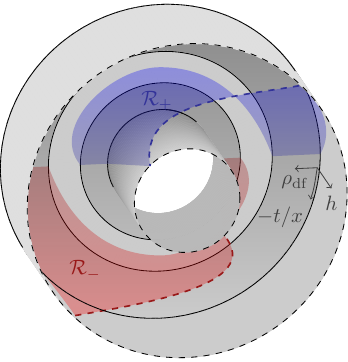}
	\end{center}
	\caption{The radial sets $\calR_+,\calR_-$ inside of the manifold-with-corners $\Sigma\cap \mathrm{bf}$ in the case when there are $d=1$ spatial dimensions. Unlike in most of the other figures so far, we have included the degrees of freedom parametrizing $\partial \bbM$, which in the $d=1$ case is a circle $\bbS^1_\varphi$ (where $\varphi=\operatorname{arctan}(t/x)$ is the angular coordinate in the figure). The inner and outer cylinders are the portions of $\Sigma\cap \mathrm{bf}$ at fiber infinity, and the interstitial region is the rest of $\Sigma\cap \mathrm{bf}$. The ``back wall'' is where $h=0$. 
	When $d=1$, the blowup used to create $\mathrm{pf}$ disconnects $\Sigma\cap \mathrm{bf} \cap \natural\mathrm{f}$ into two components, the two darker gray annuli. \textit{Note}: the cross section of this figure with $h=1$ has been displayed already as \Cref{fig:individual_flows}(a), except that figure has the flow at df drawn. The cross section with $h=0$ (the ``back plate'' of the figure) is similar to  \Cref{fig:individual_flows}(b); it is depicted in \Cref{fig:Schrflow_with_pf}. }
	\label{fig:radial_sets}
\end{figure}

\begin{proposition}\label{prop:calR}
The radial set $\calR$ can be written as $\calR= \calR_- \cup \calR_+$, where $\calR_\varsigma$ are the components of $\calR$; i.e. they are disjoint and connected. Moreover, the rescaled Hamiltonian vector field $\mathsf{H}_{p}$ is a source at $\calR_-$ and a sink at $\calR_+$. 
\end{proposition}
\begin{proof} 
We compute the Hamiltonian vector field on $\Sigma$ in new coordinates that exhibit both the vanishing set of the Hamiltonian vector field, and its source/sink structure. For simplicity, we consider the case of $P_+$. The case of $P_-$ is similar, with some changes of sign. 

First, in the interior of $\natural\mathrm{f}$, we have frequency coordinates $\taun$ and $\xin$, boundary defining function $\rho_{\mathrm{\natural\mathrm{f}}}=h$, boundary defining function $\rho_{\mathrm{bf}}$ which we take, after a rotation of the $x$ coordinates if necessary, to be $1/x_1$ where $x_1 \geq 0$ is a dominant spatial variable, and projective coordinates $t/x_1$, $x_j/x_1$, for $j \geq 2$. Also we take $\rho_{\mathrm{df}}=1$ here, as usual. We recall that $\xin \neq 0$ on $\Sigma$ in the interior of $\natural\mathrm{f}$; see Remark~\ref{rem:xi_vanishing}. Next, if $\xi_{\natural, 1} = 0$, then the Hamiltonian vector field does not vanish, as follows from the discussion above. 
So it remains to check points where $\xi_{\natural, 1} \neq 0$.

Given the relation \cref{eq:omega_defn}, and since we are specifically treating $P_+$ here, we see that $\sgn x_1 = - \varsigma \sgn \xi_{\natural, 1}$.  We introduce the coordinates 
\begin{equation}\label{eq:sw1}\begin{aligned}
s &= \frac{t}{x_1} + \frac{h (\taun + 1)}{\xi_{\natural, 1}}, \\
w_j &= \frac{x_j}{x_1} - \frac{\xi_{\natural, j}}{\xi_{\natural, 1}}, \quad j \geq 2
\end{aligned}\end{equation}
which vanish at $\calR_\varsigma$, $\varsigma = - (\sgn x_1)(\sgn \xi_{\natural, 1})$, as follows from \cref{eq:omega_defn}. 
Thus our spacetime variables are now $\rho_{\mathrm{bf}}=1/x_1, s, w$, which together with $\taun, \xin, h$ form a complete coordinate system. In these variables we can compute that 
\begin{equation}\label{eq:Hvf_nat_face 2}
\mathsf{H}_{p_0} = \xi_{\natural, 1}(  s \partial_s  + \rho_{\mathrm{bf}} \partial_{\rho_{\mathrm{bf}}} + w \cdot \partial_w ) .
\end{equation}
This is a sink at $\calR_+$ and a source at $\calR_-$; we emphasize that this is for $P_+$; it will be the other way around for $P_-$. 

This same calculation is also valid at $\natural\mathrm{f} \cap \mathrm{df}$, in a neighbourhood of $\Sigma$, since the functions $s$ and $w$ remain smooth, with linearly independent differentials, there. In this region, we normalize the Hamiltonian vector field by multiplying also by $\rhodf$, as in \cref{eq:renormalized_Hvf}. This removes the $\xi_{\natural, 1}$ factor at the front (up to a sign). So, $\calR_+$ is a sink and $\calR_-$ is a source, uniformly up to $\mathrm{df}$.

In fact, the coordinates $s$ and $w$ \emph{also} remain smooth in a neighbourhood of $\Sigma \cap \natural\mathrm{f} \cap \mathrm{pf}$, in a region where both $x_1 \geq 0$ and $\xi_1$ are dominant variables. To see this, we switch to coordinates $\rhopf = h\xi_1$ which locally defines $\mathrm{pf}$, $\rhonf = 1/|\xi_1|$ which locally defines $\natural\mathrm{f}$, $\sqrt{\tau}/\xi_1$, and $\xi_j/\xi_1$ for the frequency variables, and $\rhobf = 1/x_1$, $t/x_1, x_j/x_1$ for the spacetime variables as before. In these coordinates we can write 
\begin{equation}\label{eq:sw2}\begin{aligned}
s &= \frac{t}{x_1} + (\sgn \xi_1) \rhonf \Big( \rhopf^2 \Big( \frac{\tau}{\xi_1^2} \Big) + 1 \Big), \\
w_j &= \frac{x_j}{x_1} -  \frac{\xi_{j}}{\xi_{1}}, \quad j \geq 2
\end{aligned}\end{equation}
showing that these are smooth functions. The Hamiltonian vector field here is normalized by replacing the factor $h$ in \cref{eq:Hvf_nat_face 2} with $\rhonf$, and this has the effect of removing the $\xi_{\natural, 1}$ factor and replacing it with $\sgn \xi_1$. We thus obtain 
\begin{equation}\label{eq:Hvf_nat_face pf}
\mathsf{H}_{p_0}= (\sgn \xi_1) (  s \partial s  + \rho_{\mathrm{bf}} \partial_{\rho_{\mathrm{bf}}} + w \cdot \partial_w ),
\end{equation}
so the conclusion is the same as above. 

It remains to analyze the Hamiltonian vector field at $\mathrm{pf} \cap \mathrm{bf}$, and away from $\natural\mathrm{f}$. In this region we can use $h$ as a boundary defining function for $\mathrm{pf}$ and $\xi, \tau$ as frequency coordinates. As a boundary defining function for $\mathrm{bf}$ we can take $1/|t|$, since $t/\lVert x \rVert$ is nonvanishing in this region where the Hamiltonian vector field vanishes, i.e. $t$ is a dominant spacetime variable provided we stay outside of a neighbourhood of $\natural\mathrm{f}$. This follows from  \cref{eq:Hp0} which shows that the coefficient of $\partial_t$ is nonvanishing in this region (remembering that $\tau>0$ on $\Sigma$ for $P_+$). We thus use coordinates $\rhobf = 1/|t|$, $x_j/t$, $1 \leq j \leq n$, as spacetime coordinates. We then define
$$
v_j = \frac{x_j}{t} + \frac{\xi_j}{h^2 \tau +1}, \quad 1 \leq j \leq n,
$$
which vanish on $\calR_\varsigma$. 
In terms of coordinates $\rhobf = 1/t, v, \xi, \tau, h$, we have 
\begin{equation}\label{eq:Hvf_pf_face}
|t| ((h^2 \tau +1) \partial_t - \xi \cdot \partial_x ) = - \sgn t(h^2 \tau + 1) (  \rhobf \partial_{\rhobf} + v \cdot \partial_v ) .
\end{equation}
Thus the renormalized Hamilton vector field vanishes precisely at $\{ \rhobf = 0, v = 0 \}$.

We now check that the set $\{ \rhobf = 0, v = 0 \}$ is in the closure  \cref{eq:Rclosure}. To do this, we set $\xin = h \xi$ and write \cref{eq:Rclosure} in terms of $\tau = \taun/h^2$ and $\xi$. We find that 
$$
\omega_\varsigma(h \xi)   = \varsigma \bbR^+ (h^2 \tau \pm 1, - \xi) .
$$
Using $x/t$ as coordinates for $\omega_\varsigma$, this implies that 
$$
\frac{x}{t} = - \frac{\xi}{h^2 \tau + 1} \Longrightarrow v = 0. 
$$
Thus  $\{ v = 0, \ \rhobf = 0 \}$ is contained in the closure of \cref{eq:Rclosure}. Moreover, it is immediate from \cref{eq:Hvf_pf_face} that this is a sink at $\calR_+$ and a source at $\calR_-$. This completes the analysis of the vanishing set of $H_{p_0}$ in all regions and shows that it coincides with the closure of \eqref{eq:Rclosure}. Since the coordinates $v_j$ have differentials that are linearly independent from $\rhopf$, $\rhonf$ and the renormalized symbol $\mathsf{p}_0$,  the set $\{ v = 0, \ \rhobf = 0 \}$ is a smooth p-submanifold (see \Cref{remark:p-submanifolds}) of $\Sigma$ and, in turn, of ${}^{\calc}\overline{T}^* \bbM$. 

The connectedness of $\calR_\varsigma$ follows from the definition \cref{eq:Rclosure}, since it is the closure of the  image of $(h, \xin) \in \bbR_+ \times \bbR^d$ under the continuous map in \cref{eq:Rclosure}. The disjointness at $\natural\mathrm{f}$ follows from the observation that $x/\lVert x \rVert = \pm \varsigma  \xin/|\xin|$ on $\calR_\varsigma$, and at $\mathrm{pf}$ from the fact that we either have $x/\lVert x \rVert = \pm \varsigma  \xi/|\xi|$ on $\calR_\varsigma$ (near the equator), or $t = \varsigma \infty$ (away from the equator). 
\end{proof}

\begin{remark} This proof illustrates that the radial sets $\calR_{\varsigma}$ have two distinct limits as $h \to 0$. One is the limit at the natural face, which is smooth in the variables $\taun, \xin$ and can be viewed as the limit of the `arms' of the hyperbola $\taun^2 = |\xin|^2 + 1$. This lies over the `equator', where $t/\lVert x \rVert = 0$ (so called from the conventional diagram of compactified spacetime in which the time axis is vertical). The second is the limit at the parabolic face, which is smooth in the variables $\tau, \xi$ and can be viewed as the limit of the `tip' of the hyperbola near the point $\taun = h^2 \tau =  1, \xin = 0$. That is, it is the limit of the paraboloid that is the quadratic approximation to the graph of the parabola at the minimum point of $\tau$, viewed as a function of $\xi$ on the radial set. This part of the radial set lies over the whole of the `northern'  or `southern' hemisphere at spacetime infinity. These two different limits correspond to the natural scale and the laboratory scale as discussed in \S\ref{sec:true_intro}. The fact that the radial sets are smooth in the $\calc$-phase space, and that the Hamiltonian vector field is uniformly nondegenerate in the way in which it vanishes at these radial sets, explains the claim in the introduction that the natural and laboratory scales are the only frequency scales that need to be considered in the limit $c = 1/h \to \infty$. 
\end{remark}

\begin{remark}
It will be convenient in Section~\ref{sec:estimates} to introduce the notation 
\begin{equation}\label{eq:Rschr}
\calR_{\varsigma}^{\mathrm{Schr}} = \calR_{\varsigma} \cap \mathrm{pf} = \calR_{\varsigma} \cap \mathrm{pf} \cap \mathrm{bf} , \quad \varsigma \in \{ +, - \}, 
\end{equation} 
for the intersection of $\calR_{\varsigma}$ with the parabolic face $\mathrm{pf}$. We can identify these with the radial sets for the induced Schr\"odinger operator $\mp 2D_t - \Delta$ at the parabolic face --- see \cite[Section 3.3]{Parabolicsc}.  
\end{remark}

\begin{remark}
Writing $r = \lVert x \rVert$, we observe that 
\begin{equation}
\calR_\varsigma \cap \mathrm{cl}_{{}^{\calc}\overline{T}^* \bbM} \{r=0\} \cap (\natural\mathrm{f} \cup \mathrm{df}) = \varnothing;
\label{eq:north_south_R}
\end{equation}
see \Cref{fig:radial_sets}. 
This fact can be seen from the proof above, where we found that, at $h=0$, we have $\lVert x \rVert/t = 0$ on the radial set only at $\mathrm{pf} \cap \mathrm{bf}$.  
So, $\calR_\varsigma \cap \mathrm{cl}_{{}^{\calc}\overline{T}^* \bbM} \{r=0\} $ is covered by the coordinate chart in \eqref{it:pf_coord} in \S\ref{subsec:char}. 
\end{remark}

\begin{remark} \label{remark:p-submanifolds}
In the inclusions 
\begin{equation}
\calR_\pm \subset \Sigma \cap \mathrm{bf} \subset \mathrm{bf} \subset {}^{\calc}\overline{T}^* \bbM, 
\end{equation}
each term is a p-submanifold of the next. This means that, locally, there exists a coordinate chart so that each submanifold is given by the vanishing of a subset of the coordinates. Our proof above actually constructs such coordinate charts. For example, in the interior of $\natural\mathrm{f} \cap \mathrm{bf}$, we take the coordinate system $s, w, \rhobf, \taun, \xin, h$ above and substitute $\mathsf{p} = \taun^2 \pm 2\taun - |\xin|^2$ for $\taun$. Then we have 
\begin{equation}
\calR_\pm = \{ s = 0, w = 0, \rhobf = 0, \mathsf{p} = 0 \} \subset  
\Sigma \cap \mathrm{bf} = \{ \rhobf = 0, \mathsf{p}  = 0 \} \subset \mathrm{bf} = \{ \rhobf = 0 \}.
\end{equation}
\end{remark}


The inclusion $\calR_\pm \hookrightarrow \Sigma\cap \mathrm{bf}$ is depicted in \Cref{fig:radial_sets}.

\

A consequence of $\calR_\varsigma$ being a p-submanifold of ${}^{\calc}\overline{T}^* \bbM$ is the existence of a \emph{quadratic-defining-function} thereof. By definition, a quadratic defining-function of a submanifold is a nonnegative function that vanishes to second order at the submanifold, and such that the Hessian is nondegenerate on the normal bundle of the submanifold.

\begin{figure}[t!]
	\floatbox[{\capbeside\thisfloatsetup{capbesideposition={left,bottom},capbesidewidth=8cm,capbesidesep=none}}]{figure}[\FBwidth]
	{\caption{
			A zoomed in view of $\calR_+$, as a subset of $\Sigma$, near $\calR_+\cap \natural\mathrm{f}$, showing how $\calR_+$ is a p-submanifold of $\Sigma$. Here $\hat{t}=t/x$.}}
	{\includegraphics[scale=1]{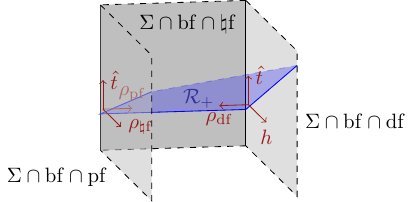}}
\end{figure}

Actually, because $\calR$ is a p-submanifold of $\Sigma$, it will be more convenient to take $\varrho_\varsigma$ to be a quadratic-defining-function for $\calR_\varsigma$ within $\mathsf{p}^{-1}(\{0\})$, which means a smooth nonnegative-valued function $\varrho_\varsigma$ such that $\mathsf{p}^2 +\varrho_\varsigma$ is a global quadratic-defining-function for $\calR_\varsigma$.

\begin{proposition}\label{prop:qdf} There exists a quadratic-defining-function $\varrho_\varsigma$ of $\calR_\varsigma$ within $\mathsf{p}^{-1}(\{0\})$ such that
	\begin{equation}
		\mathsf{H}_p \varrho_\varsigma = \mp \varsigma \iota \varrho_\varsigma \mp \varsigma F+E 
		\label{eq:Hp_qdf}
	\end{equation}
	(where $\iota,F,E$ can depend on $\pm,\varsigma$) for symbols $\iota \in C^\infty({}^{\calc} \overline{T}^* \bbM)$ and $F,E \in S^{0,0,0,0}_{\calc}$ such that 
	\begin{itemize}
		\item $\iota>\varepsilon$ on $\calR_\varsigma$, for some $\varepsilon>0$, 
		\item $F\geq 0$ everywhere, 
		\item $E$ is vanishing cubically at $\calR_\varsigma$, meaning that $E/\varrho_\varsigma^{3/2} \in L^\infty$.
	\end{itemize}
\end{proposition}

\begin{proof} This will follow easily from the local coordinate representations of $\mathsf{H}_{p_0}$ in the previous proof, once we have noticed the following: it suffices to construct $\varrho_+,\iota,F,E$ locally and then patch together the constructions with a partition of unity. Indeed, if $\calU$ is a finite open cover of ${}^{\calc}\overline{T}^* \bbM$ and $\{\chi_U\}_{U\in \calU}$ is a partition of unity subordinate to it, and if we have $\varrho_U,\iota_U \in C^\infty(U;[0,\infty))$ and 
	\begin{equation} 
		F_U,E_U\in S^{0,0,0,0}_{\calc}
	\end{equation} 
	defined in $U$ such that $\varrho_U$ is a quadratic-defining-function for $\calR_+$ in $U$, $\iota_U>\varepsilon_U>0$ on $\calR_+ \cap U$, for some constant $\varepsilon_U>0$, $F_U\geq 0$ on $U$ (in the cases we consider, we will actually have $F_U=0$, but we do not assume this yet), $E_U$ is vanishing cubically at $\calR_+$ in $U$, and $\mathsf{H}_p \varrho_U = - \varsigma \iota_U \varrho_U - \varsigma F_U+E_U$ in $U$, then 
	\begin{equation}
	\varrho_+ = \sum_{U\in \calU} \chi_U \varrho_U 
	\end{equation}
	is a quadratic-defining-function of $\calR_\varsigma$. Then, choosing $\iota \in C^\infty({}^{\calc}\overline{T}^* \bbM; [0,\infty))$ such that $\varepsilon<\iota|_{\calR} < \min_{U\in \calU} \varepsilon_U$ for some $\varepsilon>0$ and $\iota\leq \iota_U$ on each $U\in \calU$,  \cref{eq:Hp_qdf} holds if we define
	\begin{equation}
	F = \sum_{U\in \calU} \chi_U ((\iota_U - \iota)\varrho_U+F_U) ,\quad E = \sum_{U\in \calU} (\chi_U E_U + \varrho_U \mathsf{H}_p \chi_U).
	\end{equation} 
	The symbol $F$ satisfies $F\geq 0$, and $E$ is vanishing cubically at $\calR_+$ (because each $\varrho_U$ is vanishing quadratically at $\calR_+$ within the support of $\mathsf{H}_p \chi_U$ and because $\mathsf{H}_p$ vanishes at $\calR_+$, which means that so does $\mathsf{H}_p \chi_U$).
	
	Also, note that whether a given locally defined $\varrho_+$ satisfies \cref{eq:Hp_qdf} for some symbols $\iota,F,E$ with the desired properties does not depend on the choice of boundary-defining-functions $\rho_{\mathrm{f}}$ going into the definition of $\mathsf{H}_p$, so, when constructing the symbol $\varrho_+$ in local coordinates, we are free to choose whatever locally-defined boundary-defining-functions are convenient for that particular coordinate chart. 

These observations reduce the proof to a local construction in convenient coordinates. And this is almost immediate from the previous proof: for example, in the first case, in the interior of $\natural\mathrm{f} \cap \mathrm{bf}$, we take $\rho_{+} = s^2 + |w|^2 + \rhobf^2$ if $p=p_0$, and similarly for the other coordinate charts. If $p$ differed from $p_0$ by terms decaying quadratically at spacetime infinity, then $(\mathsf{H}_p-\mathsf{H}_{p_0})\varrho_+$ would vanish cubically at $\calR_+$, in which case we could use the same $\varrho_+$ as the free case but add $(\mathsf{H}_p-\mathsf{H}_{p_0})\varrho_+$ to $E$.  

However, because $p$ can differ from $p_0$ by terms which are only decaying \emph{linearly} at spacetime infinity, we need a slightly different argument. (This issue arises already in the sc-analysis of the Klein--Gordon equation for long-range metric perturbations, so is not new. The argument we give for handling such long-range terms is a variant of the standard one.)
Indeed, classical $O(\varrho_{\mathrm{bf}})$ terms in $\mathsf{H}_p$ enter the linearization of $\mathsf{H}_p$ at the radial set. (If the terms are only conormal, then it does not quite make sense to speak of the linearization.) The problem is only in the coefficients of the partial derivatives in $\mathsf{H}_p$ besides $\partial_{\varrho_{\mathrm{bf}}}$; the reason $\partial_{\varrho_{\mathrm{bf}}}$ is not an issue is that $\mathsf{H}_p-\mathsf{H}_{p_0}$ is $O(\varrho_{\mathrm{bf}})$ \emph{as a b-vector field}, which means that the coefficient of $\partial_{\varrho_{\mathrm{bf}}}$ in $\mathsf{H}_p-\mathsf{H}_{p_0}$ is $O(\varrho_{\mathrm{bf}}^2)$, and when this hits $\varrho_+$ we generate a term which is $O(\varrho_{\mathrm{bf}}^3)$ and therefore vanishing cubically at $\calR_+$ (and therefore treatable as part of the $E$ term). 

The key point is that terms like $O(\varrho_{\mathrm{bf}})\partial_{w_j}$, $O(\varrho_{\mathrm{bf}})\partial_s$ do not change the source/sink structure at $\calR_+$.
In order to see this, we can use 
\begin{equation} 
	\varrho_{+} = s^2 + |w|^2 + \Upsilon \rhobf^2
\end{equation} 
for some $\Upsilon \gg 1$, instead of the previous choice of $\varrho_+$. This alternative choice also works in the $p=p_0$ case, and we can choose $\iota$ to not depend on $\Upsilon$ (as follows immediately from the formulas for $\mathsf{H}_{p_0}$ above). 
Now, if one of $O(\varrho_{\mathrm{bf}})\partial_{w_j}$, $O(\varrho_{\mathrm{bf}})\partial_s$ hits $\varrho_+$, we generate $O(\varrho_{\mathrm{bf}}) w_j$ or $O(\varrho_{\mathrm{bf}}) s$. For each $\varepsilon>0$, these are boundable by $2^{-1}\iota \varrho_+$ if we choose $\Upsilon$ large enough relative to $\iota$. So, our new choice of $\varrho_+$ satisfies \cref{eq:Hp_qdf}, except with $\iota/2$ in place of $\iota$ and $2^{-1} \iota \varrho_+ + \varsigma (\mathsf{H}_p-\mathsf{H}_{p_0})\varrho_+ $ added to the $F$ term.
\end{proof}

\begin{proposition}
	$\calR_\pm$ are global sources/sinks, in the sense that if $U_\pm$ are open neighborhoods of $\calR_\pm$ in ${}^{\calc}\overline{T}^* \bbM$, then every integral curve of $\mathsf{H}_p$ (or $-\mathsf{H}_p$) beginning in $U_-$ and contained in $\Sigma$ reaches $U_+$ in finite parameter time.
	\label{prop:sourcesink_global}
\end{proposition}
Cf.\ \Cref{fig:individual_flows}.
\begin{figure}[t]
	\begin{center}
		\includegraphics[scale=1]{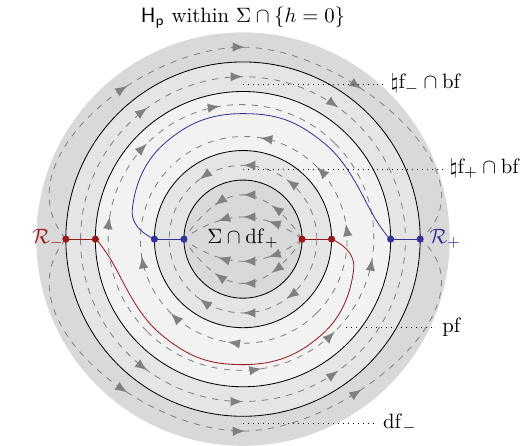}
	\end{center}
	\caption{The Hamiltonian flow in the $h=0$ cross section of \Cref{fig:radial_sets}, except we have also depicted the flow in $\mathrm{df}_\pm$. (Recall that the drawing is done in the $d=1$ case, when $\Sigma_+\cap\mathrm{df}$ consists of two components, a left-moving component and a right-moving component.) Cf.\ \Cref{fig:individual_flows}(b).}
	\label{fig:Schrflow_with_pf}
\end{figure}

\begin{proof} 
We consider only the case of $P_+$ for simplicity, with $P_-$ being analogous.

Let $\gamma$ be an integral curve of $\mathsf{H}_p$ within $\Sigma$. The goal is to show that $\gamma$ tends to $\calR_-$ in one direction and $\calR_+$ in the other. If $\gamma$ is not contained in $\{h=0\}$, then this follows from the global source/sink structure of the sc-Hamiltonian flow for the Klein--Gordon equation, so we only need to consider the case $\gamma \subset \{h=0\}$. But also, since $\Sigma\cap \{h=0\}\subseteq \natural\mathrm{f}\cup \mathrm{bf}$, where $\mathsf{H}_p=\mathsf{H}_{p_0}$, $\gamma$ is an integral curve of $\mathsf{H}_{p_0}$.
	The set $\{h=0\}$ is $\natural\mathrm{f}\cup \mathrm{pf}$, but if $\gamma\subset \mathrm{pf}$, then $\gamma\backslash \natural\mathrm{f}$ is an integral curve for the par-Hamiltonian flow for the Schr\"odinger equation (and $\calR$ is the radial set for that flow) over $\partial \bbM$, so the fact that it travels between the radial sets follows from the global source/sink structure of the flow in \cite{Parabolicsc}. So, we need only to consider the case $\gamma\subset \natural\mathrm{f}$.

To do this, we return to the coordinate system $(s, w, \rhobf, \taun, \xin, h)$ at the beginning of of the proof of Proposition~\ref{prop:calR}. We note that this is a large neighbourhood of the radial set; in fact, it covers the region where $x_1 > 0$. There is a similar coordinate system with $x_1 < 0$; here we would define $\rhobf = -1/x_1$ and the other variables are defined as in \eqref{eq:sw1}. Because of the change of sign in the definition of $\rhobf$, the Hamiltonian vector field in this case takes the form 
\begin{equation}\label{eq:Hvf_nat_face 22}
\mathsf{H}_{p_0} = -\xi_{\natural, 1}(  s \partial_s  + \rho_{\mathrm{bf}} \partial_{\rho_{\mathrm{bf}}} + w \cdot \partial_w )
\end{equation}
and this is a neighbourhood of $\calR_{-\varsigma}$, if the original is a neighbourhood of $\calR_{\varsigma}$ (remembering that $\varsigma = - (\sgn x_1)(\sgn \xi_{\natural, 1})$). The two coordinate systems cover the whole region away from the closure of $x_1 = 0$, that is, the hypersurface $x_1/\ang{z} = 0$ and have uniform transition. So it will be enough to show that every trajectory that starts at $\calR_{-\varsigma}$, in the coordinate system for $x_1/\ang{z} < -\epsilon$ for some positive $\epsilon \ll 1$, passes into the coordinate system with $x_1/\ang{z} > \epsilon$. This is clear for the trajectories in the interior of spacetime, as the Hamiltonian vector field is \eqref{eq:Hp0} with a nonzero component in the $x_1$ direction, since $\xi_{\natural, 1} < 0$. On the other hand, at spacetime infinity, if $|x_1/\ang{z}| \leq \epsilon$ then, after a spatial rotation we may assume that either $x_2$ or $t$ is a dominant variable. In the first case, we may set $\rhobf = 1/x_2$ locally, and then we find that $x_2 h H_{p_0}$ is nonzero applied to $x_1/x_2$. In the second case, we may set $\rhobf = 1/t$ locally, and then we find that $t h H_{p_0}$ is nonzero applied to $x_1/t$. In either case, we pass from the $x_1 < 0$ coordinate system to the $x_1 > 0$ coordinate system, and therefore all trajectories emanating from $\calR_{-\varsigma}$ reach $\calR_{\varsigma}$. 
\end{proof} 

\begin{remark} There may seem something paradoxical about the proof above, as it is evident from the form of the vector field \eqref{eq:sw1}, for example, that the flow exists for all parameter time and remains in the coordinate chart. The confusion arises as we used $1/x_1$ as a boundary defining function for $\mathrm{bf}$ in \eqref{eq:sw1}. This clearly is singular when $x_1/\ang{z} = 0$, and since $\rhobf^{-1}$ is used to renormalize the Hamiltonian vector field on the left-hand side, this has the effect of slowing down the flow so that the hypersurface  $x_1/\ang{z} = 0$ is never actually reached. Under a suitable local boundary defining function for $\mathrm{bf}$, as in the proof immediately above, we see that the hypersurface $x_1/\ang{z} = 0$ is reached and crossed in finite time. 
\end{remark}

We shall also  need the following, which gets at the fact that $\calR_\varsigma$ is a source/sink of $\mathsf{H}_p$ in $\{ \mathsf{p} = 0 \} \subset {}^{\calc}\overline{T}^* \bbM$ and not just in $\Sigma\cap \mathrm{bf}$.
\begin{proposition}\label{prop:sign} 
	
Let $\mathsf{s}$ be a variable spacetime order that is constant near radial sets, and let $\rho_{\mathrm{df}}, \rho_{\mathrm{bf}}, \rho_{\natural\mathrm{f}}$ and  $\rho_{\mathrm{pf}}$ be global boundary defining functions for the respective boundary faces of ${}^{\calc}\overline{T}^* \bbM$. Define 
\begin{equation} 
		a=a_{m,\mathsf{s},\ell,q}= \rho_{\mathrm{df}}^{m}\rho_{\mathrm{bf}}^{\mathsf{s}} \rho_{\natural\mathrm{f}}^{\ell} \rho_{\mathrm{pf}}^q.
\end{equation} 
Then, the symbol 
	\begin{equation} 
		\alpha=\alpha_{m,\mathsf{s},\ell,q} \in S^{0,0,0,0}_{\calc} 
	\end{equation} 
	defined by $\alpha = a^{-1} \mathsf{H}_p a$ (which is a symbol of order $(0,0,0,0)$ because $\mathsf{H}_p\in \calV_{\mathrm{b}} ({}^{\calc}\overline{T}^* \bbM)$) satisfies $\mp \varsigma \alpha|_{\calR_\varsigma} >0$ if $\mathsf{s}|_{\calR_\varsigma}>0$ and $\mp \varsigma\alpha|_{\calR_\varsigma}<0$ if $\mathsf{s}|_{\calR_\varsigma}<0$. 
	\label{prop:alpha}
\end{proposition}

\begin{proof}
The validity of the proposition does not depend on the choice of boundary-defining-functions, so we can verify the proposition by working locally and choosing whichever local boundary-defining-functions are convenient. We will use the coordinate charts from the proof of Proposition~\ref{prop:calR}. We next observe that $\mathsf{H}_p - \mathsf{H}_{p_0} \in \rho_{\mathrm{bf}} \calV_{\mathrm{b}}({}^{\calc} \overline{T}^* \bbM)$, which implies that 
	\begin{equation}
		 a^{-1}(\mathsf{H}_p-\mathsf{H}_{p_0}) a \in \rho_{\mathrm{bf}} S^{0,0,0,0}_{\calc} , 
	\end{equation}
	so this does not affect the sign of $\alpha$ at the radial sets. So, it suffices to compute 
	the signs with $\mathsf{H}_p$ replaced by $\mathsf{H}_{p_0}$.
	
We note that the boundary defining functions $\rho_{\mathrm{df}},  \rho_{\natural\mathrm{f}}$ and  $\rho_{\mathrm{pf}}$ can all be chosen to be functions of $\xi, \tau$ and $h$. These are constant under $\mathsf{H}_{p_0}$, so it is really only the factor $\rho_{\mathrm{bf}}^{\mathsf{s}}$ that plays a role. Thus the sign of $a^{-1} \mathsf{H}_{p_0} a$ is determined by the sign of the $ \rhobf \partial_{\rhobf}$ term in the various coordinate systems in the proof of Proposition~\ref{prop:calR}, which yields the signs in the statement of the Proposition. 
\end{proof}

\section{Uniform estimates}
\label{sec:estimates}

In \S\ref{sec:calculus}, we developed the theory of the pseudodifferential calculus $\Psi_{\calc}$. The upshot of that development is that $\Psi_{\calc}$ behaves as a typical pseudodifferential calculus; it is under symbolic control at the boundary hypersurfaces $\mathrm{df},\mathrm{bf},\natural\mathrm{f}$ (but not $\mathrm{pf}$) of the compactified phase space on which symbols live. Consequently, we can carry out standard microlocal arguments: elliptic estimates and propagation estimates.
Moreover, we showed in \S\ref{sec:flow} that the Hamiltonian flow has a source-to-sink structure within $\Sigma$, with source/sinks $\calR_-,\calR_+$, just like what one sees in existing microlocal treatments of the Klein--Gordon and Schr\"odinger equations. So, we can carry out a version of the microlocal analysis that is uniform in the nonrelativistic limit. This is the goal of this section.

The only consequence of the non-symbolic nature of $\Psi_{\calc}$ at $\mathrm{pf}$ is that error terms that arise from the symbolic constructions in \S\ref{subsec:elliptic_propagation}, \S\ref{subsec:radial_point} will not be suppressed there. Improving the error terms will require a global argument involving the normal operator $N(P_\pm)$. We computed in \S\ref{sec:setup} that this is the (time-dependent) Schr\"odinger operator.  

The relevant solvability theory and estimates for this model problem were treated in \cite{Parabolicsc}, and we present the adaptation needed in our setting in Appendix~\S\ref{sec:schrodinger_estimate}.
See \S\ref{subsec:remainder} for the details of this key part of our argument. For the reader unfamiliar with similar arguments, let us just point out that it should be unsurprising that understanding the nonrelativistic limit of the Klein--Gordon equation requires as input the solvability theory of the Schr\"odinger equation. Here is where that input is used.
Besides this, the arguments -- and in particular the symbolic constructions -- are all standard. We will therefore take liberty with the level of details provided.

In \S\ref{subsec:combined}, we combine the estimates for $P_-,P_+$ to get global estimates on $u$ in terms of $Pu$, with no error terms on the right-hand side. The estimates in this subsection are phrased using the $\calctwo$-Sobolev spaces $H_{\calctwo}^{m,\mathsf{s},\ell;0,0}$,  which are ideal for the task.

Roughly, the progression of error terms in the estimates below is:
\begin{multline}
	\lVert u \rVert_{H_{\calctwo}^{m,\mathsf{s},\ell;0,0}}\overset{\S\ref{subsec:radial_point}}{\longrightarrow} \lVert u \rVert_{H_{\calctwo}^{-N,-N,-N;0,0}} \\  \overset{\S\ref{subsec:remainder}}{\longrightarrow} 
	\lVert Q_+u \rVert_{H_{\calctwo}^{-N,\mathsf{s},-N; -\epsilon,-\infty}} + \lVert Q_-u \rVert_{H_{\calctwo}^{-N,\mathsf{s},-N; -\infty,-\epsilon}}  \overset{\S\ref{subsec:combined} }{\longrightarrow} h^\epsilon\lVert u \rVert_{H_{\calctwo}^{-N, \mathsf{s}, -N; 0, 0}} \overset{\text{absorb}}{\longrightarrow} 0,
\end{multline}
where $N\in \bbR$ is arbitrary, $Q_\pm$ microlocalize to one component of the characteristic set of $P$, and $\epsilon>0$ is possibly small (but nonzero). The final step (which is also to be found in \S\ref{subsec:combined}) is just absorbing the $O(h^\epsilon)$ error term involving $u$ into the left-hand side of the estimate; we can do this as long as $h$ is sufficiently small and $N$ is sufficiently large.
Theorem~\ref{thm:inhomog} is obtained after this last step.

\subsection{Elliptic and propagation estimates}
\label{subsec:elliptic_propagation}
In this section, $m,\ell,q\in \bbR$ are arbitrary and $\mathsf{s} \in C^\infty({}^{\calc}\overline{T}^* \bbM;\bbR)$ is a variable order. We always assume that $\mathsf{s}$ is constant in a neighbourhood of the radial set $\calR$; this allows us to apply  Proposition~\ref{prop:sign}. 

First, we state the elliptic estimate, which is an immediate corollary of \Cref{prop:elliptic}.
\begin{proposition}\label{prop:S5elliptic}
	If $Q,G\in \Psi_{\calc}^{0,0,0,0}$ satisfy
	\begin{equation} 
		\operatorname{WF}'_{\calc}(Q)\subseteq \operatorname{Ell}^{0,0,0,0}_{\calc}(G)\cap \operatorname{Ell}^{2,0,2,0}_{\calc}(P_\pm), 
	\end{equation}then, for any $N\in \bbN$, 
	\begin{equation}\label{eq:elliptic estimate}
		\lVert Q u \rVert_{H_{\calc}^{m,\mathsf{s},\ell,q} } \lesssim 	\lVert GP_\pm u \rVert_{H_{\calc}^{m-2,\mathsf{s},\ell-2,q} } + \lVert u \rVert_{H_{\calc}^{-N,-N,-N,q} } 
	\end{equation}
	holds for all $h>0$ and $u \in \calS'$, in the strong sense that if the right-hand side is finite then so is the left-hand side and the stated inequality holds.
	\label{eq:elliptic_P}
\end{proposition}

Note that $\lesssim$ allows the constant to depend on the upper bound $h_0>0$ of allowable $h$'s. In other words, what the notation ``$a\lesssim b$'' means is that there exists, for each $h_0$ sufficiently small, a $C>0$ such that, for all $h\in (0,h_0)$, the inequality $a\leq C b$ holds.

In \Cref{prop:sourcesink_global}, we proved that $\mathsf{H}_p$ has a global source-to-sink flow in $\Sigma$, with one of $\calR_-,\calR_+$ being the source and the other being the sink. More precisely,
\begin{itemize}
	\item $\pm\mathsf{H}_p$ flows from $\calR_-$ to $\calR_+$, i.e.\ 
	\item $\mp \mathsf{H}_p$ flows from $\calR_+$ to $\calR_-$, 
\end{itemize} 
where `$\pm$' is the sign in $P_\pm$. So, when we are looking at non-negative energy, the Hamiltonian flow goes from the radial set over the past hemisphere of $\partial \bbM$ to the radial set over the future hemisphere; see \Cref{fig:Schrflow_with_pf}. This is reversed when looking at non-positive energy.
 
In the next result, which is a standard propagation estimate holding where $\mathsf{H}_p$ does not vanish, we assume that  $\mathsf{s}$ is nonincreasing as one follows the $\mp \varsigma\mathsf{H}_p$ flow  from $\calR_\varsigma$ to $\calR_{-\varsigma}$, meaning that 
\begin{equation}\label{eq:s monotone}
	\mp \varsigma\mathsf{H}_p \mathsf{s} \leq 0 \text{ in a neighborhood of } \Sigma. 
\end{equation}
In addition, we assume that $\sqrt{\pm\varsigma\mathsf{H}_p \mathsf{s}}$ is a smooth function. The only subtlety is at the boundary of the support of $\mathsf{H}_p\mathsf{s}$, but the desired property can be satisfied by a standard bump function type construction, except for that it is shifted by $-\frac{1}{2}$ in values.

\begin{remark}
This extra assumption is used in the proof of the propagation estimate in \Cref{prop:global_propagation} below. In particular, it enables us to write terms introduced by differentiating $\mathsf{s}$ in the form $B^*B$ for some operator $B$ (at least to the leading order) so that has the favored sign in the estimate.
One can remove this assumption by proving the sharp G\r{a}rding's inequality in the current setting.
\end{remark}

\begin{proposition} Assume that $\mathsf{s}$ is as above. 
Let $\varsigma \in \{-,+\}$. Let $B,G,Q,Z \in \Psi_{\calc}^{0,0,0,0}$ satisfy
\begin{itemize}
		\item  $(\calR\cup \Sigma_{\mathrm{bad}}) \cap \operatorname{WF}'_{\calc}(B) = \varnothing$, i.e.\ the essential support of $B$ intersects neither $\calR=\calR_-\cup\calR_+$ nor the ``bad'' component $\Sigma_{\mathrm{bad}}$ of the characteristic set;
	\item for every point $q \in \operatorname{WF}'_{\calc}(B) \cap \Sigma$ there is a bicharacteristic segment $\gamma([0, r])$ with 
	$\gamma(0) \in \operatorname{Ell}_{\calc}^{0,0,0,0}(Q)$ and $\gamma(r) = q$, flowing in the same direction as $\mp \varsigma\mathsf{H}_p$. 
	\item All of the bicharacteristic segments $\gamma([0, r])$ in the previous item belong to the elliptic set of $G$:
	\begin{equation}
	\gamma([0,r])\subset \operatorname{Ell}_{\calc}^{0,0,0,0}(G)
	\end{equation} 
	i.e.\ $G$ is elliptic on all bicharacteristic segments ``between''  $\operatorname{WF}'_{\calc}(B)$ and $\operatorname{Ell}_{\calc}^{0,0,0,0}(Q)$.
	\item $\operatorname{WF}'_{\calc}(B) \backslash \operatorname{Ell}_{\calc}^{0,0,0,0}(G)  \subseteq \operatorname{Ell}_{\calc}^{0,0,0,0}(Z)$, i.e.\ 
	$Z$ is elliptic over the portion of the essential support of $B$ not already covered by the elliptic set of $G$;\footnote{Note that this portion of the essential support of $B$ must be entirely outside of $\Sigma\cup \Sigma_{\mathrm{bad}}$, hence in the elliptic set of $P_\pm$.} 

\end{itemize} 
Then, for any $N$, 
\begin{equation}
\lVert Bu \rVert_{ H_{\calc}^{ m,\mathsf{s},\ell,q } }
\lesssim \lVert Qu \rVert_{ H_{\calc}^{ m,\mathsf{s},\ell,q } } + \lVert GP_\pm u \rVert_{ H_{\calc}^{ m-1,\mathsf{s}+1,\ell-1,q}}   + \lVert ZP_\pm u \rVert_{ H_{\calc}^{ m-2,\mathsf{s},\ell-2,q}}  + \lVert u \rVert_{ H_{\calc}^{-N,-N,-N,q}} ,
\end{equation}
holds for all $u \in \calS'$, in the strong sense that if the right-hand side is finite then so is the left-hand side and the stated inequality holds.
\label{prop:global_propagation}
\end{proposition}

So, assuming that we have control of $P_\pm u$ on $\Sigma\backslash \calR$, then microlocal regularity of $u$ on $\Sigma\backslash \calR$ propagates, in the direction in which $\mathsf{s}$ is nonincreasing, along the bicharacteristics, although not (yet) as far as $\calR_{-\varsigma}$ as this cannot be reached in finite parameter time. 
See \cite{VasyGrenoble} for a careful exposition of the method of proof for this sort of estimate. The method here is identical. It is a standard positive commutator argument, where the commutant is constructed via the symbol calculus.

We will use the propositions in this subsection many times below.

\subsection{Radial point estimates}
\label{subsec:radial_point}

We now turn to the radial point estimates. We have eight different estimates: two choices of sign in $P_\pm$ (which we are suppressing in the notation --- throughout this subsection $\pm$ is the same sign in the subscript of $P_\pm$, not $\varsigma$), two choices of $\varsigma\in \{-,+\}$ specifying which radial set $\calR_\varsigma$ in $\Sigma$ we are looking at, and two sorts of estimates at each depending on $\mathsf{s}$ is larger or smaller that the threshold there. Roughly, assuming that $\mathsf{s}$ is monotone under the Hamiltonian flow (with the right sign), and assuming that $P_\pm u$ is under sufficient control, we have:
\begin{itemize}
	\item a ``below-threshold'' estimate, which allows us to propagate the control of Sobolev norms near $\calR_\varsigma$ into $\calR_\varsigma$,
	if $\mathsf{s}<-1/2$ near $\calR_\varsigma$. 

	
	\item an ``above-threshold'' estimate: if one have a priori $s_0>-1/2$-order decay near $\calR_\varsigma$, then one can control the Sobolev norm near $\calR_\varsigma$ through an ``subelliptic'' estimate (so-called as it has a loss of regularity of order 1 compared with the elliptic estimate at $\mathrm{df},\mathrm{bf},\natural\mathrm{f}$ --- compare the exponents of the $GP_\pm u$ term in \eqref{eq:misc_j54} compared to the elliptic estimate \eqref{eq:elliptic estimate}).
\end{itemize}

Both of the proofs are $\calc$-analogues of standard arguments --- see \cite{VasyGrenoble} for a detailed presentation of those standard arguments in the context of the sc-calculus. What we will do is present the details of the above threshold estimate and then explain the modifications required to prove the below threshold estimate. Afterwards, we will combine the above and below threshold estimates at different radial sets in $\Sigma$. This is where we make critical use of the fact that the order $\mathsf{s}$ is allowed to be variable. 

\begin{proposition} Fix $\varsigma \in \{-,+\}$. Suppose that $s_0>-1/2$, $m,\ell,q\in \bbR$ and $\mathsf{s}\in C^\infty({}^{\calc}\overline{T}^* \bbM)$ satisfies \eqref{eq:s monotone} is constant near $\calR_\varsigma$ as well as      
	\begin{equation} 
		\mathsf{s}|_{\calR_\varsigma}>-1/2, \; \mathsf{s} \text{ is a constant near } \calR_\varsigma.
	\end{equation} 
	Then, if $B,G ,Z,Q\in \Psi_{\calc}^{0,0,0,0}$  satisfy
	\begin{itemize}
	\item  
		$(\calR_{-\varsigma}\cup \Sigma_{\mathrm{bad}}) \cap \operatorname{WF}'_{\calc}(B) = \varnothing$,
		i.e.\ the essential support of $B$ intersects neither $\calR_{-\varsigma}$ nor the ``bad'' component $\Sigma_{\mathrm{bad}}$ of the characteristic set,
		\item $G$ is elliptic on all bicharacteristic segments between $\operatorname{WF}'_{\calc}(B) \cap \Sigma$ and $\calR_\varsigma$, including ellipticity at $\calR_\varsigma$, 
		\item $\operatorname{WF}'_{\calc}(B) \backslash \operatorname{Ell}_{\calc}^{0,0,0,0}(G) \subseteq \operatorname{Ell}_{\calc}^{0,0,0,0}(Z)$, i.e.\ the portion of the essential support of $B$  not already contained in the elliptic set of $G$ is covered by the elliptic set of $Z$, 
		\item $\calR_\varsigma \subseteq\operatorname{Ell}_{\calc}^{0,0,0,0}(Q)$, i.e.\ $Q$ is elliptic on the radial set $\calR_\varsigma$, 
	\end{itemize}
	then, for any $N\in \bbR$, the estimate 
	\begin{equation}
		\lVert B u \rVert_{H_{\calc}^{m,\mathsf{s},\ell,q}}   \lesssim \lVert G P_\pm u \rVert_{H_{\calc}^{m-1,\mathsf{s}+1,\ell-1,q}} +\lVert Z P_\pm u \rVert_{H_{\calc}^{m-2,\mathsf{s},\ell-2,q}}  + \lVert Qu \rVert_{H_{\calc}^{-N,s_0,-N,q}} + \lVert u \rVert_{H_{\calc}^{-N,-N,-N,q} }
		\label{eq:misc_j54}
	\end{equation}
	holds, in the usual strong sense, for all $u\in \calS'$.
	\label{prop:propagation_out}
\end{proposition}

\begin{remark} The improvement on Proposition~\ref{prop:global_propagation} is that here $B$ can be elliptic at $\calR_{\varsigma}$. Thus, this proposition tells us that, assuming we know a priori that $u$ has above threshold regularity at $\calR_{\varsigma}$ (as given by the finiteness of the $Qu$ norm with $s_0 > -1/2$), then if $P_\pm u$ is regular (order $s+1$) at $\calR_{\varsigma}$, it follows that $u$ is also regular (order $s$) at $\calR_{\varsigma}$. 
\end{remark}

\begin{proof}
	We prove the estimate at $\calR_+$, and the proof of the estimate at $\calR_-$ is analogous. Also, it suffices to consider the $q=0$ case, since the general case just involves multiplying both sides of the desired estimate by some power of $h$. 
	
	Let $U \subset \operatorname{Ell}_{\calc}^{0,0,0,0}(Q)$ be a neighbourhood of $\calR_+$.  It is straightforward to construct a sequence of variable orders  $\{\mathsf{s}_j\}_{j=1}^\infty$, each constant in a fixed neighbourhood of $\calR$, such that 
	\begin{itemize}
	\item $\mathsf{s}_0=\mathsf{s}$, 
		\item $\mathsf{s}_{j} - 1/2 \leq \mathsf{s}_{j+1}\leq \mathsf{s}_{j}$ everywhere,
		\item $\mathsf{s}_j|_{\calR_{\varsigma}}>-1/2$, 
		\item $\mathsf{s}_{j+1} =  \mathsf{s}_j-1/2$ in the complement of $U$, 		
		\item $\mathsf{s}_j$ is (non-strictly) monotonic under the Hamiltonian flow near $\Sigma$, in the same direction as $\mathsf{s}$, 
	\end{itemize}
	for each $j\in \mathbb{N}$. Moreover,  we can arrange that for some $J \in \mathbb{N}$, we have $\mathsf{s}_J \leq s_0$ on $U$ and $\mathsf{s}_J \leq -N$ on the complement of $U$; the second condition is of course an immediate consequence of the fourth item above, for sufficiently large $J$. Given this sequence of variable orders, we will show inductively that the estimate
	\begin{multline}
		\lVert B u \rVert_{H_{\calc}^{m,\mathsf{s},\ell,0} }  \lesssim \lVert G P_\pm u \rVert_{H_{\calc}^{m-1,\mathsf{s}+1,\ell-1,0}}  + \lVert Z P_\pm u \rVert_{H_{\calc}^{m-2,\mathsf{s},\ell-2,0} }  \\ 
		+ \lVert \overline{B} u \rVert_{H_{\calc}^{m-j/2,\mathsf{s}_j,\ell-j/2,0} }  + \lVert u \rVert_{H_{\calc}^{-N,-N,-N,0} }
		\label{eq:misc_j56}
	\end{multline}
	holds for all $u\in \calS'$, in the usual strong sense, for any $\overline{B} \in \Psi_{\calc}^{0,0,0,0}$ satisfying the same requirements as $B$ above and in addition 
	\begin{equation}
		\operatorname{Ell}_{\calc}^{0,0,0,0}(\overline{B}) \supseteq \operatorname{WF}'_{\calc}(B) \cup \calR_\varsigma. 
	\end{equation}
	The desired result \cref{eq:misc_j54} follows from knowing \cref{eq:misc_j56} for $j \geq J$ sufficiently large so that $m - j/2 \leq -N$ and $\ell - j/2 \leq -N$, since then the $\overline{B}u$ term in \cref{eq:misc_j56} satisfies the elliptic estimate 
	\begin{equation}
		\lVert \overline{B} u \rVert_{H_{\calc}^{m-j/2,\mathsf{s}_j,\ell-j/2,0} } \lesssim \lVert Qu \rVert_{H_{\calc}^{-N,s_0,-N,0} } +\lVert u \rVert_{H_{\calc}^{-N,-N,-N,0} }.
		\label{eq:misc_273}
	\end{equation}
	So, we can conclude \cref{eq:misc_j54}.

	The base case $j=1$ is the crux of the problem, whereas the inductive step is straightforward. Indeed, suppose that we know \cref{eq:misc_j56} for some particular value of $j$ and smaller values of $j$. 
	Now choose $\tilde{B}\in \Psi_{\calc}^{0,0,0,0}$ (satisfying the same hypotheses as $B$) such that 
	\begin{equation}
			\operatorname{Ell}_{\calc}^{0,0,0,0}(\overline{B})\supseteq  \operatorname{WF}'_{\calc}(\tilde{B})\supseteq\operatorname{Ell}_{\calc}^{0,0,0,0}(\tilde{B})\supseteq \operatorname{WF}'_{\calc}(B).
	\end{equation}
	Then, \cref{eq:misc_j56} applies with $\tilde{B}$ in place of $\overline{B}$. Applying
	and \cref{eq:misc_j56} with $\tilde{B}$ in place of $B$, $\mathsf{s}_j$ in place of $\mathsf{s}$, and $m-j/2,\ell-j/2$ in place of $m,\ell$, gives 
	\begin{multline}
		\lVert \tilde{B} u \rVert_{H_{\calc}^{m-j/2,\mathsf{s}_j,\ell-j/2,0} }  \lesssim \lVert G P_\pm u \rVert_{H_{\calc}^{m-1,\mathsf{s}+1,\ell-1,0}}  + \lVert Z P_\pm u \rVert_{H_{\calc}^{m-2,\mathsf{s},\ell-2,0} }  \\ 
		+ \lVert \overline{B} u \rVert_{H_{\calc}^{m-(j+1)/2,\mathsf{s}_{j+1},\ell-(j+1)/2,0} }  + \lVert u \rVert_{H_{\calc}^{-N,-N,-N,0} },
	\end{multline}
	where we used that $\mathsf{s}_j\leq \mathsf{s}$ and $\mathsf{s}_{j+1}\geq \mathsf{s}_j-1/2$. We have therefore proven \cref{eq:misc_j56} for $j+1$. 
	
	So, we only need to prove that 
	\begin{multline}
		\lVert B u \rVert_{H_{\calc}^{m,\mathsf{s},\ell,0} }  \lesssim \lVert G P_\pm u \rVert_{H_{\calc}^{m-1,\mathsf{s}+1,\ell-1,0}}  + \lVert Z P_\pm u \rVert_{H_{\calc}^{m-2,\mathsf{s},\ell-2,0} }  \\ 
		+ \lVert \overline{B} u \rVert_{H_{\calc}^{m-1/2,\mathsf{s}_1,\ell-1/2,0} }  + \lVert u \rVert_{H_{\calc}^{-N,-N,-N,0} }
		\label{eq:misc_j31}
	\end{multline}
	to conclude the proposition. 
	First, we will prove this for highly regular $u$, meaning its orders at $\mathrm{df}, \mathrm{bf}$ and $\natural\mathrm{f}$ are sufficiently large  (in which case it is not necessary to assume that $\overline{B}$ is elliptic at $\calR_\varsigma$). Then we will discuss the (really quite standard, but with standard subtleties) regularization argument needed to handle general $u$. 
	This is where we require that $\overline{B}$ is elliptic at $\calR_\varsigma$.

	Assuming that $u$ is sufficiently decaying at $\mathrm{df}$ and $\mathrm{bf}$, for any individual $h>0$, and $L^2$-symmetric pseudodifferential operator $A$, we have
	\begin{equation}
		\langle i ( [A,P_\pm] + (P_\pm-P_\pm^*)A) u,u \rangle_{L^2} = 2 \Im \langle Au,P_\pm u \rangle_{L^2}, 
		\label{eq:misc_445}
	\end{equation}
	where the $L^2$-inner product is conjugate linear in the first slot. Now suppose that $A\in \Psi_{\calc}^{m',\mathsf{s}',\ell',0}$, where $\mathsf{s}'$ is constant in a neighbourhood of $\calR_{\varsigma}$.
	
	Since $P\in \operatorname{Diff}_{\calc}^{2,0,2,0}$,  we have \begin{equation}
		[A,P_\pm] + (P_\pm-P_\pm^*) A \in \Psi_{\calc}^{m'+1,\mathsf{s}'-1,\ell'+1,0}, 
	\end{equation}
	and  this has principal symbol 
	\begin{equation}
		\sigma_{\calc}^{m'+1,\mathsf{s}'-1,\ell'+1,0}(i( [A,P_\pm] + (P_\pm-P_\pm^*) A) )  = i\{p,a\}+p_1 a = H_p a + p_1 a ,  
	\end{equation}
	where $a$ is a representative of the principal symbol of $A$, and where $p_1$ is a principal symbol of $i(P_\pm-P_\pm^*)$. (Cf.\ \Cref{prop:composition}.)
	
	We saw in \Cref{prop:P1} that $P_\pm-P_\pm^*$ lies in $\Psi_{\calc}^{1,-2,1,0}$, so $p_1$ lies in $S_{\calc}^{1,-2,1,0}$.

	There exists $\chi \in C_{\mathrm{c}}^\infty(\bbR)$ such that $\chi(t) = 1$ for $|t| \leq 1/2$, $\chi(t) = 0$ for $|t| \geq 1$, and  such that $-2\operatorname{sign}(t) \chi'(t) \chi(t) = \chi_0(t)^2$ for some $\chi_0\in C_{\mathrm{c}}^\infty(\bbR)$. (Indeed, such a $\chi$ can be built from $\exp(-1/t)$.) For each $\digamma\in \bbR^+$, let $\chi_\digamma(t) = \chi(\digamma t)$, and correspondingly let 
	\begin{equation}
	\chi_{0,\digamma}(t) = \digamma^{1/2} \chi_0(\digamma t), 
	\end{equation} 
	so that $-2\operatorname{sign}(t) \chi_\digamma'(t) \chi_\digamma(t) = \chi_{0,\digamma}^2(t)$. 
	We now consider $a$ defined by 
	\begin{equation}\label{eq:asymboldefn}
		a = \chi_\digamma(\mathsf{p})^2 \chi_\digamma(\varrho_+)^2  \rho_{\mathrm{df}}^{m'} \rho_{\mathrm{bf}}^{\mathsf{s}'} \rho_{\natural\mathrm{f}}^{\ell'},
	\end{equation}
	where $\varrho_+$ is a quadratic-defining-function for $\calR_+$ within $\mathsf{p}^{-1}(\{0\})$ as in Proposition~\ref{prop:qdf}, thus 
	$a$ is supported in a small neighborhood of $\calR_+$ that gets smaller as $\digamma$ grows. We assume that $\digamma$ is chosen large enough so that the order $\mathsf{s}'$ is constant on the support of $a$. 
	
	Let $w = \rho_{\mathrm{df}}^{-1} \rho_{\mathrm{bf}}\rho_{\natural\mathrm{f}}^{-1}$, so that $\mathsf{H}_p = 2^{-1} w^{-1} H_p$. 
	Then, letting $\iota,F,E$ be as in Proposition~\ref{prop:qdf} and $\alpha=\alpha_{m',\mathsf{s}',\ell',0}$ be as in Proposition~\ref{prop:alpha},  we have (recalling $\mathsf{p} = \rho_{\mathrm{df}}^2 \rho_{\natural\mathrm{f}}^2 \sigma_{\calc}^{2,0,2,0}(P_\pm)$) 
	\begin{multline}
		\mathsf{H}_pa = \bigg( 2 \chi_\digamma'(\mathsf{p}) \chi_\digamma(\mathsf{p}) \mathsf{p} (\rho_{\mathrm{df}}^{-2} \rho_{\natural\mathrm{f}}^{-2} \mathsf{H}_p \rho_{\mathrm{df}}^2 \rho_{\natural\mathrm{f}}^2 )\chi_\digamma(\varrho_+)^2  + \chi_\digamma(\mathsf{p})^2 \chi_{0,\digamma}(\varrho_+)^2   (\iota \varrho_+ + F - E) \\ + \alpha  \chi_\digamma(\mathsf{p})^2 \chi_\digamma(\varrho_+)^2 \bigg) \rho_{\mathrm{df}}^{m'}  \rho_{\mathrm{bf}}^{\mathsf{s}'} \rho_{\natural\mathrm{f}}^{\ell'}.
		\label{eq:misc_b41}
	\end{multline}
(Our assumption that $a$ is supported where the order $\mathsf{s}'$ is constant means that there is no term where the Hamiltonian vector field hits $\mathsf{s}'$.) 
	Assume that $\mathsf{s}'<0$ on $\calR_+$. 
	For all $\digamma>0$ sufficiently large, and for all $\delta$ sufficiently small (where what ``sufficiently small'' means depends on $\digamma$), we define symbols $b,e,h \in S^{0,0,0,0}_{\calc}$ 
\begin{align}
		\begin{split} 
			b &=   \chi_\digamma(\mathsf{p})\chi_\digamma(\varrho_+)\sqrt{\alpha-\delta \chi_\digamma(\mathsf{p})^2 \chi_\digamma(\varrho_+)^2  + 2^{-1} w^{-1} p_1}, \\ 
			e &= \chi_\digamma(\mathsf{p}) \chi_{0,\digamma}(\varrho_+) \sqrt{\iota \varrho_+ + F -E}, \\ 
			h &= 2 \chi_\digamma'(\mathsf{p}) \chi_\digamma(\mathsf{p}) \mathsf{p}(\rho_{\mathrm{df}}^{-2} \rho_{\natural\mathrm{f}}^{-2} \mathsf{H}_p \rho_{\mathrm{df}}^2 \rho_{\natural\mathrm{f}}^2 )\chi_\digamma(\varrho_+)^2.
		\end{split} 
	\end{align}
	such that 
	\begin{align}
		\begin{split}
			\mathsf{H}_p a + 2^{-1} w^{-1} p_1 a  &= (\delta \chi_\digamma(\mathsf{p})^4 \chi_\digamma(\varrho_+)^4 + b^2 + e^2+h) \rho_{\mathrm{df}}^{m'} \rho_{\mathrm{bf}}^{\mathsf{s}'} \rho_{\natural\mathrm{f}}^{\ell'}, \\
			\text{or equivalently: } H_p a + p_1 a  &= 2 (\delta \chi_\digamma(\mathsf{p})^4 \chi_\digamma(\varrho_+)^4 + b^2 + e^2+h) w \rho_{\mathrm{df}}^{m'} \rho_{\mathrm{bf}}^{\mathsf{s}'} \rho_{\natural\mathrm{f}}^{\ell'}. 
			\label{eq:misc_b412}
		\end{split}
	\end{align}
Key in our ability to do this is that, since $\mathsf{s}'<0$ on $\calR_+$, \Cref{prop:alpha} says that $\alpha>0$ on $\calR_+$. So, since the function $p_1$ lies in $S_{\calc}^{1,-2,1,0}$, implying $w^{-1} p_1 \in S_{\calc}^{0,-1,0,0}= O(\rho_{\mathrm{bf}})$, 	if $\digamma$ is sufficiently large then
	\begin{equation} 
		\alpha>-2^{-1} w^{-1} p_1
	\end{equation}
	on the support of $\chi_\digamma(\mathsf{p})\chi_\digamma(\varrho_+)$. Thus, for $\delta$ sufficiently small, $b$ is well-defined and smooth. Likewise, as long as $\digamma$ is sufficiently large, the factor $\chi_{0,\digamma}(\varrho_+)$ of $e$ is supported in a small annular neighborhood of $\calR_+$, where $\iota \varrho_+ + F > E$ (since $E$ vanishes cubically at $\calR_+$, whereas $\iota\varrho_+$ only vanishes quadratically, and since $F\geq 0$ the $F$ term only helps). So, $e$ is well-defined.

	Quantizing both sides of \Cref{eq:misc_b412}, we obtain operators  
	\begin{align}
		\begin{split} 
		A := 2^{-1}(\operatorname{Op}(a) + \operatorname{Op}(a)^*) &\in \Psi_{\calc}^{-m',-\mathsf{s}',-\ell',0} \\ 
		B_0 := \operatorname{Op}( \sqrt{w \rho_{\mathrm{df}}^{m'}\rho_{\mathrm{bf}}^{\mathsf{s}'}\rho_{\natural\mathrm{f}}^{\ell'}} b), \;
		E_0 := \operatorname{Op}(\sqrt{w \rho_{\mathrm{df}}^{m'}\rho_{\mathrm{bf}}^{\mathsf{s}'}\rho_{\natural\mathrm{f}}^{\ell'}} e) &\in \Psi_{\calc}^{(1-m')/2,(-\mathsf{s}'-1)/2,(1-\ell')/2,0} \\ 
		H := \operatorname{Op}( w \rho_{\mathrm{df}}^{m'}\rho_{\mathrm{bf}}^{\mathsf{s}'}\rho_{\natural\mathrm{f}}^{\ell'}  h) &\in  \Psi_{\calc}^{1-m',-\mathsf{s}'-1,1-\ell',0} \\ 
		\Lambda =\operatorname{Op}\big(\sqrt{w\rho_{\mathrm{df}}^{-m'}\rho_{\mathrm{bf}}^{-\mathsf{s}'}\rho_{\natural\mathrm{f}}^{-\ell'} }\big) &\in \Psi_{\calc}^{(m'+1)/2,(\mathsf{s}'-1)/2,(\ell'+1)/2,0}
		\end{split}
	\label{eq:quantizing}\end{align}
	such that 
	\begin{equation}
		i( [A,P_\pm] + (P_\pm-P_\pm^*) A ) = \delta A \Lambda^* \Lambda A  + B_0^*B_0 + E_0^*E_0  + H + R
		\label{eq:misc_ooo}
	\end{equation}
	for some $R \in \Psi_{\calc}^{-m',-2-\mathsf{s}',-\ell',0}$ from the principal symbol short-exact sequence \eqref{eq:short exact natres}. Since our quantization procedure preserves essential supports, the operators $A,B_0,E_0,H$ all have essential support inside $\operatorname{supp} \chi_{\digamma}(\mathsf{p}) \cap \operatorname{supp}\chi_\digamma(\varrho_+)$. It follows that $R$ does as well.

	Plugging \cref{eq:misc_ooo} into \cref{eq:misc_445}, we get 
	\begin{equation}
		2 \Im \langle Au,P_\pm u \rangle_{L^2} = \delta \lVert \Lambda A u \rVert_{L^2}^2 + \lVert B_0 u \rVert_{L^2}^2 + \lVert E_0 u \rVert_{L^2}^2 + \langle Hu,u \rangle_{L^2} + \langle Ru ,u \rangle_{L^2}.
	\end{equation}
	So, 
	\begin{align}
		\begin{split} 
			\lVert B_0u \rVert_{L^2}^2 +\delta \lVert \Lambda A u \rVert_{L^2}^2&\leq \lVert B_0u \rVert_{L^2}^2 +\lVert E_0u \rVert_{L^2}^2 +\delta \lVert \Lambda A u \rVert_{L^2}^2 \\ 
			&\leq    2 |\langle Au,P_\pm u \rangle_{L^2} | + |\langle Hu,u \rangle_{L^2}| + |\langle Ru,u \rangle_{L^2}|.
		\end{split} 
		\label{eq:misc_h55}
	\end{align}
	Let $\tilde{\Lambda}$  denote a two-sided parametrix for $\Lambda$ on some neighborhood of  $\operatorname{WF}'_{\calc}(A)$, which we can choose to be symmetric and have essential support contained in $O$, which is an arbitrary fixed neighborhood of $\operatorname{WF}'_{\calc}(A)$. We have 
	\begin{equation}
		\langle Au,P_\pm u \rangle = \langle \Lambda A u ,\tilde{\Lambda} P_\pm u \rangle_{L^2} + \langle u , R_0 u \rangle_{L^2} 
	\end{equation}
	for some $R_0\in \Psi_{\calc}^{-\infty,-\infty,-\infty,0}$.  Then, we use $2|\langle \Lambda A u ,\tilde{\Lambda} P_\pm u \rangle_{L^2}|\leq (\delta/2) \lVert \Lambda A u \rVert_{L^2}^2 + (2/\delta) \lVert \tilde{\Lambda} P u \rVert_{L^2}^2$ to go from \cref{eq:misc_h55} to 
	\begin{align}
		\begin{split} 
		\lVert B_0u \rVert_{L^2}^2&\leq \lVert B_0u \rVert_{L^2}^2 +(\delta/2) \lVert \Lambda A u \rVert_{L^2}^2     \\ &\leq  (2/\delta) \lVert \tilde{\Lambda} P_\pm  u \rVert_{L^2}^2 + |\langle Hu,u \rangle_{L^2}| + |\langle Ru,u \rangle_{L^2}| + C_N\lVert u \rVert_{H^{-N,-N,-N,0}_{\calc } }^2 ,
		\end{split} 
	\end{align}
	for some $C_N>0$, where the last term comes from bounding $\langle u,R_0u \rangle_{L^2}$.
	
	Note that, $h$ has a $\chi_{\digamma}'(\mathsf{p})$-factor, hence the wavefront set of $H$ is away from the characteristic set, and if $\digamma$ is sufficiently large, then $\operatorname{WF}'_{\calc}(H)$ is contained entirely in $\operatorname{Ell}_{\calc}^{2,0,2,0}(G\tilde{P})$, so an elliptic estimate yields 
	\begin{equation}
		|\langle Hu,u \rangle_{L^2}| \lesssim \lVert G P_\pm u \rVert_{H_{\calc}^{-(3+m')/2,-(1+\mathsf{s}')/2,-(3+\ell')/2,0 } }^2 + \lVert u \rVert_{H_{\calc}^{-N,-N,-N,0} }^2 . 
	\end{equation}
	Likewise, if $\digamma$ is sufficiently large, then $\operatorname{WF}'_{\calc }(R)$ will be contained in the elliptic set of $\overline{B}$ (which was introduced in \eqref{eq:misc_j56}), so that
	\begin{equation}
		|\langle Ru,u \rangle_{L^2}| \lesssim \lVert \overline{B} u \rVert^2_{H_{\calc }^{-m'/2,-(2+\mathsf{s}')/2,-\ell'/2,0} } +  \lVert u \rVert_{H_{\calc }^{-N,-N,-N,0} }^2 
	\end{equation}
	Also, if $\digamma$ is sufficiently large, then we can choose $O \supset \operatorname{WF}'_{\calc }(\tilde{\Lambda})$ so that it is contained in the elliptic set of $G$. Consequently,
	\begin{equation}
		\lVert \tilde{\Lambda} P_\pm u \rVert_{L^2} \lesssim \lVert GP_\pm u \rVert_{H_{\calc}^{-(1+m')/2,(1-\mathsf{s}')/2,-(1+\ell')/2,0} }+\lVert u \rVert_{H_{\calc }^{-N,-N,-N,0}} .
	\end{equation}
	So, altogether, we get 
	\begin{multline}
		\lVert B u\rVert_{H_{\calc}^{(1-m')/2,(-\mathsf{s}'-1)/2,(1-\ell')/2,0} } \lesssim  \lVert GP_\pm u \rVert_{H_{\calc}^{-(1+m')/2,(1-\mathsf{s}')/2,-(1+\ell')/2,0} }\\ + \lVert \overline{B} u \rVert_{H_{\calc }^{-m'/2,-(2+\mathsf{s}')/2,-\ell'/2,0} } + \lVert u \rVert_{H_{\calc}^{-N,-N,-N,0} }
		\label{eq:misc_297}
	\end{multline}
	in the case where the essential support of $B$ is contained in $B_0$. 
	This gives us microlocal control of $u$ near $\calR_\varsigma$. 
	Combining this with elliptic and propagation estimates (Propositions~\ref{prop:S5elliptic} and \ref{prop:global_propagation}), we get 
	\begin{multline}
		\lVert B u\rVert_{H_{\calc}^{(1-m')/2,(-\mathsf{s}'-1)/2,(1-\ell')/2,0} } \lesssim  \lVert GP_\pm u \rVert_{H_{\calc}^{-(1+m')/2,(1-\mathsf{s}')/2,-(1+\ell')/2,0} } \\ +\lVert Z P_\pm u \rVert_{H_{\calc}^{-(3+m')/2,-(1+\mathsf{s}')/2,-(3+\ell')/2,0}}  + \lVert \overline{B} u \rVert_{H_{\calc }^{-m'/2,-(2+\mathsf{s}')/2,-\ell'/2,0} } + \lVert u \rVert_{H_{\calc}^{-N,-N,-N,0} }
	\end{multline}
	for general $B$ as in the proposition statement.

	This almost becomes the desired estimate \cref{eq:misc_j31} if we take $m'=1-2m$, $\mathsf{s}'= -1-2\mathsf{s}$, and $\ell'=1-2\ell$. 
	The requirement that $\mathsf{s}>-1/2$ on $\calR_+$, imposed in the proposition statement, is consistent with the only requirement imposed on $\mathsf{s}'$ so far, that $\mathsf{s}'<0$ on $\calR_+$. The stated choices of $m',\mathsf{s}',\ell'$ are therefore allowed. The only reason that the estimate that results from these choices is not exactly \cref{eq:misc_j31} is because the variable order in the $\overline{B} u$ term is $\mathsf{s}-1/2$, not $\mathsf{s}_1$. But we are assuming $\mathsf{s}_1\geq \mathsf{s}-1/2$ so we have actually proved a (potentially) slightly stronger estimate than what we require. 
	
	Having now dealt with the case where $u$ is highly regular, we need to explain the regularization argument required to handle general $u\in \calS'$. 
	The only place we needed $u$ to be Schwartz was to justify those $L^2$-pairings and integration by parts (i.e., identities like $\la Au, P_\pm u \ra = \la P_\pm^*A u , u \ra$). We remove this restriction by introducing regularizers. This process is also needed in the proof of Proposition~\ref{prop:global_propagation}; however, that argument is more straightforward as there is no threshold condition in that case, so one can regularize in both differential and decay orders to arbitrarily high order. The radial point setting is slightly more delicate so we provide more details here. 
	
	The idea is to replace $a$ with $a_{\varepsilon}$, for each $\varepsilon>0$, given by
	\begin{equation}
		a_{\varepsilon} =a \rho_{\natural\mathrm{f}}^{\ell'} \Big( \frac{1}{1+\varepsilon \rho_{\mathrm{df}}^{-1} } \Big)^K\Big( \frac{1}{1+\varepsilon \rho_{\mathrm{bf}}^{-1} } \Big)^{K'},
		\label{eq:misc_299}
	\end{equation}
	for some $K,K'\in \bbR$. 
	We are not regularizing in $\natural\mathrm{f}$ here. One might expect it necessary to regularize in all orders, but all we need is to regularize enough to justify the various algebraic manipulations for each individual $h>0$.
	
	The derivative $\mathsf{H}_p a_{\varepsilon}$ is computed as before. We get 
\begin{multline}
		\mathsf{H}_pa = \Big( 2 \chi_\digamma'(\mathsf{p}) \chi_\digamma(\mathsf{p}) \mathsf{p} (\rho_{\mathrm{df}}^{-2} \rho_{\natural\mathrm{f}}^{-2} \mathsf{H}_p \rho_{\mathrm{df}}^2 \rho_{\natural\mathrm{f}}^2 )\chi_\digamma(\varrho_+)^2  \rho_{\mathrm{df}}^{m'} \rho_{\mathrm{bf}}^{\mathsf{s}'} \rho_{\natural\mathrm{f}}^{\ell'} + \chi_\digamma(\mathsf{p})^2 \chi_{0,\digamma}(\varrho_+)^2  \rho_{\mathrm{df}}^{m'} \rho_{\mathrm{bf}}^{\mathsf{s}'} \rho_{\natural\mathrm{f}}^{\ell'} (\iota \varrho_+ + F - E) \\ + \alpha_{K, K', \epsilon}  \chi_\digamma(\mathsf{p})^2 \chi_\digamma(\varrho_+)^2  \rho_{\mathrm{df}}^{m'} \rho_{\mathrm{bf}}^{\mathsf{s}'} \rho_{\natural\mathrm{f}}^{\ell'} \Big) \Big( \frac{1}{1+\varepsilon \rho_{\mathrm{df}}^{-1} } \Big)^K\Big( \frac{1}{1+\varepsilon \rho_{\mathrm{bf}}^{-1} } \Big)^{K'},
		\label{eq:misc_b4114}
	\end{multline}
where we get the right-hand side of \eqref{eq:misc_b41} multiplied by the regularizing factors in \eqref{eq:misc_299} and we get additional terms from applying $\mathsf{H}_p$ to the regularizing factors; these terms are incorporated into 
	\begin{equation}
		\alpha_{K,K',\varepsilon} := \alpha + \frac{ K\varepsilon}{\varepsilon + \rho_{\mathrm{df}} } \rho_{\mathrm{df}}^{-1} \mathsf{H}_p \rho_{\mathrm{df}} +  \frac{ K'\varepsilon}{\varepsilon + \rho_{\mathrm{bf}} } \rho_{\mathrm{bf}}^{-1} \mathsf{H}_p \rho_{\mathrm{bf}}. 
	\end{equation}
We also recall from the proof of \Cref{prop:alpha} that 
\begin{equation} 
\alpha = \mathsf{s}' \rho_{\mathrm{bf}}^{-1} \mathsf{H}_p \rho_{\mathrm{bf}} + \text{ terms vanishing at } \calR \cap \{ h = 0 \}. 
\end{equation} 
It follows that the sign of $\alpha_{K,K',\varepsilon}$ is equal to the sign of $\alpha$  (positive under the assumptions that $\mathsf{s}' < 0$ and we are working near $\calR_+$)  provided that $K' < -\mathsf{s}'$ and $h$ is sufficiently small. 
%
	Equivalently, in terms of $\mathsf{s}$, this says
	\begin{equation}
		K' <1+2\mathsf{s}|_{\calR_+} = 2(\mathsf{s}|_{\calR_+} +1/2).
		\label{eq:misc_301}
	\end{equation}
	This does allow $K' > 0$, due to the above-threshold condition $\mathsf{s}|_{\calR_+} > -1/2$, but it is limited by the upper bound \eqref{eq:misc_301}. 
	On the other hand, there is no condition on $K$ because $\rho_{\mathrm{df}}^{-1} \mathsf{H}_p \rho_{\mathrm{df}} = O(h)$  (as in the proof of \Cref{prop:alpha}). 
	So, we can take $K$ to be arbitrarily large (which corresponds to an arbitrarily large amount of regularization), but not $K'$ (we cannot induce arbitrarily much decay). 
	
	Given such $K'$, we can define, for $\digamma>0$ sufficiently large and $\delta$ sufficiently small relative to $\digamma$, we can define a symbol 
	\begin{equation}\begin{aligned}
		b_\varepsilon &=   \chi_\digamma(\mathsf{p})\chi_\digamma(\varrho_+)\sqrt{\alpha_{K,K',\varepsilon}-\delta_{K,K',\varepsilon}  \chi_\digamma(\mathsf{p})^4 \chi_\digamma(\varrho_+)^4  + 2^{-1} w^{-1} p_1  },  \\
		&\delta_{K,K',\varepsilon} = \delta \Big( \frac{1}{1+\varepsilon \rho_{\mathrm{df}}^{-1} } \Big)^K\Big( \frac{1}{1+\varepsilon \rho_{\mathrm{bf}}^{-1} } \Big)^{K'}
	\end{aligned} \end{equation} 
	in addition to $e,h$ defined as before. 
	Then, 
	\begin{equation}
			H_p a_\varepsilon + p_1 a_\varepsilon = (\delta_{K,K',\varepsilon} \chi_\digamma(\mathsf{p})^4 \chi_\digamma(\varrho_+)^4 + b^2_\varepsilon + e^2+h) w \rho_{\mathrm{df}}^{m'} \rho_{\mathrm{bf}}^{\mathsf{s}'} \rho_{\natural\mathrm{f}}^{\ell'}\Big( \frac{1}{1+\varepsilon \rho_{\mathrm{df}}^{-1} } \Big)^K\Big( \frac{1}{1+\varepsilon \rho_{\mathrm{bf}}^{-1} } \Big)^{K'}.
	\end{equation}
	Because $\alpha_{K,K',\varepsilon}$ satisfies $L^\infty$-bounds which are uniform in $\varepsilon$, the chosen $\digamma$ and $\delta$ can be independent of $\varepsilon$. 
	Then, a key fact, which follows from uniform symbolic bounds on $1/(1+\varepsilon \rho_{\mathrm{df}}^{-1})$ as $\varepsilon \to 0^+$, is that $b_\varepsilon,a_\varepsilon$ are uniform families of symbols as $\varepsilon \to 0^+$. So, quantizing, we get 
	\begin{equation}
		A_\varepsilon = 2^{-1}(\operatorname{Op}(a_\varepsilon) + \operatorname{Op}(a_\varepsilon)^*) \in L^\infty([0,1)_\varepsilon;\Psi_{\calc}^{-m',-\mathsf{s}',-\ell',0} ), 
	\end{equation}
	\begin{multline}
		B_{0,\varepsilon} =  \operatorname{Op}\Big( \sqrt{w \rho_{\mathrm{df}}^{m'}\rho_{\mathrm{bf}}^{\mathsf{s}'}\rho_{\natural\mathrm{f}}^{\ell'}} b_\varepsilon \Big( \frac{1}{1+\varepsilon \rho_{\mathrm{df}}^{-1} } \Big)^{K/2}\Big( \frac{1}{1+\varepsilon \rho_{\mathrm{bf}}^{-1} } \Big)^{K'/2}\Big) \\ \in L^\infty([0,1)_\varepsilon;\Psi_{\calc}^{(1-m')/2,(-\mathsf{s}'-1)/2,(1-\ell')/2,0})
	\end{multline}
	and so on (with uniform estimates on essential supports as well). These satisfy 
	\begin{equation} 
	i( [A_\varepsilon,P_\pm] + (P_\pm-P_\pm^*) A_\varepsilon ) = \delta A_\varepsilon \Lambda^* \Lambda A_\varepsilon  + B_{0,\varepsilon}^*B_{0,\varepsilon} + E_{0,\varepsilon}^*E_{0,\varepsilon}  + H_\varepsilon + R_\varepsilon, 
	\end{equation} 
	where all of the operators are uniform families of pseudodifferential operators of the orders given previously. We can now go through the same computation that led to \cref{eq:misc_297}, and the manipulations will make sense if the various norms and inner products which appeared make sense distributionally, for each $\varepsilon>0$. 
	That is, we want each term in 
	\begin{equation}
		2 \Im \langle A_\varepsilon u,P_\pm u \rangle_{L^2} = \delta \lVert \Lambda A_\varepsilon u \rVert_{L^2}^2 + \lVert B_{0,\varepsilon} u \rVert_{L^2}^2 + \lVert E_{0,\varepsilon} u \rVert_{L^2}^2 + \langle H_\varepsilon u,u \rangle_{L^2} + \langle R_\varepsilon u ,u \rangle_{L^2}.
		\label{eq:misc_307}
	\end{equation}
	to be well-defined, for each $\varepsilon>0$, for the $u$'s in question.
	It turns out that, if 
	$u(h) \in H_{\mathrm{sc}}^{-N,s_0}$ near the radial set in question, and therefore on the essential support of the various operators appearing above (assuming that $P_\pm u$ is regular enough to propagate that control --- otherwise, the estimates we are trying to prove all have infinite right-hand sides and therefore hold trivially), then, if $K>2m+2N$ and
	\begin{equation}
		2(\mathsf{s}|_{\calR_+} - s_0) < K'  < 1+2\mathsf{s}|_{\calR_+}, 
	\end{equation}
	(the interval being nonempty because $s_0>-1/2$), the manipulations make sense. For example, consider the $L^2$-norm of $B_{0,\varepsilon} u$. Under the stated assumptions, $B_{0,\varepsilon} u$ lies in
	\begin{equation}
		\Psi_{\mathrm{sc}}^{(1-m'-K)/2,(-\mathsf{s}'-1-K')/2}H_{\mathrm{sc}}^{-N,s_0} = \Psi_{\mathrm{sc}}^{m-K/2,\mathsf{s}-K'/2}H_{\mathrm{sc}}^{-N,s_0} \subseteq L^2,
	\end{equation}
	so the norm $\lVert B_{0,\varepsilon} u \rVert$ is finite. One similarly checks that the other terms in \cref{eq:misc_307}  are well-defined norms and distributional pairings. We will skip the details, because it suffices to do the check for each individual $h>0$, where it is part of the proof of the above-threshold radial point estimates in \cite{VasyGrenoble}.
  In addition, we note here the subtlety that the range of $K'$ only allows us to make sense of $\la AP_\pm-P_\pm^*A u , u \ra$ instead of $\la AP_\pm  u , u \ra$ and $\la P_\pm^*A u, u \ra$ individually. But this can be overcome by introducing a further regularizing family and we refer readers to the discussion after \cite[Equation~(5.62)]{VasyGrenoble} for more details.

	The argument now proceeds as before to get \cref{eq:misc_297}, except with $B_{0,\varepsilon}$ in place of $B_0$. This estimate holds for each individual $h$ and $\varepsilon>0$, and the constant in uniform in these. So, to complete the proof one just needs 
	\begin{equation}
		\lVert B_{0} u \rVert_{L^2}\leq \sup_{\varepsilon \in [0,1)}  \lVert B_{0,\varepsilon} u \rVert_{L^2} 
	\end{equation}
	when the right-hand side is known to be finite. This is an often cited consequence of Banach--Alaoglu: we can choose $\varepsilon_j\to 0^+$ such that $B_{0,\varepsilon_j}u $ converges weakly to some $w\in L^2$ satisfying
	\begin{equation}
		\lVert w \rVert_{L^2}\leq \sup_{\varepsilon \in [0,1)}  \lVert B_{0,\varepsilon} u \rVert_{L^2} .
	\end{equation} 
	But, $B_{0,\varepsilon}u\to B_0u$ in $\calS'$ (since $b_{\varepsilon}\to b$ as $\varepsilon\to 0^+$ in a large space of symbols), so we must have $w=B_{0}u$. 
\end{proof}

The combination of Propositions~\ref{prop:global_propagation} and \ref{prop:propagation_out} show that one can propagate regularity of $u$ from $\calR_\varsigma$ right across $\Sigma \setminus \calR_{-\varsigma}$, provided that $P_\pm u$ is sufficiently regular and that $u$ is known to be above threshold microlocally near $\calR_\varsigma$. It remains to propagate regularity into the other radial set $\calR_{-\varsigma}$. To do this, we need to assume that $\mathsf{s}$ is \emph{below} threshold, that is, $< -1/2$, near $\calR_{-\varsigma}$.  It suffices to work in a small neighbourhood of $\calR_{-\varsigma}$, in which we can assume that $\mathsf{s}$ is constant. For convenience, we assume that $\varsigma = +$ and express the conditions on the operators $B$, $E$, $G$ in terms of a quadratic defining function $\varrho_-$ for $\calR_{-\varsigma} = \calR_-$.

The ``below-threshold'' radial point estimate is:
\begin{proposition} Arbitrarily fix $\varsigma = +$. Suppose that $m,\ell,q\in \bbR$ and that $\mathsf{s}\in C^\infty({}^{\calc}\overline{T}^* \bbM)$ satisfies \eqref{eq:s monotone} and 
	\begin{equation} 
		\mathsf{s}|_{\calR_-}<-1/2, \; \mathsf{s} \text{ is a constant near } \calR_-.
	\end{equation} 

Assume that 	$E, B,G\in \Psi_{\calc}^{0,0,0,0}$ are such that, for some sufficiently small parameter $\epsilon >0$,  we have 
	\begin{itemize}
	\item $\operatorname{WF}'_{\calc}(B) \subset \{ \varrho_- + |\mathsf{p}| \leq \epsilon \}$, 
	\item The  region $\{ \epsilon \leq \varrho_- \leq 2\epsilon, \  \mathsf{p} = 0  \}$ is contained in $\operatorname{Ell}_{\calc}^{0,0,0,0}(E)$, 
	\item $G$ is elliptic on $\{  \varrho_- + |\mathsf{p}| \leq 4\epsilon \}$. 
	\end{itemize}
	Then
	\begin{equation}
		\lVert B u \rVert_{H_{\calc}^{m,\mathsf{s},\ell,q}}   \lesssim \lVert G P_\pm u \rVert_{H_{\calc}^{m-1,\mathsf{s}+1,\ell-1,q}} 
		+ \lVert E u \rVert_{H_{\calc}^{m,\mathsf{s},\ell,q}}  + \lVert u \rVert_{H_{\calc}^{-N,-N,-N,q} }
		\label{eq:misc_j57}
	\end{equation}
	holds in the usual strong sense, for all $u\in \calS'(\bbR^{1,d})$. 
	\label{prop:propagation_in}
\end{proposition}
\begin{proof}[Proof sketch]
	The proof proceeds very similarly to the above threshold estimate. We choose $\digamma = (2\epsilon)^{-1}$, assume $q=0$ as before,  and define symbol $a$ as in \eqref{eq:asymboldefn}. The main difference is that, for $\mathsf{s}<-1/2$ on $\calR_+$, $\mathsf{s}'=-1-2\mathsf{s}$ satisfies $\mathsf{s}'>0$ on $\calR_+$, which means that $\alpha=\alpha_{m',\mathsf{s}',\ell',0}$ in Proposition~\ref{prop:alpha} is now negative. So, instead of \cref{eq:misc_b412}, we can instead arrange
	\begin{align}
		\begin{split}
			\mathsf{H}_p a + w^{-1} p_1 a  &= (-\delta \chi_\digamma(\mathsf{p})^4 \chi_\digamma(\varrho_+)^4 - b^2 + e^2+h) \rho_{\mathrm{df}}^{m'} \rho_{\mathrm{bf}}^{\mathsf{s}} \rho_{\natural\mathrm{f}}^{\ell'} \\ 
			H_p a + p_1 a  &= 2(-\delta \chi_\digamma(\mathsf{p})^4 \chi_\digamma(\varrho_+)^4 - b^2 + e^2+h) w \rho_{\mathrm{df}}^{m'} \rho_{\mathrm{bf}}^{\mathsf{s}'} \rho_{\natural\mathrm{f}}^{\ell'},
		\end{split}
	\end{align}
	where the symbols are defined as before with the exception of $b$, which is now defined, for $\digamma$ sufficiently large and $\delta$ small, to be
	\begin{equation}\label{eq:newb}
			b =   \chi_\digamma(\mathsf{p})\chi_\digamma(\varrho_+)\sqrt{-\alpha - \delta \chi_\digamma(\mathsf{p})^2 \chi_\digamma(\varrho_+)^2  - 2^{-1}w^{-1} p_1}.
	\end{equation} 
	Quantizing as in \eqref{eq:quantizing}, we get, for $B_0 = \operatorname{Op}( \sqrt{w \rho_{\mathrm{df}}^{m'}\rho_{\mathrm{bf}}^{\mathsf{s}'}\rho_{\natural\mathrm{f}}^{\ell'}} b)$ with $b$ as in 
	\eqref{eq:newb}  and $R \in \Psi_{\calc}^{-m',-2-\mathsf{s}',-\ell',0}$,
	\begin{equation}
		i( [A,P_\pm] + (P_\pm-P_\pm^*) A ) = -\delta A \Lambda^* \Lambda A  - B_0^*B_0 + E_0^*E_0  + H + R.
		\label{eq:misc_oo1}
	\end{equation}
	The estimates of $\lVert B_0 u \rVert_{L^2}$, if $u$ is Schwartz, can now proceed as before, except, because the $E_0^2$ term in \cref{eq:misc_oo1} has the opposite sign of the $B_0^2$ term, all of our estimates will have an extra $\lVert E_0 u \rVert_{L^2}$ term on the right-hand side. But by construction, $\operatorname{WF}'_{\calc}(E_0)$ is contained in $\{ \epsilon \leq \varrho_+ \leq 2\epsilon \}$ (using the support properties of $\chi_\digamma'$ and our choice $\digamma = 1/(2\epsilon)$), so on $\operatorname{WF}'_{\calc}(E_0)$, either $E$ is elliptic (where $\mathsf{p} = 0$) or $GP$ is elliptic (where $\mathsf{p} \neq 0$).  Hence we have an elliptic estimate 
	\begin{equation}
		\lVert E_0 u \rVert_{L^2}  \lesssim \lVert G P_\pm u \rVert_{H_{\calc}^{m-2,\mathsf{s},\ell-2,0}}  + \lVert E u \rVert_{H_{\calc}^{m,\mathsf{s},\ell,0}}  + \lVert u \rVert_{H_{\calc}^{-N,-N,-N,0} }. 
	\end{equation}
	The estimate on $\lVert B_0 u \rVert_{L^2}$ then implies the desired estimate \cref{eq:misc_j57}, at least for $u$ decaying sufficiently at $\mathrm{bf}$, since $B_0$ is elliptic on $\operatorname{WF}'_{\calc}(B)$.

	To prove the estimate for general $u$ requires a regularization argument, as before. Now though, an arbitrarily high amount of regularization, both in the sense of decay and regularity, is consistent with the threshold condition $\mathsf{s}|_{\calR_+}<-1/2$. That is, using the notation from above, we can take $K,K'$ arbitrarily large. So, for each $N\in \bbR$, we can take $K,K'$ large enough so that the estimates go through for all 
	\begin{equation}
		u(h)\in H_{\mathrm{sc}}^{-N,-N}.
		\label{eq:misc_319}
	\end{equation}
	On the other hand, if \cref{eq:misc_319} is not valid then \cref{eq:misc_j57} holds vacuously since then the right-hand side is infinite.  
	\end{proof}

Combining the above-threshold estimates on $\calR_{-\varsigma}$ with the below-threshold estimates on $\calR_{\varsigma}$, and using the global propagation estimate \Cref{prop:global_propagation} in between, we get:
\begin{proposition}
	Fix $\varsigma \in \{-,+\}$. Suppose that $m,\ell,q\in \bbR$ and that $\mathsf{s}\in C^\infty({}^{\calc}\overline{T}^* \bbM)$ is monotonic under the Hamiltonian flow near $\Sigma$ and satisfies
	\begin{equation} 
		\mathsf{s}|_{\calR_\varsigma} > -1/2, \qquad \mathsf{s}|_{\calR_{-\varsigma}} < -1/2,
		\label{eq:both_thresholds}
	\end{equation}
	and in addition satisfies our two additional technical conditions: $\mathsf{s}$ is constant near each radial set and such that $|\mathsf{H}_p \mathsf{s}|^{1/2}$ is smooth near each component of the characteristic set.  
	Let $G,B,Q,Z\in \Psi^{0,0,0,0}_{\calc}$ satisfy
	\begin{itemize}
		\item $\Sigma\subset \operatorname{Ell}_{\calc}^{0,0,0,0}(G)$, i.e.\ $G$ is elliptic on the good sheet of the characteristic set,  
		\item $\operatorname{WF}'_{\calc}(B) \backslash \operatorname{Ell}_{\calc}^{0,0,0,0}(G) \subseteq \operatorname{Ell}_{\calc}^{0,0,0,0}(Z)$, i.e.\ the portion of the essential support of $B$ not already contained in the elliptic set of $G$ is covered by the elliptic set of $Z$,
		\item $\Sigma_{\mathrm{bad}} \cap \operatorname{WF}'_{\calc}(B)=\varnothing$, i.e.\ the essential support of $B$ does not intersect the ``bad'' component of the characteristic set,  
		\item $\operatorname{Ell}_{\calc}^{0,0,0,0}(Q)\supseteq \calR_{\varsigma}$, i.e.\ $Q$ is elliptic on the radial set $\calR_\varsigma$.  
	\end{itemize}
	Then, for every $N,\in \bbR$ and $s_0>-1/2$, the estimate 
	\begin{multline} \label{eq:symbolic_global}
			\lVert B u \rVert_{H_{\calc}^{m,\mathsf{s},\ell,q}}   \lesssim \lVert G P_\pm u \rVert_{H_{\calc}^{m-1,\mathsf{s}+1,\ell-1,q}} + \lVert Z P_\pm u \rVert_{H_{\calc}^{m-2,\mathsf{s},\ell-2,q}}  \\ + \lVert Q u \rVert_{ H_{\calc}^{-N,s_0,-N,q}}  + \lVert u \rVert_{H_{\calc}^{-N,-N,-N,q} } 
	\end{multline}
	holds, in the usual strong sense.
	\label{prop:symbolic_global}
\end{proposition}

In the previous estimate and in the above-threshold estimate, $Qu$ is playing the mostly qualitative role of enforcing the lack of above-threshold wavefront set at the radial set in question, but this term also entered quantitatively via \cref{eq:misc_273}. Having terms involving $u$ (but not through $P_\pm u$) on the right-hand side makes these estimates harder to use. So, we note:

\begin{proposition}
	The estimates \cref{eq:misc_j54} in \Cref{prop:propagation_out} and  \cref{eq:symbolic_global} in \Cref{prop:symbolic_global} continue to hold without the $Qu$ term on the right-hand side, but only for those $u$ and $h$ for which the left-hand side is known to be finite. Moreover, if $N$ is sufficiently large and $\mathsf{s}>s_0$ on the essential support of $Q$, then
	\begin{equation}
		\lVert Q u \rVert_{ H_{\calc}^{-N,s_0,-N,q}} \lesssim \lVert G P_\pm u \rVert_{H_{\calc}^{m-1,\mathsf{s}+1,\ell-1,q}} + \lVert Z P_\pm u \rVert_{H_{\calc}^{m-2,\mathsf{s},\ell-2,q}} + \lVert u \rVert_{H_{\calc}^{-N,-N,-N,q} } 
		\label{eq:misc_318}
	\end{equation}
	holds for such $u\in \calS'$ and $h$. 
	\label{prop:symbolic_global_better}
\end{proposition}
\begin{proof}
	It suffices to consider the case where $N$ is sufficiently large, $\mathsf{s}>s_0$ on the essential support of $Q$, and $B$ in \Cref{prop:propagation_out} or \Cref{prop:symbolic_global} is elliptic on the essential support of $Q$. (Indeed, one can use the propagation estimate \cref{prop:global_propagation} to shrink the essential support of $Q$ to an arbitrarily small neighbourhood of the above-threshold radial set.)  Then, we prove the following interpolation bound: 
	\begin{equation}
		\lVert Qu \rVert_{H_{\calc}^{-N,s_0,-N,q}} \leq C_0 \varepsilon \lVert Bu \rVert_{H_{\calc}^{m,\mathsf{s},\ell,q}} + C(\varepsilon) \lVert u \rVert_{H_{\calc}^{-N,-N,-N,q}}, 
		\label{eq:interpolation}
	\end{equation}
	which, for every $\varepsilon>0$, holds for some $C(\varepsilon)$ and $C_0>0$ independent of $\varepsilon$.

    To prove this, we choose an operator $O \in \Psi_{\calc}^{0,0,0,0}$  that is the identity microlocally on the essential support of $Q$ (i.e.\ such that the essential support of $O-1$ is disjoint from the essential support of $Q$) and such that  $\mathsf{s}>s_0$ on its essential support.  
	Then, $OB$ is elliptic on the essential support of $Q$, and we have
    \begin{equation} \label{eq:interpolation-norm0}
\lVert Qu \rVert_{H_{\calc}^{-N,s_0,-N,q}} \leq C (\lVert OBu \rVert_{H_{\calc}^{-N,s_0,-N,q}} +  \lVert u \rVert_{H_{\calc}^{-N,-N,-N,q}})
    \end{equation}
by the elliptic estimate in \Cref{prop:elliptic}. In addition, when $\min_{\calO} \mathsf{s}>s_0>-N$ (where $\calO$ is the essential support of $o$), we have the inequality
\begin{equation} \label{eq:interpolation-symbol1}
\rho_{\mathsf{bf}}^{-s_0} \leq \epsilon \rho_{\mathrm{bf}}^{-\mathsf{s}} + C(\epsilon) \rho_{\mathrm{bf}}^{N} \text{ \; on \; } \calO.
\end{equation}
Then this shows
\begin{equation} \label{eq:interpolation-norm1}
\lVert OBu \rVert_{H_{\calc}^{-N,s_0,-N,q}} \leq \epsilon \lVert OBu \rVert_{H_{\calc}^{-N,\mathsf{s},-N,q}} + C(\epsilon) \lVert OBu \rVert_{H_{\calc}^{-N,-N,-N,q}}.
\end{equation}
Substituting \cref{eq:interpolation-norm1} into \cref{eq:interpolation-norm0},  bounding the $OBu$-norm by $C_0$ times the same norm of $Bu$, and bounding $\lVert OBu \rVert_{H_{\calc}^{-N,-N,-N,q}}$ by $C \lVert u \rVert_{H_{\calc}^{-N,-N,-N,q}}$,  we obtain \cref{eq:interpolation}.

 Finally, using \cref{eq:interpolation} to replace the $Qu$ terms in the estimates \cref{eq:misc_j54} or \cref{eq:symbolic_global}  
     with a $Bu$ term, we can absorb the $Bu$ term into the left-hand side of the estimate if $\varepsilon$ is sufficiently small, at least for those $u$ and $h$ such that
	\begin{equation}
		\lVert Bu (h)\rVert_{H_{\calc}^{m,\mathsf{s},\ell,q}} <\infty.
	\end{equation}
	So, we conclude that the estimates in question hold without the $Qu$ terms if the left-hand side is finite. Plugging in the resulting estimate to \cref{eq:interpolation}, we get \cref{eq:misc_318}.
\end{proof}

\subsection{Estimates for the remainder}
\label{subsec:remainder}

Our next goal is to improve upon the estimate in \Cref{prop:symbolic_global} to get an estimate, for each $N\in \bbR$, such that $u$ appears on the right-hand side (not counting through $P_\pm u$) measured only in the very weak Sobolev norm
\begin{equation} 
	\| u \|_{H_{\calczero}^{-N,-N,-N}} . 
\end{equation} 
In order to do this, what we need to estimate is the term 
\begin{equation}
	\lVert u \rVert_{H_{\calc}^{-N,-N,-N,q} } 
	\label{eq:misc_315}
\end{equation}
appearing in \Cref{prop:symbolic_global}, \Cref{prop:symbolic_estimate_for_Cauchy}.
We want to bring in the main results of \cite{Parabolicsc} to estimate this. The reason is that, since the norms in \cref{eq:misc_315} are already negligibly low order at every boundary hypersurface of our phase space except $\mathrm{pf}$, 
and because 
\begin{equation}
	N(P_\pm)\approx P_\pm
\end{equation}
at $\mathrm{pf}$, we can interchange the two operators when estimating the norms in \cref{eq:misc_315}. Since $N(P_\pm)$ is the operator in the Schr\"odinger equation, the Fredholm estimates for the Schr\"odinger equation developed in \cite{Parabolicsc} allow us to bound $u$ in terms of $N(P_\pm)u$ and therefore $P_\pm u$ (modulo terms which are suppressed as $h\to 0^+$), which is what we want.

First, we rephrase the estimates in \cite{Parabolicsc} in terms of the $\calc$-Sobolev norms. We recall the notation $\calR_{\varsigma}^{\mathrm{Schr}}$ defined in  \eqref{eq:Rschr}. 

\begin{lemma}
	Let $m,\ell,q\in \bbR$, and let $\mathsf{s}\in C^\infty({}^{\calc}\overline{T}^* \bbM )$ be monotonic under the Hamiltonian flow and satisfy
	\begin{equation}
		\mathsf{s}|_{\calR_{-\varsigma}\cap \mathrm{pf}}<-1/2 < \mathsf{s}|_{\calR_{\varsigma}\cap\mathrm{pf}}.
	\end{equation} 
	Let $\Pi \in \Psi_{\calczero}^{0,0,0}$. 
	For any $\delta>0$ and $s_0>-1/2$, there exist $O_1 ,O_2 \in \Psi_{\calczero}^{-\infty,0,0}$ satisfying the conditions of  \Cref{prop:par_Planck_comp_full} (with $\epsilon$ there depending on $\delta,s_0$), 
such that for each $N\in \bbR$, the estimate
	\begin{equation}
		\lVert O_1  \Pi u \rVert_{H_{\calc}^{m,\mathsf{s},\ell,q}} \lesssim \lVert N(P_\pm ) O_2  \Pi u \rVert_{H_{\calc}^{-N,\mathsf{s}+1+\delta,\ell-1,q} } + \lVert u \rVert_{H_{\calczero}^{-N,-N,-N} }
		\label{eq:misc_317}
	\end{equation}
	holds for all $u\in \calS'$ and $h>0$ such that  $\operatorname{WF}_{\mathrm{par}}^{\ell',s_0}(\Pi u(h) ) \cap \calR_{-\varsigma}^{\mathrm{Schr}} = \varnothing$. 
	\label{lem:Schrodinger_input_rephrased}
\end{lemma}
The reader may take $\Pi=1$ if they like. We have kept it in to emphasize the microlocal nature of the estimate.
\begin{proof}
We first show that we can find variable orders $\overline{\mathsf{s}} \in C^\infty({}^{\mathrm{par}}\overline{T}^* \bbM )$ and $\tilde{\mathsf{s}} \in C^\infty({}^{\calc} \overline{T}^* \bbM)$ such that 
	\begin{equation}
		\overline{\mathsf{s}}|_{\calR_{-\varsigma}^{\mathrm{Schr}}} < -1/2 < \overline{\mathsf{s}}|_{\calR_\varsigma^{\mathrm{Schr}}} 
	\end{equation}
	and 
	\begin{itemize}
		\item $\overline{\mathsf{s}}$ is monotonic under the Hamiltonian flow,
		\item $\mathsf{s} \leq \overline{\mathsf{s}} \leq \tilde{\mathsf{s}}$ in some neighborhood $U$ of $\mathrm{pf}$ that is disjoint with $\mathrm{df}$. 
		Here we interpret $\overline{\mathsf{s}}$ as a smooth function on ${}^{\mathrm{par,I}} \overline{T}^* \bbM$ by extending it to be constant in $h$, and then lifting it to a smooth function on ${}^{\mathrm{par,I,res}} \overline{T}^* \bbM$ via pullback by the blowdown map $\beta^*$. Recall from Lemma~\ref{lem:phase_space_gluing} that ${}^{\mathrm{par,I,res}} \overline{T}^* \bbM$ is locally identical to  ${}^{\calc} \overline{T}^* \bbM$ over $U$.
		\item  $|\mathsf{s}-\tilde{\mathsf{s}}|<\delta$.
	\end{itemize}
We define
\begin{equation}
\epsilon_0 = \min_{\calR} |\mathsf{s}+\frac{1}{2}|.
\end{equation}	
Now take $\epsilon \in (0,\frac{ \min\{\epsilon_0,\delta\}}{3})$, and set
\begin{equation}
\overline{\mathsf{s}} = \mathsf{s}|_{\mathrm{pf}}+\epsilon.
\end{equation}
Recalling that $\mathrm{pf}$ is diffeomorphic to ${}^{\mathrm{par}}\overline{T}^* \bbM$ canonically, $\overline{\mathsf{s}}$ can be viewed as a smooth function on ${}^{\mathrm{par}}\overline{T}^* \bbM$ and then extended as above to a smooth function on $U$. 
We choose a local coordinate system $(\rho_{\mathrm{pf}},\mathsf{X})$ on $U$, where $\mathsf{X}$ are constant on the fibers of the blowdown map $\beta$. Thus $\overline{\mathsf{s}}(\rho_{\mathrm{pf}},\mathsf{X}) = \mathsf{s}(0,\mathsf{X}) + \epsilon$ in this coordinate system. 
By the smoothness of $\mathsf{s}$, we can take $U$ so small that 
\begin{equation}
\mathsf{s}(0,\mathsf{X}) - \epsilon <  \mathsf{s}(\rho_{\mathrm{pf}},\mathsf{X})< \mathsf{s}(0,\mathsf{X}) +\epsilon   \;  \text{ on } \; U,
\end{equation}
then taking $\tilde{\mathsf{s}}=\mathsf{s}+2\epsilon$, the inequality $\mathsf{s} \leq \overline{\mathsf{s}} \leq \tilde{\mathsf{s}}$ is satisfied in $U$. The other properties are satisfied by the choice of $\epsilon$ and the monotonicity of $\mathsf{s}$ itself along the Hamiltonian flow.

	Since $\overline{\mathsf{s}}$ satisfies the hypotheses required in \cite[Thm.\ 1.1]{Parabolicsc}, combining the adaptation in \Cref{sec:schrodinger_estimate}, for each $\ell'\in \bbR$ and 
    \begin{equation}  \label{eq: cal X definition, schrodinger}
	v \in \calX^{\ell',\overline{\mathsf{s}}}_{\mathrm{Schr}} := \{v\in H_{\mathrm{par}}^{\ell',\overline{\mathsf{s}}} : N(P_\pm)v \in H_{\mathrm{par}}^{\ell'-1,\overline{\mathsf{s}}+1} \}, 
    \end{equation}
	we have
	\begin{equation}
		\lVert v \rVert_{H_{\mathrm{par}}^{\ell',\overline{\mathsf{s}}}} \lesssim \lVert N(P_\pm)v\rVert_{H_{\mathrm{par}}^{\ell'-1,\overline{\mathsf{s}}+1}},
		\label{eq:par_estimate}
	\end{equation}
	where the constant does not depend on $v$. 
	Actually, what is proven in \cite[\S6]{Parabolicsc} is that \cref{eq:par_estimate} holds in the sense that if the right-hand side is finite, and if  
	\begin{equation}
		\operatorname{WF}_{\mathrm{par}}^{\ell',s_0 }(v) \cap \calR_{-\varsigma}^{\mathrm{Schr}} = \varnothing
		\label{eq:par_threshold_cond}
	\end{equation}
	holds for some $s_0>-1/2$, then the left-hand side is finite and the inequality holds.  
	Indeed, if $ N(P_\pm)v \notin H_{\mathrm{par}}^{\ell'-1,\overline{\mathsf{s}}+1}$, then the right-hand side of \cref{eq:par_estimate} is infinite, so the estimate holds trivially, and otherwise \eqref{eq:par_threshold_cond} and the various radial point and propagation estimates in \cite{Parabolicsc} show that $u\in \smash{H_{\mathrm{par}}^{\ell',\overline{\mathsf{s}}}}$ (which means $u$ is in $\smash{\calX^{\ell',\overline{\mathsf{s}}}_{\mathrm{Schr}}}$ defined in \eqref{eq: cal X definition, schrodinger}) and the inequality \cref{eq:par_estimate} holds. 
	
	Let $O_j$ be as in \Cref{prop:par_Planck_comp_full}. For each individual $h>0$, $O_j\Pi(h) $ lies in $\Psi_{\mathrm{sc}}^{-\infty,0}=\Psi_{\mathrm{par}}^{-\infty,0}$, since $\operatorname{WF}'_{\calczero}(O_j)$ is disjoint from $\mathrm{df}$. So, the incoming/outgoing condition for $\Pi u$ implies the same for $O_2\Pi u(h)$: 
	\begin{equation} 
		\operatorname{WF}_{\mathrm{par}}^{-N,s_0 }(O_2\Pi  u) \cap \calR_{-\varsigma}^{\mathrm{Schr}} = \varnothing.
	\end{equation} 

	So, if the parameter $\varepsilon>0$ in the definition of $O_1,O_2,O_3$ is sufficiently small (where what ``sufficiently small'' means depends on $U$), then by \Cref{prop:par_Planck_comp_full} we have
	\begin{align*}
		\begin{split}
			h^q\lVert O_1 \Pi u \rVert_{H_{\calc}^{m,\mathsf{s},\ell,q} }  &\lesssim \lVert O_2 \Pi u \rVert_{H_{\mathrm{par}}^{\ell-q,\overline{\mathsf{s}}} }  + \lVert   u \rVert_{H_{\calczero}^{-N,-N,-N} }  \\ &\lesssim \lVert N(P_\pm) O_2 \Pi u \rVert_{H_{\mathrm{par}}^{\ell-q-1,\overline{\mathsf{s}}+1 } } + \lVert   u \rVert_{H_{\calczero}^{-N,-N,-N} } \\
			\leq \lVert O_3 &N(P_\pm) O_2 \Pi u \rVert_{H_{\mathrm{par}}^{\ell-q-1,\overline{\mathsf{s}}+1 } } + \lVert (1-O_3)N(P_\pm) O_2\Pi u \rVert_{H_{\mathrm{par}}^{\ell-q-1,\overline{\mathsf{s}}+1 } }  + \lVert   u \rVert_{H_{\calczero}^{-N,-N,-N} }
		\end{split}
	\end{align*} 
	and
	\begin{align}
		\begin{split} 
		\lVert O_3 N(P_\pm) O_2\Pi u \rVert_{H_{\mathrm{par}}^{\ell-q-1,\overline{\mathsf{s}}+1 } }  
		&\lesssim h^q \lVert N(P_\pm ) O_2\Pi u \rVert_{H_{\calc}^{-N,\mathsf{s}+\delta+1,\ell-1,q} } + \lVert N(P_\pm )O_2\Pi  u \rVert_{H_{\calczero}^{-N,-N,-N} }.
		\end{split} 
	\end{align}
	Also, choosing elliptic, invertible $\Lambda \in \Psi_{\mathrm{par}}^{\ell-q-1,\overline{\mathsf{s}}+1}$, 
	\begin{equation}
		\lVert (1-O_3)N(P_\pm) O_2\Pi u \rVert_{H_{\mathrm{par}}^{\ell-q-1,\overline{\mathsf{s}}+1 } } \lesssim \lVert \Lambda (1-O_3)N(P_\pm) O_2\Pi u \rVert_{L^2} \lesssim \lVert u \rVert_{H_{\calc}^{-N,-N,-N}}, 
	\end{equation}
	since $(1-O_3)N(P_\pm) O_2 \in \Psi_{\calczero}^{-\infty,-\infty,-\infty}$ and therefore $\Lambda (1-O_3)N(P_\pm) O_2 \in \Psi_{\calczero}^{-\infty,-\infty,-\infty}$.
	
	Combining these inequalities establishes  \cref{eq:misc_317}. 
\end{proof}

We now apply this to bound $\lVert u \rVert_{H_{\calc}^{-N,-N,-N,q}}$. Initially, we make an additional assumption on our variable order $\mathsf{s}$, namely, that it is above $+1/2$ at the above threshold radial set $\calR_{\varsigma}$ (that is, more than $1$ over the threshold value of $-1/2$). This extra assumption will be eliminated in Proposition~\ref{prop:error_bounding2} (at the cost of a somewhat weaker estimate, but one that still suffices for our purposes). 

\begin{proposition}
	Consider the setup of \Cref{prop:symbolic_global}, and choose $N$ large enough such that $\mathsf{s}>-N$. In addition, we assume that $\mathsf{s} > 1/2$ in a neighbourhood of the above-threshold radial set $\calR_{\varsigma}$, and  we assume we are given some 
	\begin{equation} 
		\Pi,O\in \Psi_{\calczero}^{0,0,0}
	\end{equation} 
	such that $\operatorname{WF}'_{\calczero}(1-O),\operatorname{WF}'_{\calczero}(1-\Pi)$ are both disjoint from the zero section of ${}^{\calczero}\overline{T}^* \bbM$. 
	Then, 
	\begin{equation}
		\lVert \Pi u \rVert_{H_{\calc}^{-N,-N,-N,q} } \lesssim \lVert P_\pm  u \rVert_{H_{\calc}^{-N-1,\mathsf{s},-N-1,q} } + h\lVert  Ou \rVert_{H_{\calc}^{-N, \mathsf{s},-N+2,q  }} + \lVert u \rVert_{H_{\calczero}^{-N,-N,-N}}
		\label{eq:misc_3p3}
	\end{equation}
	holds, in the usual strong sense, for all $u$ and $h$ satisfying the incoming/outgoing condition 
\begin{equation}	
	\operatorname{WF}_{\mathrm{par}}^{-N,s_0 }(u(h)) \cap \calR_{-\varsigma}^{\mathrm{Schr}} = \varnothing.
\label{eq:in-out condition}\end{equation}
	\label{prop:error_bounding}
\end{proposition}
Note that $\Pi=1$ is allowed. 
\begin{proof}

We choose $\delta > 0$ sufficiently small so that $\mathsf{s}_1 := \mathsf{s} - 1 - \delta$ is above threshold at $\calR_{\varsigma}$. 
	Now let $O_j$ be as in the proof of \Cref{lem:Schrodinger_input_rephrased} (which in turn, is using \Cref{prop:par_Planck_comp_full})
	with this choice of $\delta$. That proof just required that the parameter $\varepsilon$ in the definition of $O_j$ be sufficiently small, so we may take it smaller if required without changing the conclusions of that argument. 
	Also, we can choose each $O_j$ such that $\operatorname{WF}'_{\calczero}(1-O_j)$ is disjoint from the zero section of ${}^{\calczero}\overline{T}^* \bbM$, and $\operatorname{WF}'_{\calczero}(O_j) \cap \operatorname{WF}'_{\calczero}(1-O) = \varnothing$. 
	
Microlocally away from the zero section of ${}^{\calczero}\overline{T}^* \bbM$, the left-hand side is controlled by the term $ \lVert u \rVert_{H_{\calczero}^{-N,-N,-N}}$. Thus, we may assume that $\operatorname{WF}'_{\calczero}(\Pi)$ is as close to the zero section as we wish. 

In particular,  it suffices to consider the case where $\operatorname{WF}'_{\calczero}(\Pi)\subseteq \operatorname{Ell}_{\calczero}^{0,0,0,0}(O_1)$. Then, applying \cref{eq:micro-elliptic-est}, we have 
\begin{equation}
\lVert \Pi u \rVert_{H_{\calc}^{-N,-N,-N,q} }\lesssim	\lVert O_1 u \rVert_{H_{\calc}^{-N,-N,-N,q} } + 	\lVert u \rVert_{H_{\calczero}^{-N,-N,-N}}.
\end{equation}

	The previous lemma (with  $\Pi = \Id$)  says 
	\begin{multline}
		\lVert O_1 u \rVert_{H_{\calc}^{-N,\mathsf{s}_1,-N,q}} \lesssim \lVert N(P_\pm ) O_2 u \rVert_{H_{\calc}^{-N,\mathsf{s}_1+1+\delta,-N-1,q} } + \lVert u \rVert_{H_{\calc}^{-N,-N,-N} } \\
		\lesssim \lVert N(P_\pm ) O_2 u \rVert_{H_{\calc}^{-N,\mathsf{s},-N-1,q} } + \lVert u \rVert_{H_{\calc}^{-N,-N,-N} }.
	\end{multline} 
	The first term on the right-hand side admits the bound
	\begin{multline}
		\lVert N(P_\pm) O_2  u \rVert_{H_{\calc}^{-N,\mathsf{s},-N-1,q} } \leq  \lVert O_2  P_\pm  u \rVert_{H_{\calc}^{-N,\mathsf{s},-N-1,q} } +\lVert [P_\pm,O_2]u \rVert_{H_{\calc}^{-N,\mathsf{s},-N-1,q} } \\ +\lVert (P_\pm-N(P_\pm))O_2 u \rVert_{H_{\calc}^{-N,\mathsf{s},-N-1,q} } .
	\end{multline}
	Let us examine each term on the right-hand side. First,
	\begin{equation}
			\lVert O_2  P_\pm  u \rVert_{H_{\calc}^{-N,\mathsf{s},-N-1,q} }  \lesssim 	\lVert P_\pm  u \rVert_{H_{\calc}^{-N,\mathsf{s},-N-1,q} }.
	\end{equation}
	Second, since $\operatorname{WF}_{\calczero}'([P_\pm,O_2]) \subseteq \operatorname{WF}_{\calczero}'(O_2) \cap \operatorname{WF}_{\calczero}'(1-O_2)$,  we have
		\begin{equation}
			[P_\pm,O_2]  \in \Psi_{\calczero}^{-\infty,-1,1}, 
		\end{equation}
		with essential support disjoint from the zero section at $h=0$.  
		Consequently, 
		\begin{multline}
			\lVert [P_\pm,O_2 ]u \rVert_{H_{\calc}^{-N,\mathsf{s},-N-1,q} }  \lesssim \lVert [P_\pm,O_2 ] u \rVert_{H_{\calc}^{-N,\mathsf{s},-N-1,q-1}}+ \lVert u \rVert_{H_{\calc}^{-N,-N,-N} } \\ \lesssim h\lVert Ou \rVert_{H_{\calc}^{-N,\mathsf{	s}-1,-N,q}}+ \lVert u \rVert_{H_{\calc}^{-N,-N,-N} }  .
		\end{multline}
		Finally, since $P_\pm - N(P_\pm) \in \Psi_{\calc}^{2,0,2,-1}$,
		\begin{multline}
			\lVert (P_\pm-N(P_\pm))O_2  u \rVert_{H_{\calc}^{-N,\mathsf{s},-N-1,q} } \lesssim \lVert O_2 u \rVert_{H_{\calc}^{-N+2,\mathsf{s},-N+1,q-1} } \\ \lesssim h\lVert Ou \rVert_{H_{\calc}^{-N,\mathsf{s},-N+2,q} } +\lVert u \rVert_{H_{\calczero}^{-N,-N,-N}}.
		\end{multline}
Putting these together yields \cref{eq:misc_3p3}. 
\end{proof}

In the next proposition, we improve Proposition~\ref{prop:error_bounding} by removing the extra assumption $\mathsf{s} > 1/2$ on the spacetime order at the above-threshold radial set, at the cost of weakening the factor of $h$ in front of the $Ou$ term to $h^\epsilon$ and removing the microlocalizer $O$. 

\begin{proposition}\label{prop:error_bounding2}
	Consider the setup of Proposition~\ref{prop:error_bounding}, except now let $\mathsf{s}$ be an arbitrary variable order satisfying the threshold and monotonicity conditions. We assume $N$ is large enough so that $\mathsf{s} > -N+1$.  Suppose that $\epsilon \in (0,1)$ is such that $\mathsf{s}-2\epsilon$ also satisfies the threshold conditions.
	Then, 
	\begin{equation}
	\lVert u \rVert_{H_{\calc}^{-N,-N,-N,q} } \lesssim \lVert P_\pm  u \rVert_{H_{\calc}^{-N,\mathsf{s}+1,-N,q} } + h^\epsilon\lVert  u \rVert_{H_{\calc}^{-N+1, \mathsf{s},-N+3,q  }} 
	\label{eq:misc_3p4}
	\end{equation}
	holds, in the usual strong sense, for all $u$ and $h$ satisfying the incoming/outgoing condition \eqref{eq:in-out condition}. 
	\label{prop:error_bounding_improved}
\end{proposition} 
\begin{proof}
The idea is to use a multiplier that will allow us to `trade' a power of $h$ with some extra spacetime decay on the right-hand side, using the fact that we can be profligate with spacetime decay since the norm is so weak (in the spacetime sense) on the left-hand side.  

We define 
\begin{equation} 
\tilde{u} = a u, \quad a=(1+h^2 |z|^2)^{-1/2} \in S^{0,0,0}_{\calczero}, \quad z = (t, x). 
\end{equation} 
Notice that $a$ is not a classical symbol: for each $h > 0$ it decays as $\ang{z}^{-1}$, but this is non-uniform and the optimum order as a $\calczero$-symbol is $(0,0,0)$, while  $a^{-1} \in S^{0,1,0}_{\calczero}$. We can view $a$, or $a^{-1}$, as a $\calc$-operator; however, there is a slight technicality that $a$ is \emph{not} in the $\Psi_{\calc}$ calculus as we have defined it, since it is not classical at the parabolic face. However, it \emph{is} in the sum of the natural calculus $\Psi_{\calczero}$ and the calculus $\tilde \Psi_{\mathrm{par,I,res}}^{m,\mathsf{s},\ell,q}$ defined in Section~\ref{subsec:parIres}, since neither of those calculi require more than conormal behaviour of the symbols at the boundary. We can use boundedness property \eqref{eq:action on par Sob spaces} in the special case $m = 0, \mathsf{s} = 0, \ell = 0, q =0$ to see that merely `conormal' operators $A_{\mathrm{con}}$ of order $(0,0,0,0)$, i.e.\ with symbols that are only conormal at each boundary hypersurface, map $L^2 \to L^2$ with a uniformly bounded norm as $h \to 0$. 
Then, by conjugating $A_{\mathrm{con}}$ with an elliptic element $\Lambda_{m, \mathsf{s}, \ell}$ of $\Psi_{\calc}^{m, \mathsf{s}, \ell,0}$, decomposing into the two constituent calculi $\Psi_{\calczero}$ and $\tilde \Psi_{\mathrm{par,I,res}}^{m,\mathsf{s},\ell,q}$, and using the boundedness properties of each, we see that $A_{\mathrm{con}}$ maps $\smash{H_{\calc}^{m, \mathsf{s}, \ell, q}}$ boundedly to itself for all $(m, \mathsf{s}, \ell, q)$.

In particular, since $u = a^{-1} \tilde u$ and $\ang{z}^{-1} a^{-1} \in S^{0,0,0}_{\calczero}$, we have 
\begin{equation}\label{eq:uutilde}
		\lVert u \rVert_{H_{\calc}^{-N,-N,-N,q} }  \lesssim \lVert \tilde{u} \rVert_{ H_{\calc}^{-N,-N+1,-N,q}} \lesssim \lVert \tilde{u} \rVert_{ H_{\calc}^{-N+1,-N+1,-N+1,q}}.
	\end{equation}
Thus it suffices to bound the right-hand side of \cref{eq:uutilde}.

To do this, we apply Proposition~\ref{prop:error_bounding} to $\tilde u$ and with the spacetime order $\mathsf{s}' = \mathsf{s} + 1 - \epsilon$. We obtain 
	\begin{equation}
		\lVert \tilde{u} \rVert_{ H_{\calc}^{-N+1,-N+1,-N+1,q}} \lesssim  \lVert P_\pm  \tilde{u} \rVert_{H_{\calc}^{-N,\mathsf{s}+1 - \epsilon,-N,q} } + h\lVert  \tilde{u} \rVert_{H_{\calc}^{-N+1, \mathsf{s}+1 - \epsilon,-N+3,q  }} + \lVert \tilde{u} \rVert_{H_{\calczero}^{-N+1,-N+1,-N+1}}. 
\label{eq:tilde u P}	\end{equation}
	Let us analyze each term on the right-hand side of \eqref{eq:tilde u P}. First, 
		\begin{align}
		\begin{split} 
		\lVert P_\pm  \tilde{u} \rVert_{H_{\calc}^{-N,\mathsf{s}+1 - \epsilon,-N,q} } &\leq \lVert a P_\pm  u \rVert_{H_{\calc}^{-N,\mathsf{s}+1 - \epsilon,-N,q} } + \lVert  [P_\pm ,a] u \rVert_{H_{\calc}^{-N,\mathsf{s}+1 - \epsilon,-N,q} } \\ 
		&\lesssim \lVert  P_\pm  u \rVert_{H_{\calc}^{-N,\mathsf{s}+1,-N,q} } + \lVert  [P_\pm ,a] u \rVert_{H_{\calc}^{-N,\mathsf{s}+1 - \epsilon,-N,q} }.
		\end{split} 
		\end{align}
		The first term on the right-hand side is a good term, appearing on the right-hand side of \cref{eq:misc_3p4}. 
		
	As for the second term, a computation shows that $[P_\pm ,a]$ is a first-order differential operator whose coefficients (when written in terms of $\partial_t,\partial_{x_j}$) are given by $ha$ times a symbol in $S^{0}(\overline{\mathbb{R}^{1,d}})$. 
		 Consequently, 
		\begin{equation}
		\lVert  [P_\pm ,a] u \rVert_{H_{\calc}^{-N,\mathsf{s}+1 - \epsilon,-N,q} } \leq  h \lVert  \tilde{u}  \rVert_{H_{\calc}^{-N+1,\mathsf{s}+1 - \epsilon,-N,q} }. 
		\end{equation}
		So, we can control this term provided we can control the next term on the right-hand side of \cref{eq:tilde u P}.

 To control $h\lVert  \tilde{u} \rVert_{H_{\calc}^{-N+1, \mathsf{s}+1 - \epsilon,-N+3,q  }} = \lVert  h a u \rVert_{H_{\calc}^{-N+1, \mathsf{s}+1 - \epsilon,-N+3,q  }}$, we can use the following property of $a$:  
		\begin{equation}
		h^{1-\epsilon} a \in S_{\calczero}^{0,-1+\epsilon,0} \text{ for all } \epsilon \in [0,1]. 
		\end{equation}
		(This is the exact place where `trading' between spacetime decay and decay in $h$ takes place.) 
		Thus, 
\begin{equation}
\lVert  h a u \rVert_{H_{\calc}^{-N+1, \mathsf{s}+1 - \epsilon,-N+3,q  }}\lesssim h^\epsilon \lVert u \rVert_{H_{\calctwo}^{-N+1, \mathsf{s},-N+3,q } }.
\label{eq:trading}\end{equation}
	
Finally, using the fact that $a \in S^{0,0,0}_{\calczero}$, the third term on the right-hand side of \cref{eq:tilde u P} is bounded by $\lVert u \rVert_{H_{\calczero}^{-N+1,-N+1,-N+1}}$. This term can be absorbed into the right-hand side of \cref{eq:trading}.  

Putting these estimates together yields \cref{eq:misc_3p4}. 
\end{proof}

\begin{remark} Another way to achieve the removal of the extra assumption $\mathsf{s} > 1/2$ at the above-threshold radial set in Proposition~\ref{prop:error_bounding} is to mimic the argument in \cite[Section 4.1]{VasyN0}, where a similar issue was encountered. 
\end{remark}

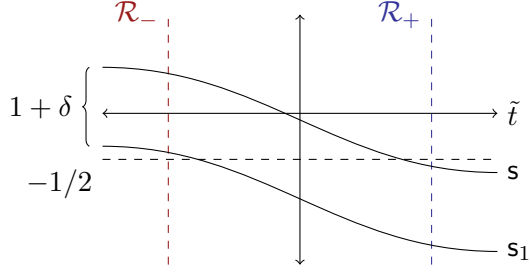
\begin{figure}
	\centering
	\begin{tikzpicture}[scale=1.75]
		\draw[<->] (0,-.9) -- (0,1);
		\draw[<->] (-1.5,.25) -- (1.5,.25) node[right] {$\tilde{t}$};
		\draw[darkred, dashed] (-1,-.9) -- (-1,1) node[left] {$\calR_-$};
		\draw[darkblue, dashed] (1,-.9) -- (1,1) node[left] {$\calR_+$};
		\draw[dashed] (-1.5,-.1) node[below left] {$-1/2$} -- (1.5,-.1);
		\draw (-1.5,.6) to[out=0, in=180] (1.5,-.2) node[right] {$\mathsf{s}$};
		\draw (-1.5,0) to[out=0, in=180] (1.5,-.8) node[right] {$\mathsf{s}_1$};
		\draw [decorate,
		decoration = {brace}] (-1.6,0) --  (-1.6,.6);
		\node[left] () at (-1.65,.3) {$1+\delta$};
	\end{tikzpicture}
	\caption{The variable orders $\mathsf{s}_1,\mathsf{s}$ involved in \Cref{prop:error_bounding} and its proof (in the $\varsigma=+$ case), as a function of a parameter $\tilde{t}$ parametrizing the Hamiltonian flow. As long as $-N< \inf \{\mathsf{s}\}$, then we can find such $\mathsf{s}_1$.}
	\label{fig:s}
\end{figure}

Before we use these estimates to improve \Cref{prop:symbolic_global}, we need to sort out the relation between the conditions 
\begin{itemize}
	\item $\operatorname{WF}_{\mathrm{par}}^{\ell,s_0}(u) \cap \calR_{-\varsigma}^{\mathrm{Schr}}=\varnothing$, 
	\item  $\operatorname{WF}_{\mathrm{sc}}^{\ell,s_0}(u) \cap \calR_{-\varsigma}(h)=\varnothing$.
\end{itemize}

Recall that the pullback map $\beta^*$ used in \Cref{lem:par_Planck_comp} is a diffeomorphism between $\mathrm{pf}$ and ${}^{\mathrm{par}}\overline{T}^* \bbM$, identifying the portions of our radial set $\calR_\varsigma$ at $\mathrm{pf}$ with the corresponding radial sets 
\begin{equation} 
	\calR_\varsigma^{\mathrm{Schr}} \subset {}^{\mathrm{par}}\overline{T}^* \bbM
\end{equation} 
for the Schr\"odinger equation. If $h$ is very small, then $\calR(h)$ is close to $\calR^{\mathrm{Schr}}$ at finite frequencies, so one should be able to propagate control between the two, if the forcing is sufficiently nice. So, it is natural to expect that the incoming/outgoing condition on $\calR_{-\varsigma}(h)$ (which is an `open' condition), phrased in terms of sc-wavefront set, implies the incoming/outgoing condition on $\calR_{-\varsigma}^{\mathrm{Schr}}$, phrased in terms of par-wavefront set. The incomparability of fiber infinity in the par-cotangent bundle with fiber infinity in the sc-cotangent bundle, depicted in Figure~\ref{fig:degen}, means that this cannot quite work as stated.
However, we do have:

\begin{lemma} Let $B\in \Psi_{\calc}^{0,0,0,0}$ be elliptic at $\calR_{-\varsigma}$, and let $\Pi \in \Psi_{\calc}^{-\infty,0,0,0}$ be such that 
	\begin{equation} 
		\operatorname{WF}'_{\calc}(\Pi )\Subset \{\rho_{\mathrm{pf}} < \varepsilon\}
	\end{equation} 
	for some $\varepsilon>0$. As long as $\varepsilon$ is sufficiently small (in terms of $B$), then, for all $m,\ell\in \bbR$, $s_0>-1/2$, 
	\begin{equation}
		 Bu(h) \in H_{\mathrm{sc}}^{m,s_0}  \Longrightarrow  \operatorname{WF}_{\mathrm{par}}^{\ell,s_0 }(\Pi u(h)) \cap \calR_{-\varsigma}^{\mathrm{Schr}} = \varnothing
	\end{equation}
	holds for all $h>0$. 
	\label{lem:automatic_threshold}
\end{lemma}
Note that $\Pi u(h)$ is smooth for each $h>0$, which is why the order $\ell$ does not matter. Indeed, $\Pi (h)\in \Psi_{\mathrm{par}}^{-\infty,0}$ for each $h>0$, even though $\Pi $, the element of $\Psi_{\calc}$, can be elliptic at points in $\natural\mathrm{f}$. The behavior at $\natural\mathrm{f}$ is irrelevant for individual  $h>0$. 
\begin{proof}	
	Let $V$ denote an open subset of ${}^{\calc}\overline{T}^* \bbM\backslash \mathrm{df}$ of the form $\beta^*V_0\backslash \mathrm{df}$ for $V_0$ an open neighborhood in ${}^{\mathrm{par}}\overline{T}^* \bbM$ of $\calR_{-\varsigma}^{\mathrm{Schr}}$. 
	Identifying $\mathrm{pf}$ with ${}^{\mathrm{par}}\overline{T}^* \bbM$, we can choose $V_0$ such that $\overline{V}_0 \Subset U \cap \mathrm{pf}$, where 
	\begin{equation} 
		U :=\operatorname{Ell}_{\calc}^{0,0,0,0}(B)
	\end{equation} 
	 is an open neighborhood of $\calR_{-\varsigma}$.  
	Let $W= \{\rho_{\mathrm{pf}} < \varepsilon\}$, $W_1=\{\rho_{\mathrm{pf}}<2\varepsilon\}$, and take $\varepsilon$ sufficiently small such that $W_1$, and therefore $W$, is disjoint from $\mathrm{df}$. 
	The key observation that we need to make is that if $\varepsilon$ is sufficiently small, then 
	\begin{equation}
		W_1 \cap V\Subset U,\label{eq:misc_349}
	\end{equation}
	which is true because $\overline{V}\cap \mathrm{pf} = \overline{V}_0\Subset U$, so $\overline{V}\cap \mathrm{pf}$ has a nonzero distance from $U^\complement$ with respect to any metric on ${}^{\calc}\overline{T}^* \bbM$. See \Cref{fig:degen}.

	\begin{figure}
		\begin{tikzpicture}[scale=.8]
			\draw[fill=lightgray!30] (0,0) circle (3);
			\draw[fill=lightgray!30] (0,0) circle (1.5);
			\fill[darkblue, fill opacity=.25] (-3,0) to[out=10,in=130] ({-1.5/1.44},{1.5/1.4+.3}) to[out=-50,in=180] (0,.3) to[out=0,in=230] ({1.5/1.44},{1.5/1.4+.3}) to[out=50, in=170] (3,0) to[out=190, in=50] ({1.5/1.44},{1.5/1.4-.3}) to[out=230, in=0] (0,-.3) to[out=180, in=-50] ({-1.5/1.44},{1.5/1.4-.3}) to[out=130, in=-10] (-3,0);
			\fill[darkgreen,opacity=.5] ({3.15/1.44-.375},{3.15/1.4+.13}) -- ({1.5/1.44},{1.5/1.4+.15}) to[out=235, in=0] (0,.1) to[out=180, in=-55] ({-1.5/1.44},{1.5/1.4+.15}) --  ({-3.15/1.44+.375},{3.15/1.4+.13}) to[out=218, in=50] ({-3.15/1.44},{3.15/1.4-.195}) -- ({-1.5/1.44},{1.5/1.4-.15}) to[out=-50, in=180] (0,-.1) to[out=0,in=230] ({1.5/1.44},{1.5/1.4-.15}) -- ({3.15/1.44},{3.1/1.4-.15}) -- ({3.15/1.44-.375},{3.15/1.4+.13}) -- ({1.5/1.44},{1.5/1.4+.15});
			\node[darkgreen, opacity=.5] at (-1.25,2.25) {$V$};
			\node[darkblue, opacity=.5] at (-2.5,.75) {$U$};
			\draw[darkblue] (-3,0) node[below right] {$\calR_{-\varsigma}$} to[out=0,in=130] ({-1.5/1.44},{1.5/1.4}) to[out=-50,in=180] (0,0) to[out=0,in=230] ({1.5/1.44},{1.5/1.4}) to[out=50, in=180] (3,0);
			\node at (0,-2.25) {$\natural\mathrm{f}$};
			\node at (0,-.8) {$\mathrm{pf}$};
			\node[above right] at ({3/1.44},{3/1.4}) {$\mathrm{df}_1$};
			\node at (-2.55,2.7) {(a)};
		\end{tikzpicture}
		\qquad 
		\begin{tikzpicture}[scale=.8]
			\draw[fill=lightgray!30] (0,0) circle (3);
			\draw[fill=lightgray!30] (0,0) circle (1.5);
			\fill[darkblue, fill opacity=.25] ({-3/1.44-.12},{3/1.4-.1}) -- ({-1.5/1.44},{1.5/1.4-.4}) to[out=-50,in=180] (0,-.3) to[out=0,in=230] ({1.5/1.44},{1.5/1.4-.4}) -- ({3/1.44+.12},{3/1.4-.1}) -- ({3/1.44-.1},{3/1.4+.1}) -- ({1.5/1.44},{1.5/1.4+.3}) to[out=220, in=0] (0,.3) to[in=-50, out=180] ({-1.5/1.44},{1.5/1.4+.3}) -- ({-3/1.44+.1},{3/1.4+.1}) -- cycle;
			\fill[darkgreen, opacity=.5] (0,3) to[out=210, in=130] ({-1.5/1.44-.1},{1.5/1.4-.1}) to[out=-50,in=180] (0,-.1) to[out=0,in=230] ({1.5/1.44+.1},{1.5/1.4-.1}) to[out=50, in=-30] (0,3) to[out=-70, in=50] ({1.5/1.44-.1},{1.5/1.4}) to[out=-130, in=0] (0,.1) to[out=180, in=-50] ({-1.5/1.44+.1},{1.5/1.4}) to[out=120, in=250] (0,3);
			\draw[darkblue] ({-3/1.44},{3/1.4}) -- ({-1.5/1.44},{1.5/1.4}) to[out=-50,in=180] (0,0) to[out=0,in=230] ({1.5/1.44},{1.5/1.4}) -- ({3/1.44},{3/1.4});
			\node at (0,-2.25) {$\natural\mathrm{f}$};
			\node at (0,-.8) {$\mathrm{pf}$};
			\node[above right] at ({3/1.44},{3/1.4}) {$\mathrm{df}$};
			\node at (-2.55,2.7) {(b)};
		\end{tikzpicture}
		\caption{
		The neighborhoods $U\supseteq \calR_{-\varsigma}$, $V\supseteq \mathrm{pf} \cap \calR_{-\varsigma}$ in \Cref{lem:automatic_threshold} and its proof, as they intersect $\mathrm{pf},\natural\mathrm{f}$ in ${}^{\mathrm{par,I,res}}\overline{T}^* \bbM$ (a) and ${}^{\calc}\overline{T}^* \bbM$ (b). The degrees-of-freedom associated with the base variables are not depicted. In the left figure, the angular figure in the diagram is the angular variable at fiber infinity in the par-cotangent bundle. In the right figure, it is just the angle of $\zeta_{\natural}$. The key fact depicted by both figures is that, near $\mathrm{pf}$, $V$ stays within $U$.}
		\label{fig:degen} 
	\end{figure}
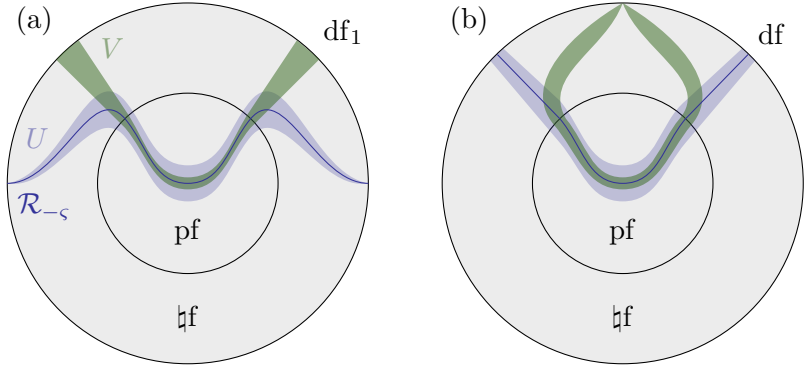

	Let $E \in \Psi_{\mathrm{par}}^{0,0}$ be elliptic at $\calR_{-\varsigma}^{\mathrm{Schr}}$, and let $Q_1 \in \Psi_{\calc}^{-\infty,0,0,0}$ have essential support in $W_1$ and be elliptic on the essential support of $\Pi $. This implies that the operator $Q_1(h)\in \Psi_{\mathrm{par}}^{-\infty,0}$ satisfies 
	\begin{equation}
		\operatorname{Ell}^{\bullet,0}_{\mathrm{par}}(Q_1(h))\supseteq \operatorname{WF}'_{\mathrm{par}}(\Pi (h))
	\end{equation}
	for each $h>0$. Here, $\bullet$ is an arbitrary order on which $\operatorname{Ell}^{\bullet,0}_{\mathrm{par}}(Q_1(h))$ does not depend. We now have
	\begin{multline}
		\operatorname{WF}_{\mathrm{par}}^{\ell,s_0 }(\Pi u(h)) \cap \calR_{-\varsigma}^{\mathrm{Schr}} = \varnothing  \iff \operatorname{WF}_{\mathrm{par}}^{\ell,s_0 }(EQ_1 \Pi u(h)) \cap \calR_{-\varsigma}^{\mathrm{Schr}} = \varnothing 
		\Longleftarrow \Lambda E Q_1 \Pi  u(h) \in L^2, 
	\end{multline}
	where $\Lambda \in \Psi_{\mathrm{par}}^{\ell,s_0}$ is elliptic. For simplicity, we can assume that it is invertible, with $\Lambda^{-1}\in \Psi_{\mathrm{par}}^{-\ell,-s_0}$. 
	
	Let $D\in \Psi_{\calc}^{0,0,0,0}$ denote a two-sided parametrix for $B$ on $U$. Then, 
	\begin{equation}
			\Lambda E Q_1 \Pi  u(h)  = 	(\Lambda E Q_1 \Pi  D \Lambda^{-1} )\Lambda B u(h) + \Lambda  E Q_1 \Pi  R u(h)
	\end{equation}
	for some $R\in \Psi_{\calc}^{-\infty,-\infty,-\infty,0}$ with essential support in the complement of $U$. 
	
	If we choose $E$ such that $\operatorname{WF}'_{\mathrm{par}}(E)\Subset V_0$, then $\operatorname{WF}'_{\calc}(\Lambda EQ_1) \Subset W_1\cap V \subset U$ (using \cref{eq:misc_349}). Then, 
	\begin{align}
		 \Lambda  E Q_1 \Pi  R(h) &\in \Psi_{\mathrm{sc}}^{-\infty,-\infty} \\
		 \Lambda E Q_1 \Pi  D \Lambda^{-1}(h) & \in \Psi_{\mathrm{sc}}^{-\infty,0}
	\end{align}
	for each $h>0$. 
	The first of these implies that $\Lambda  E Q_1 \Pi  R u(h)$ is Schwartz for each $h>0$. This, together with the latter, gives $\Lambda E Q_1 \Pi  u(h) \in L^2  \Longleftarrow  	\Lambda Bu(h) \in L^2 \Longleftarrow Bu(h) \in H_{\mathrm{sc}}^{m,s_0}$.
\end{proof}

We now achieve our goal of estimating $u$ by $P_\pm u$, up to an `absorbable' error term. 

\begin{proposition}
	Consider the setup of \Cref{prop:error_bounding2}, 
and let $B \in \Psi_{\calczero}^{0,0,0}$ satisfy: 
\begin{itemize}
\item $\operatorname{WF}'_{\calc}(B) \backslash \operatorname{Ell}_{\calc}^{0,0,0,0}(G) \subseteq \operatorname{Ell}_{\calc}^{0,0,0,0}(Z)$, i.e.\ the portion of the essential support of $B$ not already contained in the elliptic set of $G$ is covered by the elliptic set of $Z$,
\item $\Sigma_{\mathrm{bad}} \cap \operatorname{WF}'_{\calc}(B)=\varnothing$, i.e.\ the essential support of $B$ does not intersect the ``bad'' component of the characteristic set, 
\item $B$ is elliptic at the zero section over $\natural\mathrm{f}$. 
\end{itemize}
Then, there exists some $h_0$ (depending on everything except $u,h$), such that
	\begin{multline}
		\lVert B u \rVert_{ H_{\calc}^{m,\mathsf{s},\ell,q} }\lesssim \lVert  G P_\pm u \rVert_{H_{ \calc}^{m-1,\mathsf{s}+1,\ell-1,q}} + \lVert  Z P_\pm u \rVert_{H_{ \calc}^{m-2,\mathsf{s},\ell-2,q}} 
		+ \lVert P_\pm u\rVert_{H_{\calc}^{-N,\mathsf{s}+1,-N,q} }
		\\ +  h^\epsilon\lVert  u \rVert_{H_{\calc}^{-N, \mathsf{s},-N,q  }}
		\label{eq:misc_3455}
	\end{multline}
	holds, in the usual strong sense, 
	for all $h\in (0,h_0)$ and $u\in \calS'$ such that the incoming/outgoing condition $\operatorname{WF}_{\mathrm{sc}}^{-N,s_0}(u)\cap \calR_{\varsigma}(h) = \varnothing$ is satisfied. In particular, we have 
	\begin{equation}
		\lVert B u \rVert_{ H_{\calc}^{m,\mathsf{s},\ell,q} }\lesssim \lVert  P_\pm u \rVert_{H_{ \calc}^{m-1,\mathsf{s}+1,\ell-1,q}} + h^\epsilon\lVert  u \rVert_{H_{\calc}^{-N, \mathsf{s},-N,q  }} 
		\label{eq:misc_34555}
	\end{equation}
	under the same assumptions. 
	\label{prop:monosheet}
\end{proposition}
\begin{proof}
	It suffices to prove the proposition for $u \in H_{\mathrm{sc}}^{-N',-N'}$,  $N'$ sufficiently large.
	We apply \Cref{prop:symbolic_global} with the improvement noted in \Cref{prop:symbolic_global_better} to get 
	\begin{equation}
		\lVert \tilde B u \rVert_{H_{\calc}^{m,\mathsf{s},\ell,q}}   \lesssim \lVert G P_\pm u \rVert_{H_{\calc}^{m-1,\mathsf{s}+1,\ell-1,q}} + \lVert Z P_\pm u \rVert_{H_{\calc}^{m-2,\mathsf{s},\ell-2,q}}  + \lVert u \rVert_{H_{\calc}^{-N',-N',-N',q} } 
		\label{eq:misc_34s}
	\end{equation}
	for all $N'>N$. Here $\tilde B$ can be any operator of order $(0,0,0,0)$ such that
	\begin{equation}
		\operatorname{WF}'_{\calc}(\tilde B) \subset \operatorname{Ell}_{\calc}^{0,0,0,0}(G) \setminus \Sigma_{\mathrm{bad}};
	\end{equation}
	in particular, $\tilde B$ can be elliptic at $\calR_{\varsigma}$.	 
	Let $\Pi$ be as in \Cref{lem:automatic_threshold}, and we can choose $\Pi \in \Psi_{\calczero}^{0,0,0}$ such that $1-\Pi$ has essential support disjoint from the zero section over $\natural\mathrm{f}$. Then,
	\begin{multline}
		 \lVert u \rVert_{H_{\calc}^{-N',-N',-N',q} }  \leq \lVert \Pi u \rVert_{H_{\calc}^{-N',-N',-N',q} }  +\lVert (1-\Pi)u \rVert_{H_{\calc}^{-N',-N',-N',q} }   
		 \\
		 \leq  \lVert \Pi u \rVert_{H_{\calc}^{-N',-N',-N',q} }  + \lVert u \rVert_{H_{\calczero}^{-N',-N',-N'}}.
\label{eq:u decomposed Pi}	\end{multline}
	We just need to bound the $\Pi u$ term. If $h$ is such that the right-hand side of \cref{eq:misc_3455} is finite (which is the only case where the estimate is nontrivial), then \cref{eq:misc_34s} gives that $\tilde Bu(h) \in H_{\mathrm{sc}}^{m,\mathsf{s}}$, in which case \Cref{lem:automatic_threshold} says that $\operatorname{WF}_{\mathrm{par}}^{\ell_0,s_0 }(\Pi u(h)) \cap \calR_{-\varsigma}^{\mathrm{Schr}}$ for any $\ell_0$. 
	Then, \Cref{prop:error_bounding2} applies, giving the bound
	\begin{equation}
		\lVert \Pi u \rVert_{H_{\calc}^{-N',-N',-N',q} } \lesssim \lVert P_\pm  u \rVert_{H_{\calc}^{-N',\mathsf{s}+1,-N',q} } + 
		h^\epsilon\lVert  u \rVert_{H_{\calc}^{-N'+1, \mathsf{s},-N'+3,q  }}. 
	\end{equation}
	Of course we trivially have for the last term in \cref{eq:u decomposed Pi}
\begin{equation}
\lVert u \rVert_{H_{\calczero}^{-N',-N',-N'}} \lesssim h^\epsilon\lVert  u \rVert_{H_{\calc}^{-N'+1, \mathsf{s},-N'+3,q  }}. 
\label{eq:error comparison}\end{equation}
We choose $N' = N+3$, combine these estimates, and obtain \cref{eq:misc_3455}.  
\end{proof}

\subsection{Combining the estimates on both components of the characteristic set}
\label{subsec:combined}
The estimate \Cref{prop:monosheet} is still not ready to be applied to the analysis of the PDE because of the appearance of $u$ (without $P_\pm$ applied to it) on the right-hand side, in the final term. 
Even though 
\begin{equation}
	\lVert u \rVert_{H_{\calczero}^{-N,-N,-N} }\lesssim h \lVert u \rVert_{ H_{\calc}^{m,\mathsf{s},\ell,q} }
\end{equation}
if $N$ is sufficiently large, we cannot absorb this term into the left-hand side of that estimate because the left-hand side is written in terms of $Bu$. But we cannot take $B\to 1$ because it was essential that $B$ have essential support disjoint from $\Sigma_{\mathrm{bad}}$ --- the estimates done so far are only valid on $\Sigma$ (or in the elliptic region). So, if we want estimates only involving the forcing on the right-hand side, we are going to need to combine the estimates for $P_-$, $P_+$. The $\calctwo$-Sobolev norms 
 \begin{equation} 
 	\lVert - \rVert_{H_{\calctwo}^{m,\mathsf{s},\ell;q_-,q_+}}
 \end{equation} 
 are ideal for this purpose. For simplicity, we take $q_\pm = 0$. We will also initially state the estimates in terms of the $\calc$-Sobolev norms.

Unfortunately, our notation must now keep track of the sign in $P_\pm$, whereas this dependence was previously suppressed. Our convention will be to add a `$\pm$' \emph{superscript} when considering objects associated with $P_\pm$. For example, $\Sigma_{\mathrm{bad}}^-$ is the sheet of the $\calczero$-characteristic set of $P_-$ that does not intersect the zero section, and $\Sigma^-$ is the other component of the $\calc$-characteristic set. In contrast, $\Sigma_-$ will denote the negative energy sheet of the $\calczero$-characteristic set of $P$. 
Additionally, we use a subscript of $\alpha\in \bbR$ to denote that we are translating a subset of ${}^{\calczero}\overline{T}^* \bbM$ up $\alpha$ units in $\tau_{\natural}$ (not altering any points at fiber infinity). For example, $\Sigma^+_{+1}$ is the result of translating $\Sigma^+$ up one unit. This will be necessary when converting between $\calc$-Sobolev regularity and $\calctwo$-Sobolev regularity. 

Choose $Q_\pm \in \Psi_{\calczero}^{0,0,0}$, $Q_- + Q_+ = \Id$,  as in \cref{eq:Qpm-properties,eq:definition_2res norm} where the $\calctwo$-Sobolev norms were defined. In addition, we shall choose $Q_\pm$ so that 
\begin{equation}
	\operatorname{WF}'_{\calczero}(Q_-)\cap \underbrace{\operatorname{char}_{\calczero}^{0,0,0}(P)\cap \{\tau_{\natural} >0\}}_{\Sigma_+=\Sigma^-_{\mathrm{bad},-1}}  = \varnothing = \operatorname{WF}'_{\calczero}(Q_+)\cap  \underbrace{\operatorname{char}_{\calczero}^{0,0,0}(P)\cap \{\tau_{\natural} <0\}}_{\Sigma_-=\Sigma^+_{\mathrm{bad},+1}}.
	\label{eq:misc_320}
\end{equation}
This is possible because $\operatorname{char}_{\calczero}^{0,0,0}(P)$ is disjoint from the closure in ${}^{\natural}\overline{T}^* \bbM$ of the hyperplane $\taun = 0$. Concretely, we can choose 
\begin{equation}
\sigma^{0,0,0,0}_{\calc}(Q_+) = \chi \big( \frac{\taun}{\ang{\xin}} \big), 
\end{equation}
where $\chi(s) = 1$ for $s \geq 1/2$, $\chi(s) = 0$ for $s \leq -1/2$. This satisfies \eqref{eq:misc_320} at least for small $h$, since at $h=0$, $\Sigma_\pm = \{ \taun = \pm \ang{\xin} \}$. 
Note that we are looking at the $\calczero$-characteristic set of $P$, \emph{not} $P_\pm$.  Roughly, applying $Q_+$ to a distribution restricts attention to those oscillations lying on the positive-energy (meaning $\taun > 0$) component $\Sigma_+$ of the $\calczero$-characteristic set of $P$, while applying $Q_-$ restricts attention to oscillations on the negative energy $(\taun < 0$) component.

We can write any $v\in \calS'$ as 
\begin{equation} 
	v = e^{-i c^2 t} v_- + e^{ic^2 t} v_+ 
	\label{eq:vpm_def}
\end{equation} 
for $v_- = e^{ic^2 t} Q_- v$ and $v_+ = e^{-ic^2 t} Q_+ v$. 
Let $\mathsf{s}_\pm \in C^\infty({}^{\calc}\overline{T}^* \bbM)$ be variable orders as in \eqref{eq:spm orders}. Then, using the definition of the $\calctwo$-norm from \eqref{eq:definition_2res norm}, we have 
\begin{align}
	\lVert v \rVert_{H_{\calctwo}^{m,\mathsf{s},\ell;0,0}} =  \lVert v_- \rVert_{H_{\calc}^{m,\mathsf{s}_-,\ell,0} }  +\lVert v_+ \rVert_{H_{\calc}^{m,\mathsf{s}_+,\ell,0} } . 
	\label{eq:misc_379}
\end{align}

The auxiliary functions $u_\pm$ satisfy the PDE 
\begin{equation}\label{eq:u+u comparison}
	P_\pm u_\pm = e^{\mp i c^2 t } P e^{\pm i c^2   t} u_\pm= (Pu)_\pm+ f_\pm 
\end{equation}
for $(Pu)_- = e^{ic^2 t}Q_- Pu$, $(Pu)_+ = e^{-ic^2 t} Q_+Pu$, and 
\begin{equation}
	f_\pm = 
	\begin{cases}
		e^{- i c^2 t} [P, Q_+] u = -e^{- i c^2 t} [P, Q_-] u & (+\text{ case}),  \\
		e^{ic^2 t} [P, Q_-] u & (-\text{ case}).
	\end{cases}
	\label{eq:f}
\end{equation} 

\begin{proposition} \label{prop:semiFred estimate}
Let the spacetime order $\mathsf{s}$ be as in Theorem~\ref{thm:inhomog}, and let $m, \ell \in \mathbb{R}$. Then, for $h>0$ sufficiently small and for all functions $u = u(h)$ such that $u$ satisfies an above-threshold condition 
\begin{equation}\label{eq:u above threshold cond}
\operatorname{WF}_{\mathrm{sc}}^{-N,s_0}(u)\cap \calR_{\varsigma}(h) = \varnothing \text{ for some } s_0 > -\frac1{2},
\end{equation}
at the above-threshold radial set $\calR_{\varsigma}(h)$ on each component $\Sigma_\pm$ of the characteristic set of $P$, we have the estimate, 
\begin{equation}\label{eq:semiFred estimate}
\lVert u \rVert_{H_{\calctwo}^{m,\mathsf{s},\ell;0,0}} \lesssim  \lVert  Pu  \rVert_{H_{ \calctwo}^{m-1,\mathsf{s}+1,\ell-1,0,0}}.
\end{equation}
\end{proposition}

\begin{proof}
We  apply Proposition~\ref{prop:monosheet} to $u_\pm$. More precisely, we apply the Proposition to $\tilde Q_\pm u_\pm$, where $\tilde Q_\pm \in \Psi_{\calczero}^{0,0,0}$ is equal to the identity microlocally on $\operatorname{WF}'_{\calczero}(Q_\pm)$, and is microlocally trivial at $\Sigma_\mp$. Since $u_\pm$ has already been microlocalized using $Q_\pm$, the difference between $u_\pm$ and $\tilde Q_\pm u_\pm$ is a Schwartz function with seminorms rapidly vanishing in $h$, and its contribution to the left-hand side can be estimated by $\lVert u \rVert_{H_{\calczero}^{-N,-N,-N}}$. Thus, Proposition~\ref{prop:monosheet} yields 
\begin{equation}\begin{aligned}
\lVert  u_+ \rVert_{ H_{\calc}^{m,\mathsf{s_+},\ell,0} } &\lesssim \lVert  P_+ u_+ \rVert_{H_{ \calc}^{m-1,\mathsf{s_+}+1,\ell-1,0}}  + h^\epsilon\lVert  u_+ \rVert_{H_{\calc}^{-N, \mathsf{s_+},-N,0  }} +  \lVert u \rVert_{H_{\calczero}^{-N,-N,-N}}, 
\\
\lVert  u_- \rVert_{ H_{\calc}^{m,\mathsf{s_-},\ell,0} } &\lesssim \lVert  P_- u_- \rVert_{H_{ \calc}^{m-1,\mathsf{s_-}+1,\ell-1,0}}  + h^\epsilon\lVert  u_- \rVert_{H_{\calc}^{-N, \mathsf{s_-},-N,0  }} +  \lVert u \rVert_{H_{\calczero}^{-N,-N,-N}}.
\end{aligned}\end{equation}
Combined with \eqref{eq:u+u comparison}, and undoing the modulations (which has the effect of changing $\mathsf{s_\pm}$ to $\mathsf{s}$) on the commutator term,  this yields 
\begin{multline}
\lVert  u_+ \rVert_{ H_{\calc}^{m,\mathsf{s_+},\ell,0} } \lesssim   \lVert  (P u)_+ \rVert_{H_{ \calc}^{m-1,\mathsf{s_+}+1,\ell-1,0}}   + \lVert  [P, Q_+ ]u  \rVert_{H_{ \calc}^{m-1,\mathsf{s}+1,\ell-1,0}} \\ + h^\epsilon\lVert  u_+ \rVert_{H_{\calc}^{-N, \mathsf{s_+},-N,0  }} +  \lVert u \rVert_{H_{\calczero}^{-N',-N',-N'}}, 
\end{multline}
\begin{multline}
\lVert  u_- \rVert_{ H_{\calc}^{m,\mathsf{s_-},\ell,0} } \lesssim  \lVert  (Pu)_- \rVert_{H_{ \calc}^{m-1,\mathsf{s_-}+1,\ell-1,0}} + \lVert  [P, Q_- ]u  \rVert_{H_{ \calc}^{m-1,\mathsf{s}+1,\ell-1,0}} \\ + h^\epsilon\lVert  u_- \rVert_{H_{\calc}^{-N, \mathsf{s_-},-N,0  }} +  \lVert u \rVert_{H_{\calczero}^{-N',-N',-N'}},
\end{multline}
where $(Pu)_\pm$ are defined as in \Cref{eq:vpm_def}.
Using \eqref{eq:misc_379} (three times!) we obtain (remembering that $[P, Q_+] = - [P, Q_-]$) 
\begin{equation}
	\lVert u \rVert_{H_{\calctwo}^{m,\mathsf{s},\ell;0,0}} \lesssim  \lVert  Pu  \rVert_{H_{ \calctwo}^{m-1,\mathsf{s}+1,\ell-1;0,0}} + \lVert  [P, Q_+ ]u  \rVert_{H_{ \calc}^{m-1,\mathsf{s}+1,\ell-1,0}} + h^\epsilon\lVert  u \rVert_{H_{\calctwo}^{-N, \mathsf{s},-N; 0, 0  }} +  \lVert u \rVert_{H_{\calczero}^{-N',-N',-N'}}. 	\label{eq:misc_3799}
\end{equation}

We now claim we have the following estimate on the commutator term: 


\begin{equation}
 \lVert  [P, Q_+ ]u  \rVert_{H_{ \calc}^{m-1,\mathsf{s}+1,\ell-1,0}} \lesssim  \lVert  P u  \rVert_{H_{ \calctwo}^{m-2,\mathsf{s},\ell,0,0}} + \lVert u \rVert_{H_{\calczero}^{-N',-N',-N'}}.
 \label{eq:comm_term} \end{equation}
 
To show this claim, first observe that, because of the way $Q_+$ was chosen, the commutator $[P, Q_+]$ is microsupported away from $\{ h=0, \taun = \pm 1, \xin = 0 \}$ which is the location of the parabolic faces $\mathrm{pf}_\pm$. Therefore, up to an error term of the form $\lVert u \rVert_{H_{\calczero}^{-N',-N',-N'}}$, the left-hand side of \cref{eq:comm_term} is bounded by 
 \begin{equation} 
  \lVert  [P, Q_+ ]u  \rVert_{H_{ \calczero}^{m-1,\mathsf{s}+1,\ell-1}},
 \end{equation} 
  that is, the order at the parabolic face is irrelevant up to an error term $\lVert u \rVert_{H_{\calczero}^{-N',-N',-N'}}$.

  Next, we observe that $[P, Q_\pm]$ is microsupported on the elliptic set of $P$. It follows that we can use  \eqref{eq:micro-elliptic-est} to estimate 
  \begin{equation} 
  \lVert  [P, Q_+ ]u  \rVert_{H_{ \calczero}^{m-1,\mathsf{s}+1,\ell-1}} \lesssim  \lVert  \tilde Q Pu  \rVert_{H_{ \calczero}^{m-2,\mathsf{s},\ell-2}} + \lVert u \rVert_{H_{\calczero}^{-N',-N',-N'}}
  \end{equation} 
  where $\tilde Q$ is elliptic on the microsupport of $[P, Q_\pm]$, but is microsupported away from $\{ h=0, \taun = \pm 1, \xin = 0 \}$. 
  
  Finally, using this microsupport property,  we can replace the $\calczero$-norm of $\tilde Q Pu$ with a $\calctwo$-norm with arbitrary orders at the two parabolic faces, up to another error term $\lVert u \rVert_{H_{\calczero}^{-N',-N',-N'}}$, yielding \cref{eq:comm_term} and proving the claim.

Choosing $N' \geq N+\epsilon$, we can absorb the error term    $\lVert u \rVert_{H_{\calczero}^{-N',-N',-N'}}$ into the term 
$h^\epsilon\lVert  u \rVert_{H_{\calctwo}^{-N, \mathsf{s},-N; 0, 0  }}$. 
We can thus combine \cref{eq:misc_3799} and \cref{eq:comm_term}  to deduce that 
  \begin{equation}
	\lVert u \rVert_{H_{\calctwo}^{m,\mathsf{s},\ell;0,0}} \lesssim  \lVert  Pu  \rVert_{H_{ \calctwo}^{m-1,\mathsf{s}+1,\ell-1;0,0}}  + h^\epsilon\lVert  u \rVert_{H_{\calctwo}^{-N, \mathsf{s},-N;0 ,0 }}  . 	\label{eq:misc_37999}
\end{equation}
The second term can be absorbed on the left-hand side for sufficiently  small $h$. This yields \cref{eq:semiFred estimate}. 
\end{proof}
    
\begin{proof}[Proof of Theorem~\ref{thm:inhomog}] Let $u$ be an element of 
 \begin{equation} 
 \mathcal{X}^{m,\mathsf{s},\ell} := \big\{ u \in H_{\calctwo}^{m,\mathsf{s},\ell;0,0} : Pu \in H_{\calctwo}^{m-1,\mathsf{s}+1,\ell-1;0,0} \big\}.
 \end{equation} 
 Since, by assumption, $\mathsf{s}$ is above threshold at one of the two radial sets corresponding to each component $\Sigma_\pm$ of the the characteristic set of $P$, $u$ satisfies condition \eqref{eq:u above threshold cond}. It follows that the estimate \eqref{eq:semiFred estimate} of Proposition~\ref{prop:semiFred estimate} holds, which we express here in the form 
 \begin{equation}\label{eq:semiFred XY}
 \lVert u \rVert_{\mathcal{X}^{m,\mathsf{s},\ell}} \lesssim  \lVert  Pu  \rVert_{\mathcal{Y}^{m-1,\mathsf{s}+1,\ell-1}}.
 \end{equation}
 This estimate shows that $P$, viewed as a map from  $\mathcal{X}^{m,\mathsf{s},\ell}$ to $ \mathcal{Y}^{m-1,\mathsf{s}+1,\ell-1}$, is injective and has closed range. Indeed any convergent sequence in the range may be expressed uniquely in the form $Pu_j$ thanks to injectivity. The estimate \eqref{eq:semiFred XY} shows that $u_j$ is Cauchy, therefore converges to $u \in \mathcal{X}^{m,\mathsf{s},\ell}$. Then by continuity of $P$, the limit of $Pu_j$ is $Pu$, showing that the range of $P$ contains all of its limit points. 
 
 To prove invertibility, it thus suffices to show that the orthogonal complement of the range is trivial. We can view the orthogonal complement as living in the dual space $H_{\calctwo}^{-m+1, -\mathsf{s}-1, -\ell +1; 0, 0}$ via the $L^2$-pairing. So suppose that $v \in  H_{\calctwo}^{-m+1, -\mathsf{s}-1, -\ell +1; 0, 0}$ is orthogonal to the range of $P$ acting on $\mathcal{X}^{m,\mathsf{s},\ell}$. In particular, it is orthogonal to the range of $P$ acting on Schwartz functions:
 \begin{equation} 
 \ang{P \phi, v} = 0 \text{ for all } \phi \in \mathcal{S}(\mathbb{R}^{n+1}). 
 \end{equation} 
 This implies that 
 \begin{equation} 
 \ang{ \phi, P^* v} = 0 \text{ for all } \phi \in \mathcal{S}(\mathbb{R}^{n+1}),
 \end{equation} 
 and hence that $P^* v = 0$. We notice that $\tilde{\mathsf{s}} := -1 - \mathsf{s}$ is above threshold where $\mathsf{s}$ is below threshold, which by assumption is true on one of the two radial sets corresponding to each component of the characteristic set. Therefore, we can apply Proposition~\ref{prop:semiFred estimate} to $v$ with $P^*$ replacing $P$ and with exponents $(1-m, \tilde{\mathsf{s}}, 1-\ell)$ in place of $(m, \mathsf{s}, \ell)$, and we find that $P^* v = 0$ implies that $v=0$. This shows the triviality of the orthogonal complement of the range of $P$ acting between $\mathcal{X}^{m,\mathsf{s},\ell}$ and $ \mathcal{Y}^{m-1,\mathsf{s}+1,\ell-1}$, and thus shows that $P$ is a bijection, and therefore invertible. 
\end{proof}

\section*{Acknowledgements}

This work began at the \href{https://www.matrix-inst.org.au/}{MATRIX} workshop ``\href{https://www.matrix-inst.org.au/events/hyperbolic-pdes-and-nonlinear-evolution-problems/}{Hyperbolic PDEs and Nonlinear Evolution Problems}.'' We thank MATRIX for the hospitality and inspiring environment.
We are particularly grateful to Timothy Candy, whose talk at MATRIX, on \cite{CaHe}, served as the immediate impetus for this project. A.H. and Q.J. are supported by the Australian Research Council through grant FL220100072. E.S.\ is supported by the National Science Foundation under grant number DMS-2401636.
A.V. gratefully acknowledges support from the National Science
Foundation under grant number DMS-2247004 and from a Simons Fellowship.

\appendix
\section{Index of notation}
\label{sec:notation}\noindent Some standard notations:
\begin{itemize}
	\item $\triangle = -\sum_{j=1}^d \partial_{x_j}^2$ is the positive semi-definite Euclidean Laplacian.
	\item $D_{\bullet} = -i\partial_{\bullet}$, with $\bullet$ being one of the coordinates. 
	\item  $r=(x_1^2+\dots+x_d^2)^{1/2}= \lVert x \rVert$ is the usual Euclidean radial coordinate, and $\langle r\rangle = (1+r^2)^{1/2}$ is the ``Japanese bracket'' of $r$.
	\item $\calF$ is the spacetime Fourier transform, defined using the convention 
	\begin{equation}
		\calF u(\tau,\xi) = \int_{\bbR^{1+d}} e^{-it \tau -ix\cdot \xi } u(t,x) \dd t \dd^d x .
	\end{equation}
	\item $H_p$ is the Hamiltonian vector field associated to the symbol $p\in C^\infty(T^* \bbR^{1,d})$. See \cref{eq:Hp-def}.
    \item $\{ \cdot , \; \cdot \}$, the Poisson bracket. See \cref{eq:poisson_convention}.
\end{itemize}
Operators:
\begin{itemize}
	\item $P$ the Klein--Gordon operator considered; see \S\ref{subsec:operators}. 
	\item $P_0$ is usually the free Klein--Gordon equation, but sometimes includes some other terms. 
	\item $P_\pm = e^{\mp ic^2 t} P e^{\pm i c^2 t}$, see \cref{eq:P+-_form}
\end{itemize}
Manifolds-with-corners (mwc):
\begin{itemize}
	\item $\overline{\bbR^d}=\bbR^d\sqcup \infty\bbS^{d-1}$ is the radial compactification of $\bbR^d$. As a special case, $\bbM = \overline{\bbR^{1,d}}$. 
	\item  ${}^\bullet\overline{T}^* \bbM$ is a compactification of $T^* \bbR^{1,d}$ associated to the $\bullet$-calculus, the \emph{$\bullet$-phase space}. We use the label `$\mathrm{df}$' (or variants thereof) to refer to the closure of fiber infinity, which is `infinity part' of fibers in whatever phase space (i.e., cotangent bundle) we are considering. The `d' stands for differential. Also, 
	\begin{equation}
		{}^\bullet T^* \bbM={}^\bullet\overline{T}^* \bbM\backslash \mathrm{df}. 
	\end{equation}
	\item $\mathrm{df},\mathrm{bf},\natural\mathrm{f},\mathrm{pf}$ are the boundary hypersurfaces of the $\calc$-phase space. We also use `$\natural\mathrm{f}$' to refer to the analogous boundary hypersurfaces of the $\calczero$- and $\calctwo$- phase spaces.  
	\item $\mathrm{pf}_\pm$ the analogues of $\mathrm{pf}$ in the $\calctwo$-phase space. 
	\item $\rho_{\mathrm{f}}$ is, for each boundary hypersurface $\mathrm{f}$ of a mwc, a boundary-defining-function of that face.
	\item $\calV_\bullet$ is the space of sections of the dual bundle ${}^\bullet T \bbM$ to ${}^{\bullet}T^* \bbM$.
	\item $\tilde{\calV}_{\mathrm{b}}(\bullet)$  b-vector fields with a small amount of extra decay.  See \cref{eq:tildeVb-nat}, \cref{eq:tildeVb-par}, \cref{eq:tildeVb-parIres}, \cref{eq:tildeVb-general} \cref{eq:tildeV-calc}.
\end{itemize}
Pseudodifferential calculi: 
\begin{itemize}
	\item $\Psi_{\mathrm{sc}}(\bbR^N)$ the scattering- (sc-) calculus on $\overline{\bbR^N}$. 
	\item $\Psi_{\mathrm{par}}=\Psi_{\mathrm{par}}(\bbR^{1,d})$ is the parabolic scattering calculus on $\bbR^{1+d}$, defined in \cite{Parabolicsc} in order to study the Schr\"odinger equation. 
	\item $\Psi_{\mathrm{par,I,res}}$ is the result of second- microlocalizing the calculus of smooth families of elements of $\Psi_{\mathrm{par}}$, specifically second- microlocalizing at the corner of 
	\begin{equation} 
		{}^{\mathrm{par,I}}\overline{T}^* \bbM = [0,\infty)\times {}^{\mathrm{par}}\overline{T}^* \bbM.
	\end{equation}
	\item $\Psi_{\calczero}$ is the natural ($\natural$) calculus, defined in \S\ref{subsec:natural_calculus}.
	\item $\Psi_{\calc}$ is the $\calc$-calculus, defined in \S\ref{subsec:calc}.
	\item $\Psi_{\calctwo}$ is the $\calctwo$-calculus, defined in \S\ref{sec:2-res_construction}.
	\item $\Psi_\bullet^{m,s,\dots}$ is the subset of $\Psi_\bullet$ consisting of operators of orders $m,s,\dots$. 
	\item $\sigma_\bullet^{m,s,\dots}$ is the principal symbol map associated to the $\bullet$-calculus and orders $m,s,\dots$.
	\item $N(A)$ is the leading-order part of $A\in \Psi_{\calc}$ at $\mathrm{pf}$, 
	\item $\operatorname{WF}'_\bullet(A)$ is the operator wavefront of $A$ associated to the $\bullet$-calculus.
\end{itemize}
Function spaces: 
\begin{itemize}
	\item If $X$ is a manifold-with-corners and $I\subseteq \bbC$, then $C^\infty(X;I)$ denotes the set of smooth functions $f:X^\circ \to I$ such that, if $X_0\supseteq X$ is a closed manifold in which $X$ is embedded, then $f$ extends to an element of $C^\infty(X_0)$. 
	\item If $X$ is a manifold-with-boundary, then $\calS(X)$ denotes the space of Schwartz functions on $X$, i.e.\ smooth functions on $X^\circ$ decaying rapidly, together with all derivatives, at $\partial X$. 
	\item $H_\bullet^{m,s,\dots}$ is the $L^2$-based Sobolev space associated to $\Psi_\bullet$, such that $\Psi_\bullet^{m,s,\dots} H_\bullet^{m,s,\dots} =L^2$. Note that 
	\begin{equation}
		H_\bullet^{0,0,\dots}= L^2.
	\end{equation}
\end{itemize}
Miscellaneous:
\begin{itemize}
	\item ${}^\bullet\mathsf{H}_p$ is $H_p$ multiplied by just enough boundary-defining-functions to make a b-vector field on the $\bullet$-phase space.
	\item $\Sigma,\Sigma_{\pm},\Sigma_{\mathrm{bad}}$ all denote portions of various characteristic sets. In this paper, it is mainly the $\calctwo$-characteristic set of $P$ which matters. We will usually denote this $\Sigma$.
	\item $\calR^\bullet_\bullet$ denotes the radial sets. The subscript specifies whether we are looking over the future (+) or past ($-$) hemisphere of spacetime infinity, and the superscript specifies whether we are looking at positive/negative energy (i.e.\ particles vs. anti-particles).
	\item $Q=Q_-$, $Q_+$ in \S\ref{subsec:combined}, \S\ref{sec:further} are microlocalizers microlocalizing near the components $\Sigma_\pm$ of the $\calczero$-characteristic set of the Klein--Gordon operator; $Q_-$ microlocalizes to negative frequencies $\tau_\natural$ (``negative energy'') and $Q_+$ microlocalizes to positive frequencies (``positive energy'').
\end{itemize}

\section{The half-Fredholm estimate for the Schr\"odinger problem }
\label{sec:schrodinger_estimate}

Notation below will be as in \S\ref{subsec:remainder}. 
For the reader's convenience, we recall that $\calR^{\mathrm{Schr}}_\pm$ is the radial set from \cite{Parabolicsc} over the future(+)/past($-$) hemisphere of spacetime infinity, and also that,  
according to \cref{eq:Ppm normal}, $N(P_-)$ is given by
\begin{equation}
	N(P_-) =  2i\partial_t + \triangle +i  B \cdot \partial_x + V,
	\label{eq:misc_623}
\end{equation}
where $V=W-\beta-\aleph$, and where we are dropping from the notation the evaluation at $c=\infty$ present \cref{eq:Ppm normal}.

In this appendix, we discuss how to adapt the proof of the estimate \cite[Thm.\ 1.1, Thm.\ 6.1]{Parabolicsc} to obtain \eqref{eq:par_estimate}, which we recall here: for each $\ell'\in \bbR$ and variable order $\mathsf{s} \in C^\infty({}^{\mathrm{par}}\overline{T}^* \bbM )$ such that 
\begin{itemize}
	\item $\overline{\mathsf{s}}$ is monotonic under the Hamiltonian flow,
	\item $\overline{\mathsf{s}}|_{\calR_{-}^{\mathrm{Schr}} } > -1/2 > \overline{\mathsf{s}}|_{\calR_+^{\mathrm{Schr}}} $,
\end{itemize}
the estimate 
\begin{equation}
	\lVert v \rVert_{H_{\mathrm{par}}^{\ell',\overline{\mathsf{s}}}} \lesssim \lVert N(P_-)v\rVert_{H_{\mathrm{par}}^{\ell'-1,\overline{\mathsf{s}}+1}}
	\label{eq:Schrodinger_est}
\end{equation}
holds for all 
\begin{equation}  
	v \in \calX^{\ell',\overline{\mathsf{s}}}_{\mathrm{Schr}} := \{v\in H_{\mathrm{par}}^{\ell',\overline{\mathsf{s}}} : N(P_-)v \in H_{\mathrm{par}}^{\ell'-1,\overline{\mathsf{s}}+1} \}.
\end{equation}
This estimate is essentially the same as the cited estimate from ibid.\ except for the difference that the latter is stated (and proved) for compactly supported potentials. 

We begin by noting that the propagation of singularities/regularity argument from \cite[\S5]{Parabolicsc} applies unchanged, since whatever extra terms we have here do not enter the principal symbol, and the skew-symmetric part in the sub-principal symbol (which has the effect of shifting the threshold of radial point estimates) is absent due to the assumption \eqref{eq:misc_158}. Consequently, the estimate
\cite[Prop.\ 5.11]{Parabolicsc} applies: 
\begin{equation}
	\lVert v \rVert_{H_{\mathrm{par}}^{\ell',\overline{\mathsf{s}}}} \lesssim \lVert N(P_-)v\rVert_{H_{\mathrm{par}}^{\ell'-1,\overline{\mathsf{s}}+1}} + \lVert v \rVert_{H_{\mathrm{par}}^{M,N}}, 
\end{equation}
where $M,N\in \bbR$ are arbitrary. This estimate is almost the desired \cref{eq:Schrodinger_est}, but not quite.
What we need to explain is how to remove the last term. By a standard argument, it suffices to show that $N(P_-)$ is injective on the space $\smash{\calX^{\ell',\overline{\mathsf{s}}}_{\mathrm{Schr}}}$:
\begin{equation}
	v \in \ker_{\calX^{\ell',\overline{\mathsf{s}}}_{\mathrm{Schr}}}  N(P_-) \Rightarrow v=0.
	\label{eq:misc_625}
\end{equation}
This is done in \cite[\S6]{Parabolicsc}, at the level of generality considered there, but the argument there does not apply verbatim here. A modification works, however.

The basic idea of the proof of \cref{eq:misc_625} will be to use the fact that 
\begin{equation}
	v \in \ker_{\calX^{\ell',\overline{\mathsf{s}}}_{\mathrm{Schr}}}  N(P_-) \Rightarrow v\text{ Schwartz in any backwards cone}
	\label{eq:misc_six}
\end{equation}
(\cref{eq:misc_six} being a consequence of the propagation and above-threshold radial point estimates proven for $\calX_{\mathrm{Schr}}^{*,\overline{\mathsf{s}}}$ in ibid.) to 
bound the behavior as $t\to -\infty$ of the $L^2(\bbR^d)$ norm of $v(t,-)$,  
\begin{equation}
	M(t) = \int_{\bbR^d} |v(t,x)|^2 \dd^d x.
\end{equation}
What we will do is show that, if $v \in \ker_{\calX^{\ell',\overline{\mathsf{s}}}_{\mathrm{Schr}}}  N(P_-)$, then
\begin{enumerate}[label=(\roman*)]
	\item there exists a constant $C>0$ such that $M(T_1) \leq CM(T_0)$ for any $T_0<T_1$ (essentially an energy/mass estimate), \label{it:misc_h4}
	\item $\lim_{t\to-\infty} M(t) = 0$.  
\end{enumerate}
Combining these, taking $T_0\to -\infty$ in \cref{it:misc_h4}, it follows that $M(T_1)=0$ for all $T_1\in \bbR$, which implies that $v$ vanishes identically, thus completing the proof of \cref{eq:misc_625}.
The proof of the energy estimate $M(T_1) \leq CM(T_0)$ will be elementary. (If $N(P_-)$ were exactly $L^2$-symmetric, which is often true in the cases of interest, then $M$ would be constant, owing to conservation of mass.) The second required result,  $\lim_{t\to-\infty} M(t) = 0$, is obviously consistent with the fact that $v$ is Schwartz in any backwards cone, but the latter is not quite sufficient. 
Ultimately, the proof of  $\lim_{t\to-\infty} M(t) = 0$ will use notions from \cite{gell2023scattering} (specifically the notion of so-called ``module regularity''), not just \cite{Parabolicsc}, in order to control $v$ at large $x/|t|$ (a region which includes the complement of arbitrarily wide backwards cones, thus not controlled by \cref{eq:misc_six}).

We begin with the energy/mass estimate. The quantity $M(t)$ satisfies 
\begin{align} \label{eq: time derivative of mass}
	\begin{split}
		\frac{\mathrm{d} M(t)}{\mathrm{d} t}  = \int_{\bbR^d} ((\partial_t v) \overline{v} + v \partial_t\overline{v}) \dd^d x
		& = 2\Re \int_{\bbR^d} (\partial_t u) \overline{u} \dd^d x
		\\ & = -\Im \Big[ \int_{\bbR^d}( ( \triangle + i B\cdot \partial_x+V) u )\bar{u} \dd^d x \Big], 
	\end{split}
	\intertext{using $N(P_-)v=0$. Taking the imaginary part inside the integral, the result is} 
	\begin{split}  \label{eq: time derivative of mass-2}
		\frac{\mathrm{d} M(t)}{\mathrm{d} t} &= -\int_{\bbR^d}\Big[ \Big( i (\Im B)\cdot \partial_x- \frac{\nabla\cdot \overline{B}}{2}+\Im V\Big) u \Big]\bar{u} \dd^d x \\ 
		&= \int_{\bbR^d}\Big[ \Big(\frac{\nabla\cdot \overline{B}}{2}-\Im V\Big) u \Big]\bar{u} \dd^d x,
	\end{split} 
\end{align}
using \cref{eq:misc_170}, which says that $\Im B=0$ exactly. (Recall that, in this appendix, the coefficients $W,B$ etc.\ are all being evaluated at $c=\infty$.) Now, $\nabla\cdot \overline{B}, \Im V \in S^{-2}(\bbR^{1,d})$ (using \eqref{eq:misc_158}), so we can conclude that 
\begin{equation}
	\Big|\frac{\mathrm{d} M}{\mathrm{d} t}\Big| \leq C \langle t \rangle^{-2}M(t)
\end{equation}
for some $C>0$. 
We can now use Gr\"onwall's inequality. Since
\begin{equation} 
	\int_{-\infty}^{\infty} \langle t \rangle^{-2}\dd t <\infty,
\end{equation}
we can conclude from Gr\"onwall that there exists some $C>0$ such that $M(T_1) \leq C M(T_0)$ holds whenever $T_0,T_1$ are real numbers such that $T_0 \leq T_1 $. Thus, the desired energy/mass estimate holds.

Next, we turn to the (more difficult) proof that $\lim_{t\to-\infty} M(t) = 0$. The idea, which is adapted from the construction of the scattering matrix in \cite{gell2023scattering}, will be to show that $v$ has some ``backwards scattering data'' $\frakv_\infty$. This will be defined by
\begin{equation}
	\frakv_\infty(\mathsf{X}) = \lim_{t\to-\infty} (2\pi i t)^{d/2} e^{-i t |\mathsf{X}|^2 /2} v(t,t\mathsf{X}) \in \langle \mathsf{X} \rangle^{-k} L^2(\bbR^d_{\mathsf{X}})
	\label{eq:misc_632}
\end{equation}
once it is known that this limit exists (for arbitrary $k\in \bbR$). The key calculation that relates this to $M(t)$ is 
\begin{equation}
	M(t) = (2\pi)^{-d} \int_{\bbR^d} |\frakv(t,\mathsf{X})|^2 \dd^d \mathsf{X} =  (2\pi)^{-d} \lVert \frakv(t,\mathsf{X}) \rVert_{L^2(\bbR^d_{\mathsf{X}})}^2,
\end{equation}
where 
\begin{equation}
	\frakv(t,\mathsf{X}) = (2\pi i t)^{d/2} e^{-i t |\mathsf{X}|^2 /2} v(t,t\mathsf{X}) \in C^\infty( (-\infty,0)_t\times \bbR^d_{\mathsf{X}}) 
\end{equation}
is what $\frakv_\infty$ is supposed to be the limit of.
(The point is that the factor of $t^{d/2}$ in the definition of $\frakv$ cancels the factor of $t^{-d}$ in the Jacobian of the coordinate transformation $(t,x)\mapsto (t,\mathsf{X}=x/t)$.)
So, once it is known that the limit on the right-hand side of \cref{eq:misc_632}, $\lim_{t\to-\infty} \frakv(t,\mathsf{X})$,  exists in $L^2(\bbR^d)$, we get that 
\begin{equation} 
	\lim_{t\to-\infty} M(t) = (2\pi)^{-d} \lim_{t\to\infty} \lVert \frakv(t,\mathsf{X}) \rVert_{L^2(\bbR^d_{\mathsf{X}})}^2 =(2\pi)^{-d} \lVert \frakv_\infty(\mathsf{X}) \rVert_{L^2(\bbR^d_{\mathsf{X}})}^2.
	\label{eq:misc_o63}
\end{equation} 
But, the fact that $v$ is Schwartz in backwards cones in spacetime implies that $\frakv_\infty=0$ identically. So, we can immediately conclude from \cref{eq:misc_o63} that $\lim_{t\to-\infty} M(t) = 0$. Thus, what we need to do to complete our proof is establish the existence of the limit in \cref{eq:misc_632}, i.e.\ to prove the existence of backwards scattering data.

In order to establish the existence of the limit \cref{eq:misc_632}, it suffices to prove that 
\begin{equation}
	\partial_t \frakv(t,\mathsf{X}) \in L^1((-\infty,-T) ; \langle \mathsf{X} \rangle^{-k} L^2(\bbR^d_{\mathsf{X}})) 
	\label{eq:misc_k31}
\end{equation}
for all $k\in \bbR$, $T>0$. Actually, it ends up being slightly cleaner to show that 
\begin{equation}
	\partial_t \tilde{\frakv}(t,\mathsf{X}) \in L^1((-\infty,-T) ; \langle \mathsf{X} \rangle^{-k} L^2(\bbR^d_{\mathsf{X}})), 
	\label{eq:misc_k32}
\end{equation}
where 
\begin{equation}
	\tilde{\frakv}(t,\mathsf{X}) =  e^{i\tilde{V}(t,\mathsf{X})}e^{i \mathsf{X} \cdot \tilde{B}(t,\mathsf{X})} \frakv(t,\mathsf{X}),
\end{equation}
where $\tilde{V},\tilde{B}$ are given by
\begin{equation}
	\tilde{V}(t,\mathsf{X}) = \frac{1}{2} \int_{-T}^t V(\tau,\tau\mathsf{X})\dd \tau \quad\text{ and }\quad \tilde{B}_j(t,\mathsf{X})
	= -\frac{1}{2}\int_{-T}^t B_j(\tau,\tau \mathsf{X})\dd \tau. 
\end{equation}
Because $\Im V \in S^{-2}(\bbR^{1,d})$ (this was \cref{eq:misc_6x6}) and $\Im B=0$ \cref{eq:misc_170}, the exponentials $\exp( i\tilde{V})$, $\exp( i\mathsf{X} \cdot \tilde{B})$ are bounded and also bounded away from zero. Consequently, the two statements \cref{eq:misc_k31}, \cref{eq:misc_k32} are equivalent. We will prove the latter by simply computing $\partial_t \tilde{\frakv}$ and estimating each term in the resultant expression.

This paragraph contains the computation.
Letting $L$ denote the differential operator given in terms of the original coordinates $(t,x)$ by $L=-2i(\partial_t + t^{-1} x\cdot \partial_x)$, 
\begin{align}
	\partial_t \tilde{\frakv}(t,\mathsf{X})  = & \partial_t (e^{i\tilde{V}(t,\mathsf{X})}e^{ i\mathsf{X} \cdot \tilde{B}(t,\mathsf{X})}) \frakv(t,\mathsf{X}) + e^{i\tilde{V}(t,\mathsf{X})}e^{i \mathsf{X} \cdot \tilde{B}(t,\mathsf{X})} \partial_t\frakv(t,\mathsf{X})
	\\ = & \frac{i}{2} \Big[  \Big( V(t,t\mathsf{X}) -\mathsf{X} \cdot B(t,t\mathsf{X}) 
	-  \frac{id}{t} - |\mathsf{X}|^2 \Big) \tilde{\frakv}(t,\mathsf{X})
	+ e^{i\tilde{V}+i\mathsf{X}\cdot \tilde{B} - i t |\mathsf{X}|^2/2} (2\pi it)^{d/2} (Lv)(t,t\mathsf{X}) \Big] ,
	\label{eq:misc_622}
\end{align}
where the partial derivative $\partial_t$ in \cref{eq:misc_622} (but not in $Lv$) is computed while holding $\mathsf{X}$ fixed, rather than holding $x$ fixed. Using the identity 
\begin{equation}
	-2 i t^{-1} x\cdot \partial_x = -t^{-2} (it\partial_x+x)^2 + \triangle + i t^{-1} d +  |\mathsf{X}|^2,
\end{equation}
\cref{eq:misc_622} can be rewritten
\begin{equation}
	\partial_t \tilde{\mathfrak{v}}(t,\mathsf{X})= \frac{i}{2} \Big[  \Big( V(t,t\mathsf{X}) -\mathsf{X} \cdot B(t,t\mathsf{X}) 
	\Big) \tilde{\frakv}(t,\mathsf{X})
	+ e^{i\tilde{V}+i\mathsf{X}\cdot \tilde{B} - i t |\mathsf{X}|^2/2} (2\pi it)^{d/2} (Qv)(t,t\mathsf{X}) \Big]
	\label{eq:misc_636}
\end{equation}
for $Q$ the differential operator given in terms of the original coordinates $(t,x)$ by $Q = -2i \partial_t -t^{-2} (it\partial_x+x)^2 + \triangle$. The operator $Q$ can be written, using the formula \cref{eq:misc_623} for $N(P_-)$, as $Q =  N(P_-) - t^{-2}(it\partial_x+x)^2 - iB\cdot \partial_x - V $. So, \cref{eq:misc_636} becomes 
\begin{equation}
	\partial_t \tilde{\frakv}(t,\mathsf{X})  = \frac{i}{2}e^{i\tilde{V}+i\mathsf{X}\cdot \tilde{B} - i t |\mathsf{X}|^2/2} (2\pi it)^{d/2} (Q_1 v) (t,t\mathsf{X})
	\label{eq:misc_6x6}
\end{equation}
for 
\begin{equation} 
	Q_1 = N(P_-) - t^{-2}(it\partial_x+x)^2 - t^{-1}B\cdot (it\partial_x+ x).
\end{equation} 
We can now just estimate each term on the right-hand side of \cref{eq:misc_6x6}.

First, we recall the scale of Sobolev spaces from \cite{Parabolicsc, gell2023scattering}:
for $S,L \in \mathbb{R}$, $K \in \mathbb{N}$, $H_{\mathrm{par}}^{S,L;K}$ is the set of $u\in \calS'(\bbR^{1+d})$ such that 
\begin{equation}
	u \in H_{\mathrm{par}}^{S,L}, \quad V_1...V_Ku \in H_{\mathrm{par}}^{S,L},
	\label{eq:misc_0k1}
\end{equation}
for all possible combinations of $V_i \in \{\partial_{x_j}, \partial_t, x_i\partial_{x_j}-x_j\partial_{x_i}, it\partial_{x_j}+x_j, 2t\partial_t + x \cdot \partial_x, \mathrm{Id} \}$, and $H_{\mathrm{par}}^{S,L;K}$ is equipped with the norm
\begin{equation}
	\| u \|_{H_{\mathrm{par}}^{S,L;K}}: 
	= \Big( \sum  \| V_1...V_Ku \|_{H_{\mathrm{par}}^{S,L}}^2 \Big)^{1/2},
\end{equation}
where the sum is taken over those combinations of $V_i$ as above.
Then,
assuming that $v\in \calX_{\mathrm{Schr}}^{\ell',\overline{\mathsf{s}}}$ and $N(P_-)v=0$, we have,  as in the proof of \cite[Thm.\ 6.1]{Parabolicsc}\cite[Thm.\ 1.1]{gell2023scattering},
\begin{equation} 
	v \in H_{\mathrm{par}}^{S,-\frac{1}{2}-\epsilon;K}(\mathbb{R}^{1,d}), \qquad \chi_R\Big( - \frac{x}{t} \Big) \psi(t) v \in \calS(\bbR^{1+d} )
	\label{eq:u_module_regularity}
\end{equation}
for all $T>0$, $S,L\in \mathbb{R}$, $K\in \bbN$, and $\epsilon>0$, and for any $\psi\in C^\infty(\bbR)$ identically $1$ in $(-\infty,-2)$ and identically vanishing on $(-1,\infty)$, where $\chi_R \in C_{\mathrm{c}}^\infty(\bbR^d)$ is identically $1$ on a ball of radius $R$ centered at the origin. The latter statement in \cref{eq:u_module_regularity} is the precise version of the statement that $v$ is Schwartz in backwards cones. The first statement in \cref{eq:u_module_regularity} is the one we will use to control $\partial_t \tilde{\frakv}$.

First, we note that each term of $Q_1$ applied to $v$ is under control:
\begin{itemize}
	\item we are assuming that $N(P_-) v=0$, 
	\item the vector field $it \partial_{x_j} +x_j$ is one of our ``test operators'' considered in \cref{eq:misc_0k1}, and we saw already in \cref{eq:u_module_regularity} that $u$ has iterated par-Sobolev regularity under such operators.
\end{itemize}
Recalling that $B \in S^{-1}(\bbR^{1,d})$, this latter fact gives 
\begin{equation}
	(1+r^2+t^2)^{-1/2}(it\partial_x+x)^2 v, B\cdot (it\partial_x+ x)v \in   H_{\mathrm{par}}^{S,\frac{1}{2}-\epsilon;K}(\bbR^{1,d})
	\label{eq:misc_637}
\end{equation}
for all $S\in \bbR$, $\epsilon>0$, and $K\in \bbN$. We can now use the inclusion  
\begin{equation}
	\bigcap_{S\in \bbR,K\in \bbN} H_{\mathrm{par}}^{S,\frac{1}{2}-\epsilon;K}(\bbR^{1,d}) \hookrightarrow \bigcap_{K\in \bbR} t^{-\frac{(d+1)}{2}+\epsilon } \langle \mathsf{X} \rangle^{-K } L^2( (-\infty,-T)_t\times \bbR^d_{\mathsf{X}} ), 
	\label{eq:misc_9k1}
\end{equation}
where the measure used to define the $L^2$-space on the right-hand side is the Lebesgue measure on $\bbR^{1+d}_{t,\mathsf{X}}$. (The $t^{-d/2}$ factor in \cref{eq:misc_9k1} comes from the Jacobian of the transformation $(t,x)\mapsto (t,\mathsf{X})$.)
A key point is that infinite module regularity, as we are told by \cref{eq:u_module_regularity} that $v$ has with respect to the par-Sobolev spaces with $-1/2-\epsilon$ orders of decay, gives us infinite decay as $\langle \mathsf{X} \rangle\to\infty$. The reasoning is the same as in \cite[\S4.1]{gell2023scattering}. Specifically, 
\cite[Prop.\ 4.2]{gell2023scattering} allows us to convert the module regularity to regularity with respect to the ``flat module'' $\calF$ defined in \cite[eq.\ 4.5]{gell2023scattering}, when restricting to $t< -T$. One of the elements of $\calF$ is $\langle x \rangle / \langle t \rangle \sim \langle \mathsf{X} \rangle$, so iterated regularity under $\calF$ implies decay as $\langle \mathsf{X} \rangle \to\infty$.

The inclusion \cref{eq:misc_9k1}, when combined with \cref{eq:misc_637}, yields 
\begin{multline}
	t^{-2} (it\partial_x+x)^2 v, t^{-1}B\cdot (it\partial_x+ x)v  \in \bigcap_{K\in \bbR} t^{-\frac{(d+3)}{2}+\epsilon } \langle \mathsf{X} \rangle^{-K } L^2( (-\infty,-T)_t\times \bbR^d_{\mathsf{X}} ) \\ = t^{-(d+3)/2+\epsilon} \langle \mathsf{X} \rangle^{-\infty} L^2( (-\infty,-T)_t\times \bbR^d_{\mathsf{X}} ). 
\end{multline}
So, $(2\pi it)^{d/2}Q_1 v \in t^{-3/2+\epsilon} \langle \mathsf{X} \rangle^{-\infty} L^2( (-\infty,-T)_t\times \bbR^d_{\mathsf{X}} )$. 
Finally, because $\Im V, \Im B \in S^{-2}(\bbR^{1,d})$ by \cref{eq:misc_158}, $\exp(i \tilde{V}(t,\mathsf{X}) )\exp(i\mathsf{X} \cdot \tilde{B}(t,\mathsf{X}))   \in L^\infty((-\infty,0)_t\times \bbR^d_{\mathsf{X}} )$, as already discussed.

Plugging all of this into \cref{eq:misc_6x6}, we can conclude that 
\begin{equation}
	\partial_t \tilde{\frakv} \in t^{-3/2+\epsilon} \langle \mathsf{X} \rangle^{-\infty} L^2((-\infty,-T)_t\times \bbR^d_{\mathsf{X}} ).
\end{equation}
An application of the Cauchy--Schwarz inequality shows that $t^{-3/2+\epsilon} \langle \mathsf{X} \rangle^{-\infty} L^2((-\infty,-T)_t\times \bbR^d_{\mathsf{X}} )$ is a subset of 
\begin{equation}
	L^1((-\infty,-T)_t; \langle \mathsf{X} \rangle^{-k} L^2(\bbR^d_{\mathsf{X}}) ) 
\end{equation}
for every $k\in \bbR$. So, $\partial_t \tilde{\frakv} \in L^1((-\infty,-T)_t; \langle \mathsf{X} \rangle^{-k} L^2(\bbR^d_{\mathsf{X}}) )$, as desired.

\begin{remark}
	A conceptual way of thinking about the preceding proof is
	in terms of M{\o}ller/wave maps (cf.\ \cite{ReedSimon}), which, if they make sense at all, should allow us to relate sufficiently nice future/past scattering data with initial data. These are linear maps, so if a solution to the PDE has trivial scattering data in at least one of the future/past, then the solution must vanish identically. 
	The computations above amount to showing that the M{\o}ller/wave maps make sense in the present setting with a domain large enough to handle elements of $\ker_{\calX^{\ell',\overline{\mathsf{s}}}} N(P_\pm)$.
\end{remark}

\begin{remark}  The key fact $\lim_{t \to -\infty} M(t) = 0$  also follows from the proof of \cite[Prop.\ 11]{vasy2020selfadjoint}\footnote{In older versions of \cite{vasy2020selfadjoint}, such as that currently on the arXiv, this is Prop.\ 7.} when $N(P_\pm)$ is symmetric. Indeed, that argument shows that $v$ is Schwartz, not just in backwards spacetime cones but in any spacetime cone.
	Though the referred to proof was for the Klein--Gordon equation and using $\Psi_{\mathrm{sc}}$, it applies to our setting as well, with the sc-calculus replaced by $\Psi_{\mathrm{par}}$. The key point is that the coefficient of $\rho_{\mathrm{bf}}\partial_{\rho_{\mathrm{bf}}}$ in the (rescaled) Hamiltonian vector field has a definite sign near $\calR_{\mathrm{Schr}}$.
	
	The reason that, in this approach, it is necessary to add the requirement that $N(P_\pm)$ be symmetric is that the commutant constructed in \cite[Prop.\ 11]{vasy2020selfadjoint} is not a bounded operator (on any Sobolev space), hence we cannot afford an error term in the integration-by-parts step in the commutator argument.
\end{remark}

\section{Further microlocal propagation estimates} 
\label{sec:further}

In preparation for the sequel \cite{NRL_II} to this article, we prove some variants of the results in \S\ref{sec:estimates}, which we will need to analyze the Cauchy problem. 
\subsection{Estimates with additional microlocalizers}

The next result is a variant of  \Cref{prop:monosheet}, with a weaker assumption on the operator $B$. (When we study the Cauchy problem, it will be useful to allow $B$ to be essentially supported away from certain sets that intersect $\mathrm{pf}$.) Also, the final term on the right-hand side has been improved:
\begin{equation}
	 h^\epsilon\lVert  u \rVert_{H_{\calc}^{-N, \mathsf{s},-N,q  }}  \rightsquigarrow \lVert u \rVert_{H_{\calczero}^{-N,-N,-N} },
\end{equation} 
using a second application of  \Cref{prop:symbolic_global}, \Cref{prop:symbolic_global_better}. We could have carried out the same improvement in \Cref{prop:monosheet}, but it was not necessary for the proof of \Cref{thm:inhomog}. 

\begin{proposition}
	Consider the setup of \Cref{prop:error_bounding2}, with the additional assumption that $B$ is elliptic on $\calR_{\varsigma}$ (the above-threshold radial set), and let $N'\in \bbR$. Then, there exists some $h_0$ (depending on everything except $u,h$), such that
	\begin{multline}
		\lVert B u \rVert_{ H_{\calc}^{m,\mathsf{s},\ell,q} }\lesssim \lVert  G P_\pm u \rVert_{H_{ \calc}^{m-1,\mathsf{s}+1,\ell-1,q}} + \lVert  Z P_\pm u \rVert_{H_{ \calc}^{m-2,\mathsf{s},\ell-2,q}} 
		+ \lVert P_\pm u\rVert_{H_{\calc}^{-N,\mathsf{s}+1,-N,q} }
		\\ + \lVert u \rVert_{H_{\calczero}^{-N,-N,-N} }
		\label{eq:misc_345}
	\end{multline}
	holds, in the usual strong sense, 
	for all $h\in (0,h_0)$ and $u\in H_{\mathrm{sc}}^{-N',-N'}$ such that the incoming/outgoing condition $\operatorname{WF}_{\mathrm{sc}}^{-N',s_0}(u)\cap \calR_{-\varsigma}(h) = \varnothing$ is satisfied. 
	\label{prop:monosheet_NRL2}
\end{proposition}
\begin{proof}
	It suffices to prove the proposition for $N'$ sufficiently large.
	We apply \Cref{prop:symbolic_global} with the improvement noted in \Cref{prop:symbolic_global_better} to get 
	\begin{equation}
		\lVert B u \rVert_{H_{\calc}^{m,\mathsf{s},\ell,q}}   \lesssim \lVert G P_\pm u \rVert_{H_{\calc}^{m-1,\mathsf{s}+1,\ell-1,q}} + \lVert Z P_\pm u \rVert_{H_{\calc}^{m-2,\mathsf{s},\ell-2,q}}  + \lVert u \rVert_{H_{\calc}^{-N',-N',-N',q} } 
		\label{eq:misc_34ss}
	\end{equation}
	for all $N'>N$. Let $\Pi$ be as in \Cref{lem:automatic_threshold}, and we can choose $\Pi \in \Psi_{\calczero}^{0,0,0}$ such that $1-\Pi$ has essential support disjoint from the zero section. Then,
	\begin{multline}
		 \lVert u \rVert_{H_{\calc}^{-N',-N',-N',q} }  \leq \lVert \Pi u \rVert_{H_{\calc}^{-N',-N',-N',q} }  +\lVert (1-\Pi)u \rVert_{H_{\calc}^{-N',-N',-N',q} }   
		 \\
		 \lesssim  \lVert \Pi u \rVert_{H_{\calc}^{-N',-N',-N',q} }  + \lVert u \rVert_{H_{\calczero}^{-N,-N,-N}}.
	\end{multline}
	We just need to bound the $\Pi u$ term. 
	
	If $h$ is such that the right-hand side of \cref{eq:misc_345} is finite (which is the only case where the estimate is nontrivial), then \cref{eq:misc_34ss} gives that $Bu(h) \in H_{\mathrm{sc}}^{m,\mathsf{s}}$, in which case \Cref{lem:automatic_threshold} says that $\operatorname{WF}_{\mathrm{par}}^{\ell_0,s_0 }(\Pi u(h)) \cap \calR_{-\varsigma}^{\mathrm{Schr}}=\varnothing$ for any $\ell_0$. 
	Then, \Cref{prop:error_bounding2} applies, giving the bound
	\begin{equation}
		\lVert \Pi u \rVert_{H_{\calc}^{-N',-N',-N',q} } \lesssim \lVert P_\pm  u \rVert_{H_{\calc}^{-N'-1,\mathsf{s}+1,-N'-1,q} } + h^\epsilon\lVert  Ou \rVert_{H_{\calc}^{-N', \mathsf{s},-N'+2,q  }} + \lVert u \rVert_{H_{\calczero}^{-N,-N,-N}}.
	\end{equation}
	We can choose $O$ to have essential support disjoint from $\Sigma_{\mathrm{bad}}$, in which case we can estimate $Ou$ using \Cref{prop:symbolic_global}, \Cref{prop:symbolic_global_better} (this time with $G=1,Z=0$)  to get
	\begin{equation}
		\lVert O u \rVert_{H_{\calc}^{-N',\mathsf{s},-N'+2,q}}   \lesssim \lVert  P_\pm u \rVert_{H_{\calc}^{-N'-1,\mathsf{s}+1,-N'+1,q}} +  \lVert u \rVert_{H_{\calc}^{-N',-N',-N',q} } .
	\end{equation}
	Combining these, we get
	\begin{equation}
		\lVert u \rVert_{H_{\calc}^{-N',-N',-N',q} } \lesssim \lVert P_\pm  u \rVert_{H_{\calc}^{-N'-1,\mathsf{s}+1,-N',q} }  + \lVert u \rVert_{H_{\calczero}^{-N,-N,-N}} + 	h^\epsilon\lVert u \rVert_{H_{\calc}^{-N',-N',-N',q} }.
	\end{equation}
	Note that the last term is finite (since $u\in H_{\mathrm{sc}}^{-N',-N'}$), so, if $h$ is sufficiently small, we can absorb it into the left-hand side to get 
	\begin{equation}
		\lVert u \rVert_{H_{\calc}^{-N',-N',-N',q} } \lesssim \lVert P_\pm  u \rVert_{H_{\calc}^{-N'-1,\mathsf{s}+1,-N',q} }  + \lVert u \rVert_{H_{\calczero}^{-N,-N,-N}}.
	\end{equation}
	Plugging this into \cref{eq:misc_34ss} and taking $N'$ sufficiently large, we get \cref{eq:misc_345}. 
\end{proof}

The following lemma bounds $P_\pm u_\pm$ (the result of applying the conjugated differential operator to the conjugated $\pm$-energy portion of $u$, $u_\pm = e^{\mp i c^2 t} Q_\pm u $) in terms of $(Pu)_\pm$ (the conjugation of the $\pm$-energy portion of $Pu$):
\begin{lemma}
	Let $Q=Q_-$ be as in \S\ref{subsec:combined}.
	Let $G_\pm\in \Psi_{\calc}^{0,0,0,0}$, $Z_\pm\in \Psi_{\calczero}^{0,0,0}$ satisfy 
	\begin{equation}
		\operatorname{Ell}_{\calczero}^{0,0,0}(Z_\pm) \supseteq \operatorname{WF}'_{\calc}(G_\pm)_{\mp 1}\cap \operatorname{WF}'_{\calczero}(Q)\cap\operatorname{WF}'_{\calczero}(1-Q).
		\label{eq:misc_370}
	\end{equation}
	Then, for any $m,\ell,q,N\in \bbR$ and variable orders $\mathsf{s}_\pm$, 
	\begin{equation}
		\lVert G_\pm P_\pm u_\pm \rVert_{H_{\calc}^{m-1,\mathsf{s}_\pm+1,\ell-1,q}} \lesssim \lVert G_\pm (Pu)_\pm \rVert_{H_{\calc}^{m-1,\mathsf{s}_\pm+1,\ell-1,q}}+ \lVert Z_\pm P u \rVert_{H_{\calczero}^{m-2,\mathsf{s}_\pm,\ell-2} } + \lVert u \rVert_{H_{\calczero}^{-N,-N,-N}}
		\label{eq:misc_380}
	\end{equation}
	for all $u\in \calS'$. (Here, $\bullet_\pm$ are defined by \cref{eq:vpm_def}.) Consequently, letting $G \in \Psi_{\calctwo}^{0,0,0;0,0}$ be defined by 
	\begin{equation} 
		G = M_{\exp(-ic^2 t)} G_- M_{\exp(ic^2 t)}Q+ M_{\exp(ic^2 t)} G_+ M_{\exp(-ic^2 t)}(1-Q),
	\end{equation}  
	and letting $Z\in \Psi_{\calczero}^{0,0,0}$ be elliptic on the elliptic sets of both $Z_\pm$, 
	then  
	\begin{equation}
		\lVert G_\pm P_\pm u_\pm \rVert_{H_{\calc}^{m-1,\mathsf{s}_\pm+1,\ell-1,0}} \lesssim \lVert G Pu \rVert_{H_{\calctwo}^{m-1,\mathsf{s}+1,\ell-1;0,0}}+ \lVert Z P u \rVert_{H_{\calczero}^{m-2,\mathsf{s},\ell-2} } + \lVert u \rVert_{H_{\calczero}^{-N,-N,-N}}
		\label{eq:misc_382}
	\end{equation}
	holds. 
	\label{lem:forcing_bound}
\end{lemma}
Note that $\operatorname{WF}'_{\calczero}(Q)_{\pm 1}\cap\operatorname{WF}'_{\calczero}(1-Q)_{\pm 1}$ is disjoint from the zero section, so the right-hand side of \cref{eq:misc_370}, which equals to $\big( \operatorname{WF}'_{\calc}(G_\pm) \cap \operatorname{WF}'_{\calczero}(Q)_{\pm 1} \cap\operatorname{WF}'_{\calczero}(1-Q)_{\pm 1} \big)_{\mp 1}$, 
 makes sense as a subset of ${}^{\calczero}\overline{T}^* \bbM$ even though $\operatorname{WF}'_{\calc}(G_\pm)$ is only a subset of ${}^{\calc}\overline{T}^* \bbM$. 
\begin{proof}
	We discuss the $-$ case, and the $+$ case is analogous, and we will drop the subscripts on $G_\bullet,Z_\bullet,\mathsf{s}_\bullet$. 
	We have 
	\begin{align}
		\begin{split} 
		\lVert  GP_-u_- \rVert_{H_{\calc}^{m-1,\mathsf{s}+1,\ell-1,q}} &\leq 
		\lVert G(Pu)_- \rVert_{H_{\calc}^{m-1,\mathsf{s}+1,\ell-1,q}}+
		\lVert Gf_- \rVert_{H_{\calc}^{m-1,\mathsf{s}+1,\ell-1,q}} \\ 
		&= \lVert G(Pu)_- \rVert_{H_{\calc}^{m-1,\mathsf{s}+1,\ell-1,q}} + \lVert  G M_{e^{ic^2 t}}  [P,Q] M_{e^{-ic^2 t}} ( e^{ic^2 t}u) \rVert_{H_{\calc}^{m-1,\mathsf{s}+1,\ell-1,q}},
		\end{split}
	\end{align}
	where $f_-$ is as in \cref{eq:f}.
	Since $P,Q \in \Psi_{\calczero}$, we have  $M_{\exp(ic^2 t)}  [P,Q] M_{\exp(-ic^2 t)} \in \Psi_{\calczero}$, and \Cref{prop:translation} tells us that its essential support satisfies 
	\begin{equation}
		\operatorname{WF}'_{\calczero}(M_{\exp(ic^2 t)}  [P,Q] M_{\exp(-ic^2 t)}) =\operatorname{WF}_{\calczero}'([P,Q])_{+1} \subseteq \operatorname{WF}_{\calczero}'(1-Q)_{+1}.
	\end{equation}
	This set is disjoint from a neighborhood of the zero section. Consequently,
	\begin{equation}
		GM_{\exp(ic^2 t)}  [P,Q] M_{\exp(-ic^2 t)}\in \Psi_{\calc}^{1,-1,1,-\infty}. 
	\end{equation}
	Let $Q_0 \in \Psi_{\calczero}^{0,0,0}$ be elliptic on $\operatorname{WF}_{\calczero}'(GM_{\exp(ic^2 t)}  [P,Q] M_{\exp(-ic^2 t)})$. 
	Then, we have the elliptic bound
	\begin{equation}
		\lVert   GM_{e^{ic^2 t}}  [P,Q] M_{e^{-ic^2 t}} ( e^{ic^2 t}u) \rVert_{H_{\calc}^{m-1,\mathsf{s}+1,\ell-1,q}}\lesssim \lVert  Q_0 (e^{ic^2 t} u) \rVert_{H_{\calczero}^{m,\mathsf{s},\ell} } + \lVert u \rVert_{H_{\calczero}^{-N,-N,-N}}
	\end{equation}
	for any $N$.

	By \cref{eq:misc_370}, we have
	\begin{equation}
		\operatorname{WF}_{\calczero}'(GM_{\exp(ic^2 t)}  [P,Q] M_{\exp(-ic^2 t)}) \subseteq \operatorname{WF}'_{\calczero}(G)\cap 	\operatorname{WF}'_{\calczero}(Q)_{+ 1}\cap\operatorname{WF}'_{\calczero}(1-Q)_{+ 1} \Subset \operatorname{Ell}_\natural^{0,0,0}(Z)_{+1}.
	\end{equation}
	Actually, since the $\natural$-characteristic set of $P$ is disjoint from $\operatorname{WF}'_\calczero(Q)\cap \operatorname{WF}'_\calczero(1-Q)$, 
	we can improve this to 
	\begin{equation}
		\operatorname{WF}_{\calczero}'(GM_{\exp(ic^2 t)}  [P,Q] M_{\exp(-ic^2 t)}) \subseteq \operatorname{WF}'_{\calczero}(G)\cap 	\operatorname{WF}'_{\calczero}(Q)_{+ 1}\cap\operatorname{WF}'_{\calczero}(1-Q)_{+ 1} \Subset \operatorname{Ell}_\natural^{2,0,2}(ZP)_{+1}.
	\end{equation}
	Consequently, we can choose $Q_0$ such that $\operatorname{WF}'_{\calczero}(Q_0)\Subset \operatorname{Ell}_{\calczero}^{2,0,2}(ZP)_{+1}$, which is equivalent to
	\begin{equation} 
		\operatorname{WF}'_{\calczero}(Q_0)_{-1}\Subset \operatorname{Ell}_{\calczero}^{2,0,2}(ZP).
	\end{equation}
	Then, 
	\begin{equation}
		\lVert  Q_0 (e^{ic^2 t} u) \rVert_{H_{\calczero}^{m,\mathsf{s},\ell} }\lesssim \lVert  M_{e^{-ic^2 t}} Q_0 (e^{ic^2 t} u) \rVert_{H_{\calczero}^{m,\mathsf{s},\ell} } \lesssim \lVert ZPu \rVert_{H_{\calczero}^{m-2,\mathsf{s},\ell-2} } + \lVert u \rVert_{H_{\calczero}^{-N,-N,-N}}.
	\end{equation}
	Summing up, we have the desired estimate. 
\end{proof}

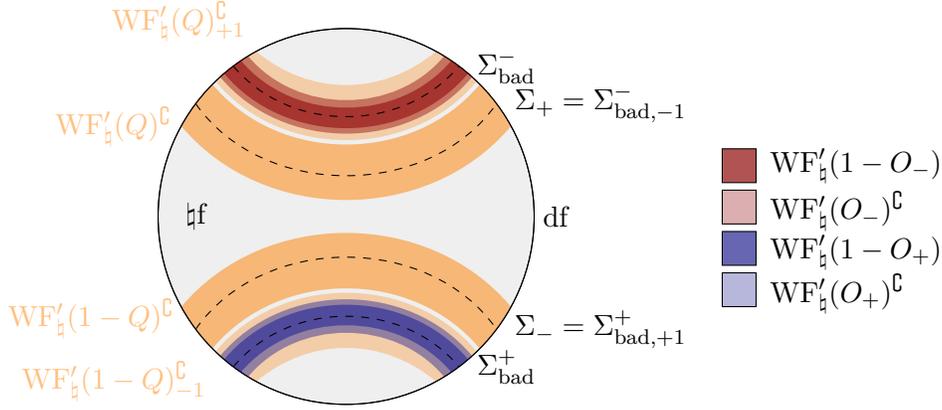
\begin{figure}
	\begin{tikzpicture}
		\filldraw[fill=lightgray!25] (0,0) circle (2.5);
		\fill[orange, opacity=.5] (1.77,1.75) to[out=230, in=-50] (-1.77,1.75) to[out=220, in=60] (-2.19,1.2) node[left] {$\operatorname{WF}'_{\calczero}(Q)^\complement $} to[out=-50,in=230] (2.19,1.2) to[out=125, in=-37] cycle;
		\fill[orange, opacity=.3] (1.73,1.79) to[out=230, in=-50] (-1.73,1.79) to[out=50, in=210] (-1.2,2.2) node[above left, opacity=.5] {$\operatorname{WF}'_{\calczero}(Q)^\complement_{+1} $} to[out=-40, in=220] (1.2,2.2) to[out=-30, in=-230] cycle;
		\fill[orange, opacity=.3] (1.73,-1.79) to[out=-230, in=50] (-1.73,-1.79) node[below left, opacity=.5] {$\operatorname{WF}'_{\calczero}(1-Q)^\complement_{-1} $} to[out=-50, in=140] (-1.2,-2.2) to[out=40, in=-220] (1.2,-2.2) to[in=220, out=30] cycle;
		\fill[orange, opacity=.5] (1.77,-1.75) to[out=-230, in=50] (-1.77,-1.75) node[above left] {$\operatorname{WF}'_{\calczero}(1-Q)^\complement\quad $} to[out=-220, in=-60] (-2.19,-1.2) to[out=50,in=-230] (2.19,-1.2) to[out=-125, in=37] (1.77,-1.75);
		\fill[darkblue, opacity=.5] (-1.3,-2.13) to[out=50,in=-230] (1.3,-2.13) to[out=30, in=-130] (1.67,-1.85) to[out=130,in=50] (-1.67,-1.85)  to[out=-50, in=150] cycle;
		\fill[darkred, opacity=.5] (-1.3,2.13) to[out=-50,in=230] (1.3,2.13) to[out=-30, in=130] (1.67,1.85) to[out=-130,in=-50] (-1.67,1.85)  to[out=50, in=-150] cycle;
		\fill[darkblue, opacity=.7] (-1.4,-2.075) to[out=50,in=-230] (1.4,-2.075) --  (1.62,-1.9) to[out=130,in=50] (-1.62,-1.9)  -- cycle;
		\fill[darkred, opacity=.7] (-1.4,2.075) to[out=-50,in=230] (1.4,2.075) --  (1.62,1.9) to[out=-130,in=-50] (-1.62,1.9)  -- cycle;
		\draw[dashed] (-1.5,2) to[out=-50,in=230] (1.5,2) node[right] {$\;\Sigma^-_{\mathrm{bad}}$};
		\draw[dashed] (-2,1.5) to[out=-55,in=235] (2,1.5) node[right] {$\;\Sigma_+=\Sigma^-_{\mathrm{bad},-1}$};
		\draw[dashed] (-1.5,-2) to[out=50,in=-230] (1.5,-2)  node[right] {$\;\Sigma^+_{\mathrm{bad}}$};
		\draw[dashed] (-2,-1.5) to[out=55,in=-235] (2,-1.5) node[right] {$\;\Sigma_-=\Sigma^+_{\mathrm{bad},+1}$};
		\node at (2.8,0) {$\mathrm{df}$};
		\draw (0,0) circle (2.5);
		\node at (-2,0) {$\natural$f};
		\begin{scope}[shift={(2,-.75)}]
			\filldraw[fill= darkred, fill opacity=.75] (3,1.2) rectangle (3.45,1.65); 
			\filldraw[fill= darkred, fill opacity=.35] (3,.65) rectangle (3.45,1.1); 
			\filldraw[fill= darkblue, fill opacity=.75] (3,.1) rectangle (3.45,0.55); 
			\filldraw[fill= darkblue, fill opacity=.35] (3,-.45) rectangle (3.45,0); 
			\node[right] at (3.5,1.425) {$\operatorname{WF}'_{\calczero}(1-O_-)$};
			\node[right] at (3.5,-.275) {$\operatorname{WF}'_{\calczero}(O_+)^\complement$};
			\node[right] at (3.5,.825) {$\operatorname{WF}'_{\calczero}(O_-)^\complement$};
			\node[right] at (3.5,.275) {$\operatorname{WF}'_{\calczero}(1-O_+)$};
		\end{scope} 
	\end{tikzpicture}
	\caption{Various subsets of ${}^{\calczero}\overline{T}^* \bbM$ mentioned in \Cref{lem:forcing_bound}, \Cref{prop:horrid} and their proofs. } 
	\label{fig:horrid}
\end{figure}

\begin{proposition}\label{prop:horrid}
	 Choose $\varsigma_-,\varsigma_+\in \{-,+\}$.
	Let $G_\bullet,Z_\bullet,B_\bullet \in \Psi_{\calc}^{0,0,0,0}$ satisfy:
	\begin{itemize}		
		\item  $\operatorname{Ell}_{\calc}^{0,0,0,0}(B_\pm) \supseteq \calR_{-\varsigma_\pm}$ , 
		\item $\operatorname{Ell}_{\calc}^{0,0,0,0}(G_\pm) \supseteq \Sigma^\pm$ (note that this is different from $\Sigma_\pm$; the latter are the components of the $\natural$-characteristic set of $P$, whereas the former are the good components of the $\calc$-characteristic sets of $P_\pm$), 
		\item $G_\pm$, $Z_\pm$ are elliptic on a sufficiently large set. Specifically:
		\begin{equation} 
			\operatorname{WF}'_{\calc}(B_\pm)\cup \big(\operatorname{WF}'_{\calc}(G_\pm)_{\mp 1}\cap \operatorname{WF}'_{\calczero}(Q)\cap\operatorname{WF}'_{\calczero}(1-Q) \big) \\ 
			\subseteq \operatorname{Ell}_{\calc}^{0,0,0,0}(G_\pm) \cup \operatorname{Ell}_{\calczero}^{0,0,0}(Z_\pm).
			\label{eq:misc_390}
			\end{equation} 
	\end{itemize}
	(For example, we can take $G_\pm,Z_\pm,B_\pm = 1$.)
	Let $\mathsf{s}_\pm,\in C^\infty({}^{\calc} \overline{T}^* \bbM;\bbR)$ be variable orders such that
	\begin{equation}
		\mathsf{s}_\pm|_{\calR^\pm_{-\varsigma_\pm}}  > -1/2 > 	\mathsf{s}_\pm|_{\calR^\pm_{\varsigma_\pm}}  .
		\label{eq:misc_331}
	\end{equation}
	In addition, suppose that $\mathsf{s}_\pm$ are monotonic under the Hamiltonian flow $\mathsf{H}_{p^\pm}$ near $\Sigma^\pm$. Let $m,\ell,q,N,m_0\in \bbR$, $s_0>-1/2$. Also, let $\mathsf{S}$ be a variable order such that $\mathsf{S}\geq \mathsf{s}_+,\mathsf{s}_-$. 
	
	Given all of this, there exists an $h_0>0$ such that 
	\begin{align}
		\begin{split} 
		\lVert  B_- u_- \rVert_{H_{\calc}^{m,\mathsf{s}_-,\ell,q}} +\lVert B_+ u_+ \rVert_{H_{\calc}^{m,\mathsf{s}_+,\ell,q}} &\lesssim \lVert G_- (Pu)_+ \rVert_{H_{\calc}^{m-1,\mathsf{s}_-+1,\ell-1,q}} + \lVert G_+ (Pu)_+ \rVert_{H_{\calc}^{m-1,\mathsf{s}_++1,\ell-1,q}}\\ 
		 &+\lVert Z_- (Pu)_- \rVert_{H_{\calc}^{m-2,\mathsf{s}_-,\ell-2,q}} + \lVert Z_+ (Pu)_+ \rVert_{H_{\calc}^{m-2,\mathsf{s}_+,\ell-2,q}}  \\ + \lVert  (Pu)_- \rVert_{H_{\calc}^{-N,\mathsf{s}_-+1,-N,q}} &+ \lVert  (Pu)_+ \rVert_{H_{\calc}^{-N,\mathsf{s}_++1,-N,q}} + \lVert Pu \rVert_{H_{\calczero}^{-N,\mathsf{S},-N}}
		 \end{split}
		 \label{eq:main}
	\end{align}
	holds (in the usual strong sense) for all $u\in \calS'$ and $h\in (0,h_0)$ such that $u$ satisfies the ``incoming/outgoing'' condition 
	\begin{equation}
		\operatorname{WF}_{\mathrm{sc}}^{m_0,s_0}(u(h) ) \cap (\calR^{\mathrm{KG},-}_{-\varsigma_-}\cup \calR^{\mathrm{KG},+}_{-\varsigma_+}) = \varnothing,
		\label{eq:misc_332}
	\end{equation}
	where $\calR^{\mathrm{KG},\pm}_\varsigma$ is the radial set of the Klein--Gordon equation in the component of the characteristic set of energy $\pm \tau>0$ and over the hemisphere of $\partial \bbM$ with $\varsigma t>0$. 
\end{proposition}
When $\varsigma_-=\varsigma_+$, then we are considering one of the advanced/retarded problems. Otherwise, we are considering one of the Feynman/anti-Feynman problems.

\begin{proof}
	Let $O_\pm \in \Psi_{\calczero}^{0,0,0}$ satisfy $\operatorname{WF}'_{\calczero}(O_\pm)\cap \Sigma_{\mathrm{bad}}^\pm = \varnothing$.  We have
	\begin{equation}
		\lVert B_\pm u_- \rVert_{H_{\calc}^{m,\mathsf{s}_\pm,\ell,q}} \leq  \lVert  B_\pm O_\pm u_\pm \rVert_{H_{\calc}^{m,\mathsf{s}_\pm,\ell,q}} + \lVert B_\pm R_\pm  (e^{\mp ic^2 t} u) \rVert_{H_{\calc}^{m,\mathsf{s}_\pm,\ell,q}}
	\end{equation}
	for 
	\begin{equation}
		R_- = (1-O_-) M_{e^{ic^2 t}} Q M_{e^{-ic^2 t}},\quad 
		R_+ = (1-O_+) M_{e^{-ic^2 t}} (1-Q) M_{e^{ic^2 t}} .
		\label{eq:R_cons}
	\end{equation}
	By \Cref{prop:translation}, $R_\pm \in \Psi_{\calc}^{0,0,0,0}$. Suppose that we can choose $O_\pm$ such that 
	\begin{equation}
		R_\pm \in \Psi_{\calczero}^{-\infty,-\infty,-\infty} .
		\label{eq:R_resid_cond}
	\end{equation}
	Then, we get
	\begin{multline}
		\lVert  B_- u_- \rVert_{H_{\calc}^{m,\mathsf{s}_-,\ell,q}} +\lVert B_+ u_+ \rVert_{H_{\calc}^{m,\mathsf{s}_+,\ell,q}}  \lesssim \lVert  B_- O_- u_- \rVert_{H_{\calc}^{m,\mathsf{s}_-,\ell,q}} +\lVert  B_+ O_+u_+ \rVert_{H_{\calc}^{m,\mathsf{s}_+,\ell,q}} \\ + \lVert u \rVert_{H_{\calczero}^{-N',-N',-N'} }
		\label{eq:misc_38s}
	\end{multline}
	for any $N'$. 
	We will come back later to why we can choose $O_\pm,R_\pm$ such that this holds. First, let us see how \cref{eq:misc_38s} yields the proposition. 
	
	Because $\operatorname{WF}'_{\calczero}(O_\pm)\cap \Sigma_{\mathrm{bad}}^\pm = \varnothing$, we can apply \Cref{prop:monosheet} with $B_\pm O_\pm$ in place of $ B$.
	There exists $\tilde{Z}_\pm$ satisfying the same hypotheses as $Z_\pm$ in that proposition, but with essential support contained in the elliptic set of $Z_\pm$. 
	The operators $B_\pm O_\pm ,G_\pm,\tilde{Z}_\pm$ satisfy the hypotheses of \Cref{prop:monosheet}, which then gives
	\begin{multline}
		\lVert  B_- O_- u_- \rVert_{H_{\calc}^{m,\mathsf{s}_-,\ell,q}} +\lVert  B_+ O_+u_+ \rVert_{H_{\calc}^{m,\mathsf{s}_+,\ell,q}} \lesssim \lVert G_- P_- u_- \rVert_{H_{\calc}^{m-1,\mathsf{s}_-+1,\ell-1,q}} + \lVert G_+ P_+u_+ \rVert_{H_{\calc}^{m-1,\mathsf{s}_++1,\ell-1,q}}   \\
		+\lVert  \tilde{Z}_- P_- u_- \rVert_{H_{ \calc}^{m-2,\mathsf{s}_-,\ell-2,q}} +\lVert  \tilde{Z}_+ P_+ u_+ \rVert_{H_{ \calc}^{m-2,\mathsf{s}_+,\ell-2,q}} 
		\\+ \lVert P_- u_-\rVert_{H_{\calc}^{-N,\mathsf{s}_-+1,-N,q} }+ \lVert P_+ u_+\rVert_{H_{\calc}^{-N,\mathsf{s}_++1,-N,q} }
		 + \lVert u \rVert_{H_{\calczero}^{-N',-N',-N'} }
		\label{eq:misc_384}
	\end{multline}
	for any $N'\in \bbR$. (Apply that lemma for $N'$ larger than $N$.)
	We now apply \Cref{lem:forcing_bound} to bound each of the terms involving $P_\pm u_\pm$ on the right-hand side. So, we apply that lemma once with the $G,Z$ being the $G_\pm,Z_\pm$ here, once with the $G,Z$ in the lemma being $\tilde{Z}_\pm,Z_\pm$, and once with them both being the identity. We then get \cref{eq:main} except with $ B_\pm O_\pm$ in place of $B_\pm$ on the left-hand side and an extra $\lVert u \rVert_{H_{\calczero}^{-N,-N,-N} }$ on the right-hand side. 
	
	Combining this with \cref{eq:misc_38s}, we get \cref{eq:main} exactly, except for the extra  $\lVert u \rVert_{H_{\calczero}^{-N,-N,-N} }$ on the right-hand side. 
	
	So, what we want to show is that, if $N'$ is sufficiently large, and $h$ is sufficiently small, then 
	\begin{equation}
		\lVert u \rVert_{H_{\calczero}^{-N',-N',-N'} } \lesssim \lVert P_- u_-\rVert_{H_{\calc}^{-N,\mathsf{s}_-+1,-N,q} }+ \lVert P_+ u_+\rVert_{H_{\calc}^{-N,\mathsf{s}_++1,-N,q} }.
		\label{eq:misc_385}
	\end{equation}
	We have
	\begin{equation}
		\lVert u \rVert_{H_{\calczero}^{-N',-N',-N'} } \lesssim \lVert u_- \rVert_{H_{\calczero}^{-N',-N',-N'} } + \lVert u_+ \rVert_{H_{\calczero}^{-N',-N',-N'} }  \lesssim \lVert u_- \rVert_{H_{\calczero}^{-N',\mathsf{s}_-,-N'} } + \lVert u_+ \rVert_{H_{\calczero}^{-N',\mathsf{s}_+,-N'} }
		\label{eq:misc_386}
	\end{equation}
	for $N'$ sufficiently large. Applying \cref{eq:misc_384} (with $B_\pm = 1$, $\tilde{Z}_\pm = 0$, $G_\pm=1$ and, some $N''$ in place of $N'$), we have 
	\begin{multline}
		\lVert u_- \rVert_{H_{\calczero}^{-N',\mathsf{s}_-,-N'} } + \lVert u_+ \rVert_{H_{\calczero}^{-N',\mathsf{s}_+,-N'} } \lesssim \lVert P_- u_-\rVert_{H_{\calc}^{-N,\mathsf{s}_-+1,-N,q} }+ \lVert P_+ u_+\rVert_{H_{\calc}^{-N,\mathsf{s}_++1,-N,q} }  \\ 
		+ \lVert u \rVert_{H_{\calc}^{-N'',-N'',-N''}}.
		\label{eq:misc_387}
	\end{multline}
	The last term satisfies 
	\begin{multline}
		\lVert u \rVert_{H_{\calc}^{-N'',-N'',-N''}} \lesssim \lVert u_- \rVert_{H_{\calc}^{-N'',-N'',-N''}}+\lVert u_+ \rVert_{H_{\calc}^{-N'',-N'',-N''}} \\ \leq h(\lVert u_- \rVert_{H_{\calczero}^{-N',\mathsf{s}_-,-N'} } + \lVert u_+ \rVert_{H_{\calczero}^{-N',\mathsf{s}_+,-N'} }) 
	\end{multline}
	if $N''$ is sufficiently large. So, we can absorb the last term on the right-hand side of \cref{eq:misc_387} into the left-hand side if $h$ is sufficiently small, and then combining \cref{eq:misc_387} with \cref{eq:misc_386} yields the desired \cref{eq:misc_385}. 
	
	This completes the proof once $O_\pm \in \Psi_{\calczero}^{0,0,0}$ with the properties required above are known to exist. We will choose $O_\pm $ such that
	\begin{align}
		\operatorname{WF}'_{\calczero}(O_\pm) \cap \Sigma^\pm_{\mathrm{bad}} &= \varnothing, \label{eq:misc_782}\\
		\operatorname{WF}'_{\calczero}(1-O_-) \cap \operatorname{WF}_{\calczero}'(Q)_{+1} &= \varnothing, \label{eq:misc_783} \\
		\operatorname{WF}'_{\calczero} (1-O_+)\cap \operatorname{WF}'_{\calczero}(1-Q)_{-1} &= \varnothing. \label{eq:misc_784}
	\end{align}
	Then, using \Cref{prop:translation} to compute the effect of conjugation on essential supports, 
	\begin{align}
		\begin{split} 
		\operatorname{WF}'_{\calczero}(R_-) &\subseteq \operatorname{WF}_{\calczero}'(1-O_-) \cap \operatorname{WF}'_{\calczero}(M_{e^{ic^2 t}} Q M_{e^{-ic^2 t}}) \\ 
		&=  \operatorname{WF}'_{\calczero}(1-O_-) \cap \operatorname{WF}_{\calczero}'(Q)_{+1} = \varnothing ,
		\end{split} 
	\intertext{by \cref{eq:misc_783}. Similarly, }
	\begin{split} 
		\operatorname{WF}'_{\calczero}(R_+) &\subseteq \operatorname{WF}_{\calczero}'(1-O_+) \cap \operatorname{WF}'_{\calczero}(M_{e^{-ic^2 t}} (1-Q) M_{e^{ic^2 t}}) \\ 
		&=  \operatorname{WF}'_{\calczero}(1-O_+) \cap \operatorname{WF}_{\calczero}'(1-Q)_{-1} = \varnothing ,
	\end{split} 
	\end{align} 
	by \cref{eq:misc_784}. So, \cref{eq:R_resid_cond} holds. The only other required condition on $O_\pm$ was \cref{eq:misc_782}. So, if we can construct $O_\pm$ satisfying \cref{eq:misc_782}, \cref{eq:misc_783}, and \cref{eq:misc_784}, then we are done.
	This can be done via the explicit quantization of a symbol if $\Sigma^+_{\mathrm{bad}}$ is disjoint from $\operatorname{WF}'_{\calczero}(1-Q)_{-1}$ and $\Sigma^-_{\mathrm{bad}}$ is disjoint from $\operatorname{WF}'_{\calczero}(Q)_{+1}$. Indeed, this is just a restatement of \cref{eq:misc_320}, translating things by one unit in $\tau_{\natural}$. 
\end{proof}

\subsection{Radial point estimates for the Cauchy problem}

When we study the Cauchy problem, it will be useful to have an estimate that allows one to propagate control from the middle of phase space outwards and into the various radial sets (at below-threshold decay).  When studying the Klein--Gordon equation for fixed $c$, ``middle' can be taken to mean the portion of the phase space over the closure of $\{t=0\}$. This does not quite work when studying the $c\to\infty$ limit because $\calR_\varsigma$ lie over the equator of $\partial \bbM$ at $\natural\mathrm{f}$. (This is because the Schr\"odinger radial sets $\calR_{\mathrm{Schr}}$ do the same thing. See \Cref{fig:individual_flows}(b).) What we do instead is propagate control from $\Sigma\backslash (U_-\cup U_+)$ throughout $\Sigma$ for some neighborhoods $U_\pm$ of $\calR_\pm$. (These $U_\pm$ can be quite large, but not large enough to together cover $\Sigma$.) Specifically:

\begin{proposition}
	Let  $U_\pm$ be neighbourhoods of $\calR_\pm$ such that $U_+ \cap U_- = \varnothing$.
	Suppose that $m,\ell,q\in \bbR$ and that $\mathsf{s}\in C^\infty({}^{\calc}\overline{T}^* \bbM)$ satisfies both of
	\begin{equation} 
		\mathsf{s}|_{\calR_{-}},\mathsf{s}|_{\calR_+} < -1/2.
	\end{equation}  
	Let $G,B,E,Z\in \Psi^{0,0,0,0}_{\calc}$ satisfy
	\begin{itemize}
		\item $\Sigma\subset \operatorname{Ell}_{\calc}^{0,0,0,0}(G)$, 
		\item  $\operatorname{WF}'_{\calc}(B) \subseteq \Big( \operatorname{Ell}_{\calc}^{0,0,0,0}(G) \cup \operatorname{Ell}_{\calc}^{0,0,0,0}(Z) \Big) \backslash \Sigma_{\mathrm{bad}}$ , 
		\item $\operatorname{Ell}_{\calc}^{0,0,0,0}(E)\supseteq \Sigma\backslash (U_-\cup U_+)$. 
	\end{itemize}
	In addition, suppose that $\mathsf{s}$ is monotonic under the Hamiltonian flow except possibly inside of some subset $K\Subset \operatorname{Ell}_{\calc}^{0,0,0,0}(E)$. 
	Then, for every $N,\in \bbR$, the estimate 
	\begin{multline}
		\lVert B u \rVert_{H_{\calc}^{m,\mathsf{s},\ell,q}}   \lesssim \lVert G P_\pm u \rVert_{H_{\calc}^{m-1,\mathsf{s}+1,\ell-1,q}} + \lVert Z P_\pm u \rVert_{H_{\calc}^{m-2,\mathsf{s},\ell-2,q}}   + \lVert E u \rVert_{ H_{\calc}^{m,\mathsf{s},\ell,q}}  + \lVert u \rVert_{H_{\calc}^{-N,-N,-N,q} } 
	\end{multline}
	holds, in the usual strong sense.
	\label{prop:symbolic_estimate_for_Cauchy}
\end{proposition}

This can be proven by gluing two below-threshold radial point estimates from \Cref{prop:propagation_in}. 
The case of interest is when $B$ is elliptic near $\mathcal{R}_\pm$ and $E$ is microlocally trivial there. Notice that the elliptic set of $E$ must intersect every bicharacteristic moving between $\calR_-$ and $\calR_+$ due to the assumed disjointness of the sets $U_-$ and $U_+$, allowing regularity to be propagated from the elliptic set of $E$ along each bicharacteristic. 

\begin{figure}
	\includegraphics{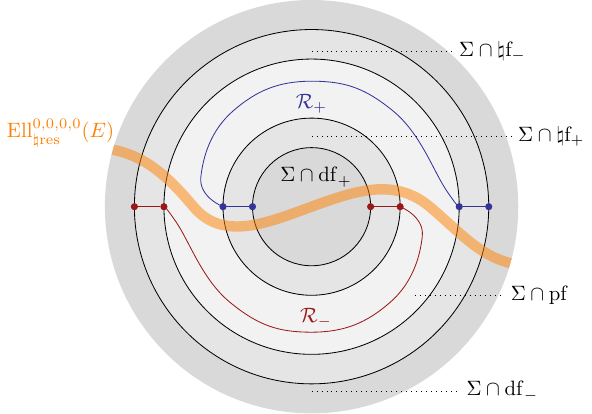}
	\caption{The sort of set on which the $E$ in \Cref{prop:symbolic_estimate_for_Cauchy} is required to be elliptic.}
\end{figure} 

In the same way that the error term in \Cref{prop:monosheet} could be improved to that in \Cref{prop:monosheet_NRL2}, the error term in the previous proposition can be improved slightly, at the cost of a global term involving the forcing:
\begin{proposition}
	Consider the setup of \Cref{prop:symbolic_estimate_for_Cauchy}. There exists some $h_0$ such that 
	\begin{multline}
		\lVert B u \rVert_{ H_{\calc}^{m,\mathsf{s},\ell,q} }\lesssim \lVert  G P_\pm u \rVert_{H_{ \calc}^{m-1,\mathsf{s}+1,\ell-1,q}} + \lVert  Z P_\pm u \rVert_{H_{ \calc}^{m-2,\mathsf{s},\ell-2,q}} 
		+ \lVert P_\pm u\rVert_{H_{\calc}^{-N,\mathsf{s}+1,-N,q} }
		\\ +\lVert E u \rVert_{ H_{\calc}^{m,\mathsf{s},\ell,q} } + \lVert u \rVert_{H_{\calczero}^{-N,-N,-N} }
	\end{multline}
	holds, in the usual strong sense, 
	for all $h\in (0,h_0)$ and $u\in \calS'$. 
	\label{prop:monosheet_cauchy}
\end{proposition}
\begin{proof}
	Same as \Cref{prop:monosheet_NRL2}, but without needing to worry about the threshold conditions. 
\end{proof}

We can combine the $\pm$ cases of this estimate, yielding an estimate like \Cref{prop:horrid}. The following simpler statement suffices for our purpose:

\begin{proposition}
	Let $\mathsf{s},E$ satisfy the assumptions of \Cref{prop:monosheet_cauchy} on each component of the characteristic set.
	Then, for any $m,\ell,q\in \bbR$, there exists an $h_0>0$ such that 
	\begin{equation} 
		\lVert   u \rVert_{H_{\calctwo}^{m,\mathsf{s},\ell;0,0}}  \lesssim \lVert  Pu \rVert_{H_{\calctwo}^{m-1,\mathsf{s}+1,\ell-1;0,0}}  +\lVert E u \rVert_{H_{\calctwo}^{m,\mathsf{s},\ell;0,0}} 
	\end{equation} 
	holds for all $u\in \calS'$ and $h\in (0,h_0)$, in the usual strong sense. 
	\label{prop:Cauchy_main_estimate}
\end{proposition}

\section{\texorpdfstring{Why not just use $\Psi_{\calczero}$?}{Why not just use the natural calculus?}}
\label{sec:whynot}

It may not be clear why $\Psi_{\calczero}$, by itself, is not sufficient to analyze the non-relativistic limit -- why, to get any sort of complete result, it is necessary to work with the second- microlocalized calculi $\Psi_{\calc},\Psi_{\calctwo}$.
The answer has to do with the dynamics of the geodesic flow in the phase space ${}^{\calczero}\overline{T}^* \bbM$. In this appendix, we look at that flow for the free Klein--Gordon operator 
\begin{equation}
	P = h^4 \partial_t^2 +h^2 \triangle +1 \in \operatorname{Diff}^{2,0,0}_{\calczero}, 
\end{equation}
where recall that $h=1/c$.
(This differs from what we call $P,P_0$ elsewhere by an overall factor of $c^2$.) 
It is the dynamics within the $\natural$-characteristic set, which is given by 
\begin{equation}
	\Sigma_\natural = \{\tau_{\natural}^2 = \xi_{\natural}^2+1\},
\end{equation}
that governs the microlocal analysis in the $\calczero$-calculus.

The symbol is $p = -h^4 \tau^2 + h^2 \xi^2 +1$, whose Hamiltonian vector field is (up to an overall sign convention) $H_p = 2h^4 \tau \partial_t  - 2 h^2 \xi\cdot \partial_x$. So, in terms of the natural frequency coordinates $\xi_{\natural} = h \xi$ and $\tau_{\natural} = h^2 \tau$, 
\begin{equation}
	{}^{\calczero}\mathsf{H}_p\propto  h^{-1} H_p = 2 h \tau_{\natural}\partial_t - 2 \xi_{\natural}\cdot \partial_x,
	\label{eq:misc_148}
\end{equation}
where the partial derivatives are taken with respect to the coordinate system $h,t,x,\tau_{\calczero},\xi_{\calczero}$, and where $\propto$ means proportionality up to a factor of the boundary-defining-functions of $\mathrm{df}$ and $\mathrm{bf}$. So, at $h=0$, 
\begin{equation}
	{}^{\calczero} \mathsf{H}_p \propto \xi_{\calczero}\cdot \partial_x . \label{eq:misc_149}
\end{equation} 
The key point is that the extra factor of $h$ on the $\partial_t$ in \cref{eq:misc_148} leads to the temporal component of the flow degenerating at $h=0$. The consequence is that ${}^{\calczero}\mathsf{H}_p$ vanishes on the closure of 
\begin{equation}
	\{\xi_{\calczero}=0 \}\subset T^* \bbR^{1,d}\times \{h=0\}
\end{equation}
in ${}^{\calczero} \overline{T}^* \bbM$. This includes a portion of the characteristic set $\Sigma_\natural$
that is not only over $\partial \bbM$ (where the rest of the vanishing set of ${}^{\calczero}\mathsf{H}_p$ is); it contains points over the interior $\bbR^{1,d}=\bbM^\circ$ as well.

This vanishing of ${}^{\calczero}\mathsf{H}_p$ means that $\calczero$-propagation estimates break down at $\{\xi_{\natural}=0 \}$. Nevertheless, depending on the structure of the flow nearby, one might hold out hope that it is possible to prove some sort of estimate. Microlocal estimates at hyperbolically trapped sets might come to mind. (If the issue were just over $\partial \bbM$, then one could hope to get away with a radial point estimate.) Unfortunately, as the formula \cref{eq:misc_149} makes clear, the vanishing of ${}^{\calczero} \mathsf{H}_p$ at $\{\xi_{\calczero}=0 \}$ is not of an amenable sort.

Note that conjugating $P$ by $\exp(\pm i t/h^2)$, the resulting operator $P_\pm$ has the same Hamiltonian flow as $P$ at $\{h=0\}$, but everything is shifted by one natural unit:
\begin{equation}
	\Sigma_{\natural}(P_\pm) = \{(\tau_{\calczero} \pm 1)^2 = \xi_{\calczero}^2+1 \}.
\end{equation}
Now, the problematic set $\{\xi_{\calczero}=0\}$ intersects $\Sigma(P_\pm)$ at the zero section of ${}^{\calczero} T^* \bbM$. This is the locus in ${}^{\calczero}T^* \bbM$ that we blow up to get ${}^{\calc} T^* \bbM$. As mentioned elsewhere in this paper, this blowup turns out to be exactly what is needed to remove from the Hamiltonian flow the degeneracy that we have just seen; see \S\ref{subsec:H}.

\section{Some facts about parabolic compactification}
\label{sec:parabolic}

The (compactified) parabolic phase space is defined by
\begin{equation} 
	{}^{\mathrm{par}}\overline{T}^* \bbM = \bbM \times \overline{(\bbR^{1,d}_{\tau,\xi  })}_{\mathrm{par}},
\end{equation} 
where $\smash{\overline{(\bbR^{1,d}_{\tau,\xi  })}_{\mathrm{par}}}$
is the parabolic compactification of $\bbR^{1,d}$. 
So, 
\begin{equation} 
	\overline{(\bbR^{1,d}_{\tau,\xi  })}_{\mathrm{par}} \cong \underbrace{\overline{\bbR_{\tau,\xi}}}_{\text{radial compactification}}
\end{equation} 
as mwcs --- they are both just $(1+d)$-dimensional balls --- but the canonical identification of their interiors does not extend smoothly, or even continuously, to the boundary. Indeed, the boundary of the mwc on the left can be identified with the (linear) rays through the origin, while the boundary of the mwc on the right, with the parabolic compactification, can be identified with the ``parabolic'' rays   
\begin{equation} 
	\{ (\tau_0 s^2, v s) :s\in \bbR^+  \} ,\quad \tau_0\in \bbR,v\in \bbR^d
\end{equation} 
(where $\tau_0,v$ are not both zero).
An explicit atlas is:
\begin{equation} 
	\overline{(\bbR^{1,d}_{\tau,\xi  })}_{\mathrm{par}} =  \bbR^{1,d}_{\tau,\xi }  \cup \Big( \bigcup_\pm  (\overline{(\bbR^{1,d}_{\tau,\xi  })}_{\mathrm{par}} \cap \{\pm \tau>0\}) \Big) \cup \Big( \bigcup_{\pm, k\in \{1,\ldots,d\}} (\overline{(\bbR^{1,d}_{\tau,\xi  })}_{\mathrm{par}} \cap \{\pm \xi_k>0\}) \Big)\;\,
\end{equation} 
where, in the sense of equality of compactifications,  
\begin{align}
	&\overline{(\bbR^{1,d}_{\tau,\xi  })}_{\mathrm{par}} \cap \{\pm\tau>0\} = [0,\infty)_{1/(\pm \tau)^{1/2}} \times \bbR^d_{\xi/(\pm\tau)^{1/2}},\\
	&\overline{(\bbR^{1,d}_{\tau,\xi  })}_{\mathrm{par}} \cap \{\pm \xi_k>0\} =  [0,\infty)_{\pm 1/\xi_k} \times \bbR_{\tau/\xi_k^2}\times  \bbR^{d-1}_{\hat{\xi}_k}, 
\end{align}
where $\hat{\xi}_k$ is the $(d-1)$-tuple $\{\xi_j/\xi_k\}_{j\neq k}$. 

We can observe that $\ang{z}^{-1}$ is a boundary defining function for spacetime infinity, and $\aang{\zeta}^{-1}$ is a boundary defining function for fiber-infinity, where 
\begin{equation} 
	\aang{\zeta} = (1 + \tau^2 + |\xi|^4)^{1/4}, \quad \zeta = (\tau, \xi).
\end{equation} 

\begin{figure}
	\begin{tikzpicture}[scale=1]
		\draw[fill=gray!10] (0,0) circle (2);
		\draw[dotted] (0,-2) -- (0, 2);
		\draw[dotted] (-2,0) -- (2,0);
		\draw[->] (0,0) -- (0,1);
		\draw[darkred] (-2,0) to[out=70, in=180] (-1.6,.2) to[out=0,in=180] (0,0) to[out=0, in=180] (1.6,.2) to[out=0,in=110] (2,0);
		\draw[darkred] (-2,0) to[out=75, in=180] (-1.7,.5) to[out=0,in=180]  (-.1,0) -- (0,0) -- (.1,0) to[out=0, in=180] (1.7,.5) to[out=0,in=105] (2,0);
		\draw[darkred] (-2,0) to[out=82, in=180] (-1.64,.75) to[out=0,in=160] (-.57,.1)  -- (-.2,0) -- (0,0) -- (.2,0) -- (.57,.1) to[out=20, in=180] (1.64,.75) to[out=0,in=98] (2,0);
		\draw[darkred] (-2,0) to[out=90, in=240] (-1.71,1) to[out=60, in=120] (-1.5,1) to[out=-55, in=180] (0,0) to[out=0, in=235] (1.5,1) to[out=60, in=120] (1.71,1) to[in=90, out=-60] (2,0);
		\draw[orange] (-1.6,1.2) to[out=-60, in=180] (0,0) to[out=0, in=240] (1.6,1.2);
		\draw[lightgray] (-2,0) to[out=-40, in=180] (0,-.75) to[out=0,in=220] (2,0);
		\draw[lightgray] (-2,0) to[out=-60, in=180] (0,-1.2) to[out=0,in=240] (2,0);
		\draw[lightgray] (-2,0) to[out=-75, in=180] (0,-1.6) to[out=0,in=255] (2,0);
		\draw[lightgray] (-2,0) to[out=-90, in=180] (0,-1.9) to[out=0,in=270] (2,0);
		\node[above] (nu) at (0,.9) {$\pm\tau$};
		\draw[->] (0,0) -- (1,0);
		\node[below] (aleph) at (1,0) {$\xi$};
		\node[darkred] at (-.5,-.25) {$\Gamma_h$};
		\node[gray] at (2,-1.5) {$\Sigma_{\mathrm{bad}}$};
		\node at (-2,2) {(a)};
		\draw (0,0) circle (2);
	\end{tikzpicture}
	\qquad 
	\begin{tikzpicture}[scale=1]
		\draw[fill=gray!10] (0,0) circle (2);
		\draw[dotted] (0,-2) -- (0, 2);
		\draw[dotted] (-2,0) -- (2,0);
		\draw[orange] (0,0) to[out=0, in=0] (0,2) to[out=180, in=180] (0,0);
		\draw[darkred] (-1,1.725) to[out=240, in=180]  (0,0) to[out=0, in=-60] (1,1.725);
		\draw[darkred] (-1.43,1.4) to[out=260, in=180]  (0,0) to[out=0, in=-80] (1.43,1.4);
		\draw[darkred] (-1.725,1) to[out=280, in=180]  (0,0) to[out=0, in=-100] (1.725,1);
		\draw[lightgray] (-1.725,-1) to[out=10, in=180]  (0,-.7) to[out=0, in=170] (1.725,-1);
		\draw[lightgray] (-1.4,-1.43) to[out=5, in=180]  (0,-1.15) to[out=0, in=175] (1.4,-1.43);
		\draw[lightgray] (-1,-1.725) to[out=0, in=180]  (0,-1.7) to[out=0, in=180] (1,-1.725);
		\draw[lightgray] (-.7,-1.88) to[out=-10, in=190] (.7,-1.88);
		\node at (-2,2) {(b)};
		\draw (0,0) circle (2);
	\end{tikzpicture} 
	\caption{(a) The family of curves $\Gamma_h=\{h^2 \tau^2 + 2\tau = \xi^2\}$ in $(\overline{\bbR^{2}_{\tau,\xi}})_{\mathrm{par}}$, depicted for several values of $h$. The limiting parabola $\Gamma_0$ is in orange. The bunching up of $\Gamma_h$ near fiber infinity as $h\to 0^+$ corresponds to the portion of the $\calc$-characteristic set located at $\natural\mathrm{f}.$ (b) The same family of curves, but in $\overline{\bbR^2_{\tau,\xi}}$, showing the importance of the choice of compactification.}
\end{figure}
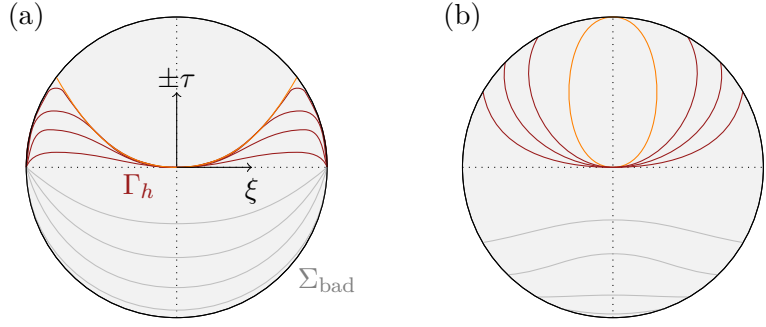

\begin{lemma}
	$\partial_\tau \in \operatorname{Diff}_{\mathrm{b}}^{1,-2}(\overline{(\bbR^{1,d}_{\tau,\xi  })}_{\mathrm{par}})$ and $\partial_{\xi_j} \in \operatorname{Diff}_{\mathrm{b}}^{1,-1}(\overline{(\bbR^{1,d}_{\tau,\xi  })}_{\mathrm{par}})$.
	\label{lem:misc_scc}
\end{lemma}
Here, the first superscript in $\operatorname{Diff}^{m,s}_{\mathrm{b}}$ is the differential order, and $s$ is the number of orders of decay.

\begin{proof}
	Since the desired result holds in the interior (where it only expresses that $\partial_\tau,\partial_{\xi_j}$ are first-order in the differential sense), it suffices to examine the situation in the coordinate charts on the right-hand sides of \cref{eq:misc_006}, \cref{eq:misc_007}. Again, we only check the $+$ case of these coordinate charts, the $-$ case being analogous.
	
	Beginning with the charts in \cref{eq:misc_006}, we have, in terms of the coordinates $\rho = 1/\tau^{1/2}$ and $\hat{\xi} = \xi/\tau^{1/2}$,  
	\begin{align}
		\frac{\partial }{\partial \tau} &= \frac{\partial \rho}{\partial \tau} \frac{\partial}{\partial \rho} + \sum_{j=1}^d \frac{\partial \hat{\xi}_j}{\partial \tau}  \frac{\partial}{\partial \hat{\xi}_j} = - \frac{\rho^3}{2}\frac{\partial}{\partial \rho} -  \frac{\rho^2}{2}\sum_{j=1}^d  \hat{\xi}_j\frac{\partial}{\partial \hat{\xi}_j}, \\ 
		\frac{\partial}{\partial \xi_j} &= \frac{\partial \hat{\xi}_j}{\partial \xi_j}\frac{\partial}{\partial \hat{\xi}_j} = \rho \frac{\partial}{\partial \hat{\xi}_j}.
	\end{align}
	So, in this coordinate chart, $\partial_\tau \in \operatorname{Diff}_{\mathrm{b}}^{1,-2}(\overline{(\bbR^{1,d}_{\tau,\xi  })}_{\mathrm{par}})$ and $\partial_{\xi_j} \in \operatorname{Diff}_{\mathrm{b}}^{1,-1}(\overline{(\bbR^{1,d}_{\tau,\xi  })}_{\mathrm{par}})$.
	
	Moving on to the coordinates in \cref{eq:misc_007}: $\rho = 1/\xi_k$, $\hat{\tau}=\tau/\xi_k^2$, and $\hat{\xi}_j = \xi_j/\xi_k$ for $j\neq k$, we have 
	\begin{equation}
		\frac{\partial}{\partial \xi_k} = \frac{\partial \rho}{\partial \xi_k}\frac{\partial}{\partial \rho} + \frac{\partial \hat{\tau}}{\partial \xi_k} \frac{\partial}{\partial \hat{\tau}} + \sum_{j\neq k} \frac{\partial \hat{\xi}_j}{\partial \xi_k} \frac{\partial}{\partial \hat{\xi}_j} = - \rho^2 \frac{\partial}{\partial \rho} - 2\hat{\tau} \rho \frac{\partial}{\partial \hat{\tau}} - \sum_{j\neq k} \hat{\xi}_j \rho \frac{\partial}{\partial \hat{\xi}_j}, 
	\end{equation}
	\begin{equation}
		\frac{\partial}{\partial \xi_j}  = \frac{\partial \hat{\xi}_j}{\partial \xi_j} \frac{\partial}{\partial \hat{\xi}_j}  = \rho \frac{\partial}{\partial \hat{\xi}_j}, \qquad 
		\frac{\partial}{\partial \tau} = \frac{\partial \hat{\tau}}{\partial \tau}\frac{\partial}{\partial \hat{\tau}} = \rho^2 \frac{\partial}{\partial \hat{\tau}},
	\end{equation}
	for $j\neq k$. 
	So again, we see that $\partial_\tau \in \operatorname{Diff}_{\mathrm{b}}^{1,-2}(\overline{(\bbR^{1,d}_{\tau,\xi  })}_{\mathrm{par}})$ and $\partial_{\xi_j} \in \operatorname{Diff}_{\mathrm{b}}^{1,-1}(\overline{(\bbR^{1,d}_{\tau,\xi  })}_{\mathrm{par}})$ locally.
\end{proof}

\section{A global proof of composition}
\label{sec:reduction}

	Because the elements of $\Psi_{\calc}$ are families of pseudodifferential operators, they admit the usual algebraic operations. For example, if $A,B\in \Psi_{\calc}$, then the composition $AB=A\circ B$ is a well-defined family of pseudodifferential operators as well, as follows from the reduction formula \cref{eq:moyal_explicit_natural}, \cite[Prop.\ 5.1 in \S5.3.3]{VasyGrenoble}. But how do we know that $AB\in \Psi_{\calc}$?
	One argument is in the body of this paper. Here we provide another, directly using the reduction formula and not relying on \cite{Parabolicsc}.

	The basic idea behind this second proof is that, in the formula for the Moyal product, each term is manifestly well-behaved on the $\calc$-phase space, because it results from applying well-behaved vector fields to symbols on that phase space. Each subsequent term is better than the previous at all of the boundary hypersurfaces except for $\mathrm{pf}$. So, $a\star b$ should lie in $S_{\calc}$ modulo a term in $C^\infty([0,1)_h;\calS(T^* \bbR^{1,d}))$. But this is also in $S_{\calc}$ (trivially). So, $a\star b \in S_{\calc}$. This sketch differs from the argument below only in terms of the technical details.

	To handle the technical details, it is useful to work with slightly larger classes of symbols than those introduced already. (This is one way of proving properties of $\Psi_{\mathrm{par}}$.)
	For $m,\ell,\ell'\in \bbR$, $\rho\in [0,1)$, and $\delta,\delta'\in [0,1/2)$ such that $\rho+2\delta<1$, let 
	\begin{equation} 
		S_{\rho,\delta,\delta'}^{m,\ell,\ell'} \subset C^\infty(\bbR^D_z\times \bbR^D_\zeta\times \bbR^D_{z'}) ,\qquad D=d+1
	\end{equation} 
	denote the set of all smooth functions $a(z,\zeta,z') $ such that 
	\begin{multline}
		\sup_{z,\zeta,z' \in \bbR^d} \langle z \rangle^{-\ell + |\alpha| } \langle z' \rangle^{-\ell' + |\beta| } (\langle z \rangle + \langle z' \rangle)^{-\delta' (|\alpha|+|\beta|+|\gamma|) } 
		\langle \zeta \rangle^{-m+(1-\rho)|\gamma| - \delta(|\alpha|+|\beta|+|\gamma|)  }  \\ \times| \partial_z^\alpha \partial_{z'}^{\beta}  
		\partial_\zeta^\gamma a(z,\zeta,z') | < \infty 
	\end{multline}
	for all  multi-indices $\alpha,\beta,\gamma\in \bbN^D$. The seminorms above naturally endow each of these symbol classes with the structure of a Fr\'echet space. 
	We will see (\Cref{prop:Moyal_helper}) that $\rho=1/2$ allows us to describe parabolic symbols.
	
	These symbol classes are closely related to the standard $S^{m,\ell,\ell'}_{\delta,\delta'}$ symbols used in \cite{VasyGrenoble}. Trivially, 
	\begin{equation} 
		S^{m,\ell,\ell'}_{\rho,\delta,\delta'} \subseteq S^{m,\ell,\ell'}_{\rho+\delta,\delta'}.
		\label{eq:symbol_inc_tech}
	\end{equation} 
	The key thing which membership in the left-hand side affords that membership in the right-hand side does not is that differentiating in the spatial variables yields only a $\langle \zeta \rangle^{\delta}$ loss in frequency decay, not $\langle \zeta \rangle^{\rho+\delta}$. The latter type of loss only occurs when differentiating in the frequency variables.
	
	The theory of the $S_{\rho,\delta,\delta'}$ symbol classes can be developed along the lines of \cite{VasyGrenoble}.  Vasy assumes, from the beginning, that the $\delta$ in $\smash{S^{m,\ell,\ell'}_{\delta,\delta'}}$ satisfies $\delta<1/2$, whereas our $\rho+\delta$ can be $1/2$ or larger.  Fortunately, the assumption that $\delta<1/2$ is not required until the proof of \cite[Prop.\ 5.1 in \S5.3.3]{VasyGrenoble}, and the proof of that proposition goes through without modification for the more general symbol classes here. The reason is explained in the next paragraph.

	The conclusion is that, for any symbol $a\in S_{\rho,\delta,\delta'}^{m,\ell,\ell'}$, we have 
	\begin{equation}
		\operatorname{Op}(a) = \operatorname{Op}(a_{\mathrm{L}}) = \operatorname{Op}(a_{\mathrm{R}}) 
	\end{equation}
	for unique symbols $a_{\mathrm{L}}(z,\zeta) \in \smash{S_{\rho,\delta,\delta'}^{m,\ell+\ell',0}}$ and $a_{\mathrm{R}}(\zeta,z') \in \smash{S_{\rho,\delta,\delta'}^{m,0,\ell+\ell'}}$ -- where the key point is that $a_{\mathrm{L}}$ is independent of $z'$ and $a_{\mathrm{R}}$ is independent of $z$ -- and moreover we have the reduction formulae
	\begin{align}
		a_{\mathrm{L}}(z,\zeta) &\sim \sum_{\alpha \in \bbN^D}\frac{i^{|\alpha|}}{\alpha!} D_{z'}^\alpha D^\alpha_\zeta a(z,\zeta,z')|_{z'=z} \\ a_{\mathrm{R}}(\zeta,z') &\sim \sum_{\alpha \in \bbN^D}\frac{(-i)^{|\alpha|}}{\alpha!} D_{z}^\alpha D^\alpha_\zeta a(z,\zeta,z')|_{z=z'}.
	\end{align}
	Here, $\sim$ has the usual meaning; for example, for each $N\in \bbN$, we can write 
	\begin{equation}
		a_{\mathrm{L}} - \sum_{|\alpha|\leq N} \frac{i^{|\alpha|}}{\alpha!} D_{z'}^\alpha D^\alpha_\zeta a(z,\zeta,z')|_{z'=z}  \in S_{\rho,\delta,\delta'}^{m-N + (\rho+2\delta)N,\ell+s-N+2\delta' N,0},
		\label{eq:hb341z1}
	\end{equation}
	with this depending smoothly on $a\in \smash{S_{\rho,\delta,\delta'}^{m,\ell,\ell'}}$. As long as $\rho+2\delta<1$, as we are assuming, the right-hand side is lower order than $a_{\mathrm{L}}$ in every sense. This is the only place that Vasy uses the inequality $\delta<1/2$, so our assumption $\rho+2\delta<1$ suffices to prove the reduction formulae for elements of the symbol classes here. The point is that in the Moyal formula, each subsequent term is differentiated twice, once in the spatial variables and once in the frequency variables. Consequently, if $\rho$ describes the loss in frequency decay when differentiating in \emph{both} frequency and space, then we need $2\rho<1$ for subsequent terms to be better than previous terms. However, if $\delta$ describes the loss only when differentiating in space and $\delta+\rho$ describes the loss when differentiating in frequency, then we only need $2\delta+\rho<1$.

	Symbols depending only on $z,\zeta$ are called \emph{left} symbols, while symbols depending only on $\zeta,z'$ are called \emph{right} symbols. The reduction formula shows that every pseudodifferential operator in the relevant class is the quantization of a left symbol, and moreover the uniqueness clause of \cite[Prop.\ 5.1]{VasyGrenoble} means that $\operatorname{Op}$ is injective when restricted to left symbols. The same statements apply to right symbols.

	One standard application of the reduction formulae is understanding compositions of pseudodifferential operators: given left symbols $a,b$,  we have 
	\begin{equation} 
		\operatorname{Op}(a)\operatorname{Op}(b) = \operatorname{Op}(a)\operatorname{Op}(b_{\mathrm{R}}) = \operatorname{Op}(a b_{\mathrm{R}} ) = \operatorname{Op}(a \star b),
	\end{equation} 
	where $a \star b$ is the \emph{Moyal(--Weyl--Groenewold product)} $a\star b = (a b_{\mathrm{R}})_{\mathrm{L}}$. The explicit formula that this gives for the Moyal product is
	\begin{equation}
		a \star b(z,\zeta) \sim  \sum_{\alpha,\beta\in \bbN^D} \frac{i^{|\alpha|} (-i)^{|\beta|} }{ \alpha! \beta!} D_\zeta^\alpha (a(z,\zeta) D_{z}^{\alpha+\beta} D_\zeta^\beta b(z,\zeta)   ) .
		\label{eq:moyal_explicit}
	\end{equation}

	For one-parameter families $a=\{a(h)\}_{h>0}$, $b=\{b(h)\}_{h>0}$ of left symbols, we define $a\star b = \{(a\star b)(h)=a(h)\star b(h)\}_{h>0}$ pointwise, for each individual value of the parameter $h$. The fact that the error term in \cref{eq:hb341z1} depends smoothly on $a,b$ means that 
	\begin{multline}
		a \in h^{-q} C^\infty([0,1)_h; S_{1/2,0,0}^{m,s,0} ),b \in h^{-q'} C^\infty([0,1)_h; S_{1/2,0,0}^{m',s',0} ) \\ \Rightarrow 
		a \star b(z,\zeta) - \sum_{|\alpha| + |\beta|\leq N} \frac{i^{|\alpha|} (-i)^{|\beta|} }{ \alpha! \beta!} D_\zeta^\alpha (a(z,\zeta) D_{z}^{\alpha+\beta} D_\zeta^\beta b(z,\zeta)   ) \in  h^{-q-q'} C^\infty([0,1)_h; S_{1/2,0,0}^{m-N/2,\ell+s-N,0}) 
		\label{eq:moyal_explicit_error}
	\end{multline}
	for each $N\in \bbN$. 
	
	In order to analyze the terms appearing in the Moyal expansion, we note:
	\begin{lemma}$
		\langle z \rangle \partial_{z_j} ,  \rho_{\natural\mathrm{f} }^{-2} \rho_{\mathrm{df}}^{-1} \partial_\tau, \rho_{\natural\mathrm{f}}^{-1} \rho_{\mathrm{df}}^{-1} \partial_{\xi_k} ,h \partial_h \in \calV_{\mathrm{b}}({}^{\calc}\overline{T}^* \bbM )$ for each $j\in \{0,\dots,d\}$ and $k \in \{1,\ldots,d\}$, i.e.\ these vector fields are tangent to all of the boundary hypersurfaces of ${}^{\calc}\overline{T}^* \bbM$.
		\label{lem:reduction_vector_fields}
	\end{lemma}
	\begin{proof}
		Easily checked in local coordinates.
	\end{proof}

	The reason for introducing the symbol classes $S_{\rho,\delta,\delta'}^{m,\ell,\ell'}$ is the following:
	\begin{lemma}
		For any $s,q,m\in \bbR$, $\ell\leq q$, and $K\in \bbN$,  
		we have 
		\begin{equation} 
			S^{m,s,\ell,q}_{\calc} \subseteq h^{-q} C^K ( [0,1)_h; S_{1/2,0,0}^{M+1+K,s,0})
		\end{equation} 
		for $M= \max\{m,(\ell-q)/2\}$.
		\label{prop:Moyal_helper}
	\end{lemma}
	So, elements of $S_{\calc}^{m,s,\ell,0}$ are just $C^K$ one-parameter families of symbols lying in some big but \emph{fixed} (meaning $h$-independent) symbol class. 
	
	It does not matter below what the particular value of $M$ is, just that it goes to $-\infty$ as $m,\ell\to -\infty$ together.
	
	\begin{proof}
		Since $S^{m,s,\ell,q}_{\calc} = h^{-q} S^{m,s,\ell-q,0}_{\calc}$, it suffices to prove the $q=0$ case. Then, $\ell\leq 0$. 
		
		In order to prove the lemma, let us estimate, for $a(h,z,\tau,\xi) \in \smash{S^{m,s,\ell,q}_{\calc}}$, the derivatives $\smash{\partial_h^j \partial_z^\alpha \partial_\tau^k \partial_\xi^\beta a}$ for $j,k\in \bbN$, $\alpha\in \bbN^D$, and $\beta\in \bbN^d$. It follows from \Cref{lem:reduction_vector_fields} that 
		\begin{equation}
			(h\partial_h)^j \partial_z^\alpha \partial_\tau^k \partial_\xi^\beta a \in \rho_{\mathrm{df}}^{-m+k+|\beta|} \rho_{\mathrm{bf}}^{-s+|\alpha|} \rho_{\natural\mathrm{f}}^{-\ell+2k+|\beta|} \rho_{\mathrm{pf}}^{j} L^\infty({}^{\calc}\overline{T}^* \bbM ).
			\label{eq:misc_012}
		\end{equation}
		The fact that we get one factor of $ \rho_{\mathrm{pf}}$ on the right-hand side for each $\partial_h$ on the left-hand side is because $a$ is assumed to be smooth at $\mathrm{pf}$ (in the appropriate sense which is uniform up to the corners).

		Next, note that 
		\begin{equation} 
			\rho_{\mathrm{df}} \rho_{\natural\mathrm{f}}^2 \lesssim \langle \zeta \rangle^{-1} \lesssim \rho_{\mathrm{df}} \rho_{\natural\mathrm{f}}.
			\label{eq:misc_j12}
		\end{equation} 
		We have recorded this as \Cref{lem:Moyal_helper2} below, and the proof is written there.
		Similarly, it is possible to show that $\rho_{\natural\mathrm{f}}\rho_{\mathrm{pf}} \in h L^\infty({}^{\calc}\overline{T}^* \bbM )$. Indeed, away from $\mathrm{pf}$, this follows from the fact that $\rho_{\natural\mathrm{f}}\cong h$ here. Away from $\mathrm{df}$, this follows because we can arrange $\rho_{\natural\mathrm{f}} \rho_{\mathrm{pf}} \cong h$.

		Combining the previous paragraph with $\rho_{\mathrm{bf}} \cong \langle z \rangle^{-1}$ and \cref{eq:misc_012} gives 
		\begin{equation}
			\partial_h^j \partial_z^\alpha \partial_\tau^k \partial_\xi^\beta a \in \langle \zeta \rangle^{-k-|\beta|/2} \langle z \rangle^{s-|\alpha|} \rho_{\mathrm{df}}^{-m}  \rho_{\natural\mathrm{f}}^{-\ell-j}  L^\infty.
			\label{eq:misc_655}
		\end{equation}
		
		Now we split into two cases:
		\begin{itemize}
			\item First suppose that $-M-j \geq 0$. Depending on $M$, this may cover the first few $j$. We can write 
			\begin{equation} 
				\rho_{\mathrm{df}}^{-m}  \rho_{\natural\mathrm{f}}^{-\ell-j}  \in(\rho_{\mathrm{df}}\rho_{\natural\mathrm{f}}^2 )^{-M-j} L^\infty\subseteq \langle \zeta \rangle^{M+j} L^\infty,
			\end{equation}   
			using the left inequality in \cref{eq:misc_j12}.
			\item If $-M-j \leq 0$, we instead note $\rho_{\mathrm{df}}^{-m}  \rho_{\natural\mathrm{f}}^{-\ell-j}  \in(\rho_{\mathrm{df}}\rho_{\natural\mathrm{f}} )^{-M-j} L^\infty$. (Here is where we use $\ell\leq 0$.) Then, using the right inequality in \cref{eq:misc_j12} instead gives 
			\begin{equation}
				(\rho_{\mathrm{df}}\rho_{\natural\mathrm{f}} )^{-M-j} L^\infty \subseteq \langle \zeta \rangle^{M+j} L^\infty, 
			\end{equation}
			as before.\footnote{This case actually does not matter for the proof of \Cref{prop:Moyal_our}, since in that proof we can restrict attention to a single $K$ but take $m,\ell$ arbitrarily negative relative to that.}
		\end{itemize}
		So, no matter what, $\rho_{\mathrm{df}}^{-m}  \rho_{\natural\mathrm{f}}^{-\ell-j}\subseteq \langle \zeta \rangle^{M+j} L^\infty$.
		
		So, \cref{eq:misc_655} implies 
		\begin{equation}
			\partial_h^j \partial_z^\alpha \partial_\tau^k \partial_\xi^\beta a \in \langle \zeta \rangle^{M+j-k-|\beta|/2} \langle z \rangle^{s-|\alpha|} L^\infty = \langle \zeta \rangle^{M+j-k-|\beta|/2} \langle z \rangle^{s-|\alpha|} L^\infty.
		\end{equation} 
		Since this holds for all $\alpha,k,\beta$, the conclusion is that $\partial^j_h a \in S_{1/2,0,0}^{M+j,s,0}$; so,  
		\begin{equation}
			a, \partial_h a,\dots,\partial_h^{K+1}a \in S_{1/2,0,0}^{M+1+K,s,0}
		\end{equation}
		for any $K\in \bbN$. This implies 
		$a \in C^K([0,1)_h ; S_{1/2,0,0}^{M+1+K,s,0} )$.
	\end{proof}
	\begin{lemma}
		\Cref{eq:misc_j12} holds.
		\label{lem:Moyal_helper2}
	\end{lemma}
	\begin{proof}
		 Indeed, away from $\mathrm{pf}$, we can take $\rho_{\mathrm{df}} = (1+ h^4 \tau^2+h^2 \xi^2 )^{-1/2}$ and $\rho_{\natural\mathrm{f}} = h$. Then, 
		\begin{equation}
			\rho_{\mathrm{df}} \rho_{\natural\mathrm{f}}^2 = \frac{h^2}{(1 + h^4 \tau^2+h^2 \xi^2)^{1/2}} = \frac{1}{(h^{-4}+\tau^2+h^{-2} \xi^2)^{1/2}} \leq \frac{1}{\langle \zeta \rangle} 
		\end{equation}
		for $h\leq 1$. 
		On the other hand, near $\mathrm{pf}$, we can take $\rho_{\mathrm{df}}=1$ and $\rho_{\natural\mathrm{f}} = (1+ \tau^2 + \xi^4)^{-1/4}$. Since $1+\xi^4 \geq \xi^2$, we have $1+\xi^4 \geq 2^{-1}(1+\xi^2)$ and therefore 
		\begin{equation} 
			\rho_{\natural\mathrm{f}}^2 \leq \frac{1}{( 2^{-1}(1+\xi^2)+\tau^2 )^{1/2}} \leq  \frac{2^{1/2}}{ \langle \zeta \rangle}.
		\end{equation} 
		So, $\rho_{\mathrm{df}}\rho_{\natural\mathrm{f}}^2\lesssim \langle \zeta \rangle^{-1}$ everywhere. 
		
		To prove the lower bound, $ \langle \zeta \rangle^{-1}\lesssim \rho_{\mathrm{df}} \rho_{\natural\mathrm{f}}$: away from pf, say $h^2\tau^2+\xi^2>1/h^2$, this follows from 
		\begin{align}
			\begin{split} 
			\rho_{\mathrm{df}} \rho_{\natural\mathrm{f}}\approx \frac{h}{(1 + h^4 \tau^2+h^2 \xi^2)^{1/2}} &=\frac{1}{ (h^{-2} + h^2\tau^2 +  \xi^2)^{1/2}} \\ 
			&\geq \frac{1}{\sqrt{2} (1+  h^2\tau^2+  \xi^2 )^{1/2}} \geq \frac{1}{\sqrt{2} \langle \zeta \rangle}.
			\end{split}
		\end{align}
		Away from df, we can instead take $\rho_{\mathrm{df}}=1$ and $\rho_{\natural\mathrm{f}} =(1+\tau^2+\xi^4)^{-1/4}$, so 
		\begin{equation}
			\rho_{\mathrm{df}}\rho_{\natural \mathrm{f}} = \frac{1}{(1+\tau^2+\xi^4)^{1/4}} \geq \frac{1}{(2(1+\tau^4)+\xi^4 )^{1/4}} \geq  \frac{1}{2^{1/4}(1+\tau^4+\xi^4)^{1/4}} \gtrsim \frac{1}{\langle \zeta \rangle}.
		\end{equation}
		So, $\rho_{\mathrm{df}}\rho_{\natural\mathrm{f}}\gtrsim \langle \zeta \rangle^{-1}$ everywhere.
	\end{proof}
	
	Finally:
	\begin{proposition}
		If $a\in S_{\calc}^{m,\mathsf{s},\ell,q}$ and $b\in S_{\calc}^{m',\mathsf{s}',\ell',q'}$, then $a\star b \in S_{\calc}^{m+m',\mathsf{s}+\mathsf{s}',\ell+\ell',q+q'}$, and \cref{eq:moyal_explicit} holds in the sense that, for each $N\in \bbN$, 
		\begin{equation}
			a\star b - \sum_{|\alpha| + |\beta| \leq N } \frac{i^{|\alpha|} (-i)^{|\beta|} }{ \alpha! \beta!} D_\zeta^\alpha  (a(z,\zeta) D_{z}^{\alpha+\beta} D_\zeta^\beta b(z,\zeta)   )  \in S_{\calc}^{m+m'-N-1,\mathsf{s}+\mathsf{s}'-N-1+\varepsilon, \ell+\ell'-N-1,q+q'}
			\label{eq:moyal_our_explicit}
		\end{equation}
		for every $\varepsilon>0$. If $\mathsf{s},\mathsf{s}'$ are constant, then we can take $\varepsilon=0$. 
		\label{prop:Moyal_our}
	\end{proposition}
	\begin{proof}
		Let $\varepsilon>0$.
		Via an asymptotic summation argument, we can find $c\in S_{\calc}^{m+m',\mathsf{s}+\mathsf{s}',\ell+\ell',q+q'}$ such that 
		\begin{equation}
			c \sim   \sum_{\alpha,\beta\in \bbN^D} \frac{i^{|\alpha|} (-i)^{|\beta|} }{ \alpha! \beta!} D_\zeta^\alpha  (a(z,\zeta) D_{z}^{\alpha+\beta} D_\zeta^\beta b(z,\zeta)), 
			\label{eq:y4vv}
		\end{equation}
		meaning that, for each $N\in \bbN$, 
		\begin{equation}
			c- \sum_{|\alpha| + |\beta| \leq N } \frac{i^{|\alpha|} (-i)^{|\beta|} }{ \alpha! \beta!} D_\zeta^\alpha  (a(z,\zeta) D_{z}^{\alpha+\beta} D_\zeta^\beta b(z,\zeta)   )  \in S_{\calc}^{m+m'-N-1,\mathsf{s}+\mathsf{s}'-N-1+\varepsilon, \ell+\ell'-N-1,q+q'}. \label{eq:moyal_our_explicit2}
		\end{equation}
		The key point here is that \Cref{lem:reduction_vector_fields} implies that each term in \cref{eq:y4vv} is suppressed by one order at each of the faces $\mathrm{df},\mathrm{bf},\natural\mathrm{f}$ compared to the previous terms. This is what allows the asymptotic summation.

		Since
		\begin{equation}
			S_{\calc}^{m+m'-N-1,\mathsf{s}+\mathsf{s}'-N-1+\varepsilon, \ell+\ell'-N-1,q+q'} \subseteq 	S_{\calc}^{m+m'-N-1,\sup_{h\leq 1/2} (\mathsf{s}+\mathsf{s}') -N-1+2\varepsilon, \ell+\ell'-N-1,q+q'} 
			\label{eq:o1124}
		\end{equation}
		in $h\in [0,1/2]$, 
		\Cref{prop:Moyal_helper} allows us to deduce that, for any $K\in \bbN$, $m_0,s_0\in \bbR$,
		\begin{equation}
			c- \sum_{|\alpha| + |\beta| \leq N } \frac{i^{|\alpha|} (-i)^{|\beta|} }{ \alpha! \beta!} D_\zeta^\alpha  (a(z,\zeta) D_{z}^{\alpha+\beta} D_\zeta^\beta b(z,\zeta)   )  \in h^{-q-q'} C^K( [0,1)_h; S_{1/2,0,0}^{m_0,s_0,0})
		\end{equation}
		if $N$ is sufficiently large relative to $K,m_0,s_0$.

		On the other hand, \Cref{prop:Moyal_helper} (together with an inclusion like \cref{eq:o1124} in the case where $\mathsf{s},\mathsf{s}'$ are variable) tells us that 
		\begin{equation} 
			a \in  h^{-q}C^K( [0,1)_h; S_{1/2,0,0}^{M,S,0}), \quad 
			b \in  h^{-q'}C^K( [0,1)_h; S_{1/2,0,0}^{M,S,0})
		\end{equation} 
		for some possibly very large $M=M(K),S\in \bbR$. \Cref{eq:moyal_explicit_error} therefore tells us that, if $N$ is sufficiently large, then
		\begin{equation}
			a\star b- \sum_{|\alpha| + |\beta| \leq N } \frac{i^{|\alpha|} (-i)^{|\beta|} }{ \alpha! \beta!} D_\zeta^\alpha  (a(z,\zeta) D_{z}^{\alpha+\beta} D_\zeta^\beta b(z,\zeta)   )  \in h^{-q-q'}C^K ([0,1)_h; S_{1/2,0,0}^{m_0,s_0,0}), 
		\end{equation}
		as well,
		so that $c - a\star b \in h^{-q-q'} C^\infty([0,1)_h; S_{1/2,0,0}^{m_0,s_0,0})$. Since $m_0,s_0\in \bbR$ were arbitrary, we conclude that 
		\begin{equation}
			c - a\star b \in \bigcap_{m_0,s_0\in \bbR} h^{-q-q'} C^\infty([0,1)_h; S_{1/2,0,0}^{m_0,s_0,0}) = h^{-q-q'} C^\infty([0,1)_h;  \calS(T^* \bbR^{1,d}) ).
		\end{equation}
		Evidently, $C^\infty([0,1)_h; \calS(T^* \bbR^{1,d}) ) \subseteq S_{\calc}^{-\infty,-\infty,-\infty,0}$.  
		So, $a \star b - c \in  S_{\calc}^{-\infty,-\infty,-\infty,q+q'}$, which implies 
		\begin{equation} 
			a\star b \in S_{\calc}^{m+m',\mathsf{s}+\mathsf{s}',\ell+\ell',q+q'},
		\end{equation} 
		as claimed.
		
		\Cref{eq:moyal_our_explicit2} says that \cref{eq:moyal_our_explicit} holds with $c$ in place of $a\star b$.
		Since $c$ differs from $a\star b$ by a residual symbol, \cref{eq:moyal_our_explicit} holds as well. 
		
		If $\mathsf{s},\mathsf{s}'$ are constant, then we can choose $c$ such that \cref{eq:moyal_our_explicit2} holds with $\varepsilon=0$, so the same applies to  \cref{eq:moyal_our_explicit}.
	\end{proof}

\printbibliography

\end{document}